\documentclass[letterpaper, 10 pt]{article}
\usepackage{style}
\usepackage{cancel}
\usepackage{blindtext}
\usepackage{pdflscape}

\usepackage{afterpage}

\usepackage{amsmath,amssymb}
\usepackage{hyperref}
\usepackage{amsthm}

\usepackage{arydshln}
\usepackage{enumitem}
\usepackage{xcolor}
\usepackage{threeparttable}
\usepackage{tikz}
\usetikzlibrary{shapes.arrows}
\usetikzlibrary{shapes}
\usetikzlibrary{decorations}
\usetikzlibrary{backgrounds}
\usetikzlibrary{calc}
\usetikzlibrary{arrows}
\usetikzlibrary{automata,positioning}
\usetikzlibrary{decorations.pathreplacing,calligraphy}
\usepackage{booktabs,multirow,makecell}
\usepackage[table]{xcolor}
\usepackage{enumitem}
\usepackage{array,xcolor,colortbl} 
\usepackage{subcaption} 
\usepackage{tabularx}
\usepackage{mathtools}
\usepackage{makecell} 
\usepackage{adjustbox}
\usepackage{rotating} 
\usepackage{afterpage}
\usepackage{cases}
\usepackage{mathtools,empheq}

\usepackage{pgfplots}
\pgfplotsset{compat=1.15}
\usepackage{adjustbox}
\usepackage{bm}
\usepackage{gincltex}
\usepackage{caption}
\usepackage{subcaption}
\usepgfplotslibrary{fillbetween}
\usepgfplotslibrary{groupplots}
\usetikzlibrary{plotmarks}
\usetikzlibrary{patterns}
\pgfplotsset{
tick label style={font=\footnotesize},
label style={font=\footnotesize},
legend style={font=\footnotesize},
}

\allowdisplaybreaks

\usepackage{multirow}
\RequirePackage{tikz}
\usepackage{pifont}
\usepackage{xcolor,colortbl}

\title{\LARGE \bf Tail Bounds for Queues with Abandonment: Constant, Moderate, Large Deviations, and
Efficient Concentration}

\author{%
{\normalsize Zedong Wang}\thanks{H. Milton Stewart School of Industrial \& Systems Engineering, Georgia Institute of Technology, Atlanta, GA 30332, USA. \texttt{\{zwang3524\}@gatech.edu}}%
    \and
    {\normalsize Siva Theja Maguluri} \thanks{H. Milton Stewart School of Industrial \& Systems Engineering, Georgia Institute of Technology, Atlanta, GA 30332, USA. \texttt{\{siva.theja\}@gatech.edu}}%
}
\date{}
\begin{document}

\maketitle
\begin{abstract}
    We study a heavily overloaded single-server queue with abandonment and derive bounds on stationary tail probabilities of the queue length. As the abandonment rate $\gamma \downarrow 0$, the centered-scaled steady-state queue length $\tilde{q}$ is known to converge in distribution to a Gaussian. However, such asymptotic limits do not quantify the pre-limit tail $\mathbb{P}(\tilde{q}>a)$ for fixed $\gamma>0$. Our goal is to obtain pre-limit bounds that are \emph{efficient} across deviation scales. For constant deviations, efficiency means Gaussian-type decay in $a$ together with a pre-limit error that vanishes as $\gamma\downarrow 0$, yielding the correct Gaussian tail in the limit. We establish such an efficient bound that is best-of-both-worlds. For larger deviations when $a$ is a function of $\gamma$, efficiency translates into exponentially tight, matching upper and lower bounds. In the moderate deviation regime, we obtain sub-Gaussian tails, while in the large deviation regime the decay becomes sub-Poisson. This deviation-scale viewpoint parallels the classical CLT-moderate-large deviation trichotomy for i.i.d.\ sums. Our bounds are obtained using a combination of Stein's method and the transform method. For constant deviations, we first prove Wasserstein-$p$ bounds between the centered-scaled queue and the limiting Gaussian via Stein's method, and then convert into tail bounds. Compared with Wasserstein-$1$ bounds that are relatively more popular in the literature, this is, to the best of our knowledge, the first work that obtains Wasserstein-$p$ bounds, which are of independent interest. The moderate and large deviation tail bounds are derived via the transform method.

    We then consider a load-balancing system of abandonment queues with heterogeneous servers, operating under the join-the-shortest-queue (JSQ) policy in the heavily overloaded regime. 
    As in the case of single-server queue, we again obtain Wasserstein-$p$ bounds w.r.t.\ a Gaussian, and efficient concentration for constant and moderate deviations. For larger deviations, our JSQ upper bounds exhibit a transition from Gaussian-type decay to sub-Weibull decay.
    All these results are obtained using Stein's method. In addition to the techniques mentioned above, a key ingredient here is establishing a state space collapse (SSC) where all queues become equal. While prior work establishes SSC via second-moment bounds of the orthogonal component of the queue length vector, we establish a stronger $p$-th moment bound that is essential for our Wasserstein-$p$ bound.

\end{abstract}

\section{Introduction}

In many service systems, e.g., microservice clusters, call centers, and ride-hailing platforms, a key defining feature is abandonment: jobs time out, customers leave, requests are canceled \cite{GarnettMandelbaumReiman2002,KooleMandelbaum2002,DaiHe2010MOOR}. In these settings, contractors usually require  hard constraints on tail probabilities of the form $\mathbb{P}(q > a)$, where $q$ is the steady-state queue length. This paper aims to provide bounds for such stationary tail probabilities for queues with abandonment. We consider two such queueing systems. The first one is a single server queue (SSQ) operating in continuous-time with memoryless arrivals, service and abandonment. The second one is a load balancing system consisting of several heterogeneous servers, where an incoming job must be routed to one of the servers upon arrival. We study this system under the join-the-shortest-queue (JSQ) policy, which dynamically assigns each arrival job to the server with currently shortest queue \cite{EschenfeldtGamarnik2018MOR,Khiyaita2012}.

To understand tail probabilities, one often adopts an asymptotic viewpoint, which is organized by the \emph{deviation scale}, i.e., the scale of $a$ in $\mathbb{P}(q > a)$. For both SSQ and JSQ under constant deviation, it is shown that in the heavily overloaded regime (when the arrival rate is larger than the service rate), the appropriately centered-scaled steady queue length $\tilde{q}$ converges to a Gaussian distribution as the abandonment rate $\gamma \to 0$ \cite{ward2003diffusion,jhunjhunwala2023jointheshortestqueueabandonmentcritically}. Note that the heavily overloaded system is still stable due to abandonment. This limit suggests that the tail probability $\mathbb{P}(\tilde{q} > a)$ should exhibit some kind of Gaussian-like decay. At the other extreme, when the deviation $a$ is large enough, the abandonment effect dominates the system dynamics, and the queue behaves like an $M/M/\infty$ queue, with abandonment playing the role of service. In this case, one expects a Poisson-like tail decay since the steady state distribution of an $M/M/\infty$ queue is Poisson, rather than Gaussian. In the middle ground, for moderate deviations, one seeks a smooth transition connecting the Gaussian and Poisson tail behaviors from the two extremes. This is similar in spirit to how classical Central Limit Theorem (CLT), moderate and large deviation principles fit together for sums of i.i.d. random variables \cite{petrov2012sums}. Figure \ref{fig: Phase Transition Diagram} schematically illustrates these tail behaviors across deviation regimes. The $x$-axis represents $\delta$ which parameterizes the deviation scale as $a_\gamma = \Theta(1/\gamma^\delta)$. 
The $y$-axis represents the tail probability. The phase transition points between the regimes will be revisited later in Section \ref{sec: main result} when we present our results. Parameter $\alpha$ in this phase transition diagram is a system parameter
relating abandonment rate to the arrival and service rates.
\begin{figure}[H]
    \centering
    \includegraphics[width=0.7\textwidth]{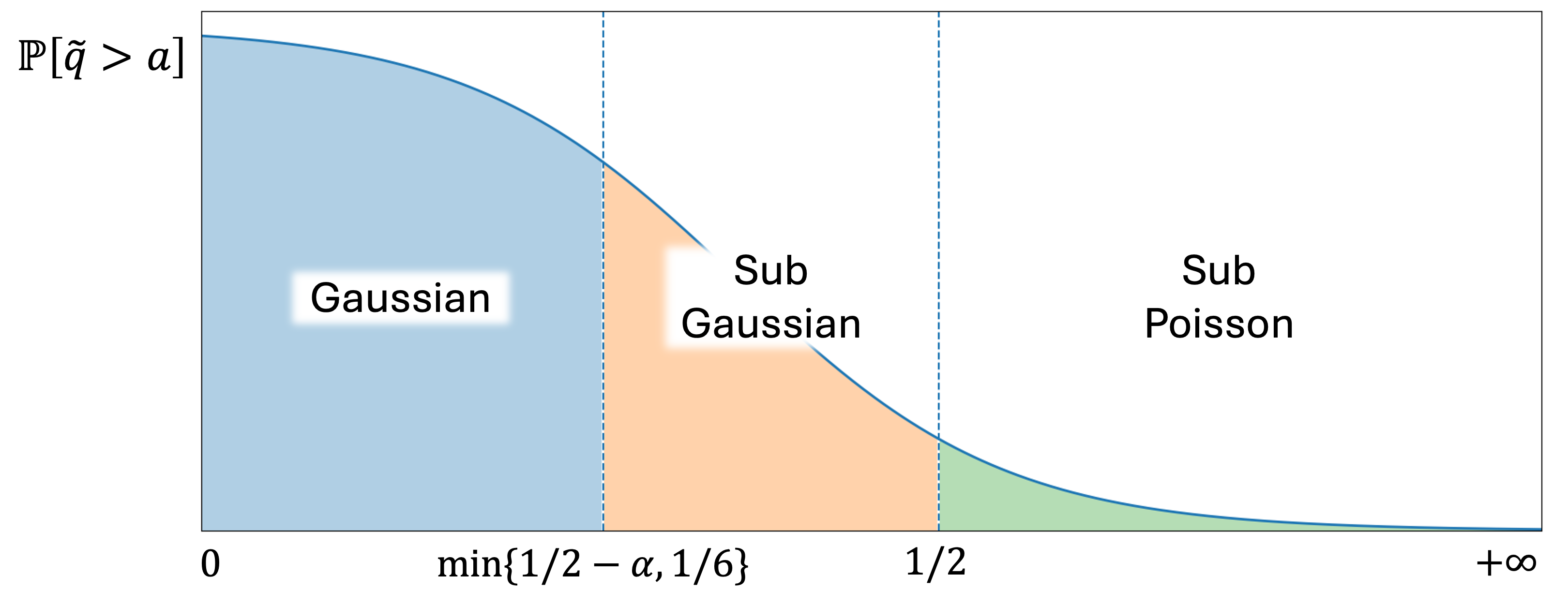}
    \caption{Phase Transition of Tail Bounds for SSQ, with $x$-axis representing $\delta$ in $a_\gamma = \Theta(1/\gamma^\delta)$.}
    \label{fig: Phase Transition Diagram}
\end{figure}


However, asymptotic results for SSQ or JSQ do not control the pre-limit tail when $\gamma$ is strictly positive. Therefore,
our goal is to derive explicit pre-limit tail bounds for both systems that are tight. Here, explicit means that the bounds do not depend on any unknown or unspecified constants. Moreover, from the aspect of limit parameter $\gamma$,
we require tightness in the sense that as $\gamma$ goes to zero, the pre-limit tail bounds directly imply different limiting tail behaviors across all deviation scales. In particular, for constant deviations, we aim to prove matching upper and lower bounds that together yield $\lim_{\gamma \downarrow 0} \frac{\mathbb{P}(\tilde{q} > a)}{\mathbb{P}(Z > a)} = 1$, where $Z$ is a standard Gaussian random variable. This implication recovers the aforementioned limit theorems. For larger deviations, we seek upper and lower tail bounds that are exponentially tight, in the sense that $\lim_{\gamma \downarrow 0} \frac{\ln \mathbb{P}(\tilde{q} > a_\gamma)}{\ln \mathbb{P}(Y > a_\gamma)} = 1$, where $Y$ is Gaussian or Poisson depending on how $a_\gamma$ scales with $\gamma$. 

Moreover, in terms of deviation level $a$, the limit heuristics points to Gaussian or Poisson behavior. Accordingly, we require pre-limit bounds to display such Gaussian-type or Poisson-type decay in $a$. We refer to such tight and explicit pre-limit tail bounds with corresponding decay as \emph{efficient concentration} inequalities. That is, these bounds are tight enough to yield the correct limiting tail behavior as $\gamma \downarrow 0$, while providing light-tailed control in $a$ for fixed $\gamma$. To the best of our knowledge, there is no such efficient concentration inequality in the queueing literature. The one exception we are aware of is \cite{braverman2024high} which considers a different system and focuses on deviations up to moderate regime. Therefore, the question we address is: \textit{Can one establish efficient concentration for both SSQ and JSQ under abandonment, and over what deviation scales?}

\paragraph{Main Contributions}
We now present our main contributions for SSQ and JSQ. For SSQ we have the following. 
\begin{itemize} 
    \item We obtain an efficient concentration inequality for centered-scaled stationary queue length $\tilde{q}$ under constant deviation $a$. 
    In particular, we show a bound of the form 
    \begin{align}
    \Phi^c(a)[1-c\sqrt{\gamma}(\ln(1/\gamma)+a^3 )] \leq \mathbb{P}(\tilde{q}>a) \leq \Phi^c(a)[1+c\sqrt{\gamma}(\ln(1/\gamma)+a^3)]
        \label{eq: main contribution, constant deviation}
    \end{align} 
    for some  constant $c$ that is explicitly presented later in the paper. Here $\Phi(\cdot)$ is the CDF of the standard Gaussian. Letting $\gamma\to 0$, this result implies the convergence $\mathbb{P}(\tilde{q}>a) \to \Phi^c(a)$. Moreover, the pre-limit tail decay in $a$ is close to that of a Gaussian. 

    \item We then consider the larger deviations, $a_{\gamma} = \Theta\left(\tfrac{1}{\gamma}\right)^\delta$. Depending on the exponent $\delta$, we have three different regimes depending on a parameter $\alpha$ which relates the abandonment rate to the arrival and service rates. First when  $\delta \in [0,\min\{1/2-\alpha,1/6\}) $ in the near-constant deviations regime, we again get efficient concentration of the form \eqref{eq: main contribution, constant deviation}. When $\delta \in [\min\{1/2-\alpha,1/6\},1/2) $, in the moderate deviations regime, we establish a sub-Gaussian tail. Finally when $\delta \in [1/2,\infty) $, in the large deviations regime,  we establish Sub-Poisson decay. 
    \item All the bounds stated above are order-wise tight, as we also establish matching lower bounds in all regimes. Thus, we achieve efficient concentration for all deviation scales. 


    \item For constant deviation, we obtain efficient concentration by proving bounds on the Wasserstein-$p$ distance between $\tilde{q}$ and the limiting Gaussian distribution for all $p > 1$. Given that most such bounds in queueing literature are limited to the setting of Wasserstein-$1$ distance, we believe that these Wasserstein-$p$ bounds are of independent interest. 
    Such bounds imply weak convergence and recover limit theorems in \cite{ward2003diffusion}, with convergence of all moments, and provide explicit convergence rate in $\mathcal{W}_p$ metric. Moreover, our Wasserstein-$p$ bounds are tight enough in terms of $p$ and $\gamma$ to imply the inequality \eqref{eq: main contribution, constant deviation}. 
    We obtain the Wasserstein-$p$ bound as follows. First, using Stein's method the Wasserstein-$p$ distance is bounded in terms of certain moment bounds of the queue length. These moment bounds are then obtained using the Lyapunov drift method, i.e., by picking appropriate Lyapunov functions and
    utilizing the zero-drift identity in the steady-state. 
    \item We obtain the moderate and large deviation bounds using the transform method. More specifically, we use Lyapunov arguments to bound the moment generating function of $\tilde{q}$, and then use Markov inequality to get the tail bounds. The lower bounds are obtained using change-of-measure arguments. 
\end{itemize}
For the JSQ system, we obtain the following results. Note that we are now studying the centered-scaled stationary queue length vector $\tilde{\mathbf{q}}$.
\begin{itemize}
    \item We establish efficient concentration for $\langle \tilde{\mathbf{q}}, \boldsymbol{\phi}\rangle$ for any direction $\boldsymbol{\phi}$ that is not orthogonal to $\boldsymbol{1}$. 
    \begin{align}
    \left|\mathbb{P}(\langle \boldsymbol{\phi},\boldsymbol{\tilde{q}}\rangle>a) - \Phi^c\left(\frac{a}{\langle\boldsymbol{\phi},\boldsymbol{1}\rangle}\right)\right| \leq c\sqrt{\gamma}\left[\ln^2(1/\gamma)+\frac{a^5}{\langle\boldsymbol{\phi},\boldsymbol{1}\rangle^5}\right]\Phi^c\left(\frac{a}{\langle\boldsymbol{\phi},\boldsymbol{1}\rangle}\right)
        \label{eq: main contribution JSQ constant}
    \end{align}
    Again letting $\gamma$ to zero, our result shows the limit $\mathbb{P}(\langle \boldsymbol{\phi},\boldsymbol{\tilde{q}}\rangle>a) \to \Phi^c\left(\frac{a}{\langle\boldsymbol{\phi},\boldsymbol{1}\rangle}\right)$. Moreover, we have pre-limit Gaussian tail for any test direction not perpendicular to $\boldsymbol{1}$. 
    \item In the moderate deviation regime when $\delta\in [0,\min\{1/4 - \alpha/2,1/10\})$, we again obtain efficient concentration of the form \eqref{eq: main contribution JSQ constant}. For 
    larger deviations when $\delta \in [\min\{1/4 - \alpha/2,1/10\},\infty)$, we provide pre-limit upper bounds yielding a smooth transition from Gaussian decay to sub-Weibull decay $\exp(-a^{1/2})$. Yet unlike SSQ, the efficient concentration holds only for moderate deviations. We believe that the phase transition and sub-Weibull tail are not tight.
    
   \item We obtain all these results by bounding the Wasserstein-$p$ distances to Gaussian. The Wasserstein-$p$ bound itself implies the weak convergence from $\tilde{\boldsymbol{q}}$ to a vector of Gaussian random variable multiplied by $\boldsymbol{1}$, which resembles the limit theorem in \cite{jhunjhunwala2023jointheshortestqueueabandonmentcritically}. 
    Compared with the SSQ case, obtaining these Wasserstein-$p$ bounds is more challenging due to the more complex queueing dynamics.
   A main challenge is establishing the state space collapse (SSC). We define $\mathbf{q}_\perp$ as the component of the queue length vector $\mathbf{q}$ orthogonal to $\boldsymbol{1}$. 
Prior work \cite{jhunjhunwala2023jointheshortestqueueabandonmentcritically} showed an SSC phenomenon that $\mathbf{q}_\perp \overset{\gamma \downarrow 0}{\longrightarrow} \boldsymbol{0}$.
This limit motivates us to first study directions $\boldsymbol{\phi}$ with $\langle \boldsymbol{\phi}, \boldsymbol{1}\rangle \neq 0$ and establish efficient concentration as in \eqref{eq: main contribution JSQ constant}.
   SSC is usually established by bounding $\mathbf{q}_\perp$ and the work \cite{jhunjhunwala2023jointheshortestqueueabandonmentcritically} studied SSC in terms of the second moment of $\mathbf{q}_\perp$. We obtain a much tighter handle on SSC by establishing bounds on the $p$-th moment of $\mathbf{q}_\perp$ for all $p > 1$, and these higher moments are an essential ingredient to get Wasserstein-$p$ bounds. 


 \end{itemize}


\paragraph{Organization of the Paper} 
In Section \ref{sec: notation and modeling}, we present the model and assumption for both SSQ and JSQ. In Section \ref{sec: summary of results}, we present a summary of our results in a table, with a brief survey of analogous results on iid sums on central limit theorem, moderate and large deviation theory. In Section \ref{sec: main result} we deliver all results on tail bounds and Wasserstein-$p$ bounds for SSQ and JSQ. 
We provide a survey of techniques used in this paper and some related methods in Section \ref{sec: survey of methods} before delving into the proof sketches in Section \ref{sec: proof sketch}. We conclude the main text with a literature review in Section \ref{sec: literature review} and pointing future research directions in Section \ref{sec: future directions}. All detailed proofs are relegated to the Appendix.

\section{Model and Assumptions} \label{sec: notation and modeling}
We introduce the mathematical models for both the SSQ and JSQ systems, as well as the details of the heavily overloaded assumptions used throughout the paper.
\subsection{Modeling and Notation}
Following settings of \cite{ward2003diffusion}, we consider a continuous-time queueing system for both SSQ and JSQ, with Poisson arrivals, exponential service times and exponential patience times. All queues are First-in-First-out with infinite capacity. For SSQ, the arrival rate is $\lambda$ and the service rate is $\mu$. Jobs are assumed to be impatient and, once their exponential patience time clock rings, they will abandon the queue, even if they are in service. Such SSQ is also known as $M/M/1+M$ or Erlang-A queue in the literature. 
We use queue length to denote the total number of customers in the queue including the one in service. 
The arrival and service process are all independent of queue length, but the abandonment process is dependent on queue length.

The above Poisson/exponential assumptions are standard in the queueing literature for studying steady-state tail behavior \cite[Section 3]{braverman2024high}, since they lead to a Markovian state descriptor and an explicit birth--death generator. At the same time, although stationary distributions admit closed-form formula under such assumption, deriving sharp concentration bounds as in \eqref{eq: main contribution JSQ constant} is still nontrivial. In particular, the birth-death structure itself is not essential to our approach. The same framework can in principle be applied to discrete-time queues with more general arrival and service distributions, given that the underlying generator can be explicitly written down (see e.g., \cite{jhunjhunwala2023jointheshortestqueueabandonmentcritically}). We use SSQ as the main model to illustrate the core ideas of the framework, and demonstrate its versatility by extending to more complex systems, with the JSQ as one such example in below. We stick to the continuous-time Markovian setting mainly for ease of presentation and for direct comparison with prior work such as \cite{ward2003diffusion,braverman2024high}.

Throughout the paper we work directly with the local dynamics described by the infinitesimal generator. The generator for SSQ $\mathcal{L}_{\mathrm{SSQ}}$ acting on functions $f:\mathbb{N}\to\mathbb{R}$ is given by
\begin{align}
(\mathcal{L}_{\mathrm{SSQ}} f)(i)
= \lambda\big(f(i{+}1)-f(i)\big)
+ (\mu + \gamma i)\,\mathbf{1}_{\{i\ge 1\}}\big(f(i{-}1)-f(i)\big), \quad i\in\mathbb{N}. \label{eq: generator of SSQ}
\end{align}

For JSQ, we study heterogeneous server, where the $i$-th server has service rate $\mu_i$. We break ties for the shortest queue by lexicographic order, i.e., we always choose the server with the smallest index among the shortest queues. So $\arg\min_{i} q_i$ also follows the lexicographic order. All jobs, including those in service are impatient with exponential patience time with rate $\gamma$. We denote the total arrival rate as $\lambda$, and total service rate as $\mu = \sum_{i=1}^n \mu_i$ with a slight abuse of notation. Similar to SSQ, 
our queue length vector denotes the total population in each queue. The dynamics of JSQ is described both in a transition rate kernel and infinitesimal generator. The transition rate kernel $\mathbf{Q}_{JSQ}(\mathbf{x}, \mathbf{x}')$ is given as follows, where we denote $\mathbf{e}_*$ as the unit vector of coordinate $\arg\min_{i} x_i$. Note that the inequalities below are coordinate-wise.
\begin{align*}
    \mathbf{Q}_{JSQ}(\mathbf{x}, \mathbf{x}') = \begin{cases}
        \lambda, & \text{if } \mathbf{x}' = \mathbf{x} + \mathbf{e}_*, \\
        \mu_i + \gamma x_i, & \text{if } \mathbf{x}' = \mathbf{x} - \mathbf{e}_i \geq \mathbf{0}, \, i \in [n], \\
        0, & \text{otherwise}.
    \end{cases}
\end{align*}
The infinitesimal generator $\mathcal{L}_{JSQ}$ for function $f:\mathbb{N}^n\to \mathbb{R}$, is given by:
\begin{align}
    \mathcal{L}_{JSQ} f(\mathbf{x}) &= \sum_{\mathbf{x}' \in \mathbb{Z}^n} \mathbf{Q}_{JSQ}(\mathbf{x}, \mathbf{x}') (f(\mathbf{x}') - f(\mathbf{x})) \notag\\
    &= \left( \lambda (f(\mathbf{x} + \mathbf{e}_*) - f(\mathbf{x})) + \sum_{i=1}^{n}(\mu_i \boldsymbol{1}_{\{ x_i > 0\}}+ \gamma x_i)(f(\mathbf{x} - \mathbf{e}_i) - f(\mathbf{x})) \right) \label{eq: generator for join shortest queue}
\end{align}
We will use $P_{t, SSQ}$ and $P_{t, JSQ}$ to denote the transition semigroup for SSQ and JSQ respectively.

A key ingredient in analyzing JSQ is to establish a state space collapse (SSC). 
Here we introduce some notation for SSC. We use $\boldsymbol{1}$ to denote the vector of all ones in $\mathbb{R}^n$, and $\mathbf{q}_\parallel$, $\mathbf{q}_\perp$ to denote the parallel and orthogonal components of the queue length vector $\mathbf{q}$ w.r.t. $\boldsymbol{1}$. Formally, we first define the line spanned by vector $\boldsymbol{1}$ and projection of queue length vector $\mathbf{q}$ onto this line,
\begin{align*}
    \mathbf{q}_\parallel := \argmin_{\mathbf{x}\in \mathcal{S}} \|\mathbf{q} - \mathbf{x}\| ,\quad \mathcal{S} := \{ \mathbf{x} \in \mathbb{R}^n: \mathbf{x} = c\boldsymbol{1}, c\in \mathbb{R}\}.
\end{align*}
Then the orthogonal component is defined with decomposition $\mathbf{q} = \mathbf{q}_\perp + \mathbf{q}_\parallel$. 

For both SSQ and JSQ, since the number of abandonments grows with the queue lengths, the expected number of abandonments is larger than the expected number of arrivals if the queue lengths are large enough. This ensures the stability of the system. Formally, the positive recurrence and existence of stationary distributions
 of both systems are shown by using the Foster-Lyapunov Theorem with $f(q)=q^2$ as the Lyapunov function in Section \ref{apx: proof for Lp norm of q_infty}. With such stability, we denote the steady-state queue length for SSQ as $q$, and the steady-state queue length vector for JSQ as $\mathbf{q} = (q_1, q_2, ..., q_n)$. By contrast, $x$ or $\mathbf{x}$ are used to denote deterministic variables.

As mentioned, an intermediate step to obtain efficient concentration is to bound the Wasserstein-$p$ distance between the law of $\tilde{q}$ (or $\tilde{q}_\Sigma$, defined below), and standard normal. The Wasserstein-$p$ distance is the optimal transport $L^p$ distance between two probability measures $\mu, \nu$ on $(\mathbb{R}, \mathcal{B}(\mathbb{R}))$, where $\mathcal{B}(\mathbb{R})$ is the Borel $\sigma$-algebra on $\mathbb{R}$.
\begin{align*}
    \mathcal{W}_p(\mu,\nu)
    &:= \left( \inf_{\pi\in\Pi(\mu,\nu)} 
    \int_{\mathbb{R}^2} |x-y|^p\,\mathrm{d}\pi(x,y) \right)^{1/p}.\\
    \Pi(\mu,\nu)
    := \bigl\{ \pi \text{ prob.\ measure on } (\mathbb{R}^2,&\mathcal{B}(\mathbb{R}^2)) :
    \pi(A\times\mathbb{R})=\mu(A),\ 
    \pi(\mathbb{R}\times A)=\nu(A),\ 
    \forall A\in\mathcal{B}(\mathbb{R}) \bigr\}.
\end{align*}

\paragraph{Notation} A list of notations is presented here for easy reference.
\begin{enumerate} [label=(\roman*)] \label{notation}
    \item  Transformation of queue length: In SSQ, $\tilde{q} := \sqrt{\gamma/\lambda}(q - (\lambda - \mu)/\gamma)$ is denoted for normalized steady-state queue length, $\hat{q} = q - (\lambda - \mu)/\gamma$ is for centered steady-state queue length. Similarly for JSQ: $\hat{q}_\Sigma:= \sum_i q_i - \frac{\lambda-\mu}{\gamma}, \tilde{q}_\Sigma:= \sqrt{\gamma/\lambda}(q_\Sigma)$, $\hat{q}_{\Sigma}, \tilde{q}_{\Sigma}$ are centered and normalized steady state random variables respectively. $\hat{\mathbf{q}}:= \mathbf{q} - \frac{\lambda-\mu}{n\gamma}\boldsymbol{1}$, $\tilde{\mathbf{q}}:= \frac{n\sqrt{\gamma}}{\sqrt{\lambda}}\hat{\mathbf{q}}$ are centered and normalized steady state random vectors. 
    \item  Norm: $\|\cdot\|_{L^p}$ is $L^p$ norm for random variables, i.e., $\|X\|_{L^p}$ $:= (\mathbb{E}|X|^p)^{1/p}$ for random variable $X$. Specifically, since we are studying steady-state queue length, $\|\cdot\|_{L^p}$ denotes $L^p$ norm for steady-state queue length random variables. 
    For vectors, $\|\cdot\|$ is Euclidean norm, i.e., $\|\mathbf{x}\| :=\sqrt{\sum_{i} x_i^2}$ for vector $\mathbf{x}$. We would also specify whether it is deterministic vector or random vector when the context is not clear.
    \item Order notation: $O(\cdot), \Omega(\cdot), \Theta(\cdot), o(\cdot), \omega(\cdot)$, and $\lesssim, \gtrsim, \asymp$ are standard order notation, with respect to $\gamma \to 0$, or $t \to 0$, as specified.
    \item Standard Normal: $Z\sim \mathcal{N}(0,1),\Phi^c(\cdot)$ is complementary cdf (ccdf) of standard normal.
    \item Restriction of operators:
Let $E:=\operatorname{supp}(\mathcal{L}(X)) \subseteq \mathbb{R}$ for a real-valued random variable $X$.
Suppose $\mathcal A$ is an operator acting on functions on $E$.
For notational convenience, we extend $\mathcal A$ to functions on $\mathbb R$ by mere restriction, i.e.,
for $f\in\mathcal C_b(\mathbb R)$, define the function on $E$ by
\[
(\mathcal A f)(x):=\big(\mathcal A(f|_E)\big)(x),\qquad x\in E,
\]
where $f|_E$ is the restriction of $f$ to $E$. Since $X\in E$ a.s., the composition $(\mathcal A f)(X)$ is well-defined as a random variable. Throughout, we will use this convention.

\end{enumerate}

\subsection{Assumptions}
We collect here the standing assumptions used throughout the paper. We assume the arrival rate $\lambda>0$, the service rate $\mu>0$ for SSQ
(or service rates $\mu_i>0$ for $i=1,\dots,n$ in JSQ) and abandonment rate $\gamma>0$. Recall in JSQ, we denote the total service rate as $\mu:=\sum_{i=1}^n\mu_i$.

\begin{assumption}[Heavily‑overloaded scaling]\label{ass:heavy_overload}
There exist fixed constants $C>1$ and $\alpha\in[0,1/2)$ such that for all sufficiently small $\gamma>0$ the system operates in the heavily overloaded regime:
\[
\frac{\lambda}{\mu}-1 \;\ge\; C\Big(\frac{\gamma}{\mu}\Big)^{\alpha}.
\]
In particular $\lambda/\mu>1$ for all $\gamma$ in the range considered.
\end{assumption}
 Assumption \ref{ass:heavy_overload} is scale-invariant with respect to each parameter $\lambda,\mu,\gamma$ since it depends only on ratios. Note that we are singling out the heavily overloaded regime studied in the paper \cite{jhunjhunwala2023jointheshortestqueueabandonmentcritically} which leads to Gaussian limit. The restriction \(\alpha\in [0,1/2)\) excludes the other regimes for other diffusion limits. Moreover, the asymptotic results in this paper focus on $\gamma \downarrow 0$, so our pre-limit bounds in below will assume $\gamma$ small enough.  Where useful we write statements as holding for all sufficiently small \(\gamma>0\) (equivalently for all \(0<\gamma\le\bar{\gamma}\)).

\section{Summary and Analogy of Results} \label{sec: summary of results}
In this section, we summarize our main results on tail bounds for both SSQ and JSQ in Table \ref{tab:tail-bounds-for-all}. For a better understanding of these results, we also provide an analogy to the classical CLT, moderate and large deviation theory for i.i.d. sums. On the CLT side,  we treat two i.i.d.\ cases with mean \(0\) and variance \(1\): (i) sub-exponential summands, which give Cramér-type moderate deviations with unspecified constants \cite{Cramer1938,petrov2012sums,BahadurRangaRao1960}; and (ii) bounded summands, for which computable pre-limit upper bounds are available \cite{austern2022efficient}. For notation, we let \(S_n=\sum_{i=1}^n X_i\), \(Z_n:=S_n/\sqrt{n}\), and recall \(\Phi^c(x):=\mathbb{P}(Z\ge x)\) for \(Z\sim\mathcal{N}(0,1)\).
 Previous results in the CLT literature are centered around non-uniform Berry-Esseen bounds, which quantify tail probability differences between $Z_n$ and the standard normal in terms of both sample size $n$ and deviation $a$. Our notion of efficient concentration originates from such bound
 that captures the dependency on both parameters. Within the scope of CLT, efficient concentration is used to imply relative ratio convergence $\mathbb{P}(Z_n>a)/\Phi^c(a)\to 1$ 
  or the moderate deviation principle, i.e., $\lim_{n\to\infty} \frac{1}{a_n^2} \ln \mathbb{P}(Z_n) = -\frac{1}{2}$ with some range of $a_n$. Beyond that range, large deviation theory provides tail bounds for the remaining deviation scales. 
  (see Table \ref{tab:tail-bounds-for-all} for more details and the CLT Survey in Section \ref{sec: literature review} for more discussion).

Turning to our queueing systems, analogous to CLT, we establish efficient concentration bounds for all deviation regimes for the SSQ, and for near-constant and moderate deviations for JSQ. For the remaining larger deviations in JSQ, we provide pre-limit upper bounds.
 The alignment between queueing systems and their CLT counterparts, as well as the summary of our results are schematically illustrated
  in Table \ref{tab:tail-bounds-for-all}.  
 It's important to emphasize, however, that such alignment is just an analogy instead of direct connection. We are studying a sequence of steady states of Markov Chains indexed by $\gamma$ instead of summation of i.i.d. random variables indexed by $n$.



In Table \ref{tab:tail-bounds-for-all}, in order to match the deviation range across models, we relate the deviations by $a_\gamma=\Theta(\gamma^{-\delta})
\longleftrightarrow
a_n=\Theta(n^{\delta})$,
with $\gamma$ the abandonment rate for SSQ/JSQ, and $n$ the sample size for CLT. This relation $1/\gamma \sim n$ is due to the fact that the variance of $q$ is of order $1/\gamma$ for SSQ, while the variance of $S_n$ is of order $n$ for CLT. 
Each column of the table corresponds to a deviation regime (indexed by \(\delta\)). 
The row \emph{Pre-limit} gives bounds that hold at finite \(\gamma\) or \(n\). The row \emph{Limit} provides the asymptotic tail behavior as \(\gamma\to0\) (SSQ/JSQ) or \(n\to\infty\) (CLT). References for each entry are included in the table.
We separate regimes by the relative growth of the deviation and the limit parameter (controlled by \(\delta\)):
\begin{enumerate}
    \item Constant deviation: \(a_\gamma\) (or \(a_n\)) is fixed, while $\gamma$ (or $n$) is small (or big) enough but fixed for the 'Pre-limit' result, and $\gamma\to0$ (or $n\to\infty$) for the 'Limit' result. This regime covers the classical Central Limit Theorem and the limit results for SSQ and JSQ;
    \item Moderate deviations: \(a_\gamma\to\infty\) with \(\gamma\to0\) (or \(a_n\to\infty\) with \(n\to\infty\)), with rate $\delta$ slower than that in  large-deviation scaling.
    \item Large deviations: \(a_\gamma\) (or \(a_n\)) grows faster than in the moderate regime. This includes \(a_\gamma\to\infty\) (or \(a_n\to\infty\)) at fixed \(\gamma\) (or fixed \(n\)).
\end{enumerate}

Since we provide pre-limit bounds for the queueing system in the next section, here we focus on the limiting tail statements in each deviation regime and their analogy with CLT. 

For SSQ, the limiting behavior exhibits the same three-phase progression familiar from CLT. First is the precise ratio convergence of the tail probability to the Gaussian tail, $\mathbb{P}(\tilde{q}>a_\gamma)/\Phi^c(a_\gamma) \to 1$. Next is the moderate deviation principle, where the logarithm of the tail probability scales as $-\frac{1}{2}a_\gamma^2$. Finally, in the large deviation regime, the exponent of the tail probability scales as
 $a_\gamma \ln a_\gamma$, reflecting a Poisson-type large deviation behavior.

\begin{landscape}
\begin{table}[t]
    \centering
    \setlength{\tabcolsep}{6pt}
    \renewcommand{\arraystretch}{1.25}
\begin{subtable}{\linewidth}
    \centering
    \renewcommand{\arraystretch}{1.25}
    {\normalfont
    \begin{tabular*}{\linewidth}{@{\extracolsep{\fill}}|l|l|l|l|l|@{}}
        \hline
         & \multicolumn{1}{c|}{Constant Deviation}
            & \multicolumn{2}{c|}{Moderate Deviation}
            & \multicolumn{1}{c|}{Large Deviation} \\
        \hline
        $a_n=\Theta(n^{\delta}),\ \delta\in$ &
        \multicolumn{1}{c|}{\(\displaystyle \{0\}\)} &
        \multicolumn{2}{c|}{\(\displaystyle (0,1/2)\)} &
        \multicolumn{1}{c|}{\([1/2,+\infty)\)} \\
        \hline

        Pre-limit Upper \cite{austern2022efficient}
        & \multicolumn{1}{c|}{
            $ \Phi^c(a_n) + \mathcal{O}(\frac{1}{\sqrt{n}})\exp(-\Omega(a_n^2))$ 
        }
        & \multicolumn{2}{c|}{%
                $[1+\exp\bigl(\mathcal{O}(1+a_n^{3}/\sqrt{n})\bigr)]\,\Phi^{c}(a_n)$ }
        & 
        \multicolumn{1}{c|}{
            $\displaystyle
                \exp\!\Bigl\{-\Omega(a_n^2)\Bigr\}$} \\
        \hline

        Pre-limit Upper 
        & \multicolumn{3}{c|}{%
                $[1+O\!\bigl(\tfrac{a_n+1}{\sqrt{n}}\bigr)]\,
                 \exp\bigl(\mathcal{O}(1+a_n^{3}/\sqrt{n})\bigr)\,\Phi^{c}(a_n)$ (No Explicit Constant)\cite{petrov2012sums}}
        & 
        \multicolumn{1}{c|}{
            $\displaystyle
                \exp\!\Bigl\{-n\,\Lambda^{*}\!\bigl(\tfrac{a_n}{\sqrt{n}}\bigr)\Bigr\}$ \cite{Cramer1938}} \\
        \hline

        Pre-limit Lower 
        & \multicolumn{3}{c|}{%
                $[1-\mathcal{O}\!\bigl(\tfrac{a_n+1}{\sqrt{n}}\bigr)]\,
                 \exp\bigl(\mathcal{O}(1+a_n^{3}/\sqrt{n})\bigr)\,\Phi^{c}(a_n)$ (No Explicit Constant)\cite{petrov2012sums}}
        & \multicolumn{1}{c|}{$ \Omega(\frac{1}{\sqrt{n}}\exp\!\Bigl\{-n\,\Lambda^{*}\!\bigl(\tfrac{a_n}{\sqrt{n}}\bigr)\Bigr\})$ \cite{BahadurRangaRao1960}}     \\
        \hline

        Limit Theorem \cite{dembo2009large}
        & 
            \multicolumn{2}{@{}>{\centering\arraybackslash}m{70mm}@{}|}{%
                $\displaystyle \frac{\mathbb{P}(Z_n>a_n)}{\Phi^{c}(a_n)} \to 1$}
        & 
            \multicolumn{1}{c|}{%
                $\displaystyle -\frac{2}{a_n^{2}}\ln \mathbb{P}(Z_n>a_n) \to 1$}
        & 
            \multicolumn{1}{c|}{
                $\displaystyle-\,\frac{\ln \mathbb{P}(Z_n>a_n)}
                             {\,n\,\Lambda^{*}\!\bigl(a_n/\sqrt{n}\bigr)} \to 1$}
        \\
        \hline

        $a_n=\Theta(n^{\delta}),\ \delta\in$ &
        \multicolumn{2}{c|}{\(\displaystyle [0\;,\; 1/6)\)} &
        \multicolumn{1}{c|}{\(\displaystyle [1/6\;,\; 1/2)\)} &
        \multicolumn{1}{c|}{\([1/2\;,\;+\infty)\)} \\
        \hline
    \end{tabular*}}
    \caption{Tail Bounds on $\mathbb{P}(Z_n>a_n)$ for CLT.}
    \label{tab:clt-tail}
\end{subtable}

\begin{subtable}{\linewidth}
    \centering
    \setlength{\tabcolsep}{6pt}
    \renewcommand{\arraystretch}{1.25}
    {\normalfont
    \begin{tabular*}{\linewidth}{@{\extracolsep{\fill}}|c|c|c|c|c|c|@{}}
        \hline
        $a_\gamma=\Theta(1/\gamma^{\delta}),\delta \in$ &
            $\{0\}$ &
            \multicolumn{2}{c|}{
                \(\displaystyle (0,\,1/2-\alpha) \)
            }&
            $[1/2-\alpha,\,1/2)$ &
            $[1/2,\,+\infty)$ \\
        \hline

        Pre-limit Upper \ref{thm: SSQ Tail bound}
            & $[1+\mathcal{O}(\frac{\ln(1/\gamma)+a_\gamma^3}{1/\sqrt{\gamma}})] \Phi^c(a_\gamma)$
        & \multicolumn{2}{c|}{
                $[1+ \mathcal{O}(a_\gamma^3\sqrt{\gamma})\exp(\mathcal{O}(a_\gamma^3\sqrt{\gamma}))]
                \Phi^c(a_\gamma)$}
            & $\mathcal{O}(\exp(-a_\gamma^{2}/2 + O(a_\gamma^3\sqrt{\gamma})))$
            & $\exp(-\Omega(\frac{a_{\gamma}}{\sqrt{\gamma}}\ln (a_\gamma\sqrt{\gamma})))$ \\
        \hline

        Pre-limit Lower \ref{thm: SSQ Tail bound} 
        & $[1-\mathcal{O}(\frac{\ln(1/\gamma)+a_\gamma^3}{1/\sqrt{\gamma}})] \Phi^c(a_\gamma)$
        & \multicolumn{3}{c|}{$[1- \mathcal{O}(a_\gamma^3\sqrt{\gamma})\exp(\mathcal{O}(a_\gamma^3\sqrt{\gamma}))]\Phi^c(a_\gamma) \lor \Omega(\exp(-a_\gamma^{2}/2 - \mathcal{O}(a_\gamma^3\sqrt{\gamma})))$  }
        & $\exp(-\mathcal{O}(\frac{a_{\gamma}}{\sqrt{\gamma}}\ln (a_\gamma\sqrt{\gamma})))$ \\
        \hline

        \makecell[c]{Limit Theorem\\ Corollary \ref{cor: Deviation Principle for M/M/1+M}}
                & \multicolumn{2}{c|}{
                        \(\displaystyle \frac{\mathbb{P}(\tilde{q}>a_\gamma)}{\Phi^c(a_\gamma)} \to 1\)} 
                & \multicolumn{2}{c|}{
                \(\displaystyle -\frac{2\ln \mathbb{P}(\tilde{q}>a_\gamma)}{a_\gamma^2} \to 1\)}
            & $\displaystyle
                -\frac{\frac{\sqrt{\gamma}}{\sqrt{\lambda}}\ln \mathbb{P}(\tilde{q}> a_\gamma)}
                            {a_\gamma \ln\!\bigl(1+a_\gamma \tfrac{\sqrt{\gamma}}{\sqrt{\lambda}}\bigr)} \to 1$ \\
        \hline
        $a_\gamma=\Theta(1/\gamma^{\delta}), \delta \in$ &
        \multicolumn{2}{@{}>{\centering\arraybackslash}m{70mm}@{}|}{ 
                \( \displaystyle \; [0\;,\;\min\{1/2-\alpha, 1/6\}) \)
        }&
        \multicolumn{2}{c|}{
                \(\displaystyle [\min\{1/2-\alpha, 1/6\}\;,\;1/2)\)
        } &
        $[1/2 \;,\;+\infty)$ \\
        \hline
    \end{tabular*}}
    \caption{Tail Bounds on $\mathbb{P}(\tilde{q}>a_\gamma)$ for SSQ, when $\lambda/\mu - 1 \ge C(\gamma/\mu)^\alpha, \alpha\in[0,1/2)$.}
\end{subtable}

\begin{subtable}{\linewidth}
    \centering
    \setlength{\tabcolsep}{6pt}
    \renewcommand{\arraystretch}{1.25}
    {\normalfont
    \begin{tabular*}{\linewidth}{@{\extracolsep{\fill}}|c|c|c|c|c|@{}}
        \hline
        $a_\gamma=\Theta(1/\gamma^{\delta}),\ \delta\in$ &
            $\{0\}$ &
            \multicolumn{2}{c|}{\(\displaystyle (0,\min\{1/4-\alpha/2,\,1/6\})\)} &
            $[\min\{1/4-\alpha/2,\,1/6\}),\,+\infty)$ \\
        \hline

        Pre-limit Upper \ref{thm: JSQ Tail bound}
        & $[1+\mathcal{O}\!\big(\tfrac{\ln^2(1/\gamma) + a_\gamma^5/\langle\boldsymbol{\phi},\boldsymbol{1}\rangle^5}{1/\sqrt{\gamma}}\big)]\Phi^c(\frac{a_\gamma}{\langle \boldsymbol{\phi},\boldsymbol{1}\rangle})$
        & \multicolumn{2}{c!{\vrule width 0.5pt}}{%
                $ [1+ \mathcal{O}(a_\gamma^5\sqrt{\gamma})\exp(O(a_\gamma^5\sqrt{\gamma}))]\,\Phi^c(\frac{a_\gamma}{\langle \boldsymbol{\phi},\boldsymbol{1}\rangle})$}
        & \multicolumn{1}{c|}{%
                $\displaystyle O\left(\exp\bigl(-\Omega(a_\gamma^{\min\{1/(4\delta)+1/2,\,2\}})\bigr)\right)$} \\
        \hline

        Pre-limit Lower \ref{thm: JSQ Tail bound} 
        & \multicolumn{2}{c!{\vrule width 0.5pt}}{%
                $ [1- \mathcal{O}(a_\gamma^5\sqrt{\gamma})
                \exp(O(a_\gamma^5\sqrt{\gamma}))]\,\Phi^c(\frac{a_\gamma}{\langle \boldsymbol{\phi},\boldsymbol{1}\rangle})$}
        & \multicolumn{2}{c|}{Open} 
        \\
        \hline
        Limit Theorem \ref{thm: JSQ Tail bound}
        & \multicolumn{2}{c!{\vrule width 0.5pt}}{%
                $\displaystyle
                \frac{\mathbb{P}(\langle \boldsymbol{\phi},\tilde{\mathbf{q}}\rangle>a_\gamma)}
                         {\mathbb{P}(\langle \boldsymbol{\phi}, Z\!\cdot\!\boldsymbol{1}\rangle>a_\gamma)}
                \to 1$}
        & \multicolumn{2}{c|}{Open} \\
        \hline
        $a_\gamma=\Theta(1/\gamma^{\delta}),\ \delta\in$ &
        \multicolumn{2}{@{}>{\centering\arraybackslash}m{55mm}@{}|}{\(\displaystyle [0\;,\;\min\{1/4-\alpha/2,1/10\})\)} &
        \multicolumn{2}{c|}{$[\min\{1/4-\alpha/2,1/10\})\;,\; +\infty)$} \\
        \hline
    \end{tabular*}}
    \caption{Tail Bounds on $\mathbb{P}(\langle \boldsymbol{\phi},\tilde{\mathbf{q}}\rangle>a_\gamma), \langle \boldsymbol{\phi}, \boldsymbol{1}\rangle \ne 0$ for JSQ, Same $\alpha$ as above.}
    \label{tab:jsq-tail}
\end{subtable}

    \caption{ Pre-limit and limit tail bounds comparison for all regimes: CLT,  SSQ, and JSQ.}
    \label{tab:tail-bounds-for-all}
\end{table}

\end{landscape}
\paragraph{Ratio Convergence}
We first note that the convergence of relative ratio is strictly stronger than the moderate deviation principle. Whenever $\mathbb{P}(\tilde{q}>a_\gamma)/\Phi^c(a_\gamma) \to 1$ holds along a sequence of $a_\gamma$ with $\sup_{\gamma>0} a_\gamma > 0$, the moderate deviation principle $\frac{-2}{a_\gamma^2} \ln \mathbb{P}(\tilde{q}>a_\gamma) \to 1$ holds automatically. This is because $\Phi^c(a_\gamma)$ has exponent $-\frac{1}{2}a_\gamma^2$ bounded away from $0$ in the limit. Thus, it is the strongest form of limiting Gaussian approximation and one naturally expects it to hold over the widest deviation range. For SSQ, we show this ratio convergence holds for $a_\gamma = o(1/\gamma^{\min\{1/2-\alpha,1/6\}})$. The appearance of the $1/6$ transition is in similar spirit to CLT transition. Indeed for CLT, a third-order expansion of the characteristic function \cite[Section XVI, Theorem 1]{feller1991introduction} leads to an additional $a_n^3/\sqrt{n}$ term in the Cramér-type moderate deviation tail for CLT. Consequently, this correction vanishes in the limit only if $a_n = o(n^{1/6})$, which is precisely the critical scaling for the transition between ratio convergence and moderate deviation principle in CLT, see the CLT subtable for both pre-limit bounds and limit results in Table \ref{tab:tail-bounds-for-all}.

Our SSQ pre-limit bounds display an analogous correction, with an exponent term of $a_\gamma^3\sqrt{\gamma}$, which leads to the same $1/6$ transition for SSQ. Notably, our pre-limit bounds are obtained via Wasserstein-$p$ bounds instead of characteristic function expansion. The matching $1/6$ threshold is therefore consistent with the heuristic that both approaches ultimately rely on the moment conditions of the underlying distribution. In other words, moment control in the Wasserstein-$p$ bounds is similar to moment control in the characteristic function expansion. In contrast, the $1/2-\alpha$ transition is more specific to our queueing systems. In particular, \cite{ward2003diffusion} shows that $\tilde{q}$ converges to a Gaussian as $\gamma \to 0$ only when $\alpha < 1/2$. Thus, a Gaussian tail ratio of the form $\mathbb{P}(\tilde{q}>a_\gamma)/\Phi^c(a_\gamma) \to 1$ can only hold when $\alpha < 1/2$, which explains the presence of the $1/2-\alpha$ transition.

\paragraph{Larger Deviation}
For the moderate deviation regimes, both the CLT and SSQ exhibit a quadratic (Gaussian) exponent, with the leading term in $\ln \mathbb{P}(\tilde{q}>a_\gamma)$ being $-\frac{1}{2}a_\gamma^2$. Moreover, the boundary between moderate and large deviation regimes occurs when the deviation level is the intrinsic "system scale" of the problem, i.e., $a_\gamma = \Theta(1/\sqrt{\gamma})$ for SSQ, mirroring the CLT threshold of $a_n = \Theta(\sqrt{n})$.
At this scale, the quadratic approximation underlying the moderate deviation principle breaks down. Technically, in CLT this quadratic behavior comes from a cumulant expansion of the summand log-MGF, while for SSQ our Lyapunov drift method yields an analogous expansion. In the large deviation regime, both rate functions arise from optimizing the exponential tilting principle, with SSQ producing a Poisson-type rate according to the abandonment dynamics, and CLT producing the Fenchel-Legendre transform of the log-MGF of the summands. 

We summarize the limit results for all deviation regimes for SSQ in the following proposition.
\begin{proposition}\label{cor: Deviation Principle for M/M/1+M}
     We assume the heavily overloaded regime \ref{ass:heavy_overload} with $\gamma\leq \gamma_0$ (see \eqref{eq: gamma assumption for SSQ W-p}). We denote $\tilde{q}:=\frac{\sqrt{\gamma}}{\sqrt{\lambda}}(q-\frac{\lambda-\mu}{\gamma})$. Then
     , with $ a_\gamma = D_{\delta}/\gamma^\delta$ for appropriate choice of constants $D_{\delta}$,
     we have,
    \begin{align*}
    &\lim_{\gamma \to 0} \mathbb{P}(\tilde{q} > a_{\gamma})/\Phi^c(a_\gamma) = 1, && \textnormal{for } 0\leq\delta< \min\{ 1/2 -\alpha, 1/6\}, \\
    &\lim_{\gamma \to 0} -\frac{2}{a_\gamma^2} \ln \mathbb{P} (\tilde{q} > a_\gamma) = 1, && \textnormal{for } 0\leq\delta\leq1/2,\\
    &\lim_{\gamma \to 0} -\frac{\sqrt{\gamma}}{\sqrt{\lambda}a_\gamma \log(1+a_\gamma \sqrt{\gamma}/\sqrt{\lambda})} \ln \mathbb{P}(\tilde{q} > a_\gamma) = 1, && \textnormal{for } \delta>1/2.
    \end{align*}
\end{proposition}

For JSQ, we also establish Gaussian tail ratio convergence in the constant-deviation regime, recovering the previous limit behavior. Moreover, we extend this ratio limit to the near-constant range of $o(1/\gamma^{\min\{1/4-\alpha/2, 1/10\}})$. Compared with SSQ, this range is narrower due to the heavy-tailed effect of the perpendicular component of the queue length vector (see Section \ref{sec: SSC for JSQ} for more discussion). 


Finally, we stress that all of the above limiting statements follow directly from our pre-limit tail bounds by letting $\gamma \to 0$. These implications are summarized in the corresponding rows of Table \ref{tab:tail-bounds-for-all}. We now turn to the main pre-limit theorems for SSQ and JSQ in the next section.

\section{Main Result} \label{sec: main result}
In this section, we present the theorems for Wasserstein-$p$ distance and tail bounds for SSQ with abandonment and for the load balancing system (JSQ). The proof sketch is deferred to later sections.


\subsection{Tail Bound for Single Server Queue with Abandonment}\label{sec: Tail Bound for M/M/1+M}
As illustrated in Table \ref{tab:tail-bounds-for-all} and in Figure \ref{fig: Phase Transition Diagram}, the tail bounds for SSQ will exhibit a three-stage behavior, transitioning from Gaussian decay in constant deviation, to sub-Gaussian decay in moderate deviation, and finally to sub-Poisson decay in large deviation. As we will see in this section, this "near-constant is Gaussian, far-from-constant is Poisson" philosophy will appear not only in the tail bounds themselves, but also in our Wasserstein-$p$ bounds and in the key ingredients driving them (see Theorem \ref{thm: Gaussian Wasserstein-$p$ upper bound for M/M/1+M} and the proof discussion in Section \ref{sec: proof sketch for Wasserstein-$p$ bounds, SSQ}).

\begin{theorem}\label{thm: SSQ Tail bound}
     Under the heavily overloaded Assumption \ref{ass:heavy_overload} with $\gamma\leq \gamma_0$ (see \eqref{eq: gamma assumption for SSQ W-p}), the following inequalities hold,
    \begin{enumerate}[label=(\alph*), ref=\thetheorem~(\alph*)]
        \item \label{thm: SSQ Tail bound, first regime}
        For constant deviation $a>0$ and $\gamma\leq (\gamma_a \land \gamma_0)$ (see \eqref{eq: gamma_a}), there exists $D_{1,\lambda,\mu}'$ independent of $\gamma$ and $a$ (see \eqref{eq: Tail bound, first regime}), such that
        \begin{align*}
            \left|\,\mathbb{P}(\tilde{q} > a) - \mathbb{P}(Z > a)\,\right| 
            \leq D_{1,\lambda,\mu}'\,\sqrt{\gamma}\,\big[\ln(1/\sqrt{\gamma}) + a^2/2\big]\,e^{-a^2/2}
        \end{align*}

        \item \label{thm: SSQ Tail bound, second regime}
        For $D_{Tail,2,l} \leq a_{\gamma} \leq D_{Tail,2,u}/\gamma^{\delta}$, $\delta \in (0,1/2-\alpha)$ (see \eqref{eq: near constant deviation setup}), we have
        \begin{align*}
            \left|\mathbb{P}(\tilde{q} > a_{\gamma}) - \mathbb{P}(Z > a_{\gamma})\right| 
            &\leq \frac{e D_{1,\lambda,\mu}}{\sqrt{2\pi}}\, \sqrt{\gamma}\, \left[\ln(1/\sqrt{\gamma}) + \frac{a_{\gamma}^2}{2}\right] e^{-a_{\gamma}^2/2 \left(1 - \frac{D_{1,\lambda,\mu}\sqrt{\gamma}[\ln(1/\sqrt{\gamma}) + a_{\gamma}^2/2]}{a_{\gamma}}\right)} 
            + \sqrt{\gamma}\, e^{-a_{\gamma}^2/2}
            \\
            \mathbb{P}(\tilde{q} > a_{\gamma}) &\geq \frac{2}{3}\exp\left\{
                -\frac{a_{\gamma}^2}{2}
                + \triangle \ln\left(1+\frac{\frac{\sqrt{\lambda}}{\sqrt{\gamma}} a_{\gamma}+\triangle/2}{\lambda/\gamma}\right)
                + \frac{\triangle^2}{8\lambda/\gamma}
                - \frac{\left(\frac{\sqrt{\lambda}}{\sqrt{\gamma}} a_{\gamma}+\triangle/2\right)^3}{3(\lambda/\gamma)^2}
            \right\} =: L_{\lambda,\mu}(a_{\gamma},\gamma),
        \end{align*}
        where $\triangle:= 8\sqrt{\frac{\sqrt{\lambda}}{\sqrt{\gamma}} a_{\gamma} + \lambda/\gamma}$.
        
        \item \label{thm: SSQ Tail bound, third regime}
        For $D_{Tail,3,l}/\gamma^{1/2-\alpha} \leq a_{\gamma} \leq D_{Tail,3,u}/\sqrt{\gamma}$ (see \eqref{eq: sub-Gaussian deviation setup}),
        \begin{align*}
            L_{\lambda,\mu}(a_{\gamma},\gamma) \leq \mathbb{P}(\tilde{q} > a_{\gamma})
            \leq \min\Bigg\{&\; \mathbb{P}(Z > a_{\gamma}) + \frac{e}{2\sqrt{\pi}}\, a_{\gamma}\, e^{-\frac{a_{\gamma}^2}{2}\big(1-\sqrt{1/2}\big)^2} 
            + \exp\!\left(-\frac{a_{\gamma}^2}{2e^2 D_{2,\lambda,\mu}}\right), \\
            &\; A_{\lambda,\mu}\, \exp\!\left(-\frac{a_{\gamma}^2}{2\lambda/\gamma} + \frac{a_{\gamma}^3}{2(\lambda/\gamma)^2}\right)
            \Bigg\}
        \end{align*}
        with $D_{2,\lambda,\mu}$ as in \eqref{eq: Wasserstein-$p$ bound, second regime}, $L_{\lambda,\mu}(a_{\gamma},\gamma)$ as above, and $A_{\lambda,\mu}$ as in \eqref{eq: MGF for q_infty - lambda/mu/gamma, Poisson style}.

        \item \label{thm: SSQ Tail bound, fourth regime}
        For $a_{\gamma} \geq D_{Tail,4,l}/\sqrt{\gamma}$ (see \eqref{eq: large deviation setup}),
        \begin{align*}
            \mathbb{P}(\tilde{q} > a_{\gamma})
            &\leq
            \min\Bigg\{\,
            \mathbb{P}(Z>a_{\gamma}) + \frac{e}{2\sqrt{2\pi}}\,a_{\gamma}\,e^{-a_{\gamma}^2/8}
            + \exp\!\Big(-\frac{\sqrt{\lambda}}{2eD_{3,\lambda,\mu}}\cdot\frac{a_{\gamma}}{\sqrt{\gamma}}\log(1+a_{\gamma}\sqrt{\gamma})\Big),\\
            &\qquad\qquad\qquad\qquad\;
            A_{\lambda,\mu}\exp\!\Big(\frac{\sqrt{\lambda}}{\sqrt{\gamma}}a_{\gamma}
            - \Big(\frac{\lambda}{\gamma}+\frac{a_{\gamma}\sqrt{\lambda}}{\sqrt{\gamma}}\Big)
              \log\Big(1+\frac{a_{\gamma}}{\sqrt{\lambda/\gamma}}\Big)\Big)
            \Bigg\},\\[6pt]
            \mathbb{P}(\tilde{q} > a_{\gamma})
            &\geq \frac{2}{3}\exp\Big\{ 
            -\Big(\frac{\sqrt{\lambda}}{\sqrt{\gamma}} a_{\gamma}+ \triangle+ \frac{\lambda}{\gamma}\Big)
            \log\Big(1+\frac{\frac{\sqrt{\lambda}}{\sqrt{\gamma}} a_{\gamma}+\triangle/2}{\lambda/\gamma}\Big) + \frac{\sqrt{\lambda}}{\sqrt{\gamma}} a_{\gamma} + \frac{\triangle}{2}
            \Big\}.
        \end{align*}
        with $D_{3,\lambda,\mu}$ as in \eqref{eq: Wasserstein-$p$ bound, third regime}, $\triangle$ as in \ref{thm: SSQ Tail bound, second regime}, and $A_{\lambda,\mu}$ as in \eqref{eq: MGF for q_infty - lambda/mu/gamma, Poisson style}.
    \end{enumerate}
\end{theorem}

\paragraph{Constant Deviation} First we consider the case when the deviation $a$ is a constant. Theorem \ref{thm: SSQ Tail bound, first regime} provides an upper bound on  the absolute error between the tail probability $\mathbb{P}(\tilde{q} > a)$ and the tail of the standard Gaussian. This immediately implies an upper bound on the queue length tail  $\mathbb{P}(\tilde{q} > a)$ of the form, 
\begin{align}
\mathbb{P}(\tilde{q}>a)\leq [1+\mathcal{O}(\sqrt{\gamma}(\ln(1/\gamma)+a^3)]\Phi^c(a) . \label{eq: efficient concentration SSQ, for writing}
\end{align}
Note that one can obtain this from Theorem \ref{thm: SSQ Tail bound, first regime}
using the Mills ratio \cite[Section~3.1.2]{gasull2014approximating}, 
\begin{align}
 \frac{105}{91+110a} \leq \sqrt{2\pi}e^{a^2/2}\Phi^c(a) \leq \frac{44}{35+28a}. \label{eq: Mills Ratio}
\end{align}
For $\gamma$ small enough, we have that the upper bound has a Gaussian tail with exact exponent $-a^2/2$. Moreover, as $\gamma\to 0$, the bound in Theorem \ref{thm: SSQ Tail bound, first regime}  implies that $|\mathbb{P}(\tilde{q} > a) - \Phi^c(a)| \to 0$. Thus, this result establishes efficient concentration. We note that a similar efficient concentration result can be obtained for the lower tail $\mathbb{P}(\tilde{q} < -a)$, and thus the two-sided tail $\mathbb{P}(|\tilde{q}| > a)$, by a similar argument for Inequality \eqref{eq: Tail bound from Wasserstein-$p$ distance, lower tail}. We choose to present the one-sided tail for simplicity.

Even though the pre-exponent factor $a^3$ in the tail decay is probably not tight, the exponential factor $-a^2/2$  is tight due to the discrete nature of 
 $\tilde{q}$.   To see this, note that  $\tilde{q}$ is lattice distributed on $\{\sqrt{\gamma/\lambda} \cdot k, k \in \mathbb{Z}\}$, and consider the value of $\mathbb{P}(\tilde{q} > a)$ between two adjacent lattice points. There exists a point $a$ in this lattice such that,
\begin{align*}
    \left|\mathbb{P}(\tilde{q} > a) - \mathbb{P}(Z > a)\right| &\geq \frac{1}{2} \mathbb{P}(Z \in (a - \frac{\sqrt{\gamma}}{\sqrt{\lambda}} , a)) \geq \frac{1}{2\sqrt{2\pi}}\cdot \frac{\sqrt{\gamma}}{\sqrt{\lambda}} \exp(-a^2/2). 
\end{align*}
This argument also suggests that  the rate of convergence $\sqrt{\gamma}\ln(1/\gamma)$ 
is sub-optimal at most by a logarithmic factor. We obtain all these results by proving a
Wasserstein-$p$ distance bound between $\tilde{q}$ and standard Gaussian.

We emphasize that the tail bound in Theorem \ref{thm: SSQ Tail bound, first regime} is \textit{non-uniform} in deviation $a$.
By contrast, a bound on the Kolmogorov distance between  $\tilde{q}$ and a standard Gaussian would provide a \textit{uniform} upper bound on $\left|\mathbb{P}(\tilde{q} > a) - \mathbb{P}(Z > a)\right|$ that does not decays with $a$ and is thus looser for large $a$.

Note that up to this point we focus on the constant-deviation regime, where $a$ is fixed while $\gamma$ is taken sufficiently small, and we derive explicit finite-$\gamma$ error bounds. So the efficient concentration result \eqref{eq: efficient concentration SSQ, for writing} holds only for $a$ up to some threshold $a_{max}$. Our next target is to understand, in the other extreme, the pre-limit tail for the classical large deviation principle, i.e., when $a$ grows to infinity while $\gamma$ is fixed. Such large deviation results fall under the regime of Theorem \ref{thm: SSQ Tail bound, fourth regime} and the Gaussian tail no longer holds. Instead, we see the tail decays in sub-Poisson style.


\paragraph{Large Deviation} 
Theorem \ref{thm: SSQ Tail bound, fourth regime} presents upper and lower bounds in the large deviation regime when 
$a_\gamma = \Theta(1/\gamma^{\delta})$ with
$\delta > 1/2$, including the case when $a_\gamma \to \infty$ as $\gamma$ is fixed. 
We present two upper bounds on the tail $\mathbb{P}(\tilde{q}>a_\gamma)$ that are obtained via different methods. One is via Wasserstein-$p$ distance bound while the other is via transform method (see Section \ref{sec: proof sketch for tail bounds, SSQ}). Both upper bounds have the same order and are matched by a lower bound that is order wise tight. We include both to offer a more explicit and readily computable upper bound, and to demonstrate the power of Wasserstein-$p$ distance bounds in deriving tail bounds.
In essence, we have an exponent that is sub-Poisson decay in $a_\gamma$, as
\begin{align*}
    \mathbb{P}(\tilde{q}>a_\gamma) =\exp\left(-\Theta \big(\frac{a_{\gamma}}{\sqrt{\gamma}}\ln (a_\gamma\sqrt{\gamma})\big) \right)
\end{align*}

As mentioned in Introduction, in large deviation, abandonment dominates service completions, so the system behaves like an $M/M/\infty$ queue with rates $\lambda$ and $\gamma$, yielding the sub-Poisson tail mirroring its Poisson($\lambda/\gamma$) stationary distribution.

Given the behaviors of tail in two extreme regimes, which include $\gamma\to 0$ with $a$ constant and $a\to \infty$ with $\gamma$ constant, we next bridge these two regimes by allowing both $a_\gamma$ and $\gamma$ to vary simultaneously in the middle ground. In particular, the moderate deviation regime exhibits sub-Gaussian tail behavior with the exponent smoothly transitions from Gaussian to sub-Poisson as the deviation $a_\gamma$ grows.


\paragraph{Moderate Deviation} Now we consider the setting when the deviation $a_{\gamma}$ grows as the parameter $\gamma$ vanishes, i.e., we pick $a_{\gamma} = \Theta(1/\gamma^\delta)$ for some $\delta>0$. 
Theorem~\ref{thm: SSQ Tail bound, second regime} continues to provide an absolute error bound between $\mathbb{P}(\tilde{q} > a_{\gamma})$ and $\Phi^c(a_{\gamma})$, in line with Theorem~\ref{thm: SSQ Tail bound, first regime} when $ \delta\in(0,1/2-\alpha)$. 
In particular, we still have a Gaussian tail in the pre-limit even though the exponent is now not a simple $-x^2/2$ but also has dependence on $\gamma$. 

Unlike the constant deviation case, if we consider the limit as $\gamma \to 0$, note that the tail of the limiting Gaussian also goes to zero since the deviation $a_{\gamma}$ depends on $\gamma$. So, here it is more appropriate to think of convergence in terms of $  \lim_{\gamma \to 0} \frac{\mathbb{P}(\tilde{q} > a_{\gamma})}{\mathbb{P}(Z > a_{\gamma})} {\longrightarrow}1$ instead of $\lim_{\gamma \to 0}\left|\,\mathbb{P}(\tilde{q} > a_{\gamma}) - \mathbb{P}(Z > a_{\gamma})\,\right|$. From Theorem~\ref{thm: SSQ Tail bound, second regime} and using the Mills ratio for the tail of standard Gaussian, we have  that 
\begin{align}
    \left|\, \frac{\mathbb{P}(\tilde{q} > a_{\gamma})}{\mathbb{P}(Z > a_{\gamma})} - 1 \right| \leq \mathcal{O}(a_\gamma^3\sqrt{\gamma})\exp(\mathcal{O}(a_\gamma^3\sqrt{\gamma}))\label{eq: efficient concentration for SSQ, second regime, for writing}
\end{align}

Thus, we have that as $\gamma\to0$, the tail bound of scaled queue lengths converges to that of Gaussian in the sense $  \lim_{\gamma \to 0} \frac{\mathbb{P}(\tilde{q} > a_{\gamma})}{\mathbb{P}(Z > a_{\gamma})} {\longrightarrow}1$ if $\delta\in[0,\min \{1/2 - \alpha, 1/6\})$.
Note that we need $\delta <1/6$ due to the exponent term $a^3$. The right hand side of \eqref{eq: efficient concentration for SSQ, second regime, for writing} has $\mathcal{O}(a_\gamma^3\sqrt{\gamma})$ which is of order $\mathcal{O}({\gamma}^{\frac{1}{2}-3\delta})$, and it goes to zero only when $\delta<1/6$.
Thus, we develop efficient concentration in the regime $\delta\in[0,\min \{1/2 - \alpha, 1/6\})$. 
Also note that within this regime, the rate of convergence to the limiting Gaussian is now much slower than $\sqrt{\gamma}\ln(1/\gamma) $ that we saw in the constant deviation regime. 
Methodologically, these results 
are again obtained via Wasserstein-$p$ distance bound as in the constant deviation regime.

Meanwhile, in the regime when $\delta>1/6$, the bound \eqref{eq: efficient concentration for SSQ, second regime, for writing} does not provide a nontrivial lower bound. So, in Theorem \ref{thm: SSQ Tail bound, second regime}, we present a lower bound of the form 
\begin{align*}
\mathbb{P}(\tilde{q} >a_\gamma)\geq \Omega\left(\exp (-a_\gamma^2/2-\mathcal{O}(a_\gamma^3\sqrt{\gamma}))\right).
\end{align*}
This lower bound is applicable for the entire regime $\delta \in (0,1/2)$, and is obtained using the transform method to complement the Wasserstein-$p$ distance method.






Continuing further, for $\delta\in[1/2-\alpha,1/2]$, 
Theorem~\ref{thm: SSQ Tail bound, third regime} provides sub-Gaussian upper and lower bounds. Note that while one has Gaussian-type tails here, unlike the previous regimes, we do not have a limit theorem of ratio converging to $1$. Again we obtain two upper bounds that are both of the same order, similar to large deviation regime, and close it with a lower bound that is exponentially tight. 

The whole moderate regime serves as a smooth transition in the exponent of tail from Gaussian $a^2$ for constant deviation to sub-Poisson $\frac{a}{\sqrt{\gamma}}\log(1+a\sqrt{\gamma})$ for large deviation. This transition is smooth in the sense that when $\delta=1/2$, the results from both the moderate and large deviations are applicable and they have the same scaling. 



We have now covered all regimes of deviation $a_\gamma$ as a function of $\gamma$ and have depicted the full picture of tail bounds for SSQ in the heavily overloaded regime as in Figure \ref{fig: Phase Transition Diagram}. We conclude this section by emphasizing the above pre-limit tail bounds in Theorem \ref{thm: SSQ Tail bound} can directly imply the limit theorems in Proposition \ref{cor: Deviation Principle for M/M/1+M}.

\subsection{Wasserstein-$p$ Distance for SSQ} \label{sec: Wasserstein-$p$ Distance for M/M/1+M}
In above section, we presented tail bounds for SSQ in heavily overloaded regime. The key ingredient to achieve efficient concentration in Theorem \ref{thm: SSQ Tail bound, first regime} is via bounding Wasserstein-$p$ distance. Such connection will be established formally in Section \ref{sec: survey of methods} and with more details in Section \ref{sec: proof sketch for tail bounds, SSQ}. In this section, we first present our results for Wasserstein-$p$ distance bounds for SSQ.

\begin{theorem} \label{thm: Gaussian Wasserstein-$p$ upper bound for M/M/1+M}
    Consider SSQ in heavily overloaded regime \ref{ass:heavy_overload} and $\gamma\leq \gamma_0$, \eqref{eq: gamma assumption for SSQ W-p}, the Wasserstein-$p$ distance, between law $\mathcal{L}(\tilde{q})$ and standard normal distribution $\mathcal{L}(Z):=\mathcal{N}(0,1)$ satisfies:
\begin{subnumcases}{\mathcal{W}_p(\mathcal{L}(\tilde{q}), \mathcal{L}(Z)) \leq}
    D_{1,\lambda,\mu} \cdot p\sqrt{\gamma} & $p \in [1, 1/\gamma^{1-2\alpha-2\epsilon}]$ \label{eq: Stein Wasserstein-$p$ upper p's function}\\
    D_{2,\lambda,\mu} \cdot \sqrt{p} & $p \in (1/\gamma^{1-2\alpha}, D_{4,\lambda,\mu}/\gamma)$ \label{eq: Triangle Wasserstein-$p$ upper p's function, midlle}\\
    D_{3,\lambda,\mu} \cdot \frac{p\sqrt{ \gamma}}{\log(1+\gamma p/\lambda)} & $p \in [D_{4,\lambda,\mu}/\gamma, \infty)$ \label{eq: Triangle Wasserstein-$p$ upper p's function}
\end{subnumcases}
\begin{subnumcases}{\mathcal{W}_p(\mathcal{L}(\tilde{q}), \mathcal{L}(Z)) \geq}
    C'_{7, \lambda, \mu} \cdot \gamma \exp\left(-\frac{\lambda/\mu - \ln(\lambda/\mu) -1 }{p\gamma/\mu}\right) & $p \in [1, D_{5,\lambda,\mu}
    /\gamma)$ \label{eq: Optimal Coupling Wasserstein-$p$ lower bound p's function}\\
    D_{4,\lambda,\mu}' \cdot \frac{p\sqrt{\gamma}}{\log(1+\gamma p/\lambda)} & $p \in [D_{5,\lambda,\mu}/\gamma, \infty)$\label{eq: Triangle Wasserstein-$p$ lower bound p's function}
\end{subnumcases}

Where $\epsilon \in (0,1/2-\alpha)$ is any constants. All other constants involved are defined explicitly in \eqref{eq: constants for SSQ Wasserstein-$p$ bounds}, \eqref{eq: Wasserstein-$p$ bound, first regime}, \eqref{eq: Wasserstein-$p$ bound, second regime}, \eqref{eq: Wasserstein-$p$ bound, third regime} and \eqref{eq: D_4, lambda, mu SSQ}.
\end{theorem}

For fixed $p$, upper bound \eqref{eq: Stein Wasserstein-$p$ upper p's function} shows that $\lim_{\gamma\to 0} W_p(\mathcal{L}(\tilde{q}), \mathcal{L}(Z)) =0$. From \cite[Theorem~6.9]{Villani2009}, convergence in Wasserstein-$p$ metric implies convergence of distributions and convergence up to $p$-th moments. Thus it recovers the heavily overloaded weak convergence result \cite{ward2003diffusion}. It also implies convergence of moments up to order $p$. Moreover, the bound \eqref{eq: Stein Wasserstein-$p$ upper p's function} yields explicit pre-limit convergence rate $\mathcal{O}(\sqrt{\gamma})$ for $p$ fixed in Wasserstein-$p$ metric. Therefore, it implies convergence rate for the more commonly used metrics such as Wasserstein-$1$ distance in queueing literature (see Section \ref{sec: literature review} for survey of queueing literature). 

We highlight that when $p$ is fixed and $\gamma$ small enough, upper bound \eqref{eq: Stein Wasserstein-$p$ upper p's function} is of order $\mathcal{O}(p\sqrt{\gamma})$. This exact dependency $\mathcal{O}(p\sqrt{\gamma})$ on both $(p, \gamma)$ is crucial to achieve Theorem \ref{thm: SSQ Tail bound, first regime} and \ref{thm: SSQ Tail bound, second regime}. In order to achieve the Gaussian-decay in efficient concentration, one will need $\|\tilde{q}-Z\|_{L^p}$ to scale as $\mathcal{O}(\sqrt{p})$. In the first regime of $p$, the bound \eqref{eq: Stein Wasserstein-$p$ upper p's function} provides such $\mathcal{O}(\sqrt{p})$ since $p< 1/\gamma$.
 Meanwhile, to ensure a convergence rate to Gaussian in efficient concentration whose dominating order is $\mathcal{O}(\sqrt{\gamma})$ (ignoring the logarithmic factor), one needs such $\mathcal{O}(\sqrt{\gamma})$ dependency on the limit parameter. Thus, the order $\mathcal{O}(p\sqrt{\gamma})$ provides the right balance between $p$ and $\gamma$ for $p$ up to $1/\gamma^{1/2- \alpha}$, which is the regime of $p$ we will use to achieve efficient concentration in Theorem \ref{thm: SSQ Tail bound, first regime} and \ref{thm: SSQ Tail bound, second regime}.
Similarly, in CLT literature, a dependency $\mathcal{O}(p/\sqrt{n})$ in Wasserstein-$p$ bounds is established in \cite{fang2022wasserstein,austern2022efficient} for sums of $n$ weakly dependent random variables, which is also used to achieve efficient concentration results. A more straightforward connection between this joint dependency on $(p, \gamma)$ and efficient concentration will be provided in Section \ref{sec: survey of methods}.

Meanwhile, the Wasserstein-$p$ bounds translate directly into moment growth bounds for $\tilde{q}$. In the regime $p\lesssim 1/\gamma$, our Wasserstein-$p$ bound is $\mathcal{O}(\sqrt{p})$. Combining this with the triangle inequality and $\|Z\|_{L^p} = \mathcal{O}(\sqrt{p})$, we have $\|\tilde{q}\|_{L^p}$ scales as $\mathcal{O}(\sqrt{p})$. This $\sqrt{p}$ growth of $L^p$ norm is exactly what yields the sub-Gaussian exponents in Theorem \ref{thm: SSQ Tail bound, second regime} and \ref{thm: SSQ Tail bound, third regime}. For large $p$, Wasserstein-$p$ bound scales as $\mathcal{O}(p/\log(p))$, suggesting that $\|\tilde{q}\|_{L^p} = \mathcal{O}(p/\log(p))$. This is consistent with the Poisson-like tail behavior in the large deviation regime of Theorem \ref{thm: SSQ Tail bound, fourth regime}. In the proof (see Section \ref{sec: proof sketch for tail bounds, SSQ}), we will choose $p$ as a function of deviation $a$, so these two distinct dependencies on $p$ translate into a three stage behavior, i.e., Gaussian for near constant deviation, sub-Gaussian in the moderate range, and sub-Poisson in the large deviation regime. This aligns with Figure \ref{fig: Phase Transition Diagram} and the philosophy "near-constant is Gaussian, far-from-constant is Poisson".



As stated above, different dependency on $(p,\gamma)$ in different ranges of $p$ is crucial to achieve tight tail bounds in different deviation regimes.
Consequently, we use different techniques here to obtain upper and lower bounds for different $p$ ranges. Upper bound \eqref{eq: Stein Wasserstein-$p$ upper p's function} is derived from Stein's method. \eqref{eq: Triangle Wasserstein-$p$ upper p's function, midlle}, \eqref{eq: Triangle Wasserstein-$p$ upper p's function} and \eqref{eq: Triangle Wasserstein-$p$ lower bound p's function} are obtained via triangle inequality combined with moment control. Lower bound \eqref{eq: Optimal Coupling Wasserstein-$p$ lower bound p's function} is from quantile coupling. 
Note that we achieve the lower bounds in the entire range of $p$ by two different techniques, which allows us to examine the tightness of upper bounds in all regimes.
See Section \ref{sec: proof sketch for Wasserstein-$p$ bounds, SSQ} for more details on the proof techniques and their strengths and limitations.



\subsection{Tail Bounds for Load Balancing with Abandonment: Join-The-Shortest-Queue} \label{sec: Tail bounds for JSQ}
In this section, we study pre-limit tail bounds for the load balancing system under JSQ policy, recalling that its dynamic is described in \eqref{eq: generator for join shortest queue}. We first present the tail bounds for the stationary distribution of the normalized queue length vector. According to the Cramér-Wold Theorem \cite{samanta1989non}, convergence in distribution of random vector $\tilde{\mathbf{q}} \in \mathbb{R}^n$ is equivalent to the convergence of all linear combinations of their components, i.e., $\langle \tilde{\mathbf{q}}, \boldsymbol{\phi} \rangle$ for all unit-length $\boldsymbol{\phi}$. Thus, we study pre-limit tail bounds for the projection of the $\tilde{\mathbf{q}}$ onto an arbitrary test direction $\boldsymbol{\phi} \in \mathbb{R}^n$ with unit norm. Similar to SSQ, we summarize the tail bounds for different regimes of deviation $a$.

\begin{theorem} [Tail Bounds for Load Balancing Policy] \label{thm: JSQ Tail bound} 
Consider a test direction $\boldsymbol{\phi}\in\mathbb{R}^n$ with unit norm $\| \boldsymbol{\phi} \| = 1$. Under the heavily overloaded regime \ref{ass:heavy_overload} with $\gamma\leq\gamma_2$ as defined in \eqref{eq: gamma assumption for JSQ tail bound}, the following concentration inequalities hold for the normalized queue length vector $\tilde{\mathbf{q}}:=n\sqrt{\gamma}/\sqrt{\lambda}(\mathbf{q}-\frac{\lambda-\mu}{n\gamma}\boldsymbol{1})$. We distinguish two cases depending on the alignment of $\boldsymbol{\phi}$ with the all-ones vector $\boldsymbol{1}$. When $\langle \boldsymbol{\phi},\boldsymbol{1}\rangle \neq 0$, we have the following results.
    \begin{enumerate}[label=(\alph*), ref=\thetheorem~(\alph*)]
        \item \label{thm: JSQ Tail bound, first regime}
        For fixed deviation $a$ and $\gamma \leq \gamma_a$ (see \eqref{eq: gamma_a definition JSQ first regime}), there exists a constant $B_{1,\lambda,\mu,n}$ (see \eqref{eq: B1, B2 for JSQ tail bound}), independent of $\gamma$ and $a$, such that:
        \begin{align*}
            \left| \mathbb{P}\big( \langle \tilde{\mathbf{q}}, \boldsymbol{\phi} \rangle > a \big) 
            - \mathbb{P}\big( Z \cdot \langle \boldsymbol{\phi}, \boldsymbol{1} \rangle > a \big) \right|
            &\leq \frac{e^3 B_{1,\lambda,\mu,n}}{\sqrt{2\pi} |\langle \boldsymbol{\phi}, \boldsymbol{1} \rangle|} \cdot \sqrt{\gamma} \left[ \left( \frac{a^2}{2 \langle \boldsymbol{\phi}, \boldsymbol{1} \rangle^2} + \ln \frac{1}{\sqrt{\gamma}} \right)^2 + \sqrt{2\pi} \right] \\
            &\qquad \cdot \exp\left( -\frac{a^2}{2 \langle \boldsymbol{\phi}, \boldsymbol{1} \rangle^2} \right)
        \end{align*}
        \item \label{thm: JSQ Tail bound, second regime}
        For deviations $2 \leq a_\gamma \leq E_{Tail,1,u}/\gamma^\delta$, with $\delta \leq \min\{1/4-\alpha/2,\,1/6\}$ (see \eqref{eq: a's regime for MDP JSQ}), the following bound holds:
        \begin{align*}
            &\left|\,\mathbb{P}\big(\langle \tilde{\mathbf{q}},\boldsymbol{\phi}\rangle > a_\gamma\big) 
            - \mathbb{P}\big(Z\cdot\langle\boldsymbol{\phi},\boldsymbol{1}\rangle > a_\gamma\big)\,\right| \\
            &\quad\leq \frac{e B_{1,\lambda,\mu,n}}{\sqrt{2\pi}\,|\langle\boldsymbol{\phi},\boldsymbol{1}\rangle|}\cdot \sqrt{\gamma}\,
            \left[\left(\frac{a_\gamma^2}{2\langle\boldsymbol{\phi},\boldsymbol{1}\rangle^2} + \ln\frac{1}{\sqrt{\gamma}}\right)^2 + \sqrt{2\pi}\right] \\
            &\qquad\qquad \cdot \exp\left\{
                -\frac{a_\gamma^2}{2\langle\boldsymbol{\phi},\boldsymbol{1}\rangle^2}
                \left[1 - e B_{1,\lambda,\mu,n}\left(
                    \frac{a_\gamma^4}{2\langle\boldsymbol{\phi},\boldsymbol{1}\rangle^4}\sqrt{\gamma}
                    + \frac{2\sqrt{\gamma}\ln^2(1/\sqrt{\gamma})}{a_\gamma}
                \right)\right]
            \right\}
        \end{align*}
        \item \label{thm: JSQ Tail bound, third regime}
        For deviation $a_\gamma = 1/\gamma^\delta$, $\delta \in [\min\{1/4-\alpha/2,\,1/6\},\,\infty)$, there exists a constant $c$ (see \eqref{eq: constant c in subweibull regime, JSQ}) such that:
        \begin{align*}
            \mathbb{P}\big(\langle \tilde{\mathbf{q}},\boldsymbol{\phi}\rangle > a_\gamma\big)
             \leq  \mathbb{P}\big(Z\cdot\langle\boldsymbol{\phi},\boldsymbol{1}\rangle > a_\gamma\big) +
           \frac{e}{2\sqrt{2\pi}}\,a_\gamma\,\exp\left(-\frac{a_\gamma^2}{8}\right)
            + \exp\left(-c\,a_\gamma^{\min\{1/(4\delta)+1/2,\,2\}}\right)
        \end{align*}
        \end{enumerate}
        Alternatively, if $\langle \boldsymbol{\phi},\boldsymbol{1}\rangle = 0$, then the following holds. \begin{enumerate}[label=(\alph*), ref=\thetheorem~(\alph*),start=4 ]
            \item \label{thm: JSQ Tail bound, orthogonal case} For $a_\gamma \geq E_{Tail, 2, u} \cdot \sqrt{\gamma}$ (see \eqref{eq: a's regime for worst case direction JSQ}), there exists constant $G_{\lambda,\mu,n}$ defined in \eqref{eq: tail bound for worst case direction}, such that:
            \begin{align*}
                \mathbb{P}(\langle \tilde{\mathbf{q}},\boldsymbol{\phi}\rangle > a_\gamma) &\leq \exp\left( -G_{\lambda,\mu,n} a_\gamma^{1/2}/\gamma^{1/4}\right)
            \end{align*}
        \end{enumerate}
\end{theorem}
\noindent 

\paragraph{Constant $\&$ Moderate Deviation} First let's consider constant and moderate deviation regimes in Theorem \ref{thm: JSQ Tail bound, first regime} and \ref{thm: JSQ Tail bound, second regime}. The upper bounds in this part compare the tail probability of $\tilde{\mathbf{q}}$ and that of a Gaussian random variable times all-one vector $Z \cdot \langle \boldsymbol{\phi}, \boldsymbol{1} \rangle$. It is directly implied by the bounds and Mills ratio \eqref{eq: Mills Ratio}that
\begin{align*}
    \mathbb{P}(\langle \tilde{\mathbf{q}}, \boldsymbol{\phi} \rangle > a_\gamma) \leq
    [1+ \mathcal{O}(a_\gamma^5\sqrt{\gamma})\exp(\mathcal{O}(a_\gamma^5\sqrt{\gamma}))]\,\Phi^c(\frac{a_\gamma}{\langle \boldsymbol{\phi},\boldsymbol{1}\rangle}).
\end{align*}
Such Gaussian decay w.r.t $a$ and rate of convergence w.r.t $\gamma$ in Theorem \ref{thm: JSQ Tail bound, first regime} establish the efficient concentration inequality in \eqref{eq: main contribution JSQ constant}. 
For the absolute difference in \ref{thm: JSQ Tail bound, first regime}, similar to SSQ, due to lattice distribution of $\tilde{\mathbf{q}}$, the Gaussian exponent $\exp(-a^2/(2\langle \boldsymbol{\phi},\boldsymbol{1}\rangle^2))$ is necessary and tight in terms of deviation $a$. In terms of $\gamma$, the decaying $\sqrt{\gamma}\ln^2(1/\gamma)$ gives convergence rate towards a single Gaussian times $\boldsymbol{1}$. We note that the convergence rate is between $\tilde{\mathbf{q}}$ and $Z\cdot \boldsymbol{1}$ and thus it provides error bound for state-space collapse (SSC) phenomenon under JSQ policy. This SSC phenomenon indicates that under limit, as $\gamma \to 0$, all components of $\tilde{\mathbf{q}}$ are approximately equal.

Now we consider the moderate deviation regime where $a_\gamma$ grows with vanishing $\gamma$, i.e., we pick $a_{\gamma} = \Theta(1/\gamma^\delta)$ for some $\delta>0$. The SSC and the efficient concentration inequality can be extended from \ref{thm: JSQ Tail bound, first regime} to moderate deviation \ref{thm: JSQ Tail bound, second regime}, with $\delta$ restricted to $[0, \min\{1/4 - \alpha/2, 1/10\})$. Similar to SSQ, we derive relative ratio here,
\begin{align*}
    \left|\, \frac{\mathbb{P}(\langle \tilde{\mathbf{q}} , \boldsymbol{\phi} \rangle
    > a_{\gamma})}{\mathbb{P}(Z \cdot \langle \boldsymbol{\phi}, \boldsymbol{1} \rangle
    > a_{\gamma})} - 1 \right| \leq \mathcal{O}(a_\gamma^5\sqrt{\gamma})\exp(\mathcal{O}(a_\gamma^5\sqrt{\gamma})).
\end{align*}
Taking limit as $\gamma \to 0$, we have that the tail probability of $\tilde{\mathbf{q}}$ converges to that of $Z \cdot \langle \boldsymbol{\phi}, \boldsymbol{1} \rangle$ for all deviations $a_\gamma = \Theta(1/\gamma^\delta)$ with $\delta$ in the above range. This establishes the deviation principle for JSQ (see limit results for JSQ in Table \ref{tab:tail-bounds-for-all}).
At boundary, we need $\delta\leq1/10$ because of the pre-exponent term in \ref{thm: JSQ Tail bound, second regime}, which requires $a_\gamma^5 \sqrt{\gamma} = o(1)$ to ensure the pre-exponent does not diverge. 



\paragraph{Large Deviation}
In the large-deviation regime, SSC has weaker effect on tails. To see this, recall the definition of $\mathbf{q}_\perp$ and $\mathbf{q}_\parallel$ as the perpendicular and parallel components of $\mathbf{q}$ with respect to $\boldsymbol{1}$.
 In large deviation, $\mathbf{q}_\perp$ contributes a heavy (sub-Weibull) tail, i.e., \ $\exp(-c\,a^{1/2}/\gamma^{1/4})$, whereas $\mathbf{q}_\parallel$ remains sub-Poisson decay. The sub-Weibull term dominates and masks the Poisson decay. Therefore, unlike SSQ, the exponent in large deviation gradually transitions from sub-Gaussian to sub-Weibull. Theorem \ref{thm: JSQ Tail bound, third regime} depicts this smooth transition,
 \begin{align*}
    \mathbb{P}(\langle \tilde{\mathbf{q}},\boldsymbol{\phi}\rangle > a_\gamma) \leq \exp\!\Big(-\Omega\!\Big(a_\gamma^{\min\!\big\{\tfrac{1}{4\delta}+\tfrac{1}{2},\,2\big\}}\Big)\Big),
 \end{align*}
 for $\delta \in [\min\{1/4 - \alpha/2, 1/6\}, \infty)$.
 
 This discrepancy between parallel and perpendicular components stems from the form of SSC we obtain for JSQ (Theorem~\ref{thm: SSC JSQ}). In particular, Theorem~\ref{thm: SSC JSQ} controls $[\mathbb{E}\|\mathbf{q}_\perp\|^{p}]^{1/p}$ only at order $\mathcal{O}(p^{2})$, in contrast with the $\mathcal{O}(p/\log p)$ behavior along the parallel component $\|\mathbf{q}_\parallel\|$. Another feature for large deviation, compared with first two regimes, is that we currently lack order-wise tight lower bounds, due to the lack of matching lower bounds on $\mathbb{E}\|\mathbf{q}_\perp\|^{p}$. Additionally, for the phase transition in the upper bound, compared with SSQ, the phase transition for exponent from sub-Gaussian to sub-Weibull, occurs earlier for JSQ, at $\delta= \tfrac{1}{4}-\tfrac{\alpha}{2}$ (vs.\ $\delta=\tfrac{1}{2}-\alpha$ for SSQ). 
 We believe that the sub-Weibull tail and the phase transition are not tight and leave them as open questions. Nevertheless, Theorem~\ref{thm: JSQ Tail bound} is the first to provide pre-limit tail bounds for JSQ in the heavily overloaded regime with all ranges of deviations.


According to the discrepancy between $\mathbf{q}_\perp$ and $\mathbf{q}_\parallel$, there is a worst-/best-case phenomenon regarding the test vector $\boldsymbol{\phi}$ in large deviation. The more this test vector $\boldsymbol{\phi}$ points towards $\boldsymbol{1}$, the tail behaves more like sub-Poisson. Conversely, the closer $\boldsymbol{\phi}$ is to being orthogonal to $\boldsymbol{1}$, the tail behaves more like sub-Weibull. Particularly when $\boldsymbol{\phi} \perp \boldsymbol{1}$, the tail bound  is purely sub-Weibull, as shown in Theorem \ref{thm: JSQ Tail bound, orthogonal case}, which holds for any deviation scale $a_\gamma$ given $\gamma$ a small quantity. 
Meanwhile, the sub-Weibull part of bound in Theorem \ref{thm: JSQ Tail bound, third regime}, i.e., $\exp(-\Omega(a_\gamma^{\min\{1/(4\delta)+1/2,2\}}))$, is invariant of $\boldsymbol{\phi}$, and is overly conservative when $\boldsymbol{\phi}$ is close to $\boldsymbol{1}$.
One of our efforts here is to improve the conservative exponent when $\boldsymbol{\phi}$ is mildly deviating from $\boldsymbol{1}$, which exhibits sub-Poisson tail similar to SSQ. 

\begin{proposition}[Refinement for Large Deviation]
    \label{pro: refinement JSQ}
    Under the assumption \ref{ass:heavy_overload}, with $\gamma\leq \gamma_1$ defined in \eqref{eq: gamma assumption for JSQ W-p}, and same notations as in Theorem \ref{thm: JSQ Tail bound}. Let $\boldsymbol{\phi}\in\mathbb{R}^n$ with $\|\boldsymbol{\phi}\|=1$ and $\boldsymbol{\phi}\geq 0$ coordinate-wise. Then, for any $a \geq e^2n\sqrt{\lambda}/\sqrt{\gamma}$, we have
\begin{align*}
    \mathbb{P}(\langle \tilde{\mathbf{q}},\boldsymbol{\phi}\rangle > a) \leq \exp\left( - \left(\frac{1}{\langle \boldsymbol{\phi}, \boldsymbol{1}\rangle} \frac{\sqrt{\lambda}}{2n\sqrt{\gamma}}a + \frac{\lambda-\mu}{n\gamma} \right) \log\left( \frac{1}{\langle \boldsymbol{\phi}, \boldsymbol{1}\rangle} \frac{\sqrt{\gamma}}{n\sqrt{\lambda}}a + \frac{\lambda-\mu}{n\sqrt{\lambda\gamma}}  \right) \right).
\end{align*}
\end{proposition}

We finally point out similar arguments can be applied to deviation $a_\gamma <0$, and the tail bound is similar to the positive deviation case. We omit the details here for brevity.

\subsection{State Space Collapse for JSQ}\label{sec: SSC for JSQ}
In this section, we discuss the State Space Collapse (SSC) phenomenon for JSQ model. We bound the $p$-th moment of the perpendicular component $\mathbf{q}_\perp$ of the queue length vector $\mathbf{q}$ towards $\boldsymbol{1}$ independently of $\gamma$. This bound gives key ingredient for bounding Wasserstein-$p$ distance, and thus for tail bounds and formalization of SSC. Recall that SSC refers to the phenomenon that all queues are approximately equal to each other in the limit, and that $\mathbf{q}_\perp$ and $\mathbf{q}_\parallel$ are the perpendicular and parallel components of $\mathbf{q}$ with respect to $\boldsymbol{1}$. 
The following theorem is the $p$-th moment bound of $\|\mathbf{q}_\perp\|$.
\begin{theorem} [State Space Collapse for JSQ]\label{thm: SSC JSQ}
    For JSQ under in heavily overloaded condition \ref{ass:heavy_overload}, the following inequality holds:
    \begin{align*}
        \mathbb{E}[\|\mathbf{q}_{\perp}\|^{p}]^{1/p} &\leq \max\{ E_{n,\lambda,\mu}^{(1)} p^2, E_{n,\lambda,\mu}^{(2)} p \} \leq E_{n,\lambda,\mu} p^2 \end{align*}
    Where $\mathbf{q}_{\perp} = \mathbf{q} - \frac{1}{n}(\sum_{i=1}^n q_{i})\boldsymbol{1}$.
    $E_{n,\lambda,\mu}^{(i)}, i=1,2$ and $E_{n,\lambda,\mu}$ are constants independent of $p,\gamma$, their closed forms are given in \eqref{eq: p-moment for q_perp, state space collapse}.
\end{theorem}

Intuitively, load balancing (JSQ) tries to eliminate discrepancy among different queues by routing the arrival towards the shortest queue. Meanwhile, abandonment tends to force the larger queues to decrease more than the smaller ones. These forces drive all queues, from both increment and decrement sides, to approximately equalize. 
Theorem \ref{thm: SSC JSQ} upper bounds the $p$-th moment of $\mathbf{q}_{\perp}$ independently of $\gamma$, while Theorem \ref{thm: Wasserstein-$p$ distance JSQ} implies $p$-th moment of $\sum_i q_{i}$ is of order $1/\gamma^{p/2}$. Thus the perpendicular component $\mathbf{q}_\perp$ becomes negligible as $\gamma\leq\gamma_1$ small enough. Theorem \ref{thm: SSC JSQ} therefore indicates that all components of $\mathbf{q}$ are approximately equal to each other, as $\gamma\leq \gamma_1$ small enough. This formalizes the heuristics from force analysis and manifests the SSC phenomenon even when one of the balancing force, abandonment rate is small.

\paragraph{Relation to Literature} Theorem \ref{thm: SSC JSQ} achieves, to the best of our knowledge, the first uniform $L^p$ norm bound on $\|\mathbf{q}_\perp\|$ for continuous time JSQ model with abandonment. A closely related result is in \cite[Theorem~3.2]{jhunjhunwala2023jointheshortestqueueabandonmentcritically}, where the authors establish second moment bound for $\|\mathbf{q}_\perp\|$ in a discrete-time setting. While the settings are not directly comparable, our theorem strengthens the moment control in the continuous-time model by upgrading from $L^2$ norm to general $L^p$ norm with explicit $p$-dependency. However, notice that this does not imply exponential bound, i.e., $\mathbb{E}[\exp(\theta\|\mathbf{q}_{\perp}\|)]$ is uniformly bounded for some $\theta>0$ independent of $\gamma$. This is because the Taylor series of exponential, $\mathbb{E}[\exp(\theta\|\mathbf{q}_{\perp}\|)]$ to moment,  $\sum_{p=1}^{\infty} \frac{\theta^p}{p!} \|\mathbf{q}_{\perp}\|_{L^p}^p$ diverges due to $p^2$ dependency. By contrast, \cite{hurtado2020transform} establishes such exponential bound for classical JSQ model without abandonment. However, their arguments involve Hajek's geometric tail for hitting time \cite{hajek1982hitting}, which requires negative drift when $\|\mathbf{q}_{\perp}\|$ is beyond some compact set, as well as almost surely bounded jump size. Due to technical challenges from the abandonment, we cannot identify compact set for $\|\mathbf{q}_\perp\|$ outside which the drift of $\|\mathbf{q}_{\perp}\|$ is always negative and thus their arguments do not apply.



\subsection{Wasserstein-$p$ Distance for JSQ} \label{sec: Wasserstein-$p$ distance for JSQ}
As a key stepping stone to tail bounds in Theorem \ref{thm: JSQ Tail bound}, we next present bounds on the Wasserstein-$p$ between laws of two scalar-valued random variables. Recalling that $\tilde{\mathbf{q}} := \frac{n\sqrt{\gamma}}{\sqrt{\lambda}} (\mathbf{q} - \frac{\lambda-\mu}{n\gamma}\boldsymbol{1})$, our goal is to establish the following bound
\begin{align*}
    \mathcal{W}_p\left (\mathcal{L} \big( \langle \tilde{\mathbf{q}}, \boldsymbol{\phi}\rangle\big), \mathcal{L}\big(\langle \boldsymbol{\phi}, Z\cdot\boldsymbol{1} \rangle\big)\right ), \quad \forall \boldsymbol{\phi} \in \mathbb{R}^n, \|\boldsymbol{\phi}\|=1.
\end{align*}
The reduction from random vector to scalar random variable is from Cramér-Wold Theorem which equalizes convergence of a random vector $\tilde{\mathbf{q}}$ to convergence of all linear combinations of their components. 
Thus convergence in above metric will imply convergence in distribution of $\tilde{\mathbf{q}}$ to $Z\cdot \boldsymbol{1}$. The all-one vector structure of the limit further manifests the SSC phenomenon.

We first study the Wasserstein-$p$ distance between two random variables: $\tilde{q}_{\Sigma}:= \sqrt{\gamma/\lambda}(\sum_i q_{i} - \frac{\lambda-\mu}{\gamma})$ and standard normal $Z\sim \mathcal{N}(0,1)$. Similar to SSQ, we present the dominating terms in the upper and lower bounds for $p$ in different ranges. 

\begin{theorem}
    \label{thm: Wasserstein-$p$ distance JSQ}
    \textnormal{(Wasserstein-$p$ distance for JSQ)} In the heavily overloaded regime \ref{ass:heavy_overload}, with $\gamma\leq\gamma_1$ \eqref{eq: gamma assumption for JSQ W-p}, the Wasserstein-$p$ distance between law $\mathcal{L}(\tilde{q}_{\Sigma})$ and standard normal distribution $\mathcal{L}(Z):=\mathcal{N}(0,1)$ is bounded in:
    \begin{subnumcases}{\mathcal{W}_p(\mathcal{L}(\tilde{q}_{\Sigma}), \mathcal{L}(Z)) \leq E'_{\lambda,\mu,n} \cdot}
     p\sqrt{\gamma}, & if $p \in (1, 1/\gamma^{\delta}]$ \label{eq: Stein Wasserstein-$p$ JSQ p's func}\\
     \sqrt{p} , & if $p \in [1/\gamma^{1/2-\alpha}, 1/\gamma)$ \label{eq: Stein Wasserstein-$p$ JSQ p's func, mid}\\
     \frac{p\sqrt{\gamma}}{\log(1+\frac{n\gamma}{\lambda}p)}, & if $p \in [1/\gamma, \infty)$ \label{eq: Stein Wasserstein-$p$ JSQ p's func, large}
    \end{subnumcases}
    \begin{subnumcases}{\mathcal{W}_p(\mathcal{L}(\tilde{q}_{\Sigma}), \mathcal{L}(Z)) \geq}
        C_{OT,\lambda,\mu}' \cdot \sqrt{\gamma} \exp\left(-\frac{\lambda }{p\gamma}\right), & $p \in [1, D_{5,\lambda,\mu}/\gamma]$\label{eq: OT Wasserstein-$p$ lower bound p's function, JSQ}\\
        D_{4,\lambda,\mu}' \cdot \frac{p\sqrt{\gamma}}{\log(1+\gamma p/\lambda)} & $p \in [D_{5,\lambda,\mu}/\gamma, \infty)$\label{eq: Triangle Wasserstein-$p$ lower bound p's function, JSQ}
    \end{subnumcases}
    Where $\delta \in(0,1/2-\alpha)$ is any small constant, and all constants independent of $p,\gamma$, with closed forms in \eqref{eq: D_4, lambda, mu SSQ}, \eqref{eq: constants for Wasserstein-$p$ JSQ}.
\end{theorem}

For $p$ fixed, we use upper bound in the first regime above \eqref{eq: Stein Wasserstein-$p$ JSQ p's func} combined with SSC Theorem \ref{thm: SSC JSQ}. When $\gamma\to 0$, we can analyze the order of the Wasserstein-$p$ bound as follows.
\begin{align*}
    \mathcal{W}_p&\left (\mathcal{L} \big( \langle \tilde{\mathbf{q}}, \boldsymbol{\phi}\rangle\big), \mathcal{L}\big(\langle \boldsymbol{\phi}, Z\cdot\boldsymbol{1} \rangle\big)\right ) \overset{(a)}{\leq} |\langle \boldsymbol{\phi}, \boldsymbol{1}\rangle|\cdot \mathcal{W}_p(\mathcal{L}(\tilde{q}_{\Sigma}), \mathcal{L}(Z)) + \frac{n\sqrt{\gamma}}{\sqrt{\lambda}}\|\langle  \mathbf{q}_{\perp}, \boldsymbol{\phi}\rangle\|_{L^p}\\
    &\overset{(b)}{\leq} \sqrt{n}\|\boldsymbol{\phi}\| \mathcal{W}_p(\mathcal{L}(\tilde{q}_{\Sigma}), \mathcal{L}(Z)) + \|\boldsymbol{\phi}\|\cdot\frac{n\sqrt{\gamma}}{\sqrt{\lambda}}\mathbb{E}[\|\mathbf{q}_{\perp}\|^{p}]^{1/p} \leq \mathcal{O}(\sqrt{\gamma})
\end{align*}
$(a)$ follows from triangle inequality and the decomposition $\mathbf{q} = \mathbf{q}_{\parallel} + \mathbf{q}_{\perp}$. $(b)$ follows from Cauchy-Schwarz and the final $\mathcal{O}(\sqrt{\gamma})$ bound holds from order analysis on \eqref{eq: Stein Wasserstein-$p$ JSQ p's func} and SSC \ref{thm: SSC JSQ}. Thus we have convergence in Wasserstein metric as $\gamma\to 0$. From Cramér-Wold theorem, we have convergence in distribution: $\frac{n\sqrt{\gamma}}{\sqrt{\lambda}} (\mathbf{q} - \frac{\lambda-\mu}{n\gamma}\boldsymbol{1}) \xrightarrow{d} Z\cdot \boldsymbol{1}$ as $\gamma\to 0$. 
This limit result coincides with the diffusion-limit result in \cite{jhunjhunwala2023jointheshortestqueueabandonmentcritically}. However, our work studies a continuous-time Markov process, while \cite{jhunjhunwala2023jointheshortestqueueabandonmentcritically} consider a discrete-time model. Here, the pre-limit rate of convergence towards $Z\cdot \boldsymbol{1}$ is also $\mathcal{O}(\sqrt{\gamma})$, same as SSQ. We also emphasize that the joint order $\mathcal{O}(p\sqrt{\gamma})$ is crucial for establishing tail bounds in Theorem \ref{thm: JSQ Tail bound}, whose rationale is similar to SSQ case (see discussion after Theorem \ref{thm: Gaussian Wasserstein-$p$ upper bound for M/M/1+M}).

Similar to SSQ, to achieve a tight Wasserstein-$p$ bound in terms of both $p$ and $\gamma$, we use different approaches for different ranges of $p$. In particular,
we use a mixture of Stein's method, triangle inequality and quantile coupling (see Section \ref{sec: proof sketch}). Consequently, the upper bounds exhibit similar forms as in SSQ, i.e., $\mathcal{O}(p\sqrt{\gamma})$ for small $p$, $\mathcal{O}(\sqrt{p})$ for moderate $p$, and $\mathcal{O}(\frac{p\sqrt{\gamma}}{\log(1+p\gamma)})$ for large $p$. Yet their transition occur halfway earlier than SSQ, at $p=1/\gamma^{1/2-\alpha}$ and $p=1/\gamma$ respectively, instead of $p=1/\gamma^{1-\alpha}$ and $p=1/\gamma$ in SSQ. This early shift is due to an exponential decay on a key ingredient for Wasserstein-$p$ bound, i.e.,
 $\sum_i \mathbb{P}(q_{i}=0)$ (see discussion after Proposition \ref{pro: zero probability sum for JSQ}), whereas in SSQ, the decay is Gaussian (see Lemma \ref{lem: tight bound on P(q_infty = 0)}).
  Such shift ultimately incurs the earlier phase transition in tail bounds in Theorem \ref{thm: JSQ Tail bound} due to the connection between Wasserstein-$p$ distance and tail bounds (see Lemma \ref{lemma: tail bound via wasserstein}). Consequently, we can see a phase transition of tail exponent from sub-Gaussian to sub-Weibull at $\delta = 1/4 - \alpha/2$ in Theorem \ref{thm: JSQ Tail bound}, which is earlier than the phase transition at $\delta=1/2-\alpha$ for SSQ.

The lower bound of Wasserstein-$p$ bounds in Theorem \ref{thm: Wasserstein-$p$ distance JSQ} shares similar form as in SSQ. Note that \eqref{eq: OT Wasserstein-$p$ lower bound p's function, JSQ} has weaker exponential term than that in SSQ \eqref{eq: Optimal Coupling Wasserstein-$p$ lower bound p's function}, which is again due to the exponential decay bound on $\sum_i \mathbb{P}(q_{i}=0)$. Yet it provides impossibility result for improving upper bound \eqref{eq: Stein Wasserstein-$p$ JSQ p's func} or \eqref{eq: Stein Wasserstein-$p$ JSQ p's func, mid} for small $p$ range. Meanwhile, the lower bound for large $p$ range \eqref{eq: Triangle Wasserstein-$p$ lower bound p's function, JSQ} matches the order of upper bound \eqref{eq: Stein Wasserstein-$p$ JSQ p's func, large} to show its tightness.

\section{Our Approach} \label{sec: survey of methods}
Before the proof sketch, we first present preliminary techniques that are referenced in later sections when proving our main results. We will elaborate on our choice of techniques while surveying existing methods for concentration bounds. Specifically, we will discuss their strengths and limitations when bounding tail probability $\mathbb{P}(\tilde{q} > a)$ and highlight the rationale behind our choice in the context of queueing systems with abandonment. For simplicity, our survey focuses on techniques highly related to our work in this section, and we defer a more comprehensive review of related literature to Section \ref{sec: literature review}.

\subsection{Drift Method and Transform Method} Once a moment bound or a moment generating function (MGF) bound is available for the target random variable, the Markov inequality provides a series of fundamental concentration inequalities \cite{boucheron2003concentration}. Specifically, for the random variable $\tilde{q}$ with existing $p$-th moment and MGF, the Markov inequality yields:
\begin{align*}
    \mathbb{P}(|\tilde{q}| > a) &\leq \frac{\mathbb{E}[|\tilde{q}|^p]}{a^p},  \quad \mathbb{P}(|\tilde{q}| > a) \leq e^{-\theta a} \mathbb{E}[e^{\theta |\tilde{q}|}].
\end{align*}
In queueing systems, Lyapunov drift method \cite{EryilmazSrikant2012QSys} and transform method \cite{hurtado2020transform} systematically study moment and MGF bounds for steady-state $\tilde{q}$ in various stochastic systems. Yet Markov inequality with fixed order $p$ or fixed MGF parameter $\theta$ only yields polynomial or sub-exponential decay in deviation $a$, insufficient for efficient concentration like the form \eqref{eq: main contribution, constant deviation}. 

Optimizing $\theta$ in transform method according to deviation $a$ can achieve a sub-Gaussian upper bound of the form $\mathbb{P}(\tilde{q} > a)\leq O(1)e^{-a^2/2(1+o(1))}$ (see Theorem \ref{thm: SSQ Tail bound, third regime}). Paired with a change-of-measure-based lower bound \cite{theodosopoulos2005reversionchernoffbound}, this sub-Gaussian bound captures the Gaussian-type exponent $-\frac{a^2}{2} (1\pm o(1))$ in both upper and lower bounds. Taking the limit $\gamma \to 0$, the upper and lower bounds converge simultaneously to an exact exponent $-a^2/2$. We use this exponential tightness of transform method, to achieve efficient concentration in moderate and large deviation regimes (see Theorem \ref{thm: SSQ Tail bound, third regime} and \ref{thm: SSQ Tail bound, fourth regime}).
Consequently, for SSQ, our pre-limit bounds from transform method imply limit results with forms $\lim \frac{\ln \mathbb{P}(\tilde{q} > a)}{ \ln \mathbb{P}(Y > a)} = 1$ (see the second and third limit with different $Y$ in Corollary \ref{cor: Deviation Principle for M/M/1+M}).

However, using the transform method together with Markov inequality fails to recover the exact Gaussian tail $\Phi^c(a)$. The resulting bound typically misses some pre-exponential factor. In other words, pre-limit bounds obtained via this route cannot show $\lim_{\gamma \to 0} \mathbb{P}(\tilde{q} > a) = \Phi^c(a)$. This issue arises from the nature of Markov inequality.
We use an example with a standard normal $Z\sim \mathcal{N}(0,1)$ to illustrate. Given the $p$-th moment and MGF for $Z$ are $2^{p/2}\Gamma((p+1)/2)/\sqrt{\pi}$ and $e^{\theta^2/2}$, optimizing $p$ or $\theta$ leads to
\begin{align*}
    \mathbb{P}(Z > a) &\leq \inf_{p>0} \frac{2^{p/2}\Gamma((p+1)/2)}{\sqrt{\pi} a^p} \leq \frac{1}{\sqrt{2}} \exp(-a^2/2 + \frac{1}{6a^2}), \text{with } p = a^2 - 1; \\
    \mathbb{P}(Z > a) &\leq \inf_{\theta>0} e^{-\theta a + \theta^2/2} = e^{-a^2/2}, \text{with } \theta = a,
\end{align*}
which inherently miss the pre-exponential factor $1/a$ from Mills ratio \eqref{eq: Mills Ratio}. Meanwhile, from transform method for queueing systems, one starts from pre-limit MGF bounds that typically yields pointwise convergence of MGF, e.g., $\mathbb{E}[e^{\theta \tilde{q}}] \overset{\gamma \to 0}{\longrightarrow} e^{\theta^2/2}$ as in \cite{jhunjhunwala2023jointheshortestqueueabandonmentcritically}. An interchange of limit and $\inf_{\theta>0}$ (if justified) would then lead to the same right hand sides as for the standard normal case above. Thus we face the same issue of missing pre-exponential factor $1/a$. Modulo this missing pre-exponential factor in upper bound, a lower bound via change of measure \cite{theodosopoulos2005reversionchernoffbound} will only deviate further from limit tail probability $\Phi^c(a)$. This mismatch between the pre-limit bounds and limit tail appears in \cite[Theorem~5]{jhunjhunwala2023exponential}, where the authors use Markov inequality with transform method for single-server queue in heavy-traffic. Their upper bound has a pre-exponent error $a$, compared with the exponential limit tail. Thus unlike efficient concentration \eqref{eq: main contribution, constant deviation}, their pre-limit bound cannot recover the limit tail probability directly.

\subsection{Concentration Bounds via Wasserstein-$p$ Distance} 

An alternate approach for showing $\lim_{\gamma \to 0} \mathbb{P}(\tilde{q} > a) = \Phi^c(a)$
is to target at upper bounds in the form $\mathbb{P}(\tilde{q} > a) \leq \Phi^c(a) + o(1)$ directly. Kolmogorov distance $\sup_{a \in \mathbb{R}} |\mathbb{P}(\tilde{q} \leq a) - \mathbb{P}(Z \leq a)|$, well studied for distributional approximation \cite{ross2011fundamental, eichelsbacher2023kolmogorov}, can provide this convergence in terms of limit parameter $\gamma$ towards the Gaussian limit tail. However, it is loose since it does not reflect the decay in $a$. Due to the lattice structure of $\tilde{q}$ (see Section \ref{sec: Tail Bound for M/M/1+M}), there exists an $a$ such that $|\mathbb{P}(\tilde{q} > a) - \mathbb{P}(Z > a)| \gtrsim \exp(-a^2/2)$. It suggests at best a sub-Gaussian decay in $a$ for the upper bound and Kolmogorov distance cannot capture this decay. An appealing approach to achieve the sub-Gaussian decay in $a$ is introduced by recent works in \cite[Appendix B]{austern2022efficient} and \cite[Theorem~2]{fang2022wasserstein}. They connect tail bounds with Wasserstein-$p$ distances, building upon development of Wasserstein-$p$ distance in \cite{bonis2020steins,ledoux2015stein}. The connection is made by adding and subtracting an intermediate Gaussian variable, which is coupled with $\tilde{q}$ via Wasserstein-$p$ distance. Specifically, for any $\rho \in [0,1)$ and $a>0$, we have
\begin{align} 
    \mathbb{P}(\tilde{q} > a) &= \mathbb{P}\left( \tilde{q} - Z + Z > a \right) \leq \mathbb{P}\left( \tilde{q} - Z \geq (1 - \rho) a \right) + \mathbb{P}\left( Z \geq \rho a \right)  \notag\\
    &=\Phi^c(\rho a) + \mathbb{P}\left( \tilde{q} - Z \geq (1 - \rho) a \right) \overset{(a)}{\leq}\Phi^c(\rho a) + \frac{\mathbb{E}[|\tilde{q}-Z|^p]}{((1-\rho) a)^p} \notag\\
    &\overset{(b)}{\leq} \Phi^c(a) + (1-\rho)a \cdot \phi(\rho a ) + [(1-\rho)a]^{-p}\mathcal{W}_p^p(\mathcal{L}(\tilde{q}), \mathcal{L}(Z)). \notag
\end{align}
Where inequality $(a)$ is from Markov inequality with moment bound. Inequality
$(b)$ follows from Taylor Expansion $\Phi^c(\rho a) \leq \Phi^c(a) + (a - \rho a) \sup_{\tilde{x} \in [\rho a, a]} \phi(\tilde{x})$. Note that we choose a coupling $(\tilde{q},Z)$ such that the $L^p$ distance between $\tilde{q}$ and $Z$ is within error $\epsilon$ to the Wasserstein-$p$ distance $\mathcal{W}_p(\mathcal{L}(\tilde{q}), \mathcal{L}(Z))$. Such error $\epsilon$ is arbitrarily small since Wasserstein-$p$ distance is defined as infimum over all couplings. Thus we let $\epsilon \downarrow 0$ to obtain the last line. Meanwhile, similar argument shows the lower bound
\begin{align} 
    \mathbb{P}(Z > a) &\leq \mathbb{P}(\tilde{q} > a) + \mathbb{P}(|\tilde{q} - Z| \geq (1 - \rho) a) + \mathbb{P}(a \leq Z \leq(2-\rho)a) \notag\\
    &\leq \mathbb{P}(\tilde{q} > a) + [(1-\rho)a]^{-p}\mathcal{W}_p^p(\mathcal{L}(\tilde{q}), \mathcal{L}(Z)) + (1-\rho)a \cdot \phi( a) \notag
\end{align}
Together, we have the following lemma connecting tail bounds with Wasserstein-$p$ distance in absolute difference.
\begin{lemma} \label{lemma: tail bound via wasserstein}
For any $\rho \in [0,1)$, $a > 0$, under the same notation as above, we have
\begin{align}
    \left|\mathbb{P}(\tilde{q} > a) - \mathbb{P}(Z > a)\right| \leq (1-\rho)a \cdot \phi(\rho a ) + [(1-\rho)a]^{-p}\mathcal{W}_p^p(\mathcal{L}(\tilde{q}), \mathcal{L}(Z)) \label{eq: Tail bound from Wasserstein-$p$ distance}
\end{align}
\end{lemma}
\noindent 
Empirically from Markov inequality, for tail bounds on $\mathbb{P}(Z>a)$, optimizing over $p$ with moment $\mathbb{E}[|Z|^p]$ leads to tighter pre-exponential factor compared with optimizing over $\theta$ with MGF $\mathbb{E}[e^{\theta Z}]$ (see our illustrative example in Section 5.1). Thus raising to power $p$ instead of exponential form is preferred here, leading to the Wasserstein-$p$ approach instead of Wasserstein-Orlicz approach.

We note that a similar argument would provide a lower tail probability bound $\mathbb{P}(\tilde{q} < -a)$ with the same form of \eqref{eq: Tail bound from Wasserstein-$p$ distance} as below. Thus we can obtain efficient concentration for both upper and lower tails. We focus on upper tail in the whole paper for simplicity, and the same results hold for lower tail.
\begin{align}
    \left|\mathbb{P}(\tilde{q} < -a) - \mathbb{P}(Z < -a)\right| \leq (1-\rho)a \cdot \phi(\rho a ) + [(1-\rho)a]^{-p}\mathcal{W}_p^p(\mathcal{L}(\tilde{q}), \mathcal{L}(Z)) \label{eq: Tail bound from Wasserstein-$p$ distance, lower tail}
\end{align}

Now the bulk of efforts is to achieve tight Wasserstein-$p$ bounds between $\tilde{q}$ and $Z$ for any $p > 1$, and optimizing over the choice $(\rho, p)$ based on the results of $\mathcal{W}_p(\mathcal{L}(\tilde{q}), \mathcal{L}(Z))$. For example, as we show in Section \ref{sec: main result}, once we establish $\mathcal{W}_p(\mathcal{L}(\tilde{q}), \mathcal{L}(Z)) \lesssim p\sqrt{\gamma}$, Lemma \ref{lemma: tail bound via wasserstein} with optimal choice of $(\rho, p)$ leads to efficient concentration in the form \eqref{eq: main contribution, constant deviation}. Notice that the more popular Wasserstein-1 distance in the queueing literature \cite{hurtado2022load,braverman2017steinMPhn} is insufficient since it only yields polynomial decay in $a$ via the above lemma,
\begin{align}
    \left|\mathbb{P}(\tilde{q} > a) - \mathbb{P}(Z > a)\right| \leq (1-\rho)a \cdot \phi(\rho a ) + \frac{\mathcal{W}_1(\mathcal{L}(\tilde{q}), \mathcal{L}(Z))}{(1-\rho)a} \label{eq: Tail bound from Wasserstein-1 distance}.
\end{align}
We thus develop Wasserstein-$p$ bounds with sharp dependence on $(p,\gamma)$ for all $p > 1$. However, for SSQ, the scale $\mathcal{W}_p(\mathcal{L}(\tilde{q}), \mathcal{L}(Z)) \lesssim p\sqrt{\gamma}$ only holds up to moderate deviation regimes. Therefore, efficient concentration is only obtained up to moderate deviation scales (see Theorem \ref{thm: SSQ Tail bound, first regime}, \ref{thm: SSQ Tail bound, second regime}). So we blend with transform method for larger deviation regimes to reveal tail behaviors under all scales (see Theorem \ref{thm: SSQ Tail bound, third regime} and \ref{thm: SSQ Tail bound, fourth regime}).

\subsection{Wasserstein-$p$ Distance via Stein's Method}
 A key method for bounding Wasserstein-$p$ distance between laws of random variable $X$, and $Z \sim \mathcal{N}(0,1)$ is a variant of Stein's method via dynamic formulation \cite{ambrosio2005gradient,bonis2020steins}. Specifically, this variant upper bounds Wasserstein-$p$ distance by what we refer to as a
\textit{generator comparison} below. Before we state the result, we first introduce some notations and assumptions with explanations.  

\begin{assumption}[Stein identity] \label{assumption: Stein identity}
    Given a real-valued random variable $X$ with support $E$, there exists a linear operator $\mathcal{A}$ such that $f|_ E \in \mathcal{D}(\mathcal{A})$ for all test function $f \in \mathcal{C}_b(\mathbb{R})$, $\mathcal{A}f(X)$ is integrable, and
    $\mathbb{E}[\mathcal{A}f(X)] = 0$.
\end{assumption}
We note that operator $\mathcal{A}$ typically has domain $\mathcal{D}(\mathcal{A})$ that is a subset of $\mathbb{R}^E$. Thus we adopt the convention of notation that $\mathcal{A}f(x) := (\mathcal{A}f|_E)(x)$ for $x \in E$, where $f|_E$ is the restriction of $f$ on $E$ (see Section \ref{sec: notation and modeling} for more details).
Consider a CTMC $(X_t)_{t \geq 0}$ with stationary distribution $\mathcal{L}(X)$ and we denote its infinitesimal generator as $\mathcal{A}$. Then we have $\mathbb{E}[\mathcal{A}f(X)] = 0$ for all smooth test functions $f$ since $X$ is stationary, satisfying Assumption \ref{assumption: Stein identity}. This assumption characterizes $X$ via the operator $\mathcal{A}$, and we will compare it with that of $Z$. 
 We next introduce the function class $\mathcal{F}$ that we will work with for the subsequent assumptions.
\begin{definition}[Test functions] \label{assumption: function class}
    We define function class $\mathcal{F} := \{f \in \mathcal{C}^{\infty}(\mathbb{R}): \exists D,m > 0 \text{ s.t. } |f^{(k)}(x)| \leq D^k \sqrt{k!} (1 + |x|)^m, \forall k \geq 1, x \in \mathbb{R}\}$.
\end{definition}
In short, $\mathcal{F}$ is a class of smooth functions with derivatives that are bounded by the growth rate $\mathcal{O}(k^{k/2})$ times a polynomial. We now impose the expansion condition on operator $\mathcal{A}$, using $\mathcal{F}$ as the test function class.
\begin{assumption}[Kramers-Moyal expansion] \label{assumption: Kramers-Moyal expansion}
    Under Assumption \ref{assumption: Stein identity}, we further assume that there exists a family of functions $a^{\mathcal{A}}_k(x)$, such that $\mathcal{A}f(x) = \sum_{k=1}^{\infty} a^{\mathcal{A}}_k(x) f^{(k)}(x)$ for all $x \in \operatorname{supp}(\mathcal{L}(X))$ and all $f \in \mathcal{F}$.
\end{assumption}
This expansion is commonly referred to as the Kramers-Moyal expansion in the statistical physics literature \cite{kramers1940brownian}. From the analytic viewpoint, such expansions are closely related to the theory of pseudo-differential operators \cite{hormander2007analysis}. From the probabilistic viewpoint, they can be interpreted as a generalized Taylor expansion for generators of Markov processes. Before turning to this probabilistic interpretation, we first identify the coefficient functions $a^{\mathcal{A}}_k(\cdot)$ appearing in the Kramers-Moyal expansion of $\mathcal{A}$.
\begin{lemma}[Identity for Kramers-Moyal coefficients] \label{lemma: identity for Kramers-Moyal coefficients}
    Suppose the Kramers-Moyal expansion in Assumption \ref{assumption: Kramers-Moyal expansion} holds. Given any $k \geq 1$ and any $y \in \operatorname{supp}(\mathcal{L}(X))$, 
     we have $a^{\mathcal{A}}_k(y) = \frac{1}{k!}(\mathcal{A}(x - y)^k) \Big|_{x = y}$.
\end{lemma}
We emphasize that $y$ is a fixed point in the support of $\mathcal{L}(X)$, and
the action of $\mathcal{A}$ should be understood as acting on the function $(x - y)^k$ with $x$ as variable. We then evaluate the result function at $x=y$. This yields a mapping from $\operatorname{supp}(\mathcal{L}(X))$ to $\mathbb{R}$ for each $k \geq 1$, which is exactly $a^{\mathcal{A}}_k(\cdot)$. We now prove the above identity as follows.
\begin{proof}
    Given any $k \geq 1$ and any $y \in \operatorname{supp}(\mathcal{L}(X))$, it can be shown that the function $(x - y)^k$ belongs to $\mathcal{F}$. Thus, we apply the Kramers-Moyal expansion of $\mathcal{A}$ to the function $(x - y)^k$ and evaluate the result at $x=y$, we have
    \begin{align*}
        (\mathcal{A}(x - y)^k)\Big|_{x = y} \overset{(a)}{=} \sum_{j=1}^{\infty} a^{\mathcal{A}}_j(y) \frac{d^j}{dx^j}(x - y)^k\Big|_{x = y} \overset{(b)}{=} a^{\mathcal{A}}_k(y) k!.
    \end{align*}
    Here the equality (a) holds because of the definition of Kramers-Moyal expansion in Assumption \ref{assumption: Kramers-Moyal expansion}. The equality (b) holds because of the following case-by-case analysis,
\begin{align*}
    \frac{d^j}{dx^j}(x - y)^k\Big|_{x = y} = \begin{cases}
        0, & j \neq k, \\
        k!, & j = k, 
    \end{cases}
\end{align*}
\end{proof}
Alternatively, we can identify $a^{\mathcal{A}}_k(x)$ via an analogy of extracting moments from characteristic functions. Consider applying $\mathcal{A}$ to the function $e^{i\theta x}$ for $\theta \in \mathbb{R}$ and $i$ as the imaginary unit, with $x$ as variable. Since $e^{i\theta x} \in \mathcal{F}$, from Assumption \ref{assumption: Kramers-Moyal expansion}, we have $\mathcal{A}e^{i\theta x} = \sum_{k=1}^{\infty} a^{\mathcal{A}}_k(x) (i\theta)^k e^{i\theta x}$. Thus we can obtain $a^{\mathcal{A}}_k(x)$ by the following formula,
\begin{align*}
    a^{\mathcal{A}}_k(x) = \frac{1}{i^k k!} \frac{d^k}{d\theta^k} \left(e^{-i\theta x} \mathcal{A}e^{i\theta x}\right) \Big|_{\theta = 0}.
\end{align*}
We now turn to the probabilistic interpretation of the 
 Kramers-Moyal expansion via examples from Markovian processes. First consider a CTMC $(X_t)_{t \geq 0}$ with stationary distribution $\mathcal{L}(X)$ and generator $\mathcal{A}$. For any smooth test function $f$, suppose Kramers-Moyal expansion holds for $\mathcal{A}$, then we have
\begin{align}
    \mathcal{A}f(y) &:= \lim_{t \to 0} \frac{1}{t} \mathbb{E}[f(X_t) - f(X_0) \mid X_0 = y] \notag\\
    &\overset{(a)}{=} \sum_{k=1}^{\infty} a^{\mathcal{A}}_k(y) f^{(k)}(y) \notag\\
    &\overset{(b)}{=} \sum_{k=1}^{\infty} \frac{1}{k!}(\mathcal{A}(x - y)^k) \Big|_{x = y} f^{(k)}(y) \notag\\
    &\overset{(c)}{=} \lim_{t \to 0} \frac{1}{t} \sum_{k=1}^{\infty} \frac{1}{k!} \mathbb{E}[(X_t - X_0)^k \mid X_0 = y] f^{(k)}(y), \label{eq: Taylor expansion requirement, assumption (b) example}
\end{align}
where equality (a) follows from the Kramers-Moyal expansion in Assumption \ref{assumption: Kramers-Moyal expansion}, equality (b) follows from Lemma \ref{lemma: identity for Kramers-Moyal coefficients}. Equality (c) is from applying the definition of generator $\mathcal{A}$ to the function $(x-y)^k$ for each $k \geq 1$.
Therefore, Assumption \ref{assumption: Kramers-Moyal expansion} 
 should be interpreted as requiring that the generator $\mathcal{A}$ admit an infinite Taylor expansion, together with justification of the interchange of all limit and summation. 
 
 We next present two representative examples of Markovian processes satisfying Assumption \ref{assumption: Kramers-Moyal expansion}. First, consider OU process $dX_t = -X_t dt + \sqrt{2} dB_t$ with $\mathcal{A}$ becomes the OU generator $\mathcal{L}_{OU}f(x) = -x f'(x) + f''(x)$. To verify the Kramers-Moyal expansion, fix a small time $s > 0$, we have $X_s \overset{d.}{=} e^{-s} X_0 + \sqrt{1 - e^{-2s}} Z = X_0 - (s+o(s))X_0 + (\sqrt{2s} +o(s))Z$, thus
\begin{align*}
    \mathbb{E}[X_s - X_0 \mid X_0 = x] &= -x s + o(s), \\
    \mathbb{E}[(X_s - X_0)^2 \mid X_0 = x] &= 2s + o(s), \\
    \mathbb{E}[(X_s - X_0)^k \mid X_0 = x] &= o(s), \quad k \geq 3.
\end{align*}
Therefore, taking limit $s \to 0$, we have $\mathcal{L}_{OU}f(x) = \lim_{s \to 0} \frac{1}{s} \sum_{k=1}^{\infty} \frac{1}{k!} \mathbb{E}[(X_s - X_0)^k \mid X_0 = x] f^{(k)}(x)$. This is precisely the Kramers-Moyal expansion for $\mathcal{L}_{OU}$. In particular, only the first two coefficients are nonzero. 

The second example is a Poisson counting process with rate $\lambda > 0$. Its generator $\mathcal{A}$ is given by $\mathcal{A}f(x) = \lambda (f(x+1) - f(x))$. If $f \in \mathcal{F}$, then direct Taylor expansion at $x$ gives
\begin{align*}
    \mathcal{A}f(x) &= \lambda (f(x+1) - f(x)) \overset{(a)}{=} \lambda \sum_{k=1}^{\infty} \frac{1}{k!} f^{(k)}(x). 
\end{align*}
Equality (a) holds when $f$ is in function class $\mathcal{F}$. Thus the Kramers-Moyal expansion holds for $\mathcal{A}$ with $a_k^{\mathcal{A}}(x) = \lambda/k!$ for all $k \geq 1$. These coefficients are consistent with the identity of $a_k^{\mathcal{A}}(x)$ in Lemma \ref{lemma: identity for Kramers-Moyal coefficients} and Equation \ref{eq: Taylor expansion requirement, assumption (b) example}, since we have $\mathbb{E}[(X_s - X_0)^k \mid X_0 = x] = \lambda s + o(s)$ for all $k \geq 1$. Together with the OU example, this shows the two possible types of Kramers-Moyal expansions for continuous-time Markov processes, namely, a second-order expansion in the diffusion case and an infinite-order expansion in the jump case. By Pawula's theorem \cite{pawula2003generalizations}, these are the only two possibilities. In other words, a continuous-time Markov process either has an infinite-order Kramers-Moyal expansion, or a finite-order expansion with order 2.

We now impose the last assumption on integrability of the expansion coefficients to proceed.


\begin{assumption} [Integrability of Kramers-Moyal coefficients] \label{assumption: integrability of Kramers-Moyal coefficients}
    Under Assumption \ref{assumption: Stein identity} and \ref{assumption: Kramers-Moyal expansion}
    , we assume for all constant $D > 0$ that the inequality holds $\mathbb{E}[\sum_{k=1}^{\infty} D^k\sqrt{k!}|a^{\mathcal{A}}_k(X)|] < \infty$.
\end{assumption}
Having explained all assumptions above, we are now ready to introduce the key proposition on Wasserstein-$p$ bounds via generator comparison. 
\begin{proposition}[Wasserstein-$p$ bound via generator comparison] \label{pro: Wasserstein-$p$ bound without exchangeability}
    Let $X$ be a real-valued random variable satisfying $\mathbb{E}[|X|^p] < \infty$ for some $p > 1$. We assume Assumptions \ref{assumption: Stein identity}, \ref{assumption: Kramers-Moyal expansion} and \ref{assumption: integrability of Kramers-Moyal coefficients}
     hold for $X$ with operator $\mathcal{A}$. Then for any $t_0 > 0$, we have the following upper bound for Wasserstein-$p$ distance between $\mathcal{L}(X)$ and $\mathcal{L}(Z)$, 
    \begin{align}
        W_p(\mathcal{L}(X), \mathcal{L}(Z)) &\leq g_0(t_0,p) + \sum_{k=1}^{\infty} g_k(t_0,p) \mathbb{E}[|a_{k}^{\mathcal{A}}(X) - a_{k}^{\mathcal{L}_{OU}}(X)|^p]^{1/p} \notag\\
        &= g_0(t_0,p) + g_1(t_0,p) \mathbb{E}[| a_{1}^{\mathcal{A}}(X) + X|^p]^{1/p} \notag\\
        &+ g_2(t_0,p) \mathbb{E}[|a_{2}^{\mathcal{A}}(X) - 1|^p]^{1/p} + \sum_{k=3}^{\infty} g_k(t_0,p) \mathbb{E}[|a_{k}^{\mathcal{A}}(X)|^p]^{1/p}, \label{eq:Final Stein Wasserstein-$p$ bound}
    \end{align}
    where $Z\sim N(0,1)$ and the coefficients $g_k(t_0,p)$ for $k\ge 0$
are continuous in $(t_0,p)$ and are given explicitly in \eqref{eq: W-p bound final}.
\end{proposition}
Note that when applying the above proposition, we will construct the operator $\mathcal{A}$ such that $a_{k}^{\mathcal{A}}(X)$ is affine with respect to $X$ for all $k \geq 1$. Thus $\mathbb{E}[|a_{k}^{\mathcal{A}}(X)|^p]$ is finite as long as $\mathbb{E}[|X|^p] < \infty$. Recall from Lemma \ref{lemma: identity for Kramers-Moyal coefficients} that $a_{k}^{\mathcal{A}}(y)$ are characterized by $\tfrac{1}{k!}(\mathcal{A}(x - y)^k)(y)$. Thus $a_{k}^{\mathcal{A}}(X)$ is a random variable. We use $a_{k,X}^{\mathcal{L}_{OU}}(X)$ to denote the corresponding coefficients for the generator of OU process in the Taylor expansion. 

We next parse each term in the upper bound \eqref{eq:Final Stein Wasserstein-$p$ bound}.
The first term $g_0(t_0,p)$ is a residual term, increasing in $t_0$ and satisfies $g_0(t_0,p) \to 0$ as $t_0 \to 0$. It is introduced to offset the computational issue in $g_k(t_0,p)$ for $k \geq 3$, since $g_k(t_0,p)$ is decreasing in $t_0$ and satisfies $\lim_{t_0 \to 0} g_k(t_0,p) = \infty$ (see the definition in \eqref{eq: W-p bound final}). We will choose $t_0$ to balance the decay in $g_0(t_0,p)$ and growth in $g_k(t_0,p)$ for $k \geq 3$ when applying this proposition. 

The remaining three terms quantify the difference between the generator of \(X\) and that of the standard normal target \(Z\). Indeed, from Assumption \ref{assumption: Kramers-Moyal expansion}, the operators $\mathcal{A}$ and $\mathcal{L}_{OU}$ can be represented symbolically as $\mathcal{A} = \sum_{k=1}^{\infty} a_{k}^{\mathcal{A}}(x) \frac{d^k}{dx^k}$ and $\mathcal{L}_{OU} = \sum_{k=1}^{\infty} a_{k}^{\mathcal{L}_{OU}}(x) \frac{d^k}{dx^k}$ in terms of differential operator basis. Thus these terms compare the coefficients of the different derivative orders in the Kramers-Moyal expansions of \(\mathcal A\) and \(\mathcal L_{\mathrm{OU}}\). Since the coefficients of \(\mathcal L_{\mathrm{OU}}\) are determined by the drift and diffusion structure of the OU process, Proposition~\ref{pro: Wasserstein-$p$ bound without exchangeability} may be viewed as a Wasserstein-\(p\) diffusion approximation result.

Equivalently, from Lemma \ref{lemma: identity for Kramers-Moyal coefficients}, the generator comparison can be viewed as applying $\mathcal{A} - \mathcal{L}_{OU}$ to the centered monomials $(\cdot - X)^k$ and evaluating at $X$, then taking $L^p$ norm. In this way, the centered monomials form a convenient and tractable family of test functions to compare the generators.

In the canonical Stein's method, one compares generators by acting their difference on the solution of the Stein equation. One then expands this action into powers of derivatives through the Taylor expansion, (or Kramers-Moyal expansion) whose coefficients are determined by the action of the generator difference on the centered monomials\cite{barbour1990stein,gurvich2014diffusion}. 
Therefore, our generator comparison can be viewed as a variant of Stein's method. It quantifies Wasserstein-\(p\) distance by applying the difference of generators directly to the centered monomials, without the intermediate step of controlling the solution to Stein equation. 

\begin{remark}
Although Proposition~\ref{pro: Wasserstein-$p$ bound without exchangeability} is stated for \(p>1\), it can be extended to the case \(p=1\). Indeed, by continuity of the Wasserstein-\(p\) distance in \(p\), letting \(p\downarrow 1\) yields the Wasserstein-\(1\) distance on the left-hand side. Moreover, by the definition of \(g_k(t_0,p)\) in \eqref{eq: W-p bound final}, the right-hand side also converges by continuity in \(p\). Hence the proposition remains valid for \(p=1\), and may be viewed as an auxiliary Wasserstein-\(1\) estimate in the spirit of canonical Stein's method \cite[Theorem 3.7]{ross2011fundamental}, albeit with different constants.
\end{remark}
For JSQ we will work with a generalized version of the above proposition (see Section \ref{sec: proof for Wasserstein-$p$ bound without exchangeability}).
 This version allows a transformation of the queue vector $\tilde{\boldsymbol{q}}$
into a one dimensional projection $\tilde{q}_\Sigma$ before applying the generator comparison. The definition of $\tilde{q}_\Sigma$ is given in Section \ref{sec: Wasserstein-$p$ distance for JSQ}.

Having outlined the rationale of our choice of techniques, and explained the key Proposition \ref{pro: Wasserstein-$p$ bound without exchangeability} with its connections to tail bounds in Lemma \ref{lemma: tail bound via wasserstein}, we are now ready to sketch how to apply this machinery to our queueing system with abandonment.

\section{Proof Sketch} \label{sec: proof sketch}

With discussion on methods in Section \ref{sec: survey of methods} and main results in Section \ref{sec: main result}, we now present proof sketches for our main results. Given that Wasserstein-$p$ distance is a key stepping stone. For the sake of coherence, we first sketch proof for the ingredients, i.e., Wasserstein-$p$ distance for SSQ, and State Space Collapse for JSQ, before tail bounds, reversing their order from Section \ref{sec: main result}. We will highlight the main steps and discuss our intermediate lemmas and propositions.

\subsection{Proof Sketch for Theorem \ref{thm: Gaussian Wasserstein-$p$ upper bound for M/M/1+M}: Wasserstein-$p$ Distance} \label{sec: proof sketch for Wasserstein-$p$ bounds, SSQ}
As illustrated in Section \ref{sec: Wasserstein-$p$ Distance for M/M/1+M}, different techniques are applied to yield different bounds for Wasserstein-$p$ distance. Thus we separate the proof sketch according to different techniques.
\subsubsection{Stein's Method}
We use the variant of Stein's method, as summarized in Proposition \ref{pro: Wasserstein-$p$ bound without exchangeability}, to upper bound the Wasserstein-$p$ distance. We first summarize the main steps of proving this proposition (see Appendix \ref{sec: proof for Wasserstein-$p$ bound without exchangeability} for proof) here.

We start from the dynamic formulation of Wasserstein-$p$ distance \cite{ambrosio2005gradient,bonis2020steins} between two distributions $\mathcal{L}(X), \mathcal{L}(Z)$. Unlike the definition of Wasserstein-$p$ distance via couplings (see Section \ref{sec: notation and modeling}), which takes the infimum over all couplings $\mathcal{L}(X), \mathcal{L}(Z)$. The dynamic formulation takes the infimum over all probability flows (represented by a velocity field) connecting the two target distributions (see Lemma \ref{lemma: Benamou Brenier for W-p} for a formal statement). Its objective function is the time-integral of $L^p$ norm of this velocity field along the process. Its constraints are boundary conditions that the process starts from $\mathcal{L}(X)$ and ends at $\mathcal{L}(Z)$, and that the time-evolution of the distribution is governed by the velocity field. Thus, in order to bound Wasserstein-$p$ between $\tilde{q}$ and $Z$, we consider an Ornstein-Uhlenbeck (OU) process starting from $\tilde{q}$ at time 0, converging to $Z$ as time goes to infinity, as a feasible flow connecting the two distributions.
Using this choice of process, we obtain Proposition \ref{pro: Wasserstein-$p$ bound without exchangeability} by bounding the $L^p$ norm of the velocity field via the difference between the operators $\mathcal{A}$ and $\mathcal{L}_{OU}$.

Having explained the key Proposition \ref{pro: Wasserstein-$p$ bound without exchangeability}, we next utilize this proposition when $p\leq \lambda/(2\gamma)$ to achieve \eqref{eq: Stein Wasserstein-$p$ upper p's function}. The proof consists of the following three steps:
\begin{enumerate}
    \item Apply the generator comparison with tightest parameters to reduce the Wasserstein-$p$ distance bound to two ingredients: the probability of an empty queue and the $L^p$ norm of the normalized queue length.
    \item Bound the probability of an empty queue via recursive equations and Laplace's Method
    \item Bound the $L^p$ norm of normalized queue length via transform method and change of measure.
\end{enumerate}

\paragraph{1. Generator Comparison and Reduction}
Recall the generator comparison result in Proposition \ref{pro: Wasserstein-$p$ bound without exchangeability}, we apply it to the SSQ model to bound Wasserstein-$p$ distance. To achieve tightest generator comparison, we carefully choose the parameters in Proposition \ref{pro: Wasserstein-$p$ bound without exchangeability} to optimize. Specifically, we choose $t_0$ as a function of $(p, \gamma)$, and we let $\mathcal{A}$ be scaled infinitesimal generator of the SSQ process, i.e., $\mathcal{A} = \frac{1}{s}\mathcal{L}_{SSQ}$ with $s = 1/\gamma$ and $\mathcal{L}_{SSQ}$ in \eqref{eq: generator of SSQ}.


With these careful choices of parameters, we apply Proposition \ref{pro: Wasserstein-$p$ bound without exchangeability} to the SSQ model and
reduce the $\mathcal{W}_p$ bound to bounding two key quantities: 1) $\mathbb{P}(q = 0)$, the probability of empty queue at equilibrium, and 2) $\| \tilde{q}\|_{L^p}$, the $L^p$ norm of normalized queue length. The following lemma summarizes this reduction, 
\begin{lemma} \label{lem: Wasserstein-$p$ bound SSQ reduction}
    \textnormal{(Wasserstein-$p$ Bound Reduction for SSQ)}
    With the same notations above, for $p \leq \lambda/(2\gamma)$: 
    \begin{align*}
        W_p(\mathcal{L}(\tilde{q}), \mathcal{L}(Z))&\leq C_{10, \lambda, \mu} \cdot \big(p\sqrt{\gamma} + \frac{1}{\sqrt{\gamma}} \|\mathbf{1}_{q=0} \|_{L^p} +\gamma\sqrt{p} \| q - \frac{\lambda-\mu}{\gamma}\|_{L^p} \big)
    \end{align*}
    Here $C_{10, \lambda,\mu}$ is constant independent of $\gamma$ and $p$, defined in \eqref{eq: C_10, lambda, mu}.
\end{lemma}
Now, in order to bound Wasserstein-$p$ distance, it suffices to bound $\mathbb{P}(q = 0)$ and $\| \hat{q}\|_{L^p} = \|q - \frac{\lambda-\mu}{\gamma}\|_{L^p}$.


\paragraph{2. Probability of Empty Queue} 
To compute $\mathbb{P}(q = 0)$, we apply the generator of the SSQ to the indicator function $\mathbf{1}_{\{q = x\}}$ for $x \in \mathbb{N}$, which yields a recursive formula for its calculation.
  Using this recursive formula, we provide tight upper and lower bounds by the interplay of Riemann integral and summation, in the spirit of Laplace's method.
\begin{lemma}\label{lem: tight bound on P(q_infty = 0)}
    \textnormal{(Tight Gaussian Tail Bound on $\mathbb{P}(q = 0)$)}
    For SSQ, under assumption \ref{ass:heavy_overload}, the probability of empty queue at equilibrium $\mathbb{P}(q = 0)$ is bounded in
    \begin{align*}
        \mathbb{P}(q = 0) &\leq C_{\lambda,\mu}' \sqrt{\gamma/\mu} \exp(-\frac{1}{\gamma/\mu} (\lambda/\mu - 1 - \ln (\lambda/\mu))) \\
        \mathbb{P}(q = 0) &\geq D_{\lambda,\mu}' \sqrt{\gamma/\mu} \exp(-\frac{1}{\gamma/\mu} (\lambda/\mu - 1 - \ln (\lambda/\mu))) 
    \end{align*}
    Where $C_{\lambda,\mu}'$ and $D_{\lambda,\mu}'$ are constants $C_{\lambda,\mu}' =  \lambda/\mu(e\vee \frac{(1+C)(2+C)}{C^2+C-1}), D_{\lambda,\mu}' = \frac{\sqrt{\mu}}{(2\sqrt{\mu} + \sqrt{2\pi e^2 \lambda})}$
\end{lemma}
Intuitively, for small $\gamma$, the centered-scaled queue length $\tilde{q}$ is approximately Gaussian. However, queue length is supported only on non-negative integers, hence the probability mass at $0$ accounts for Gaussian lower tail, i.e., $\mathbb{P}(q = 0) = \mathbb{P}(\tilde{q}=-\frac{\lambda - \mu}{\sqrt{\lambda \gamma}})\approx \Phi(-\frac{\lambda - \mu}{\sqrt{\lambda \gamma}})$. Given that $\lambda/\mu-1-\ln(\lambda/\mu)\ge \tfrac{(\lambda/\mu-1)^2}{2\lambda/\mu}$ for $\lambda\ge\mu$, the exponent $\exp(\frac{1}{\gamma/\mu}(\lambda/\mu-1-\ln(\lambda/\mu)))$ indeed indicates Gaussian tail with precise $-(\frac{\lambda - \mu}{\sqrt{\lambda \gamma}})^2/2$ exponent.

\paragraph{3. $L^p$ Norm of Normalized Queue Length} For $\| \hat{q}\|_{L^p}$, we apply the transform method \cite{hurtado2020transform} with two exponential Lyapunov functions, which are $\exp(\pm \theta (q -\frac{\lambda-\mu}{\gamma}))$,$\theta>0$, to upper and lower bound the MGF of queue length. An integration-by-parts identity then connects the $L^p$ norm with MGF via concentration inequality, and provide an upper bound for $L^p$ norm. For lower bound, we use change of measure to tilt the distribution and then apply concentration inequality in the inverse way to bound the $L^p$ norm from below. The result is summarized as follows.
\begin{proposition} \label{pro: Lp norm of q_infty}
    \textnormal{($L^p$ norm of queue length)} For SSQ, under assumption \ref{ass:heavy_overload}, the $L^p$ norm of the centered queue length is bounded in the below interval
    \begin{align}
        \| \hat{q} \|_{L^p} &\leq \min\{(C_{1,\lambda,\mu}\sqrt{\frac{1}{\gamma}} \sqrt{p} + C_{2,\lambda,\mu} p),  (C_{3,\lambda,\mu} \frac{p}{\log(1+\gamma p/\lambda)})\} \label{eq: Sub-Poisson upper bound on L-p norm}\\
        \| \hat{q} \|_{L^p}&\geq C_{4,\lambda,\mu} \exp\left( -C_{5,\lambda,\mu} \frac{1}{ \gamma p} \right) \frac{p}{\log(1+\gamma p/\lambda)} \label{eq: Sub-Poisson lower bound on L-p norm}
    \end{align}
    All the above constants are only dependent on $\lambda, \mu, C$ and $\alpha$. Their closed forms are given in \eqref{eq: constants for L-p norm}
\end{proposition}

\noindent The two terms inside the minimization for \eqref{eq: Sub-Poisson upper bound on L-p norm} behave differently depending on the size of $p$. For large $p$, the first term grows roughly linearly in $p$, while the second grows as $p/\log(p)$. For small $p$, the first term is of order $1/\sqrt{\gamma}$, whereas the second term scales like the $1/\gamma$ and becomes much larger.
Thus one cannot dominate the other across all regimes of $p$. After scaling by $\tilde{q}=\frac{\sqrt{\gamma}}{\sqrt{\lambda}}\hat{q}$, the first term behaves as $\mathcal{O}(\sqrt{p})$ when $p$ is small, matching the $L^p$ norm of Gaussian distribution. The second term behaves as $\mathcal{O}(\frac{p}{\log(1+p)})$ when $p$ is large, matching the $L^p$ norm of sub-Poisson. Consequently, this mixture of bounds also exhibits the “small deviation $\leftrightarrow$ Gaussian” and “large deviation $\leftrightarrow$ Poisson” intuition in Figure \ref{fig: Phase Transition Diagram}.

This mixture of bounds, combined with the Gaussian tail of Lemma \ref{lem: tight bound on P(q_infty = 0)}, suffice to achieve Wasserstein-$p$ distance results via the above reduced problem \ref{lem: Wasserstein-$p$ bound SSQ reduction}

\subsubsection{Triangle Inequality}
Apart from the Stein's method, triangle inequality also enables us to reduce upper and lower bound of Wasserstein-$p$ distance to upper and lower bound of $L^p$ norm. Thus the ingredient is gained from Proposition \ref{pro: Lp norm of q_infty} directly. Consequently, these triangle inequalities will yield tight bound for large $p$:
\begin{align*}
    W_p(\mathcal{L} (\tilde{q}), \mathcal{L}(Z)) &\leq \| \tilde{q}\|_{L^p} + \|Z\|_{L^p}  \quad W_p(\mathcal{L}(\tilde{q}), \mathcal{L}(Z)) \geq |\| \tilde{q}\|_{L^p} - \|Z\|_{L^p}| 
\end{align*}
\subsubsection{Quantile Coupling}
While the above triangle inequality provides tight bound for large $p$, it gives a trivial lower bound for small $p$.
Quantile coupling gives a lower bound suitable for small $p$, reducing Wasserstein-$p$ distance to bounding $\mathbb{P}(q = 0)$ again.
\begin{proposition}
    \textnormal{(Quantile Coupling Lower Bound for SSQ)} \label{pro: Quantile Coupling Lower Bound for SSQ}
    With the same notations in this section, for all $p > 1$:
\begin{align*}
    W_p(\mathcal{L}(\tilde{q}), \mathcal{L}(Z)) &= \left[ \int_{0}^{1} \left| F_{\mathcal{L}(\tilde{q})}^{-1}(u) - F_{\mathcal{L}(Z)}^{-1}(u) \right|^p du \right]^{1/p} \geq C_{quantile, \lambda, \mu} \cdot \sqrt{\gamma} \cdot\mathbb{P}(q= 0)^{1/p},
\end{align*} 
where $C_{quantile, \lambda, \mu}$ is constant independent of $\gamma$ and $p$, defined in \eqref{eq: Gaussian Wasserstein-$p$ Lower bound, P_0}.
\end{proposition}

\noindent Here the quantile functions $F_{\mathcal{L}(\tilde{q})}^{-1}, F_{\mathcal{L}(Z)}^{-1}$ are the inverse of cumulative distribution functions for $\tilde{q}$ and $Z$ respectively: $F_{\mathcal{L}(\tilde{q})}^{-1}(u) = \inf\{x \in \mathbb{R}: \mathbb{P}(\tilde{q} \leq x) \geq u\}$. It is shown in \cite[Theorem 8.1]{MR520959} that coupling $(\tilde{q}, Z)$ via quantile functions, i.e., $(\tilde{q}, Z) = (F_{\mathcal{L}(\tilde{q})}^{-1}(U), F_{\mathcal{L}(Z)}^{-1}(U))$ with $U \sim \text{Uniform}(0,1)$, achieves the optimal Wasserstein-$p$ distance between 1-dimensional distributions. The above proposition thus utilizes this coupling to reduce lower bound of Wasserstein-$p$ distance to idle time lower bound, which is already achieved in Lemma \ref{lem: tight bound on P(q_infty = 0)}.

\subsubsection{Result of Different Techniques}
In the above derivations, we use Stein's method, triangle inequality and quantile coupling together to achieve complementary Wasserstein-$p$ bound. We discuss and compare the results in the following. Particularly, we focus on the joint dependency on $p$ and $\gamma$ for these bounds. This joint dependency is of interest because when developing larger deviation theorem for $\mathbb{P}(\tilde{q} > a_\gamma)$, we will choose $p$ as a function of $a_\gamma$ and thus of $\gamma$. We summarize the dominating terms from different bounding techniques in different regimes of $p$ in Table \ref{tab:Wasserstein-$p$ p's regimes}. For illustration purposes, we dissociate $(p,\gamma)$ using $p$ as the independent variable and $\gamma$ fixed. Thus the different regimes of $p$ are represented via $p = \Theta(1/\gamma^\delta)$ with different $\delta$. Although  each cell in Table \ref{tab:Wasserstein-$p$ p's regimes} can be expressed completely via $\gamma$ by plugging in $p = \Theta(1/\gamma^\delta)$, we still keep $p$ in the expressions to build up intuition. 
The highlighted cells in Stein's method and the triangle inequality stress the minima of upper bounds between different bounding techniques. Meanwhile, the highlighted cell in lower bound stresses the order-wise tight lower bound comparing to upper bound.
\begin{table}[h]
    \centering
    \begin{tabular}{lccc} 
      \toprule
      $p=\Theta(1/\gamma^\delta)$ & Upper (Stein) & Upper (Triangle inequality)  &   Lower bound \\
      \midrule
    $\delta\in[0,1-2\alpha]$ & \cellcolor{green!20}$O(p\sqrt{\gamma})$ & $O(p\sqrt{\gamma}/\gamma^{(1/2 - \delta/2 )})$ & $\Omega(\exp(-1/\gamma^{1-2\alpha})) \text{(Quantile)}$ \\
    $\delta\in[1-2\alpha,1)$ & $O\left(\frac{1}{\sqrt{\gamma}}\exp\left(-\frac{1}{p\gamma^{1-2\alpha}}\right)\right)$ & \cellcolor{green!20}$O(p\sqrt{\gamma}/\gamma^{(1/2 - \delta/2 )})$ & $\Omega\left(\gamma^{1/2+1/(2p)}\right)\text{(Quantile)}$ \\
    $\delta\in[1,\infty)$ & -
     & \cellcolor{green!20}$O\left(\frac{p\sqrt{\gamma}}{\log(1+p\gamma/\lambda)}\right)$ & \cellcolor{green!20}$\Omega\left(\frac{p\sqrt{\gamma}}{\log(1+p\gamma/\lambda)}\right) \text{(Triangle inequality)}$ \\
      \bottomrule
    \end{tabular}
    \caption{Wasserstein-$p$ distance's dependency on $p$ and $\gamma$}
    \label{tab:Wasserstein-$p$ p's regimes}
\end{table}



As a tradeoff, Stein's method yields weaker p's dependency than triangle inequality, but it promotes the dependency on $\gamma$ and thus achieves rate of convergence. This is because when deriving Stein's method for Wasserstein-$p$ distance in Proposition \ref{pro: Wasserstein-$p$ bound without exchangeability}, we utilize OU process and its velocity field for a feasible candidate. This process is designed mostly for quadratic cost and thus for Wasserstein-2 distance \cite{jordan1998variational}. However, this candidate is not tailored to the cost $|x-y|^p$ for all $p\ge 1$. Since large $p$ emphasizes rare large displacements, the OU candidate is sharp near constant $p$ but conservative for large $p$. 



\subsection{Proof Sketch for Theorem \ref{thm: SSQ Tail bound}: Tail Bounds} \label{sec: proof sketch for tail bounds, SSQ}
In this section, we sketch the proof for tail bounds in Theorem \ref{thm: SSQ Tail bound} and elaborate how results on the reduced problem of Wasserstein-$p$ bound in Lemma \ref{lem: tight bound on P(q_infty = 0)} and Proposition \ref{pro: Lp norm of q_infty} impact the final tail bounds. 


\subsubsection{Concentration Argument with Wasserstein-$p$, First and Second Regimes}

In Lemma \ref{lemma: tail bound via wasserstein}, we connect tail bounds with Wasserstein-$p$ distance. We apply the dominating terms of Wasserstein-$p$ bounds in Theorem \ref{thm: Gaussian Wasserstein-$p$ upper bound for M/M/1+M} for different range of $p$ , with different choices of $\rho$ and $p$. We achieve the various tail bounds in Theorem \ref{thm: SSQ Tail bound}. Specifically, recalling that $a=\Theta(1/\gamma^\delta)$, we choose $\rho:= 1 - e\mathcal{W}_p/a$, and $p$ based on different $\delta$ as
\begin{align*}
    p:=\begin{cases}
        a^2/2 + \ln (1/\sqrt{\gamma}), & \text{for } \delta < 1/2 - \alpha, \\
        (1/\gamma)^{1/2 - \alpha}/(2e^2 D_{2,\lambda,\mu}^2) , & \text{for } \delta \in [1/2 - \alpha, 1/2], \\
        \sqrt{\lambda}/(2eD_{3,\lambda,\mu}) \cdot (a/\sqrt{\gamma})\log(1+a\sqrt{\gamma}), & \text{for } \delta > 1/2
    \end{cases}
\end{align*}
to optimize the tail bounds in different regimes. The cross-over of different $\delta$ comes from the cross-over of different dominating terms in Wasserstein-$p$ bounds, which is implied by Gaussian tail of $\mathbb{P}(q = 0)$ in Lemma \ref{lem: tight bound on P(q_infty = 0)}. These cross-overs then determine the phase transitions in Theorem \ref{thm: SSQ Tail bound}. And the exponent of tail bounds mainly comes from the values of $p$ and thus from different $L^p$ norms in Proposition \ref{pro: Lp norm of q_infty}. The choice of $p$ when $ a < (1/\gamma)^{1/2 - \alpha}$ will yield efficient concentration for first and second regimes in Theorem \ref{thm: SSQ Tail bound}. The choices of $p$ for other regimes gives computable upper bounds in Theorem \ref{thm: SSQ Tail bound}, yet they have loose lower bounds. We thus supplement them with transform method.

\subsubsection{Transform Method for Third and Fourth Regimes}
We exploit exponential Lyapunov function to sharply upper and lower bound the MGF of $\hat{q}$, and use Markov inequality for upper bound. Meanwhile, tilting the measure with exponential function, in the spirit of inverse Chernoff bound \cite{theodosopoulos2005reversionchernoffbound}, gives lower bound. As mentioned in Section \ref{sec: survey of methods}, the transform method is mainly used for deriving efficient concentration with limit forms $\lim_{\gamma \downarrow 0} \frac{\ln \mathbb{P}(\tilde{q} > a_\gamma)}{\ln \mathbb{P}(Z > a_\gamma)}$. For the third and fourth regimes, the upper bounds from transform method supplement the Wasserstein-$p$ argument, while the lower bounds indeed provide matching exponent. 

\subsection{Proof Sketch for JSQ} In this section, we sketch the proof towards tail bounds for JSQ model, from State Space Collapse Theorem \ref{thm: SSC JSQ} to Wasserstein-$p$ distance Theorem \ref{thm: Wasserstein-$p$ distance JSQ}. Compared with SSQ, the main technical challenge for JSQ model is to handle its multi-dimensional state space. This challenge is addressed via State Space Collapse (SSC) phenomenon. Once SSC is achieved, we reduce the multi-dimensional problem to an one-dimensional problem on the total queue length, similar to SSQ model. We then apply the techniques in SSQ model to reduce Wasserstein-$p$ distance bounds to similar reduced problems, which are further bounded via Lyapunov drift method. 

\subsubsection{State Space Collapse} Our approach to prove SSC is via properly-chosen Lyapunov function and the spirit of Hajek's geometric tail for hitting time\cite{EryilmazSrikant2012QSys,hajek1982hitting} for the perpendicular component $\mathbf{q}_{\perp}$. Hajek's geometric tail requires two conditions. First, boundedness for a single jump, i.e., $|f(\mathbf{x}_\perp) - f(\mathbf{x}_\perp')| \leq D$, for $D>0$ and $\mathbf{x}, \mathbf{x}'$ s.t. $\mathbf{Q}_{JSQ}(\mathbf{x},\mathbf{x}')>0$, which is verified in the JSQ model. Second, a negative drift when $\|\mathbf{x}_{\perp}\|$ is beyond some compact set, i.e., $\mathcal{L}_{JSQ} f(\mathbf{x}_{\perp}) \leq -\xi\boldsymbol{1}_{\{\mathbf{x}_{\perp}\in \mathbb{R}^n\setminus B\}}$, for $\xi >0, f: \mathbb{R}^n\to \mathbb{R}^+, B$ compact. For this condition, we choose $\|\mathbf{q}_\perp\|$ as the Lyapunov function $f(\cdot)$. But instead of searching for a compact set depending only on $\|\mathbf{q}_\perp\|$, we identify the negative drift set for $\mathcal{L}_{JSQ}\|\mathbf{q}_\perp\|$ with dependence on both $\|\mathbf{q}_\perp\|$ and $q_\Sigma$, (see \eqref{eq: gen first order diff two control})
\begin{align*}
    \mathcal{L}_{JSQ} \|\mathbf{q}_\perp\| &\leq -\xi, \text{ if } \|\mathbf{q}_\perp\| > L(q_\Sigma)
\end{align*}
with quantity $\xi>0$ that involves $\lambda,\mu,\gamma,n$, and function $L:\mathbb{R}\to\mathbb{R}$ that is affine. Such set is identified via boundedness of single jump and drift analysis of $\|\mathbf{q}_\perp\|^2$. Notice that $\mathbf{q}_\perp \perp \mathbf{q}_\parallel := (q_\Sigma/n)\boldsymbol{1}$, the above set's dependence is dissociated into two orthogonal components. Thus the dependence on $q_\Sigma$ is inevitable.
We exploit the boundedness for finite $p$-th moment of this additional dependence $L(q_{\Sigma})$, which is suitable for applying Lyapunov drift method. We then choose polynomial Lyapunov function $\|\mathbf{q}_\perp\|^{p+1}$, apply drift method and achieve a $p$-th order polynomial inequality to bound $\mathbb{E}[\|\mathbf{q}_\perp\|^p]$. The root estimation of this polynomial finally complete the $p$-th moment bound for $\|\mathbf{q}_\perp\|$. 

\subsubsection{Wasserstein-$p$ Distance Reduction} 
Similar to SSQ, upper and lower bounds on $\mathcal{W}_p(\mathcal{L}(\tilde{q}_{\Sigma})$, $\mathcal{L}(Z))$ are derived via Stein's method, triangle inequality and quantile coupling. For Stein's method, we start with a general version of Proposition \ref{pro: Wasserstein-$p$ bound without exchangeability} (see Proposition \ref{pro: Wasserstein-$p$ bound without exchangeability final version}). This general version allows us to transform $\tilde{\boldsymbol{q}}$ into $\tilde{q}_\Sigma$ and bound Wasserstein-$p$ distance between $\tilde{q}_\Sigma$ and $Z$ via a modified generator comparison. We then apply this general version of Stein's method to the JSQ model,
 and choose $\mathcal{A}$ as the scaled generator of JSQ model with $t_0 = -\frac{1}{2} \ln(1 - p\gamma/\lambda)$. Wasserstein-$p$ distance is then reduced to bounding 1) probability mass at zero $\sum_i \mathbb{P}(\mathbf{q}_{i}=0)$ and 2) moment bounds on $\hat{q}_\Sigma$. Meanwhile, triangle inequality and quantile coupling are applied to introduce an extra ingredient for bounding Wasserstein-$p$ distance, i.e.,  $\mathbb{P}(\sum_i \mathbf{q}_{i}=0)$.

\paragraph{1. Probability Mass at Zero}
For $\sum_i \mathbb{P}(\mathbf{q}_{i}=0)$, due to lack of time-reversibility from the routing mechanism, we do not adapt detailed balance proof. We instead choose Lyapunov function $\sum_i\exp(-\sqrt{\gamma}\theta q_i)$ and apply Lyapunov drift method to bound $\sum_i \mathbb{P}(\mathbf{q}_{i}=0)$. Meanwhile, for $\mathbb{P}(\sum_i \mathbf{q}_{i}=0)$ we apply coupling argument. Our bounds are summarized as follows.
\begin{proposition}\textnormal{(Probability Mass at Zero)}
    \label{pro: zero probability sum for JSQ}
    Under heavily overloaded assumption \ref{ass:heavy_overload}, the probability mass at zero for JSQ model is bounded as
    \begin{align*}
        \sum_{i=1}^{n} \mathbb{P}(\mathbf{q}_{i} = 0) &\leq \exp\left( -\frac{C}{2n\lambda} \cdot\frac{\lambda-\mu}{\sqrt{\gamma}} \right) \\
        \exp\left( -\frac{\lambda}{\gamma} \right) \leq \mathbb{P}(\sum_{i=1}^{n} \mathbf{q}_{i} = 0) &\leq C_{\lambda,\mu}' \sqrt{\gamma/\mu} \exp\left( -\frac{1}{\gamma/\mu} (\lambda/\mu - 1 - \ln (\lambda/\mu)) \right) 
    \end{align*}
    $C_{\lambda,\mu}'$ is the same as in Lemma \ref{lem: tight bound on P(q_infty = 0)}, $C$ is from the heavily overloaded assumption \ref{ass:heavy_overload}. 
\end{proposition}

We achieve exponential tail for $\sum_i \mathbb{P}(\mathbf{q}_{i}=0)$ w.r.t $1/\sqrt{\gamma}$. Compared with Gaussian tail in SSQ, this exponential tail will lead to the early happening of phase transition in tail bound, see Theorem \ref{thm: JSQ Tail bound, second regime}. For $\mathbb{P}(\sum_i \mathbf{q}_{i}=0)$, coupling argument is applied to compare JSQ model with either SSQ(i.e., $M/M/1+M$) or $M/M/\infty$ queue, leading to upper and lower bounds correspondingly. This upper bound differs from lower bound by only logarithmic term in the exponent.

\paragraph{2. Moment Bounds on $\hat{q}_\Sigma$} 
For moment bounds on $\hat{q}_\Sigma$, we use drift method with Lyapunov function $\sum_i\exp(-\theta q_i)$, $\sum_i\exp(\theta q_i)$ with $\theta>0$, combined with Proposition \ref{pro: zero probability sum for JSQ} to achieve sharp bound on the two-sided MGF of $\hat{q}_\Sigma$, i.e., $\mathbb{E}[\exp(\theta |\hat{q}_\Sigma|)]$. From MGF towards $L^p$ norm, we use arguments including integration-by-parts for upper bound, or coupling for lower bound. 
\begin{proposition}\textnormal{(Moment Bounds on $\hat{q}_{\Sigma}$)}
    \label{pro: Moment Bounds on hat q for JSQ}
    Under the assumption \ref{ass:heavy_overload}, with $\hat{q}_{\Sigma}:= \sum_i q_i - \frac{\lambda-\mu}{\gamma}$, we have the moment bounds for any $p\geq 1, p\in \mathbb{R}$:
    \begin{align*}
        \mathbb{E}[|\hat{q}_{\Sigma}|^p]^{1/p} &\leq \max\{C_{\lambda,\mu,n}'' \cdot\frac{\sqrt{p}}{\sqrt{\gamma}} + C_{\lambda,\mu,n}'\cdot p, \: n(A_{\lambda,\mu,n} +\frac{1}{n}) \frac{p}{\log(1+\frac{n\gamma}{\lambda}p)}\} \\
        \mathbb{E}[|\hat{q}_{\Sigma}|^p]^{1/p} &\geq  C_{4,\lambda,\mu} \exp(-8(15+A_{\lambda,\mu})\frac{\lambda}{\gamma p}) \frac{p}{\log(1+\frac{p\gamma}{\lambda})}
    \end{align*}
    Constants $C_{\lambda,\mu,n}'$, $C_{\lambda,\mu,n}''$ and $A_{\lambda,\mu,n}$ are defined in \eqref{eq: subexpo L-p norm JSQ} and \eqref{eq: MGF upper bound JSQ}, respectively. 
\end{proposition}
We justify the choice of Lyapunov functions $\sum_i\exp(-\sqrt{\gamma}\theta q_i)$, $\sum_i\exp(\theta q_i)$, and $\sum_i\exp(-\theta q_i)$ as follows. The drift analysis for JSQ with abandonment yields heterogeneous Ordinary Differential Equations (ODEs). We select Lyapunov functions that make the heterogeneous terms negligible, which requires the main terms to dominate in magnitude. This guides our first choice of exponential Lyapunov functions with $\sqrt{\gamma}$ scaling in the exponent. Additionally, to handle heterogeneous servers with different service rates $\mu_i$, we use additive forms such as $\sum_i\exp(\theta q_i)$ rather than $\exp(\theta \sum_i q_i)$ to decouple the dynamics across servers.
 With the exponential Lyapunov functions above, 
 we achieve homogeneous differential inequalities from the heterogeneous ODE. Finally we translate inequalities back to bounds on the expectation of Lyapunov functions in the spirit of Gronwall's inequality \cite{gronwall1919note}.

\paragraph{Results} With connection to tail bounds as in Lemma \ref{eq: Tail bound from Wasserstein-$p$ distance}, the joint dependency on $p$ and $\gamma$ is of focus. Here we summarize the order analysis from above methods. We conclude the minima and maxima for upper and lower bounds respectively, labeling the method used to derive the bounds.
\begin{table}[h]
    \centering
    \begin{tabular}{lll} 
        \toprule
        $p=O(1/\gamma^\delta)$ & Upper bound & Lower bound \\
        \midrule
        $\delta\in[0, 1/2- \alpha)$ & $O(p\sqrt{\gamma})$ (Stein's Method) & $\Omega(\exp(-1/\gamma))$ (Quantile Coupling) \\
        $\delta\in[1/2-\alpha, 1)$ & $O(\sqrt{p})$ (Triangle Inequality) & $\Omega(\sqrt{\gamma}\exp(-1/(p\gamma)))$ (Quantile Coupling)\\
        $\delta \in [1,\infty]$ & $O\left(\frac{p\sqrt{\gamma}}{\log(1+p\gamma)}\right)$ (Triangle Inequality) & $\Omega\left(\frac{p\sqrt{\gamma}}{\log(1+p\gamma/\lambda)}\right)$(Triangle Inequality)\\
        \bottomrule
    \end{tabular}
    \caption{Wasserstein-$p$ distance's dependency on $p$ and $\gamma$}
    \label{tab:Wasserstein-$p$ p's regimes, JSQ total queue length}
\end{table}

\subsubsection{Tail Bounds} 
We build upon the $\mathcal{W}_p$ bounds on $\mathcal{W}_p (\mathcal{L}(\langle \tilde{\mathbf{q}}, \boldsymbol{\phi}\rangle), \mathcal{L}(Z\cdot\langle\boldsymbol{\phi},\boldsymbol{1}\rangle))$ to connect to tail probability as in \eqref{eq: Tail bound from Wasserstein-$p$ distance}. The $\mathcal{W}_p$ bound is
achieved via both Theorem \ref{thm: Wasserstein-$p$ distance JSQ} and SSC Theorem \ref{thm: SSC JSQ}. We then apply Lemma \ref{lemma: tail bound via wasserstein} with proper choice of $p, \rho$ to fully characterize the tail behavior in different regimes of deviation $a$ (see Theorem \ref{thm: JSQ Tail bound, first regime}, \ref{thm: JSQ Tail bound, second regime} and \ref{thm: JSQ Tail bound, third regime}).  
 Meanwhile, we specify the orthogonal case $\langle\boldsymbol{\phi},\boldsymbol{1}\rangle = 0$ separately. We use SSC Theorem \ref{thm: SSC JSQ} to transform $L^p$ bound on $\|\mathbf{q}_{\perp}\|$ into tail probability of $\langle \tilde{\mathbf{q}}, \boldsymbol{\phi}\rangle$ (see Theorem \ref{thm: JSQ Tail bound, orthogonal case}). The refinement for large deviation when $\boldsymbol{\phi}\geq 0$ is achieved via coupling argument. Notice that with the lack of matching lower bound for $\|\mathbf{q}_{\perp}\|$, we are not able to apply transform method as in SSQ to achieve a matching lower bound. Thus cannot achieve efficient concentration inequalities for large deviation and leave it as open.


\section{Literature Review} \label{sec: literature review}
This section reviews related literature in steady-state diffusion approximations for queueing systems and tail bounds. For the former, we focus on works related to pre-limit analysis and quantitative convergence rates. For tail bounds, we focus on works related to efficient concentration and different limit results, e.g., deviation principles. To avoid redundancy with survey of methods in Section \ref{sec: survey of methods}, this review surveys different approaches for concentration bounds than those in Section \ref{sec: survey of methods}. Additionally, we survey classical results in CLT that are closely related to our summary of results in Section \ref{sec: summary of results}.



\paragraph{Stein's Method in Queueing Theory}
Asymptotic analysis of queueing systems often relies on diffusion approximations, with early works tracing back to Kingman's heavy-traffic limit~\cite{Kingman1961}.
A standard route to asymptotic results is via process-level diffusion approximations, which yield functional limit theorems for the queue-length or workload processes by showing weak convergence in path space (e.g., Skorokhod space of càdlàg functions) \cite{whitt2002stochastic}. 
Separately, variable-level diffusion approximations for steady-state queueing performance metrics are often obtained by showing a weak convergence towards the stationary distribution of the diffusion limit. The machinery of this route
includes Lyapunov-drift and transform methods \cite{hurtado2020transform}. These approaches, however, generally do not yield pre-limit tail bounds.




On the other hand, Stein's method has been a focus to solve a fundamental question in probability: How to quantify the distance between two distributions in a computable way? Different distributional distances have been developed for Gaussian and beyond \cite{chen2001non}, with deployment to different models \cite{haque2026quantifyingnormalityconvergencerate, wang2026steadystatebehaviorconstantstepsizestochastic}. And different techniques have been proposed to construct the Stein's identity, including size-bias coupling \cite{arratia2018sizebias}, zero-bias coupling \cite{arratia2018sizebias} and exchangeable pairs \cite {ross2011fundamental}. 
However, all above techniques are limited in queueing systems. Specifically, the coupling methods often require a proper probabilistic construction, e.g., integration-by-part identity for Stein Kernel \cite{chatterjee2010newapproachstrongembeddings}, or direct assumption on the existence of Stein Kernel \cite{ledoux2015stein}, or summation of random variables for zero-bias coupling \cite{ross2014fundamental, Bhattacharjee_2016strongembeddings}. In our single-server queue, such construction is not available. Exchangeable pairs are natural for Markovian systems, but they impose a time-reversibility condition that is not satisfied by the JSQ system.



To utilize Stein's method for queueing systems, one typically chooses a generator comparison framework, which was initially proposed by Barbour \cite{barbour1990stein} and further introduced to queueing theory by Gurvich \cite{gurvich2014diffusion}. 
 Across Erlang-A/Erlang-C and $M/Ph/n{+}M$ systems, Braverman et al. adopt such framework and yield quantitative diffusion error bounds in metrics such as Wasserstein-1 and Kolmogorov distance\cite{braverman2017steinMPhn}. They also achieve pre-limit tail bounds for Erlang-C model \cite[Theorem~4.1]{braverman2017steinthesis}. Subsequent work of Gaunt et al. applies Stein's method to the single-server $M/G/1$ and $G/G/1$ queues in heavy traffic, obtaining exponential approximations in Wasserstein-1 distance \cite{gaunt2020stein}. For routing dynamics, Hurtado-Lange et al. show a rate of convergence for
  many-server haevy-traffic load-balancing system \cite{hurtado2022load}. 
  Meanwhile, \cite{besanccon2018stein} develops a functional Stein's method yielding quantitative path-wise convergence rates of $M/M/1$ and $M/M/\infty$ queues. They extend the quantitative Wasserstein-1 bound to stochastic differential equations driven by Poisson measures, applicable to a wide class of continuous-time Markov chains \cite{besanccon2024diffusive}.
More recently, Braverman et al. utilize higher-order information of the generator to achieve refined diffusion limits for Erlang-C model \cite{braverman2024high}. Concurrently, Chatterjee et al. \cite{chatterjee2026higherorderapproximationssojourntimes} develop a Stein method for the $M/G/1$ queue that controls a Zolotarev-type distance of order $p$. While Zolotarev-$p$ distance is linked to Wasserstein-$p$ distance via functional inequalities (see, e.g.,\cite{bołbotowski2025sharpinequalitieszolotarevwasserstein} and \cite[Theorem~3.1]{Rio2009}), it requires the two distributions to match on the first $\lceil p-1 \rceil$ moments for finiteness. For queueing systems with abandonment, since the first moment already deviates from 0 (the Gaussian mean), the Zolotarev-$p$ distance becomes infinite and inapplicable.  
 To the best of our knowledge, all existing Stein-based steady-state results for queues are formulated in Wasserstein-1 (or Zolotarev-$p$-type) distances. In contrast, a key ingredient of our efficient concentration results is a quantitative convergence rate in genuine Wasserstein-$p$ distance, $p>1$, for the stationary queue length.

\paragraph{Survey for Tail Bounds} In Section \ref{sec: survey of methods}, we survey tools including drift method, transform method, Wasserstein-$p$ method. Highly related to Wasserstein-$p$ method, quantile coupling inequalities \cite{krishnapur2020one} also provide a way to attain efficient concentration. For CLT, Komlós-Major-Tusnády construction \cite{KMTstrongembedding} and strong embeddings \cite{chatterjee2010newapproachstrongembeddings} yield couplings between the sum of i.i.d. random variables and Gaussian w.r.t Orlicz norm \cite{Rio2009}. Yet the above methods mostly focus on i.i.d. or weakly dependent random variables, and attain at best sub-exponential tails \cite{krishnapur2020one}, unclear to achieve the Gaussian tail decay as in \eqref{eq: main contribution, constant deviation} for queueing systems.

An approach that tackles the loose tail decay from Kolmogorov distance is via finer control of the solution to the Poisson equation associated with the diffusion limit. Specifically, one studies the Poisson equation $\mathcal{L}f = \boldsymbol{1}_{\{x \leq a\}} - \mathbb{P}(Z \leq a)$ that yields non-uniform Kolmogorov distance bounds under expectation \cite{chen2010normal}. The diffusion generator $\mathcal{L}$ is chosen via the diffusion limit, e.g., $\mathcal{L}_{OU}$ for Ornstein-Uhlenbeck process. Austern et al. in \cite[Theorem~1]{austern2022efficient} utilized a variant of this framework involving zero-bias coupling to achieve efficient concentration inequality in CLT. For queueing theory, Braverman et al. \cite{braverman2017steinthesis,braverman2024high} obtained bounds on the derivatives of the solution $f$ to this Poisson equation. They combined the gradient bounds with iteratively applying generator comparison, drift method and the Poisson equation itself to achieve efficient concentration under quality-and-efficiency-driven regime \cite[Theorem~4.1]{braverman2017steinthesis}. They extended the framework to a family of refined diffusion approximations using higher-order information of the generator \cite{braverman2024high}. Yet their results focus on Erlang-C model without abandonment, and only cover moderate deviations. We instead seek a full characterization under all deviation scales for queues with abandonment.

Another classical tool for efficient concentration-type bound is via Edgeworth Expansions \cite{blanchet2007uniform, ABATE1994223}, which also stems from CLT. Characteristic-function-based expansions yield convergence rate for finite sample sizes $n$ \cite[Ch.~16]{feller1991introduction}, leading to Berry-Esseen bound or non-uniform Berry-Esseen bound \cite{bikelis1966estimates, nagaev1965some}. But the tail error only has cubic decay in $a$ as the underlying argument relies only on third moment information. Moment Generating Function (MGF)-based versions of Edgeworth Expansion \cite[Ch.~8]{dembo2009large,BahadurRangaRao1960, petrov2012sums} utilize MGF information of summand to provide asymptotic expansions for tail probabilities. Appealingly, they imply the efficient concentration as shown in Cramér-type moderate-deviation theorems \cite{Cramer1938, petrov2012sums}. Yet these asymptotic expansions carry unspecified constants (see Section \ref{sec: summary of results}), unsuitable for explicit pre-limit bounds. Additionally, MGF is unknown in JSQ model, making Edgeworth Expansion inapplicable.

\paragraph{CLT survey} In Section \ref{sec: summary of results}, we treat two types of results in CLT, distinguished by their conditions on summands: sub-exponential and bounded. For sub-exponential summands, Cramér-type pre-limit bounds \cite[Chapter.~8, Theorem~2]{petrov2012sums} yield an efficient concentration-type bound for deviations \(a_n=o(n^{1/2})\). Thus, they imply relative ratio convergence for \(a_n=o(n^{1/6})\) and a moderate deviation principle (MDP) for \(a_n=o(n^{1/2})\). But they have unspecified constants in the exponents for pre-limit bound. For bounded summands, \cite[Theorems.~1-2]{austern2022efficient} provide pre-limit upper bounds: an efficient concentration upper bound at constant deviation \(a_n=\Theta(1)\), extended to \(a_n=o(n^{1/2})\), with all constants computable. Beyond the above \(a_n=o(n^{1/2})\) moderate deviation, Cramér \cite{Cramer1938} and Bahadur-Ranga Rao \cite{BahadurRangaRao1960} give matching upper/lower large-deviation bounds for sub-exponential summands, again with unspecified constants. By contrast, \cite{austern2022efficient} and classical Hoeffding-type bounds \cite{boucheron2003concentration} provide explicit upper bounds for bounded summands.

\section{Future Directions} \label{sec: future directions}

There are several possible avenues to pursue with regard to the results, methodology, and extending to other stochastic systems or limit approximations.
\begin{enumerate}
    \item  As noted in the main text, an open problem is to obtain matching lower bounds for tail probabilities in the JSQ system and to determine whether the phase transition in the tail exponent is sharp. These questions are tied to the strength of state-space collapse (SSC): does a strong form of SSC hold, and can we prove order-wise lower bounds for the $p$-th moment of $\mathbf{q}_\perp$? Another key factor is the tightness of the idle-time probability $\sum_i \mathbb{P}(q_i=0)$; whether it exhibits Gaussian decay will determine the location of phase transition.
    \item One main tool for efficient concentration is bounding the Wasserstein-$p$ distance via Stein's method. In the SSQ case, however, the generator-coupling approach is not sharp for large $p$ compared with a triangle-inequality argument. This motivates refined Stein bounds for $\mathcal{W}_p$. Extending the current results from Wasserstein-$p$ to Wasserstein-Orlicz distances is a natural next step. 
    \item It is important to explore stochastic systems without assuming exponential distributions for inter-arrival, service, or abandonment times. One possible extension is to examine discrete-time models with various distributions for these times, as discussed in \cite{hurtado2020transform}. The frameworks of Stein's and Lyapunov drift methods can be adapted by considering generator $\mathcal{A}$ in Proposition \ref{pro: Wasserstein-$p$ bound without exchangeability} as an one-step dynamics of the discrete-time Markov chain. Another extension could involve generalizing to non-Markovian systems, such as $G/G/n+G$ queues under JSQ routing. This would necessitate enlarging the state space by adding residual-life variables, resulting in piecewise-deterministic Markov processes. The generator $\mathcal{A}$ in this context would be obtained via Palm-inversion formula as in \cite{braverman2025diffusionapproximationerrorqueueing}, where the boundary condition is more complex than the Markovian case.
    \item Additionally, it is crucial to investigate other relevant regimes, such as classical heavy traffic, where the steady-state limits are often non-Gaussian reflected diffusions. Some examples include the Halfin-Whitt regime without abandonment \cite{braverman2020steady} and Critically Loaded regime with abandonment
    \cite{jhunjhunwala2023jointheshortestqueueabandonmentcritically}.
     A significant challenge is the lack of a Stein framework on Wasserstein-$p$ error bounds for these non-Gaussian limits. Developing such a framework would enable the derivation of efficient concentration inequalities in different heavy-traffic regimes.
\end{enumerate}

\bibliographystyle{plain}
\bibliography{references}

\newpage
\appendix

\section{Proof for Proposition \ref{pro: Wasserstein-$p$ bound without exchangeability}: Stein's Method for Wasserstein-$p$ bound} \label{sec: proof for Wasserstein-$p$ bound without exchangeability}

We provide the proof for a generalized version of Proposition \ref{pro: Wasserstein-$p$ bound without exchangeability} in this section. We first restate the assumptions and the proposition in a more general form, and then provide the proof. This general version introduces a transformation $T$ that first normalizes the random vector of interest, and then compares the distribution of $T(X)$ with a standard normal distribution. The original version of Proposition \ref{pro: Wasserstein-$p$ bound without exchangeability} is a special case where $T$ is the identity mapping. When applying to our SSQ system, we will choose $T$ to be $T(x) = \frac{\sqrt{\gamma}}{\sqrt{\lambda}}(x - \lambda/\gamma)$, such that $T(q) = \tilde{q}$ is the centered-scaled queue length. For JSQ, we will choose $T(x) = \sqrt{\gamma/\lambda}(\sum_i x_{i} - \frac{\lambda-\mu}{\gamma})$ such that $T(\mathbf{q}) = \tilde{q}_{\Sigma}$ is the centered-scaled total queue length. We start with the following assumptions.


\begin{assumption}[Stein identity] \label{assumption: Stein identity final version}
    Given a random vector $X\in \mathbb{R}^n$ and a measurable function $T: \mathbb{R}^n \to \mathbb{R}$. Let $E$ be the support of $\mathcal{L}(T(X))$. We assume that there exists a linear operator $\mathcal{A}$ such that for any test function $f \in \mathcal{C}_b(\mathbb{R})$, $\mathcal{A}(f \circ T)(X)$ is integrable, and
    $\mathbb{E}[\mathcal{A}(f \circ T)(X)] = 0$.
\end{assumption}

We assume the operator $\mathcal{A}$ admits a Kramers-Moyal expansion as follows, recalling that $\mathcal{F}$ is the function class defined in \eqref{assumption: function class}.
\begin{assumption}[Kramers-Moyal expansion] \label{assumption: Kramers-Moyal expansion final version}
    Under Assumption \ref{assumption: Stein identity final version}, we further assume that there exists a family of functions $a^{\mathcal{A}, T}_k(\cdot)$, such that $\mathcal{A}(f \circ T)(x) = \sum_{k=1}^{\infty} a^{\mathcal{A}, T}_k(x) f^{(k)}(T(x))$ for all $x \in \operatorname{supp}(\mathcal{L}(X))$ and all $f \in \mathcal{F}$.
\end{assumption}

We characterize the coefficients $a^{\mathcal{A}, T}_k(\cdot)$ in the following lemma, which is a direct consequence of applying Kramers-Moyal expansion to the test function $f(x) = (x - T(y))^k$ for $k \geq 1$ and $y \in \operatorname{supp}(\mathcal{L}(X))$.
\begin{lemma}[Identity for Kramers-Moyal coefficients] \label{lemma: identity for Kramers-Moyal coefficients final version}
    Suppose the Kramers-Moyal expansion in Assumption \ref{assumption: Kramers-Moyal expansion final version} holds. Given any $k \geq 1$ and any $y \in \operatorname{supp}(\mathcal{L}(X))$, 
     we have $a^{\mathcal{A}, T}_k(y) = \frac{1}{k!}(\mathcal{A}(T(x) - T(y))^k) \big|_{x=y}$.
\end{lemma}
The final assumption is on the integrability of Kramers-Moyal coefficients.
\begin{assumption} [Integrability of Kramers-Moyal coefficients] \label{assumption: integrability of Kramers-Moyal coefficients final version}
    Under Assumption \ref{assumption: Stein identity final version} and \ref{assumption: Kramers-Moyal expansion final version}
    , we assume for all constant $D > 0$ that the inequality holds $\mathbb{E}[\sum_{k=1}^{\infty} D^k\sqrt{k!} |a^{\mathcal{A}, T}_k(X)|] < \infty$.
\end{assumption}

We are now ready to present the proposition that bounds the Wasserstein-$p$ distance between $\mathcal{L}(T(X))$ and $\mathcal{L}(Z)$ via generator comparison. 

\begin{proposition}[Wasserstein-$p$ bound via generator comparison] \label{pro: Wasserstein-$p$ bound without exchangeability final version}
    Let $X\in \mathbb{R}^d$ be a random vector, and $T: \mathbb{R}^d \to \mathbb{R}$ be a measurable function. Suppose $T(X)$ satisfies $\mathbb{E}[|T(X)|^p] < \infty$ for some $p > 1$. We assume Assumptions \ref{assumption: Stein identity final version}, \ref{assumption: Kramers-Moyal expansion final version} and \ref{assumption: integrability of Kramers-Moyal coefficients final version}
     hold for $X$ with operator $\mathcal{A}$ and mapping $T$. Then for any $t_0 > 0$, we have the following upper bound for Wasserstein-$p$ distance between $\mathcal{L}(T(X))$ and $\mathcal{L}(Z')$, 
    \begin{align}
        W_p(\mathcal{L}(T(X)), \mathcal{L}(Z')) &\leq g_0(t_0,p) + \sum_{k=1}^{\infty} g_k(t_0,p) \mathbb{E}[|a_{k}^{\mathcal{A}, T}(X) - a_{k}^{\mathcal{L}_{OU}, T}(X)|^p]^{1/p} \notag\\
        &= g_0(t_0,p) + g_1(t_0,p) \mathbb{E}[| a_{1}^{\mathcal{A}, T}(X) + T(X)|^p]^{1/p} \notag\\
        &+ g_2(t_0,p) \mathbb{E}[|a_{2}^{\mathcal{A}, T}(X) - 1|^p]^{1/p} + \sum_{k=3}^{\infty} g_k(t_0,p) \mathbb{E}[|a_{k}^{\mathcal{A}, T}(X)|^p]^{1/p}, \label{eq:Final Stein Wasserstein-$p$ bound final version}
    \end{align}
    where $Z'\sim N(0,1)$ and the coefficients $g_k(t_0,p)$ for $k\ge 0$
are continuous in $(t_0,p)$ and are given explicitly in \eqref{eq: W-p bound final}.
\end{proposition}

In the following, we will work with change of variable $Y:=T(X)$. We let $Z$ be a standard normal variable independent of $X$, which we will use to construct the OU flow in the proof.
And we let $Z'$ be a standard normal variable, which is the target distribution we want to compare with in the proposition. We use $Z'$ here to distinguish from the standard normal $Z\perp X$ that we will use to construct the OU flow. This proof is largely based on the spirit of Austern et al. \cite[Inequality C.4]{austern2022efficient}, Fang et al. \cite[Proposition 6.1]{fang2022wasserstein} and Bonis \cite[Theorem 3]{bonis2020steins}. 

\paragraph{Proof for Proposition \ref{pro: Wasserstein-$p$ bound without exchangeability final version}}
We first introduce the dynamic formulation of Wasserstein-$p$ distance from Benamou-Brenier \cite{ambrosio2005gradient}, which is a key ingredient in the proof.
\begin{lemma} 
    \label{lemma: Benamou Brenier for W-p}
    \textnormal{Benamou-Brenier for Wasserstein-$p$ distance \cite[Theorem~8.3.1]{ambrosio2005gradient}\cite[Inequality~3.8]{ledoux2015stein}} 
    Suppose $X, Z$ are two random variables in $\mathbb{R}^d$ with probability measures $\mathcal{L}(X), \mathcal{L}(Z) \in \mathcal{P}_p(\mathbb{R}^d)$ respectively. 
    For a weakly continuous curve of probability measures $\mu_t : I \to \mathcal{P}_p(\mathbb{R}^d)$ with open interval $I \subset \mathbb{R}$. If it  satisfies the continuity equation for some Borel vector field $v_t$ with $\int_I \int_{\mathbb{R}^d} \|v_t(x)\|_2^p d\mu_t(x) dt < \infty$ and
        $\partial_t \mu_t + \text{div}(\mu_t v_t) = 0, \; \forall (x,t) \in \mathbb{R}^d\times I$. Then curve $\mu_t: I\to \mathcal{P}_p(\mathbb{R}^d)$ is absolutely continuous and for Lebesgue-almost everywhere $t \in I$, we have
        \begin{align*}
            \lim_{ s\to t} \frac{W_p(\mu_s, \mu_t)}{|s-t|} \leq \left\{\int_{\mathbb{R}^d} \|v_t(x)\|_2^p d\mu_t(x) \right\}^{1/p} 
        \end{align*}
\end{lemma}
Using the above lemma, we can prove the corollary below, and we will mainly work with this corollary in the proof since our application is one-dimensional.
    \begin{corollary}[Dynamic formulation of $W_p$]
    Let $\mathcal{L}(X), \mathcal{L}(Z') \in \mathcal{P}_p(\mathbb{R})$ and suppose there exists a weakly continuous curve $\mu_t:(0,\infty)\to\mathcal{P}_p(\mathbb{R})$ and a Borel vector field $v_t$ such that
    $\partial_t \mu_t + \partial_x(\mu_t v_t)=0$,
    $\sup_{t\in(0,\infty)} \|x\|_{L^p(\mu_t)}<\infty$. Moreover, suppose
    $\mu_t \Rightarrow \mathcal{L}(X)$ as $t\downarrow 0$, and
    $\mu_t \Rightarrow \mathcal{L}(Z')$ as $t\uparrow \infty$.
    Then
    \[
    W_p(\mathcal{L}(X),\mathcal{L}(Z'))
    \le
    \inf\left\{
    \int_{0}^{\infty}
    \left(\int_{\mathbb{R}} |v_t(x)|^p\,d\mu_t(x)\right)^{1/p}\!dt
    :\ \partial_t \mu_t + \partial_x(\mu_t v_t)=0,\ 
    \mu_t \overset{t\downarrow 0}{\Rightarrow} \mathcal{L}(X),\ 
    \mu_t \overset{t\uparrow \infty}{\Rightarrow} \mathcal{L}(Z')
    \right\},
    \]
    where $\Rightarrow$ denotes weak convergence.
    \end{corollary}

\noindent \paragraph{Discussion and Notation.}
For notation, $\mathcal{P}_p(\mathbb{R}^d)$ denotes the space of probability measures on $\mathbb{R}^d$ with finite $p$-th moment, 
\begin{align*}
    \mathcal{P}_p(\mathbb{R}^d) := \left\{
        \mu \in \mathcal{P}(\mathbb{R}^d) : \int_{\mathbb{R}^d} |x - \bar{x}|^p d\mu(x) < \infty, \text{for some } \bar{x} \in \mathbb{R}^d
    \right\}
\end{align*} 

 Weakly continuous curve $\mu_t : I \to \mathcal{P}_p(\mathbb{R}^d)$ requires that for all bounded continuous function $f: \mathbb{R}^d \to \mathbb{R}$, the map $t \mapsto \int_{\mathbb{R}^d} f(x) d\mu_t(x)$ is continuous on $I$, i.e.,:
\begin{align*}
    \lim_{n \to \infty} t_n \to t \implies \lim_{n \to \infty} \int_{\mathbb{R}^d} f(x) d\mu_{t_n}(x) = \int_{\mathbb{R}^d} f(x) d\mu_t(x)
\end{align*}
The original theorem in \cite[Theorem~8.3.1]{ambrosio2005gradient} only requires the continuity equation to hold in weak (distributional) sense, i.e., for all $\varphi \in C_c^{\infty}(\mathbb{R}^d\times I)$:
\begin{align*}
    \int_{I} \int_{\mathbb{R}^d} \partial_t \varphi(x,t) + \langle \nabla\varphi(x,t) , v_t(x) \rangle d\mu_t(x) dt = 0,
\end{align*}
while we provide the strong form of continuity equation in the lemma, and we verify the strong form holds for our case because of the smoothness of the density of $\mu_t$ in Lemma \ref{lemma: OU flow property}. For the interval $I$, we take $I = (0,\infty)$ in our application.
The statement in \cite[Theorem~8.3.1]{ambrosio2005gradient} also contains the converse part: if $\mu_t$ is absolutely continuous, then there exists a Borel vector field $v_t: \mathbb{R}^d \times I \to \mathbb{R}^d$ such that $v_t \in L^p(\mu_t)$ and $\|v_t\|_{L^p(\mu_t)} \leq |\mu'|(t)$ for Lebesgue-a.e. $t \in I$. And the continuity equation holds. This converse statement is not used in our proof and we only depends on the direction in Lemma \ref{lemma: Benamou Brenier for W-p}. See also \cite[Introduction]{lisini2007characterization} for more discussion on the continuity equation and absolutely continuous curves in Wasserstein space.

 
The following lemma verifies that Ornstein-Uhlenbeck (OU) flow provides a valid continuity equation connecting $\mathcal{L}(Y)$ and $\mathcal{L}(Z')$. 
Consider $F_t$ as the state at time $t>0$ for OU flow starting from $Y$.
\begin{lemma}\label{lemma: OU flow property}
Let $p>1$ and let $X\in \mathbb{R}^d$ a random vector with law $\mathcal{L}(X)\in\mathcal{P}_p(\mathbb{R}^d)$. Let $T:\mathbb{R}^d\to\mathbb{R}$ be a Borel measurable map such that $Y:=T(X)$ is a real random variable with $\mathbb{E}|Y|^p<\infty$. Let $Z\sim\mathcal{N}(0,1)$ be independent of $X$. For $t>0$ define
\[
F_t := e^{-t}Y + \sqrt{1-e^{-2t}}\,Z,
\]
and let $f_t$ denote the density of $F_t$ (w.r.t.\ Lebesgue measure). Then the following hold.
\begin{enumerate}
    \item (OU representation) \textnormal{\cite[Appendix~A.2]{fang2022wasserstein}} The family $(F_t)_{t>0}$ is the Ornstein--Uhlenbeck semigroup acting on $Y$; equivalently, $(F_t)$ is a version of the solution to the SDE
        \[
            dF_t = -F_t\,dt + \sqrt{2}\,dB_t,\qquad F_0=Y,
        \]
        where $(B_t)$ is a standard Brownian motion independent of $Y$.

    \item (Smooth positive density) \textnormal{\cite[Lemma 3.1]{johnson2001entropy}} For every $t>0$ the law of $F_t$ admits a strictly positive, smooth density $f_t$.
    \item (Commutation of OU semigroup and differentiation) \textnormal{\cite{bakry2013analysis} \cite[Eq. 16 $\&$ Lemma 4]{bonis2020steins}} For every $t>0 , k\in\mathbb{N}$, and $g \in C^{\infty}_c(\mathbb{R})$,
        \[
            \frac{d^k}{dx^k} P_t g(x) = e^{-kt} P_t\!\Bigl(\frac{d^k}{dx^k} g\Bigr)(x),\qquad x\in\mathbb{R},
        \]
        where $P_t g(x):=\mathbb{E}[g(F_t)\mid F_0=x]$ is the OU semigroup acting on $g$. Moreover,
        \[
         P_t \frac{d^k}{dx^k}g(x) = \frac{1}{\sqrt{1- e^{-2t}}^k} \mathbb{E}[h_k(Z) g(e^{-t}x + \sqrt{1-e^{-2t}}Z)], \qquad x\in\mathbb{R},
        \]
        where $h_k(x) = (-1)^k e^{x^2/2} \frac{d^k}{dx^k}e^{-x^2/2}$ is the $k$-th Hermite Polynomial.

    \item (Continuity equation) For every $t>0$ the density $f_t$ satisfies the continuity equation
        \[
            \partial_t f_t(x) + \frac{d}{dx}\bigl(f_t(x)\,\rho_t(x)\bigr)=0,\qquad x\in\mathbb{R},
        \]
        where $\rho_t(x):=-\partial_x\log f_t(x)-x$ is the score function of $f_t$.

    \item (Score representation) \textnormal{\cite[Lemma 2]{bonis2020steins}} For each $t>0$ the score admits the conditional expectation representation
        \[
            \rho_t(F_t)\;\overset{d.}{=}\; \mathbb{E}\!\Bigl[e^{-t}Y-\frac{e^{-2t}}{\sqrt{1-e^{-2t}}}\,Z\;\Big|\;F_t\Bigr],
        \]
        and, consequently,
        \[
            t\mapsto \|\rho_t(F_t)\|_{L^p(\mathcal{L}(F_t))}\in L^1((0,\infty)).
        \]

    \item (Moment bound and weak continuity) The curve $t\mapsto\mathcal{L}(F_t)$ is weakly continuous on $(0,\infty)$, and
        \[
            \sup_{t>0}\|F_t\|_{L^p}<\infty,
        \]
        (equivalently $\sup_{t>0}\|x\|_{L^p(\mathcal{L}(F_t))}<\infty$).
\end{enumerate}
\end{lemma}
\noindent We delay the proof of this lemma after the main proof for sake of  coherence. With this lemma, we can apply Benamou Brenier Inequality \ref{lemma: Benamou Brenier for W-p} to the OU flow $F_t$ connecting $\mathcal{L}(Y)$ and $\mathcal{L}(Z)$, and obtain the following bound for Wasserstein-$p$ distance via property (4) in Lemma \ref{lemma: OU flow property}.
\begin{align}
    W_p(\mathcal{L}(Y), \mathcal{L}(Z')) &\leq \int_{0}^{\infty} \| \rho_t(F_t)\|_{L^p} dt, \;p > 1  \quad \text{when }\mathbb{E}[|Y|^p]<\infty, \label{eq: BB bound for OU flow}
\end{align}
where $\rho_t$ is the score function of $F_t$ defined in property (4) in Lemma \ref{lemma: OU flow property}, and the above inequality holds since $\mathcal{L}(F_t)$ is a valid curve connecting $\mathcal{L}(Y)$ and $\mathcal{L}(Z')$ with the velocity field given by $v_t(x) = \rho_t(x)$ (see property (6) in Lemma \ref{lemma: OU flow property} for the weak continuity and the $p$-th moment bound for $\mathcal{L}(F_t)$). Moreover, from property (5) in Lemma \ref{lemma: OU flow property}, we have $\int_{0}^{\infty} \| \rho_t(F_t)\|_{L^p} dt < \infty$. Therefore, starting from the Benamou Brenier dynamic formulation \ref{lemma: Benamou Brenier for W-p}, we take integral on both sides and obtain the above bound for Wasserstein-$p$ distance.

We proceed to provide the following lemma, which yields a zero's identity that is used to formalize the generator comparison. We first summarize the intuition behind the lemma before providing the proof.

In canonical Stein's method, we subtract the zero term $\mathbb{E}[\mathcal{A} f(X)] = 0$ to the Stein equation at no cost. This subtraction yields a generator comparison $\mathbb{E}[\mathcal{A} f(X) - \mathcal{L}_{OU} f(X)]$ in the distributional distance bound. However, when working with the Benamou-Brenier formulation in \eqref{eq: BB bound for OU flow}, the upper bound involves an expectation with respect to $\mathcal{L}(F_t)$ rather than $\mathcal{L}(X)$. Thus, we must construct a zero term that is measurable with respect to $\sigma(F_t)$. This zero term is constructed via the zero's identity $\mathbb{E}[\mathcal{A} f(X)] = 0$. To change measure from $\mathcal{L}(X)$ to $\mathcal{L}(F_t)$, we choose a Lyapunov function $f$ that pushes forward $\mathcal{L}(X)$ to $\mathcal{L}(F_t)$, given by $f = P_t G$ for arbitrary $G$. We then apply commutation relations and Gaussian integration by parts repeatedly to eliminate all derivatives except the first-order derivative, which establishes the zero's identity.


Recall that $\mathcal{A}$ is the operator that satisfies Assumption \ref{assumption: Stein identity final version}, \ref{assumption: Kramers-Moyal expansion final version} and \ref{assumption: integrability of Kramers-Moyal coefficients final version},
 and $\sigma(Z)$ is independent of $\sigma(X)$.



\begin{lemma}\textnormal{(Zero's identity under OU flow) }
    \label{claim: bonis lemma 4}Under the above conditions, with $t >0$,
     the following construction,
    \begin{align*}
        \tau_t := \sum_{k=1}^{\infty} \frac{e^{-kt}}{\sqrt{1-e^{-2t}}^{k-1}}  h_{k-1}(Z) a^{\mathcal{A}, T}_k(X),
    \end{align*}
    satisfies $\mathbb{E}[\tau_t \mid F_t] = 0$ almost surely.
\end{lemma}
\begin{proof}[Proof of Lemma \ref{claim: bonis lemma 4}]
Consider any $g \in C_c^\infty(\mathbb{R})$ and let $G$ be an antiderivative of $g$, i.e., $G' = g$, we verify that $P_t G$ is in $C_b^\infty(\mathbb{R})$. This is satisfied since $\|P_t G\|_\infty$ is bounded by $\|G\|_\infty$. Additionally, from property (3) in Lemma \ref{lemma: OU flow property}, all the derivatives of $P_t G$ are bounded. Thus $P_t G \in C_b^\infty(\mathbb{R})$. From property (3) in Lemma \ref{lemma: OU flow property}, with the notation $D:= \frac{\max\{\|G\|_\infty,1\}}{\sqrt{e^{2t}-1}}$, we can quantify the bound on $k$-th derivative of $P_t G$ as follows,
\begin{align}
    \| (P_t G)^{(k)} \|_{\infty} &\leq \|G\|_{\infty} \frac{1}{\sqrt{e^{2t}-1}^{k}} \mathbb{E}[|h_k(Z)|] \notag\\
    &\overset{(a)}{\leq} \sqrt{k!}
    \|G\|_{\infty} \frac{1}{\sqrt{e^{2t}-1}^{k}} \notag\\
    &\leq \sqrt{k!}D^k. \label{eq: bound on k-th derivative of P_t G}
\end{align}
Inequality $(a)$ holds from the bound on Hermite polynomial's moment from \cite[Lemma 3]{bonis2020steins}
 as $\|h_{k}(Z)\|_{L^p} \leq \sqrt{p}^k \sqrt{k!}$. We use $p=1$ here. Therefore $P_t G$ satisfies the bound on $k$-th derivative in Definition \ref{assumption: function class}, and $P_t G \in \mathcal{F}$.

Now we have the following equalities
\begin{align}
    0 &\overset{(a)}{=}  \mathbb{E}\big[[\mathcal{A} ((P_t G) \circ T)](X)\big], \text{ for }  g \in C_c^\infty(\mathbb{R}), G' = g \notag\\
    &\overset{(b)}{=}  \mathbb{E}[\sum_{k=1}^{\infty}  (P_t G)^{(k)}(Y)  a^{\mathcal{A}, T}_k(X)]\notag\\
    &\overset{(c)}{=} \mathbb{E}[\sum_{k=1}^{\infty}e^{-kt} P_t g^{(k-1)}(Y)  a^{\mathcal{A}, T}_k(X) ] \notag\\
    &\overset{(d)}{=} \mathbb{E}[\underbrace{\sum_{k=1}^{\infty} \frac{e^{-kt}}{\sqrt{1-e^{-2t}}^{k-1}} h_{k-1}(Z)   a^{\mathcal{A}, T}_k(X)}_{:=\tau_t} g(F_t) ] \notag
\end{align}
Equality $(a)$ is from Assumption \ref{assumption: Stein identity final version}, where we apply the Stein identity $\mathbb{E}[[\mathcal{A} (f\circ T)](X)] = 0$
to the test function $f(x) = (P_t G)(x)$.
Equality $(b)$ is from Kramers Moyal Expansion in Assumption \ref{assumption: Kramers-Moyal expansion final version} and the change of variable $Y=T(X)$. We are able to apply this assumption since $P_t G \in \mathcal{F}$ according to upper bound \eqref{eq: bound on k-th derivative of P_t G}. 
Equality $(c)$ is from the commutation relation of OU semigroup and derivative operator as $\frac{d}{dx}P_t f(x) = e^{-t}P_t \frac{d}{dx}f(x)$ in property (3) of Lemma \ref{lemma: OU flow property}.  Equality $(d)$ is from 
integration by parts formula for Hermite polynomials (see also property (3) in Lemma \ref{lemma: OU flow property}),
\begin{align*}
    e^{-kt} P_t\!\Bigl(\frac{d^k}{dx^k} g\Bigr)(x) &= \frac{d^k}{dx^k} P_t g(x) = \frac{1}{\sqrt{e^{2t}-1}^k} \mathbb{E}[h_k(Z) g(e^{-t}x + \sqrt{1-e^{-2t}}Z)], \; \forall x \in \mathbb{R}.
\end{align*}
We multiply both sides by $e^{kt}$, and take $k-1$ instead of $k$ to obtain
\begin{align*}
    P_t g^{(k-1)}(x) = \frac{1}{\sqrt{1-e^{-2t}}^{k-1}} \mathbb{E}[h_{k-1}(Z) g(e^{-t}x + \sqrt{1-e^{-2t}}Z)], \; \forall x \in \mathbb{R},
\end{align*}
which justifies equality $(d)$. 

Note that the above equation holds for any $g \in C_c^\infty(\mathbb{R})$, and that $C_c^\infty(\mathbb{R})$ is dense in $C_o(\mathbb{R})$ under supremum norm, we have $\mathbb{E}[\tau_t \cdot g(F_t)] = 0$ for all $g \in C_o(\mathbb{R})$. The following standard argument shows $\mathbb{E}[\tau_t \mid F_t] = 0$ almost surely. We define a signed measure $\nu(A) := \mathbb{E}[\tau_t \cdot \mathbf{1}_{F_t \in A}]$ for all Borel set $A \subseteq \mathbb{R}$. According to the upper bound in \eqref{eq: bound on k-th derivative of P_t G}, we have the following bound on $\mathbb{E}[|\tau_t|]$
\begin{align*}
    \mathbb{E}[|\tau_t|] &= \mathbb{E}[\sum_{k=1}^{\infty} \frac{e^{-kt}}{s\sqrt{1-e^{-2t}}^{k-1}} \cdot \mathbb{E}[|h_{k-1}(Z)| \cdot | a^{\mathcal{A}, T}_k(X)| \mid \sigma(X) \lor \sigma(Z)]] \\
    &\overset{(a)}{\leq} \mathbb{E}[\sum_{k=1}^{\infty} \frac{e^{-kt}}{s\sqrt{1-e^{-2t}}^{k-1}} \cdot |h_{k-1}(Z)| \cdot | a^{\mathcal{A}, T}_k(X)| ] \\
    &\overset{(b)}{=} \sum_{k=1}^{\infty} \frac{e^{-kt}}{s\sqrt{1-e^{-2t}}^{k-1}} \cdot \mathbb{E}[|h_{k-1}(Z)| \cdot | a^{\mathcal{A}, T}_k(X)| ] \\
    &\overset{(c)}{=} \sum_{k=1}^{\infty} \frac{e^{-kt}}{s\sqrt{1-e^{-2t}}^{k-1}} \cdot \|h_{k-1}(Z)\|_1 \cdot \mathbb{E}[| a^{\mathcal{A}, T}_k(X)| ] \overset{(d)}{<} \infty,
\end{align*}
Inequality $(a)$ is from Jensen's inequality. Equality $(b)$ is from Tonelli's theorem since all the terms are non-negative. Inequality $(c)$ is from the independence between $\sigma(Z)$ and $\sigma(X)$. Inequality $(d)$ is from Assumption \ref{assumption: integrability of Kramers-Moyal coefficients final version} and Tonelli's theorem, since $\|h_{k-1}(Z)\|_1 \leq \sqrt{(k-1)!}$. Thus, $\nu$ is a finite signed measure. Since $\int g d\nu = 0$ for all $g \in C_o(\mathbb{R})$, by uniqueness from Riesz representation theorem, we have $\nu \equiv 0$. Therefore, $\mathbb{E}[\tau_t \mid F_t] = 0$ almost surely.

\end{proof}

With the above identity, we add $\tau_t$ to the integrand of the right hand side of \eqref{eq: BB bound for OU flow} to obtain
\begin{align*}
    \|\rho_t(F_t)\|_{L^p} &\overset{(a)}{=} \mathbb{E}\left[ \bigl|\mathbb{E}[ e^{-t} Y -\frac{e^{-2t}}{\sqrt{1-e^{-2t}}} Z \mid F_t]\bigr|^p\right]^{1/p} \\
     & \overset{(b)}{=}
    \mathbb{E}\left[ \bigl|\mathbb{E}[ e^{-t} Y -\frac{e^{-2t}}{\sqrt{1-e^{-2t}}} Z +\tau_t \mid F_t]\bigr|^p\right]^{1/p} \\
    &\overset{(c)}{\leq} \mathbb{E}\left[\bigl|e^{-t} Y -\frac{e^{-2t}}{\sqrt{1-e^{-2t}}} Z + \tau_t\bigr|^p\right]^{1/p} \\
        &\overset{(d)}{=} \mathbb{E}\left[\bigl|e^{-t} Y -\frac{e^{-2t}}{\sqrt{1-e^{-2t}}} h_1(Z) + \sum_{k=1}^{\infty} \frac{e^{-kt}}{\sqrt{1-e^{-2t}}^{k-1}}  h_{k-1}(Z) a^{\mathcal{A}, T}_k(X)\bigr|^p\right]^{1/p} \\
    &\overset{(e)}{\leq}  e^{-t} \cdot \left\| h_0(Z)(Y + a^{\mathcal{A}, T}_1(X))\right\|_{L^p} +  \frac{e^{-2t}}{\sqrt{1-e^{-2t}}} \cdot \left\| h_1(Z)(a^{\mathcal{A}, T}_2(X) -1)\right\|_{L^p}  \\
        &\quad + \sum_{k=3}^{\infty} \frac{e^{-kt}}{\sqrt{1-e^{-2t}}^{k-1}} \cdot\left\|  h_{k-1}(Z)a^{\mathcal{A}, T}_k(X) \right\|_{L^p} \\
    &\overset{(f)}{=} e^{-t} \cdot \left\| Y + a^{\mathcal{A}, T}_1(X)\right\|_{L^p} +  \frac{e^{-2t}\|h_1(Z)\|_{L^p}}{\sqrt{1-e^{-2t}}} \cdot \left\| a^{\mathcal{A}, T}_2(X) -1\right\|_{L^p} \\
        &\quad + \sum_{k=3}^{\infty} \frac{e^{-kt}\|h_{k-1}(Z)\|_{L^p}}{\sqrt{1-e^{-2t}}^{k-1}} \cdot\left\|  a^{\mathcal{A}, T}_k(X) \right\|_{L^p} 
\end{align*}
Equation $(a)$ is from the representation of score function $\rho_t(F_t)$ in property (5) of Lemma \ref{lemma: OU flow property}. Equaiton $(b)$ is from Lemma \ref{claim: bonis lemma 4}, where we add $\tau_t$ to the integrand and use the fact that $\mathbb{E}[\tau_t \mid F_t] = 0$ almost surely. Inequality $(c)$ is from Jensen's inequality. Equation $(d)$ is from the definition of $\tau_t$.
Inequality $(e)$ is from the definition of $\tau_t$ and triangle inequality of $L^p$ norm. Equation $(f)$ is from the independence between $\sigma(Z)$ and $\sigma(X)$.


The above bound holds for all $t>0$. Thus, integrating over $t\in (0,\infty)$, with Fubini-Tonelli theorem, we have
\begin{align}
    \mathcal{W}_p(\mathcal{L}(Y), \mathcal{L}(Z')) &\leq \int_{0}^{\infty} \| \rho_t(F_t)\|_{L^p} \,dt \notag\\
    &\overset{(a)}{\leq} \int_{0}^{t_0} \Big\| e^{-t} Y - \frac{e^{-2t}}{\sqrt{1-e^{-2t}}}\,Z \Big\|_{L^p} \,dt
        + \int_{t_0}^{\infty} \| \rho_t(F_t)\|_{L^p} \,dt \notag\\
    &\leq \underbrace{\int_{0}^{t_0} \Big\| e^{-t}Y - \frac{e^{-2t}}{\sqrt{1-e^{-2t}}}\,Z \Big\|_{L^p}\,dt }_{:=(b_0)}
        \notag\\
    &\quad + {\int_{t_0}^{\infty} e^{-t}\,dt} \;\Big\| \,a^{\mathcal{A}, T}_1(X) + Y\Big\|_{L^p}\notag\\
    &\quad + \int_{t_0}^{\infty} \frac{e^{-2t}\,\|h_1(Z)\|_{L^p}}{\sqrt{1-e^{-2t}}}\,dt \;\Big\| a^{\mathcal{A}, T}_2(X) - 1 \Big\|_{L^p} \notag\\
    &\quad + \sum_{k=3}^{\infty}\int_{t_0}^{\infty}
                \frac{e^{-kt}\,\|h_{k-1}(Z)\|_{L^p}}
                    {\bigl(1-e^{-2t}\bigr)^{(k-1)/2}}\,dt\;
                \Big\|  a^{\mathcal{A}, T}_k(X) \Big\|_{L^p}
                 \label{eq: W-p bound constants}
\end{align}
Inequality $(a)$ is from splitting the integral into two parts, and applying the score representation in property (5) of Lemma \ref{lemma: OU flow property} for the first part. The second part is from the bound on $\|\rho_t(F_t)\|_{L^p}$ above that leads to bound \eqref{eq: W-p bound constants}.

We handle the first summand $(b_0)$ as follows,
\begin{align}
    (b_0) &= \int_{0}^{t_0} \| e^{-t} Y - \frac{e^{-2t}}{\sqrt{1-e^{-2t}}} Z\|_{L^p} \notag\\
    &\overset{(a)}{\leq} \int_{0}^{t_0} e^{-t} \|Y - G\|_{L^p} + \|e^{-t} G - \frac{e^{-2t}}{\sqrt{1-e^{-2t}}} Z\|_{L^p} dt \notag\\ 
    &\overset{(b)}{\leq} (1- e^{-t_0}) \cdot (W_p(\mathcal{L}(Y),\mathcal{L}(Z')) +\epsilon) + \int_{0}^{t_0}\|e^{-t} G - \frac{e^{-2t}}{\sqrt{1-e^{-2t}}} Z\|_{L^p} dt \notag\\
    &\overset{(c)}{=} (1- e^{-t_0}) \cdot (W_p(\mathcal{L}(Y),\mathcal{L}(Z')) +\epsilon) + \|Z\|_{L^p} \int_{0}^{t_0} \frac{e^{-t}}{\sqrt{1-e^{-2t}}}dt \notag\\
    &\overset{(d)}{\leq} (1- e^{-t_0}) \cdot W_p(\mathcal{L}(Y),\mathcal{L}(Z')) + \|Z\|_{L^p} (\frac{\pi}{2} - \sin^{-1}(e^{-t_0})) \label{eq: W-p bound b0}
\end{align}
Here $G \sim \mathcal{N}(0,1)$ is another standard normal random variable. Inequality $(a)$ is from triangle inequality of $L^p$ norm.
For inequality $(b)$ and $(c)$, we first note that there exists a coupling of $(Y, G, Z)$ such that we can couple $Y$ and $G$ such that $\|Y - G\|_{L^p}$ is within $\epsilon$ from the Wasserstein distance $W_p(\mathcal{L}(Y),\mathcal{L}(Z'))$. Letting $\epsilon \to 0$ justifies inequality $(b)$. Additionally, we can couple $G$ and $Z$ such that
$G$ is independent of $Z$, leading to $e^{-t} G - \frac{e^{-2t}}{\sqrt{1-e^{-2t}}} Z \overset{d}{=} \frac{e^{-t}}{\sqrt{1-e^{-2t}}} Z$ in inequality $(c)$.
 Inequality $(d)$ holds since we have upper bound for arbitrary $\epsilon > 0$. Taking $\epsilon \to 0$ gives the bound $(d)$. Thus after rearranging the terms, we have 
\begin{align*}
    e^{-t_0} W_p(\mathcal{L}(Y), \mathcal{L}(Z')) \leq W_p(\mathcal{L}(Y), \mathcal{L}(Z')) - (b_0) +\|Z\|_{L^p} (\frac{\pi}{2} - \sin^{-1}(e^{-t_0})).
\end{align*}
Thus we can plug this bound into \eqref{eq: W-p bound constants} and rearrange to get the proposition \ref{pro: Wasserstein-$p$ bound without exchangeability final version} as follows.
\begin{align}
    W_p(\mathcal{L}(Y), \mathcal{L}(Z')) &\leq \underbrace{e^{t_0}\|Z\|_{L^p} (\frac{\pi}{2} - \sin^{-1}(e^{-t_0})) }_{:=g_0(t_0,p)} \notag\\
    &\quad + \underbrace{e^{t_0}\int_{t_0}^{\infty} e^{-t}\,dt}_{:=g_1(t_0,p)} \;\Big\| \,a^{\mathcal{A}, T}_1(X) + Y\Big\|_{L^p}\notag\\
    &\quad + \underbrace{e^{t_0}\int_{t_0}^{\infty} \frac{e^{-2t}\,\|h_1(Z)\|_{L^p}}{\sqrt{1-e^{-2t}}}\,dt}_{:=g_2(t_0,p)} \;\Big\|a^{\mathcal{A}, T}_2(X) - 1 \Big\|_{L^p} \notag\\
    &\quad + \sum_{k=3}^{\infty} \underbrace{e^{t_0}\int_{t_0}^{\infty}
                \frac{e^{-kt}\,\|h_{k-1}(Z)\|_{L^p}}
                    {\bigl(1-e^{-2t}\bigr)^{(k-1)/2}}\,dt}_{:=g_k(t_0,p)}\;
                \Big\|  a^{\mathcal{A}, T}_k(X) \Big\|_{L^p}
                 \label{eq: W-p bound final}
\end{align}

Recall that we denote $Y = T(X)$ and $Z' \sim \mathcal{N}(0,1)$. This upper bound provides the desired conclusion in Proposition \ref{pro: Wasserstein-$p$ bound without exchangeability final version}.
Now in order to conclude the proof of Proposition \ref{pro: Wasserstein-$p$ bound without exchangeability final version}, it suffices to prove lemma \ref{lemma: OU flow property}.
\paragraph{Proof of Lemma \ref{lemma: OU flow property}}
    For the first two parts, readers are referred to the references cited in each part. The commutation relation in part 3 can be shown via the Gaussian integration by parts (Stein's lemma).  For part 4, the continuity equation is verified via Fokker-Planck equation for OU process. Recall that $f_t$ is the density of $F_t$ w.r.t. Lebesgue measure, then: 
    \begin{align*}
        \partial_t f_t(x) &= \frac{d}{dx}(x f_t(x)) + h''_t(x) \\
        &= \frac{d}{dx}(f_t(x) \cdot (x + \frac{d}{dx} \log f_t(x))) = -\frac{d}{dx}(f_t(x) \cdot \rho_t(x)) 
    \end{align*}
    Part 5, i.e., $\|\rho_t(F_t)\|_{L^p(\mathcal{L}(F_t))} \in L_1((0,\infty))$, is shown as follows. We first verify the probabilistic representation of score function. For any test function $g \in C_c^\infty(\mathbb{R})$, by the continuity equation above, we first have
    \begin{align*}
        \mathbb{E}[\mathcal{L}_{OU} g(F_t)] &\overset{(a)}{=} \frac{d}{dt} \mathbb{E}[g(F_t)] = \int_{\mathbb{R}} g(x) \partial_t f_t(x) dx \\
        &\overset{(b)}{=} \int_{\mathbb{R}} g(x) \cdot -\frac{d}{dx}(f_t(x) \cdot \rho_t(x)) dx \overset{(c)}{=} \int_{\mathbb{R}} g'(x) f_t(x) \rho_t(x) dx = \mathbb{E}[g'(F_t) \rho_t(F_t)].
    \end{align*}
    Here inequality $(a)$ is from the generator definition of OU semigroup. Equality $(b)$ is from continuity equation. $(c)$ is from integration by parts, and the boundary term vanishes since $g$ has compact support. Note that this identity above holds for all $g \in C_c^\infty(\mathbb{R})$. Also note that if any random variable $Y$ satisfies $\mathbb{E}[g'(F_t) Y] = \mathbb{E}[\mathcal{L}_{OU} g(F_t)]$ for all $g \in C_c^\infty(\mathbb{R})$, then $\mathbb{E}[g'(F_t) Y] = \mathbb{E}[g'(F_t) \rho_t(F_t)]$ for all such $g$. Thus $Y = \rho_t(F_t)$ almost everywhere with respect to $\mathcal{L}(F_t)$ due to the density of $C_c^\infty(\mathbb{R})$ in $L^1(\mathbb{R},\sigma(F_t))$. With Gaussian integration by parts, it can be verified that the following choice of $Y$ satisfies the above equation,
    \begin{align*}
        \rho_t(F_t) \overset{d.}{=} \mathbb{E}[ e^{-t} X -\frac{e^{-2t}}{\sqrt{1-e^{-2t}}} Z \mid F_t].
    \end{align*}
    Thus it is a valid probabilistic representation of score function. By triangle inequality, when $\mathbb{E}[|X|^p]<\infty$,
    \begin{align*}
        \int_{0}^{\infty}\|\rho_t(F_t)\|_{L^p(\mathcal{L}(F_t))} dt &= \int_{0}^{\infty} \|\mathbb{E}[ e^{-t} X -\frac{e^{-2t}}{\sqrt{1-e^{-2t}}} Z \mid F_t] \|_{L^p(\mathcal{L}(F_t))} dt \\
        &\leq \int_{0}^{\infty} e^{-t} dt \mathbb{E}[|X|^p]^{1/p} +
        \int_{0}^{\infty} \frac{e^{-2t}}{\sqrt{1-e^{-2t}}}dt \mathbb{E}[|Z|^p]^{1/p} \\
        &\leq \mathbb{E}[|X|^p]^{1/p} +  \mathbb{E}[|Z|^p]^{1/p} < \infty
    \end{align*}
    In part 6, for the moment control, we first bound by triangle inequality,
    \begin{align*}
        \sup_{t > 0}\|F_t\|_{L^p} &\leq \sup_{t > 0} \{e^{-t} \|X\|_{L^p} + \sqrt{1-e^{-2t}} \|Z\|_{L^p}\} \leq \|X\|_{L^p} + \|Z\|_{L^p} < \infty   
    \end{align*}
    For the curve $t\mapsto F_t$ to be weakly  continuous, we need to show that for all bounded continuous function $g: \mathbb{R} \to \mathbb{R}$, the map $t \mapsto \mathbb{E}[g(F_t)]$ is continuous on $(0,\infty)$. Since $g$ is bounded continuous, for $t_0 > 0$, by Dominated Convergence Theorem,
    \begin{align*}
        \lim_{t \to t_0} \mathbb{E}[g(F_t)] &= \lim_{t \to t_0} \mathbb{E}[g(e^{-t}X + \sqrt{1-e^{-2t}}Z)] = \mathbb{E}[\lim_{t \to t_0} g(e^{-t}X + \sqrt{1-e^{-2t}}Z)] = \mathbb{E}[g(F_{t_0})]  \quad \square
    \end{align*}

\section{Proof for Single Server Queue}
Before the proof for each lemma (or theorem), we first introduce the stability of the single server queue (SSQ) and the well-definedness of moment generating function (MGF) for its stationary distribution. Let $V(q)=q^{2}$ and consider the SSQ generator defined in \eqref{eq: generator of SSQ}. For $\forall q\geq 0$:
\begin{align*}
    \mathcal{L}_{SSQ} V(q) &= \lambda(2q+1) + (\mu + \gamma q)(-2q+1)\mathbf{1}_{q\geq 1} \\
    &\leq -2\gamma q^2 + (2\lambda - 2\mu + \gamma) q + (\lambda + \mu) \\
    &\leq -\gamma q^2 + \frac{(2\lambda - 2\mu + \gamma)^2}{4\gamma} + (\lambda + \mu)
\end{align*}
Since the drift is negative outside the compact set $\{q\in[0, \frac{(2\lambda - 2\mu + \gamma)^2}{4\gamma^2} + \frac{\lambda + \mu}{\gamma}]\}$, the process is stable and has unique stationary distribution by Foster-Lyapunov theorem \cite{hajek1982hitting}. 
For well-definedness of MGF, we first develop stochastic comparison, coupling SSQ with the $M/M/\infty$ birth-death process $q_{M/M/\infty}$ with birth rate $\lambda$ and death rate $\gamma q$ (the $\mu=0$ degeneration). Since for $q\geq 0$ the SSQ death rate satisfies $\mu+\gamma q\ge\gamma q$, a standard monotone coupling yields $q_t\le q_{M/M/\infty,t}$ a.s. for all $t$ when started from the same initial state. Hence $q\le_{\mathrm{st}} q_{M/M/\infty}$ , and $q_{M/M/\infty}\stackrel{d}{=}\text{Poisson}(\lambda/\gamma)$. Therefore, for all $\theta\in\mathbb{R}$,
\begin{align}
\mathbb{E}\big[e^{\theta q}\big] &\le 1\lor \mathbb{E}\big[e^{\theta q_{M/M/\infty}}\big] = 1\lor \exp\!\Big(\tfrac{\lambda}{\gamma}(e^{\theta}-1)\Big) < \infty, \label{eq: well-definedness of MGF for SSQ}
\end{align}
so the MGF of $q$ is well defined (see also Strassen's coupling theorem \cite{strassen1965existence}).

In this section, we will assume $\gamma \leq \gamma_0$ for some small enough $\gamma_0$, depending on $\lambda$ and $\mu$. And we will recall this assumption repeatedly in the proof with explicit form of $\gamma_0$ later.
\begin{align}
    \gamma \leq \gamma_0 := \min \left\{
        \frac{e^{4\sqrt{2}}\lambda}{25},
        \frac{\mu^2}{(2^2 e^{4\sqrt{2}}\lambda)},
        \mu, (1/2eD_{1,\lambda,\mu})^{-\frac{1}{1/2-\alpha-\epsilon}}, 
        (2eD_{1,\lambda,\mu})^{-\frac{2}{1/2-\alpha-\epsilon}},
        (2\sqrt{2}eD_{1,\lambda,\mu})^{-1}, \lambda/(A_{\lambda,\mu}^2). 
    \right\}
    \label{eq: gamma assumption for SSQ W-p}
\end{align}

\subsection{Proof for Lemma \ref{lem: Wasserstein-$p$ bound SSQ reduction}: Wasserstein-$p$ Bound Reduction for SSQ} \label{sec: proof of Wasserstein-p bound reduction for SSQ}

Before the proof, we first introduce the following lemma to convert sub-exponential MGF condition to $L^p$ norm bound from \cite{vershynin2018high}. We delay the proof of this lemma at the end of this subsection.

\begin{lemma} \label{lem: sub-exponential Lp norm} 
    \textnormal{(Subexponential from MGF to $L^p$ bound) }For sub-exponential random variable $X \sim SE(\beta, \alpha)$, defined by the condition:
    \begin{align*}
        \mathbb{E}[e^{\theta|X|}]  \leq A\cdot e^{\beta^2 \theta^2/2}, \quad \forall \theta \in [0,1/\alpha]
    \end{align*}
    We have upper bound for $L^p$ norm, $\forall p \geq 1$, in terms of $\beta$ and $\alpha$.
    \begin{align*}
        \mathbb{E}[|X|^p] &\leq 2e\sqrt{2\beta}A^{1/p} \cdot \bigg( \beta \sqrt{p} + \alpha p \bigg)
    \end{align*}
\end{lemma}

For standard normal distribution, we can view it as $SE(1, 0)$, thus give us:
\begin{align}
    \mathbb{E}[|Z|^p] &\leq 2e\sqrt{2\pi} \cdot \sqrt{p} \label{eq: Lp norm of Z}
\end{align}

\begin{proof}[Proof of Lemma \ref{lem: Wasserstein-$p$ bound SSQ reduction}]
For this reduction, we study the below $p$'s range:
\begin{align}
    1 < p \leq \frac{\lambda}{2} (\frac{1}{\gamma}) \label{eq: Initial p's regime for Wasserstein distance SSQ}
\end{align} which is non-empty since $\gamma\leq \gamma_0$. This assumption on parameters guarantees $\gamma < \lambda/2$ and the above interval is non-empty.
We will apply Proposition \ref{pro: Wasserstein-$p$ bound without exchangeability final version} to the SSQ setting. First, we defined the operator $\mathcal{A}$ as the scaled infinitesimal generator of the SSQ process.  Formally, $\mathcal{A}$ is defined as follows,
\begin{align}
    \mathcal{A} f(x) &:= \frac{1}{\gamma}\mathcal{L}_{SSQ}f (x) = 
     \lim_{t \to 0} \frac{1}{\gamma t} \bigg( \mathbb{E}[f(q_t) - f(q_0) \mid q_0 = x] \bigg), \quad \forall f: \mathbb{N} \to \mathbb{R}, \label{eq: implicit formula for A, SSQ}
\end{align}

We let the transformation $T_1$ be defined as $T_1(x) := \frac{\sqrt{\gamma}}{\sqrt{\lambda}} (x-\frac{\lambda -\mu}{\gamma})$. Thus we have $\tilde{q}= T_1(q)$. For this choice of $\mathcal{A}$ and $T_1$, we can compute the action of $\mathcal{A}$ on function $f \circ T_1$ explicitly as follows.
\begin{align}
    \mathcal{A} (f \circ T_1)(x) &= \lim_{t \to 0} \frac{1}{\gamma t} \bigg( \mathbb{E}[f(T_1(q_t)) - f(T_1(q_0)) \mid q_0 = x] \bigg) \notag\\
    &= \lim_{t \to 0} \frac{1}{\gamma t} \bigg( \mathbb{E}[f(\frac{\sqrt{\gamma}}{\sqrt{\lambda}}(q_t - \frac{\lambda -\mu}{\gamma})) - f(\frac{\sqrt{\gamma}}{\sqrt{\lambda}}(x - \frac{\lambda -\mu}{\gamma})) \mid q_0 = x] \bigg) \notag\\
    &= \frac{\lambda}{\gamma} \bigg[ f\bigg(\frac{\sqrt{\gamma}}{\sqrt{\lambda}}(x+1 - \frac{\lambda -\mu}{\gamma})\bigg) - f\bigg(\frac{\sqrt{\gamma}}{\sqrt{\lambda}}(x - \frac{\lambda -\mu}{\gamma})\bigg) \bigg] \notag\\
    &\quad + \frac{1}{\gamma}\bigg[ \mu + \gamma x \bigg] \cdot \bigg[ f\bigg(\frac{\sqrt{\gamma}}{\sqrt{\lambda}}(x-1 - \frac{\lambda -\mu}{\gamma})\bigg) - f\bigg(\frac{\sqrt{\gamma}}{\sqrt{\lambda}}(x - \frac{\lambda -\mu}{\gamma})\bigg) \bigg] \mathbf{1}_{\{x \geq 1\}}. \label{eq: explicit formula for A, SSQ}
\end{align}

 Additionally, from the well-definedness of MGF in \eqref{eq: well-definedness of MGF for SSQ}, we know that $T_1(q)$ has finite moments of all orders. Recall that for notation convention, we will apply the operator $\mathcal{A}$ on function $f:\mathbb{R} \to \mathbb{R}$ by restricting its domain to $\operatorname{supp}(\mathcal{L}(q))$. Now we verify the assumptions in Proposition \ref{pro: Wasserstein-$p$ bound without exchangeability final version} as follows. 

First, for any function $f\in C_b(\mathbb{R})$, we want to check that $\mathcal{A} (f \circ T_1) \in L^1(\mathcal{L}(q))$. Since $f$ is bounded, it suffices to check that function $q$ is in $L^1(\mathcal{L}(q))$. By the well-definedness of MGF in \eqref{eq: well-definedness of MGF for SSQ}, the first moment of $q$ is finite and thus $q$ is in $L^1(\mathcal{L}(q))$. By the definition of invariant distribution, we have $\mathbb{E}[\mathcal{A} (f \circ T_1)(q)] = 0$. Thus Assumption \ref{assumption: Stein identity final version} holds.

Secondly, we verify Assumption \ref{assumption: Kramers-Moyal expansion final version} and \ref{assumption: integrability of Kramers-Moyal coefficients final version} as follows. We first prove the following explicit formula for $a^{\mathcal{A}, T_1}_k(x)$, $\forall k \geq 1$ via Lemma \ref{lemma: identity for Kramers-Moyal coefficients final version}.
\begin{align*}
    a^{\mathcal{A}, T_1}_k(x) = \frac{1}{k!\gamma}\bigg(\frac{\sqrt{\gamma}}{\sqrt{\lambda}}\bigg)^k \cdot \bigg( \lambda + (-1)^k [\mu + \gamma x] \mathbf{1}_{\{x \geq 1\}} \bigg), \quad \forall k \geq 1.
\end{align*}

This explicit formula is derived by applying the operator $\mathcal{A}$ to the function $(\cdot - T_1(x))^k \circ T_1$ and evaluating at $x$ as follows. 
\begin{align*}
    a^{\mathcal{A}, T_1}_k(x) &\overset{(a)}{=} \frac{1}{k!} [\mathcal{A} ((\cdot - T_1(x))^k \circ T_1)] (x) \\
    &= \frac{\lambda}{k!\gamma} \bigg[ \bigg(\frac{\sqrt{\gamma}}{\sqrt{\lambda}}(x+1 - \frac{\lambda-\mu}{\gamma}) - \frac{\sqrt{\gamma}}{\sqrt{\lambda}}(x - \frac{\lambda-\mu}{\gamma})\bigg)^k \bigg] \\
    &\quad + \frac{1}{k!\gamma}\bigg[ \mu + \gamma x \bigg] \cdot \bigg[ \bigg(\frac{\sqrt{\gamma}}{\sqrt{\lambda}}(x-1 - \frac{\lambda-\mu}{\gamma}) - \frac{\sqrt{\gamma}}{\sqrt{\lambda}}(x - \frac{\lambda-\mu}{\gamma})\bigg)^k \bigg] \mathbf{1}_{\{x \geq 1\}} \\
    &= \frac{\lambda}{k!\gamma}\bigg[ \bigg(\frac{\sqrt{\gamma}}{\sqrt{\lambda}}\bigg)^k - 0 \bigg] + \frac{1}{k!\gamma}\bigg[ \mu + \gamma x \bigg] \cdot \bigg[ (-\frac{\sqrt{\gamma}}{\sqrt{\lambda}})^k - 0 \bigg] \mathbf{1}_{\{x \geq 1\}} \\
    &= \frac{1}{k!\gamma}\bigg(\frac{\sqrt{\gamma}}{\sqrt{\lambda}}\bigg)^k \cdot \bigg( \lambda + (-1)^k [\mu + \gamma x] \mathbf{1}_{\{x \geq 1\}} \bigg), \quad \forall k \geq 1. 
\end{align*}

 For Assumption \ref{assumption: Kramers-Moyal expansion final version}, we need to show that the infinite Taylor series converges to \(\mathcal{A} (f \circ T_1)(x)\). By Taylor's theorem with Lagrange remainder, for each fixed $n \in \mathbb{N}$, we let $\xi_{\pm}$ be the two points (
    subscripts $\pm$ represent the two points respectively
 ) between $T_1(x)$ and $T_1(x \pm 1)$ such that the following expansion holds,
\begin{align*}
    f(T_1(x \pm 1)) &= \sum_{j=0}^{n} \frac{1}{j!} f^{(j)}(T_1(x)) (T_1(x \pm 1) - T_1(x))^j + \frac{1}{(n+1)!} f^{(n+1)}(\xi_{\pm}) (T_1(x \pm 1) - T_1(x))^{n+1}\\
    &= \sum_{j=0}^{n} \frac{1}{j!} f^{(j)}(T_1(x)) \bigg(\pm \frac{\sqrt{\gamma}}{\sqrt{\lambda}}\bigg)^j + \frac{1}{(n+1)!} f^{(n+1)}(\xi_{\pm}) \bigg(\pm \frac{\sqrt{\gamma}}{\sqrt{\lambda}}\bigg)^{n+1}
\end{align*}
Therefore, using the definition of $\mathcal{A}$ in \eqref{eq: explicit formula for A, SSQ}, we can write the operator $\mathcal{A}$ applied to function $f \circ T_1$ as
\begin{align*}
    \mathcal{A} (f \circ T_1)(x) &= \frac{\lambda}{\gamma} \bigg[ f\bigg(T_1(x + 1)\bigg) - f\bigg(T_1(x)\bigg) \bigg] + \frac{1}{\gamma}\bigg[ \mu + \gamma x \bigg] \cdot \bigg[ f\bigg(T_1(x - 1)\bigg) - f\bigg(T_1(x)\bigg) \bigg] \mathbf{1}_{\{x \geq 1\}} \\
    &= \frac{\lambda}{\gamma} \bigg[ \sum_{j=1}^{n} \frac{1}{j!} f^{(j)}(T_1(x)) \bigg(\frac{\sqrt{\gamma}}{\sqrt{\lambda}}\bigg)^j + \frac{1}{(n+1)!} f^{(n+1)}(\xi_{+}) \bigg(\frac{\sqrt{\gamma}}{\sqrt{\lambda}}\bigg)^{n+1} \bigg] \\
    &\quad + \frac{1}{\gamma}\bigg[ \mu + \gamma x \bigg] \cdot \mathbf{1}_{\{x \geq 1\}} \cdot \bigg[ \sum_{j=1}^{n} \frac{1}{j!} f^{(j)}(T_1(x)) \bigg(-\frac{\sqrt{\gamma}}{\sqrt{\lambda}}\bigg)^j + \frac{1}{(n+1)!} f^{(n+1)}(\xi_{-}) \bigg(-\frac{\sqrt{\gamma}}{\sqrt{\lambda}}\bigg)^{n+1} \bigg] \\
    &:= \sum_{j=1}^{n} \frac{\lambda}{\gamma} \bigg[ \frac{1}{j!} f^{(j)}(T_1(x)) \bigg(\frac{\sqrt{\gamma}}{\sqrt{\lambda}}\bigg)^j \bigg] + \sum_{j=1}^{n} \frac{1}{\gamma}\bigg[ \mu + \gamma x \bigg] \cdot \mathbf{1}_{\{x \geq 1\}} \cdot \bigg[ \frac{1}{j!} f^{(j)}(T_1(x)) \bigg(-\frac{\sqrt{\gamma}}{\sqrt{\lambda}}\bigg)^j \bigg]  + R_n\\
    &= \sum_{j=1}^{n} \frac{1}{j!} f^{(j)}(T_1(x))(a^{\mathcal{A}, T_1}_j(q))(x) + R_n.
\end{align*}
The remainder term $R_n$ has the following upper bound on its absolute value,
\begin{align*}
    |R_n| &\leq \frac{\lambda}{\gamma} \cdot \bigg| \frac{1}{(n+1)!} f^{(n+1)}(\xi_{+}) \bigg(\frac{\sqrt{\gamma}}{\sqrt{\lambda}}\bigg)^{n+1} \bigg| + \frac{1}{\gamma}\bigg| \mu + \gamma x \bigg| \cdot \mathbf{1}_{\{x \geq 1\}} \cdot \bigg| \frac{1}{(n+1)!} f^{(n+1)}(\xi_{-}) \bigg(-\frac{\sqrt{\gamma}}{\sqrt{\lambda}}\bigg)^{n+1} \bigg| \\
    &\overset{(a)}{\leq} \frac{1}{\gamma(n+1)!} D_f^{n+1} \sqrt{(n+1)!}  
    ) \cdot \bigg(\frac{\sqrt{\gamma}}{\sqrt{\lambda}}\bigg)^{n+1} \bigg(\lambda(1+|\xi_{+}|)^{m} + |\mu+ \gamma x| (1+|\xi_{-}|^{m}))\bigg)
\end{align*}
Inequality $(a)$ follows from the definition of $\mathcal{F}$ in \ref{assumption: function class}. Such assumption imposes a constant $D_f > 0$ such that $f^{(n+1)}(x) \leq D_f^{n+1} \sqrt{(n+1)!} (1+|x|)^{m}$ with some $m$. Thus the remainder term $R_n$ converges to $0$ as $n \to \infty$, and we conclude that the equality in Assumption \ref{assumption: Kramers-Moyal expansion final version} holds. 

Finally, we check the condition $\mathbb{E}[ \sum_{k=1}^{\infty} D^k\sqrt{k!}|a^{\mathcal{A}, T_1}_k(q)| ] < \infty$ for all $D > 0$.
\begin{align*}
    \mathbb{E}\bigg[ \sum_{j=1}^{\infty} D^j\sqrt{j!} |a^{\mathcal{A}, T_1}_j(q)| \bigg] &= \frac{1}{\gamma}\mathbb{E}\bigg[ \sum_{j=1}^{\infty} \frac{D^j}{\sqrt{j!}} \bigg| \bigg(\frac{\sqrt{\gamma}}{\sqrt{\lambda}}\bigg)^j \cdot \bigg( \lambda + (-1)^j [\mu + \gamma q] \mathbf{1}_{\{q \geq 1\}} \bigg) \bigg| \bigg] \\
    &\overset{(a)}{=} \frac{1}{\gamma}\sum_{j=1}^{\infty} \frac{D^j}{\sqrt{j!}} \bigg(\frac{\sqrt{\gamma}}{\sqrt{\lambda}}\bigg)^j \cdot \mathbb{E}\bigg[ \bigg| \lambda + (-1)^j [\mu + \gamma q] \mathbf{1}_{\{q \geq 1\}} \bigg| \bigg] \\
    &\leq \frac{1}{\gamma}\sum_{j=1}^{\infty} \frac{D^j}{\sqrt{j!}} \bigg(\frac{\sqrt{\gamma}}{\sqrt{\lambda}}\bigg)^j \cdot \bigg( \lambda + \mu + \mathbb{E}[\gamma |q|] \bigg) \overset{(b)}{<}\infty
\end{align*}
Here inequality $(a)$ is from Tonelli's theorem. Inequality $(b)$ is from the well-definedness of MGF in \eqref{eq: well-definedness of MGF for SSQ}, which implies that $\mathbb{E}[|q|] < \infty$, and from the fact that $\sum_{j=1}^{\infty} \frac{A^j}{\sqrt{j!}}< \infty$ for all constant $A > 0$. 
Thus Assumption \ref{assumption: integrability of Kramers-Moyal coefficients final version} holds. Consequently, we can apply Proposition \ref{pro: Wasserstein-$p$ bound without exchangeability final version}.

Starting from Proposition \ref{pro: Wasserstein-$p$ bound without exchangeability final version}, we recall that the Wasserstein-$p$ distance between $\mathcal{L}(\tilde{q})$ and $\mathcal{L}(Z)$ is bounded as follows,
\begin{align}
    W_p(\mathcal{L}(\tilde{q}), \mathcal{L}(Z))
    &\le
    g_0(t_0,p)\;
        + 
    \underbrace{g_1(t_0,p)\;
        \Bigl\| \,(a^{\mathcal{A}, T_1}_1(q) +\tilde{q}\,\Bigr\|_{L^p}}_{\mathclap{(b_1)}} \notag\\
    &\quad +\;
    \underbrace{g_2(t_0,p)\;
        \Bigl\| \,a^{\mathcal{A}, T_1}_2(q)-1\,\Bigr\|_{L^p}}_{\mathclap{(b_2)}} +
    \underbrace{\sum_{k=3}^{\infty}
        g_k(t_0,p)\;
        \Bigl\| \,a^{\mathcal{A}, T_1}_k(q)\,
        \Bigr\|_{L^p}}_{\mathclap{(b_3)}} .
    \notag
\end{align}
The explicit forms of $g_k(t_0,p)$ are given in \eqref{eq: W-p bound final}. We  $t_0 = -\frac{1}{2}\log(1-p\gamma/\lambda)$. Then we bound each term in \eqref{eq:Final Stein Wasserstein-$p$ bound} as follows. First notice that assumption \eqref{eq: Initial p's regime for Wasserstein distance SSQ} implies $e^{t_0} \leq \sqrt{2}$. Thus
\begin{align}
    g_0(t_0,p) &= e^{t_0}\|Z\|_{L^p} (\frac{\pi}{2} - \sin^{-1}(e^{-t_0})) \notag\\
    &\leq \sqrt{2} \|Z\|_{L^p} (\frac{\pi}{2} - \sin^{-1}(e^{-t_0})) \notag\\
    &= \sqrt{2} \|Z\|_{L^p} (\frac{\pi}{2} - \sin^{-1}(\sqrt{1- \frac{p\gamma}{\lambda}})) \notag\\
    &\overset{(a)}{\leq}  4e\sqrt{\pi} \sqrt{p} \cdot \frac{\sqrt{p\gamma}}{\sqrt{\lambda(1-p\gamma/\lambda)}} \notag\\
    &\overset{(b)}{\leq} \frac{8e\sqrt{\pi}}{\sqrt{\lambda}} \sqrt{\gamma} p. \label{eq: b0 upper bound SSQ}
\end{align}
Here inequality $(a)$ holds from the inequality $\frac{\pi}{2} - \sin^{-1}(x) \leq \frac{\sqrt{1-x^2}}{x}$ for $x \in (0,1)$ and setting $x = \sqrt{1- \frac{p\gamma}{\lambda}}$. The inequality can be seen from the monotonicity of function $u(x) = \frac{\sqrt{1-x^2}}{x} -\frac{\pi}{2} + \sin^{-1}(x) $ on $(0,1)$ as follows,
\begin{align*}
    u'(x) = \frac{-1}{x^2\sqrt{1-x^2}} + \frac{1}{\sqrt{1-x^2}} = \frac{x^2 - 1}{x^2\sqrt{1-x^2}} < 0, \quad \forall x \in (0,1), \quad \text{and }   u(1) = 0.
\end{align*}
Inequality $(b)$ holds from assumption \eqref{eq: Initial p's regime for Wasserstein distance SSQ}, since we have $\sqrt{p\gamma/\lambda} \leq 1/\sqrt{2}$.

Next, we bound $(b_1)$. Recalling the explicit formula of $\mathcal{A}$ in \eqref{eq: explicit formula for A, SSQ}, and that by assumption on $p$ in \eqref{eq: Initial p's regime for Wasserstein distance SSQ}, we have $e^{t_0}\leq \sqrt{2}$. With the definition of $g_1(t_0,p)$ in \eqref{eq: W-p bound final}, we compute $(b_1)$  as follows.
\begin{align}
    (b_1) &= g_1(t_0,p)\;
        \Bigl\| a^{\mathcal{A}, T_1}_1(q) +\tilde{q}\,\Bigr\|_{L^p} \notag\\
    &= \sqrt{2} \int_{-\frac{1}{2}\log(1-p\gamma/\lambda)}^{\infty} e^{-r} dr \cdot \| a^{\mathcal{A}, T_1}_1(q) + \tilde{q}\|_{L^p} \notag\\
    &\overset{(a)}{=} \sqrt{2} \int_{-\frac{1}{2}\log(1-p\gamma/\lambda)}^{\infty} e^{-r} dr \cdot \| \frac{1}{\gamma}\cdot \frac{\sqrt{\gamma}}{\sqrt{\lambda}} \bigg[ \lambda - (\mu + \gamma q) \mathbf{1}_{\{ q \geq 1\}} \bigg] + \tilde{q}\|_{L^p} \notag\\
    &\overset{(b)}{=} \sqrt{2}\int_{-\frac{1}{2}\log(1-p\gamma/\lambda)}^{\infty} e^{-r} dr \cdot \frac{\sqrt{\gamma}}{\sqrt{\lambda}} \cdot \| \frac{\lambda-\mu}{\gamma} - q + \frac{1}{\gamma} \mu \mathbf{1}_{\{ q = 0\}} +  q - \frac{\lambda-\mu}{\gamma} \|_{L^p} \notag\\
    &= \sqrt{2}\int_{-\frac{1}{2}\log(1-p\gamma/\lambda)}^{\infty} e^{-r} dr \cdot \frac{\sqrt{\gamma}}{\sqrt{\lambda}} \cdot \| \frac{\mu}{\gamma} \mathbf{1}_{\{ q = 0\}}\|_{L^p} \notag\\
    &\leq \sqrt{2}\int_{0}^{\infty} e^{-r} dr\cdot \frac{\mu}{\sqrt{\lambda}\sqrt{\gamma}} \| \mathbf{1}_{\{ q = 0\}}\|_{L^p} = \frac{\sqrt{2}\mu}{\sqrt{\lambda}\sqrt{\gamma}} \| \mathbf{1}_{\{ q = 0\}}\|_{L^p}. \label{eq: b1 upper bound SSQ}
\end{align}
Equality $(a)$ is from plugging in the explicit formula of $\mathcal{A}$ and the definition of $q = T_1^{-1}(\tilde{q})$, where $T_1(x) = \frac{\sqrt{\gamma}}{\sqrt{\lambda}}(x - \frac{\lambda - \mu}{\gamma})$ is the normalization function defined before. Equality $(b)$ is plugging in the definition of $\tilde{q}$ and using the fact that $q\in \mathbb{N}$ to simplify the indicator function.

Similarly, we bound $(b_2)$. Recalling that $e^{t_0}\leq \sqrt{2}$, we have
\begin{align}
    (b_2) &= g_2(t_0,p)\;
        \Bigl\| \,a^{\mathcal{A}, T_1}_2(q)-1\,\Bigr\|_{L^p} \notag\\
    &\overset{(a)}{=} \sqrt{2}\int_{-\frac{1}{2}\log(1-p\gamma/\lambda)}^{\infty} \frac{e^{-2r}\|h_1(Z)\|_{L^p}}{\sqrt{1-e^{-2r}}} dr \cdot \| a^{\mathcal{A}, T_1}_2(q) - 1\|_{L^p} \notag\\
    &\overset{(b)}{=} \sqrt{2}\int_{-\frac{1}{2}\log(1-p\gamma/\lambda)}^{\infty} \frac{e^{-2r}\|h_1(Z)\|_{L^p}}{\sqrt{1-e^{-2r}}} dr \cdot \| \frac{1}{2\gamma} \bigg(
         \frac{\gamma}{\lambda}(\lambda + \mu + \gamma q - \mu \mathbf{1}_{\{ q = 0\}}) - 1 \bigg)\|_{L^p} \notag\\
    &= \sqrt{2}\int_{-\frac{1}{2}\log(1-p\gamma/\lambda)}^{\infty} \frac{e^{-2r}\|h_1(Z)\|_{L^p}}{\sqrt{1-e^{-2r}}} dr \cdot \frac{\gamma}{\lambda}\| \frac{1}{2\gamma} \bigg(
        \lambda + \mu + \gamma q - \mu \mathbf{1}_{\{ q = 0\}} - 2\lambda \bigg)\|_{L^p} \notag\\
    &\overset{(c)}{\leq} \sqrt{2}\int_{-\frac{1}{2}\log(1-p\gamma/\lambda)}^{\infty} \frac{e^{-2r}\|h_1(Z)\|_{L^p}}{\sqrt{1-e^{-2r}}} dr \cdot \frac{\gamma}{\lambda} (\frac{1}{2}\|q - \frac{\lambda-\mu}{\gamma}\|_{L^p} + \frac{\mu}{2\gamma} \| \mathbf{1}_{\{ q = 0\}}\|_{L^p}) \notag\\
    &\overset{(d)}{\leq} 2\sqrt{2}e \sqrt{2\pi} \sqrt{p} \int_{-\frac{1}{2}\log(1-p\gamma/\lambda)}^{\infty} \frac{e^{-2r}}{\sqrt{1-e^{-2r}}} dr \cdot \frac{\gamma}{\lambda} (\frac{1}{2}\|q - \frac{\lambda-\mu}{\gamma}\|_{L^p} + \frac{\mu}{2\gamma} \| \mathbf{1}_{\{ q = 0\}}\|_{L^p}) \notag\\
    &\overset{(e)}{=} 2\sqrt{2}e\sqrt{2\pi} \sqrt{p} (1 - \sqrt{1 - (1- \frac{p\gamma}{\lambda})}) \cdot \frac{\gamma}{2\lambda} \cdot (\| q - \frac{\lambda-\mu}{\gamma} \|_{L^p} + \frac{\mu}{\gamma} \cdot \| \mathbf{1}_{\{ q = 0\}}\|_{L^p}) \notag\\
    &\leq 2\sqrt{2}e\sqrt{2\pi} \frac{\gamma\sqrt{p}}{2\lambda} \cdot (\| q - \frac{\lambda-\mu}{\gamma} \|_{L^p} + \frac{\mu}{\gamma} \cdot \| \mathbf{1}_{\{ q = 0\}}\|_{L^p}). \label{eq: b_2 upper bound SSQ}
\end{align}
Equality $(a)$ is from plugging in the definition of $g_2(t_0,p)$ in \eqref{eq: W-p bound final}. Equality $(b)$ is from plugging in the explicit formula of $\mathcal{A}$ in \eqref{eq: explicit formula for A, SSQ} and the definition of $q = T_1^{-1}(\tilde{q})$. Inequality $(c)$ is from the triangle inequality. Inequality $(d)$ is from the $L^p$ norm for $h_1(Z)=Z$ as in \eqref{eq: Lp norm of Z}. Equation $(e)$ is from explicitly solving the integration.

For $(b_3)$, we first split the summation into two parts based on whether index $k$ is odd or even. For odd $k \geq 3$, we 
compute the following identity using Lemma \ref{lemma: identity for Kramers-Moyal coefficients final version},
\begin{align}
    \|a^{\mathcal{A}, T_1}_k(q)\|_{L^p} &= \frac{\sqrt{\gamma}^k}{k!\sqrt{\lambda}^k} \| \frac{1}{\gamma} \bigg(\lambda + (\mu + \gamma  q) (-1)^k - (-1)^k \mu \mathbf{1}_{\{ q = 0\}} \bigg)\|_{L^p} \label{eq: conditional expectation of tilde q SSQ} 
\end{align}
And recall that the first part of summation in $(b_3)$ is given by
\begin{align*}
    (b_{3,1}) &:= \sum_{k\geq 3, odd}
        g_k(t_0,p)\;
        \Bigl\| \,a^{\mathcal{A}, T_1}_k(q)\,
        \Bigr\|_{L^p} \\
        &= \sum_{k\geq 3, odd}
        g_k(t_0,p)\;
        \Bigl\| \,a^{\mathcal{A}, T_1}_k(q)\,
        \Bigr\|_{L^p} \\
        &\overset{(a)}{\leq} \sqrt{2}\sum_{k\geq 3, odd} \int_{-\frac{1}{2}\log(1-p\gamma/\lambda)}^{\infty} \frac{e^{-kr}\|h_{k-1}(Z)\|_{L^p}}{k!\sqrt{1-e^{-2r}}^{k-1}} dr \cdot \notag\\
    &\quad \frac{\sqrt{\gamma}^k}{\sqrt{\lambda}^k} \bigg(\frac{\mu}{\gamma} \| \mathbf{1}_{q=0}\|_{L^p} + \| q - \frac{\lambda-\mu}{\gamma}\|_{L^p} \bigg) 
\end{align*}
Here inequality $(a)$ holds from identity \eqref{eq: conditional expectation of tilde q SSQ}, definition of $g_k(t_0,p)$ in \eqref{eq: W-p bound final}, and triangle inequality.

Continuing, we first upper bound the integral inside the right hand side for the above inequality.
\begin{align}
    \sqrt{2}\sum_{k\geq 3, odd} \frac{\sqrt{\gamma}^k}{\sqrt{\lambda}^k}&\int_{-\frac{1}{2}\log(1-p\gamma/\lambda)}^{\infty} \frac{e^{-kr}\|h_{k-1}(Z)\|_{L^p}}{k!\sqrt{1-e^{-2r}}^{k-1}} dr \notag\\
    &\overset{(a)}{=} \sqrt{2}\sum_{k\geq 3, odd} \frac{\sqrt{\gamma}^k}{\sqrt{\lambda}^k}\frac{\|h_{k-1}(Z)\|_{L^p}}{k!} \int_{0}^{\sqrt{1-p\gamma/\lambda}} \frac{x^{k-1}}{(1-x^2)^{(k-1)/2}} dx \notag\\
    &\overset{(b)}{\leq} \sqrt{2}\sum_{k\geq 1} \frac{\sqrt{\gamma}^{2k+1}}{\sqrt{\lambda}^{2k+1}}\frac{\|h_{2k}(Z)\|_{L^p}}{(2k+1)!} \int_{0}^{\sqrt{1-p\gamma/\lambda}} \frac{x^{2k}}{(1-x^2)^{k}} dx \notag\\
    &\overset{(c)}{\leq} \sqrt{2}\sum_{k\geq 1} \frac{\sqrt{\gamma}^{2k+1}}{\sqrt{\lambda}^{2k+1}} \frac{p^k}{(2k+1)\sqrt{(2k)!}}
    \int_{0}^{\sqrt{1-p\gamma/\lambda}} \frac{x^{2k}}{(1-x^2)^{k}} dx \notag\\
    &\overset{(d)}{\leq} \sqrt{2}\sum_{k=1}^{\infty}\frac{\sqrt{\gamma}^{2k+1}p^k}{\sqrt{\lambda}^{2k+1}} \frac{\pi^{1/4}k^{1/4}e^{19/300}}{(2k+1)2^k k!} \int_{0}^{\sqrt{1-p\gamma/\lambda}} \frac{x^{2k}}{(1-x^2)^{k}} dx \notag\\
    &\overset{(e)}{\leq} \sqrt{2}\sum_{k=1}^{\infty} \frac{\sqrt{\gamma}^{2k+1}p^k}{\sqrt{\lambda}^{2k+1}} \frac{\pi^{1/4}k^{1/4}e^{19/300}}{(2k+1)2^k k!}(1-\frac{p\gamma}{\lambda})^k \int_{0}^{\sqrt{1-p\gamma/\lambda}} \frac{1}{1-x^2} dx \notag\\
    &\overset{(f)}{\leq} 2\sqrt{2}e \cdot\frac{\sqrt{\gamma}^{3}p}{\sqrt{\lambda}^{3}} \sum_{k\geq 1} \frac{1}{2^k k!} (1-\frac{p\gamma}{\lambda})^k \int_{0}^{\sqrt{1-p\gamma/\lambda}} \frac{1}{1-x^2} dx \notag\\
    &= 2\sqrt{2}e \cdot \frac{\sqrt{\gamma}^{3}p}{\sqrt{\lambda}^{3}}  \sum_{k=1}^{\infty} \frac{1}{2^k k!} (1-\frac{p\gamma}{\lambda})^k \log(\frac{1+\sqrt{1-p\gamma/\lambda}}{1-\sqrt{1-p\gamma/\lambda}}) \notag\\
    &\leq 2\sqrt{2}e \cdot \frac{\sqrt{\gamma}^{3}p}{\sqrt{\lambda}^{3}}  \sum_{k=1}^{\infty} \frac{1}{2^k k!} \log(\frac{1+\sqrt{1-p\gamma/\lambda}}{1-\sqrt{1-p\gamma/\lambda}})\notag \\
    &\leq 2\sqrt{2}e^2\cdot \frac{\sqrt{\gamma}^{3}p}{\sqrt{\lambda}^{3}} \log(\frac{1+\sqrt{1-p\gamma/\lambda}}{1-\sqrt{1-p\gamma/\lambda}}) \notag\\
    &\overset{(g)}{\leq} 2\sqrt{2}e^2\cdot \frac{\sqrt{\gamma}^{3}p}{\sqrt{\lambda}^{3}} \log (\frac{4\lambda}{p\gamma}) \notag\\
    &\overset{(h)}{\leq} 8\sqrt{2}e^2 \frac{\gamma \sqrt{p}}{\lambda}. \label{eq: b_3,1 integral upper bound SSQ}
\end{align}
Equality $(a)$ holds from change of variables $x:=e^{-r}$. Inequality $(b)$ is re-ordering the index by writing these odd indices as $2k+1$. Inequality $(c)$ is the upper bound on $\|h_{2k}(Z)\|_{L^p}$ as in \cite[Lemma 3]{bonis2020steins}. We quote here for convenience,
\begin{align*}
    \|h_{k}(Z)\|_{L^p} \leq \sqrt{p}^k \sqrt{k!}.
\end{align*}
Inequality $(d)$ holds from the lower bound estimation for Stirling Approximation (see \cite{robbins1955remark}),
\begin{align}
    \sqrt{(2m)!} \geq e^{-\frac{19}{300}} 2^m m! \frac{1}{(\pi m)^{1/4}}. \notag
\end{align}
Inequality $(e)$ is bounding the integrand using the boundary values. Inequality $(f)$ is to get a clean right hand side with the constants. Inequality $(g)$ is from assumption on $p$ in \eqref{eq: Initial p's regime for Wasserstein distance SSQ}. Inequality $(h)$ is from $\log(x)\leq 2\sqrt{x}$. Using this bound on integral, we obtain bound on odd index summation as,
\begin{align}
    (b_{3,1}) \leq 8e^2 \frac{\gamma \sqrt{p}}{\lambda} \bigg( \|  \hat{q}\|_{L^p} + \frac{\mu}{\gamma}\|\sum_{i=1}^{n}\mathbf{1}_{q_i=0}\|_{L^p}\bigg) \label{b_3,1 upper bound SSQ}
\end{align}

For the summation involving even indices $k$, we bound them in the following.
\begin{align}
    (b_{3,2}) &:= \sum_{k\geq 4, even} g_k(t_0,p)\;
        \Bigl\| a^{\mathcal{A}, T_1}_k(q)\,
        \Bigr\|_{L^p} \notag\\
    &\overset{(a)}{\leq} \sqrt{2}\sum_{k\geq 4, even} \int_{-\frac{1}{2}\log(1-p\gamma/\lambda)}^{\infty} \frac{e^{-kr}\|h_{k-1}(Z)\|_{L^p}}{k!\sqrt{1-e^{-2r}}^{k-1}} dr \cdot \notag\\
    &\quad \frac{\sqrt{\gamma}^k}{\sqrt{\lambda}^k} \bigg(\frac{\mu}{\gamma} \| \mathbf{1}_{q=0}\|_{L^p} + \| q - \frac{\lambda-\mu}{\gamma}\|_{L^p} + \frac{2\lambda}{\gamma} \bigg) \notag 
\end{align}
Inequality $(a)$ is from identity \eqref{eq: conditional expectation of tilde q SSQ}, definition of $g_k(t_0,p)$ in \eqref{eq: W-p bound final}, and triangle inequality. Similarly as in odd case, we first bound the integral inside the summation for the right hand side of the above upper bound. 
\begin{align}
    \sqrt{2}\sum_{k\geq 4, even} \frac{\sqrt{\gamma}^k}{\sqrt{\lambda}^k}&\int_{-\frac{1}{2}\log(1-p\gamma/\lambda)}^{\infty} \frac{e^{-kr}\|h_{k-1}(Z)\|_{L^p}}{k!\sqrt{1-e^{-2r}}^{k-1}} dr \notag\\
    &\overset{(a)}{\leq} \sqrt{2}\sum_{k\geq 3, odd} \frac{\sqrt{\lambda}^{k+1}}{\sqrt{\gamma}^{k+1}} \frac{\|h_{k}(Z)\|_{L^p}}{(k+1)!} \int_{0}^{\sqrt{1-p\gamma/\lambda}} \frac{x^k}{(1-x^2)^{k/2}}dx \notag \\
    &\leq \sqrt{2}\sum_{k=1}^{\infty} \frac{\gamma^{k+1}}{\lambda^{k+1}} \frac{\|h_{2k+1}(Z)\|_{L^p}}{(2k+2)!} \int_{0}^{\sqrt{1-p\gamma/\lambda}} \frac{x^{2k+1}}{\sqrt{1-x^2}^{2k+1}}dx \notag\\
    &\leq \sqrt{2}\sum_{k=1}^{\infty} \frac{\gamma^{2}}{\lambda^{2}p^{k-1}} \frac{\|h_{2k+1}(Z)\|_{L^p}}{(2k+2)!} \int_{0}^{\sqrt{1-p\gamma/\lambda}}  \frac{x}{\sqrt{1-x^2}^{3}}dx \; (1-p\gamma/\lambda)^k \notag \\
    &\overset{(b)}{\leq} \sqrt{2}\sum_{k=1}^{\infty} \frac{\gamma^{2}}{\lambda^{2}p^{k-1}} \frac{p^{k+1/2}}{(2k+2)\sqrt{(2k+1)!}}(\frac{\sqrt{\lambda}}{\sqrt{\gamma p}} - 1) \notag \\
    &\leq \sqrt{2}\frac{\gamma^{3/2}p}{\lambda^{3/2}} \frac{1}{(2k+2)\sqrt{(2k+1)!}}\notag \\
    &\overset{(c)}{\leq} \sqrt{2}\frac{\gamma^{3/2}p}{\lambda^{3/2}} \frac{2e (\pi k)^{1/4}}{2k+2} \sum_{k=1}^{\infty} \frac{1}{2^k k!} \leq \sqrt{2}e^2 \frac{\gamma^{3/2}p}{\lambda^{3/2}} \label{eq: b_3,2 integral upper bound SSQ}
\end{align}
 Similar to the case for odd indices, here inequality $(a)$ is from shifting the index and change of variables. The inequalities below $(a)$ are from re-indexing and bounding the integral from boundary terms. Inequality $(b)$ is using \cite[Lemma 3]{bonis2020steins} for bounds on $\|h_{2k}(Z)\|_{L^p}$ and solving the integral explicitly. Inequality $(c)$ is again using the Stirling lower bound approximation in the following way,
\begin{align*}
    \sqrt{(2m+1!} \geq \sqrt{(2m+1)} \sqrt{(2m)!} \geq \sqrt{2m+1} e^{-\frac{19}{300}} 2^m m! \frac{1}{(\pi m)^{1/4}}
\end{align*}
Using \eqref{eq: b_3,2 integral upper bound SSQ} we have
\begin{align}
    (b_{3,2}) \leq \sqrt{2}e^2 \frac{\gamma^{3/2}p}{\lambda^{3/2}} \bigg(\frac{\mu}{\gamma} \|  \mathbf{1}_{q=0}\|_{L^p} + \|  q - \frac{\lambda-\mu}{\gamma}\|_{L^p} + \frac{2\lambda}{\gamma} \bigg) \label{eq: b_3,2 upper bound SSQ}
\end{align} 
With the above bounds on each terms of the upper bound in Proposition \ref{pro: Wasserstein-$p$ bound without exchangeability final version}, i.e., \eqref{eq: b0 upper bound SSQ}, \eqref{eq: b1 upper bound SSQ}, \eqref{eq: b_2 upper bound SSQ}, \eqref{eq: b_3,1 integral upper bound SSQ}, \eqref{eq: b_3,2 upper bound SSQ}, we now bound Wasserstein-$p$ distance as 
\begin{align}
    W_p(\mathcal{L}(\tilde{q}), \mathcal{L}(Z))&\leq \frac{8e\sqrt{\pi}}{\lambda} \gamma p^{3/2} + 2\sqrt{2}e^2\frac{\sqrt{\gamma}p}{\sqrt{\lambda}} \notag\\
    &\quad + \sqrt{2}\bigg(\frac{\mu}{\sqrt{\lambda}} \frac{1}{\sqrt{\gamma}} + e\sqrt{2\pi}\frac{\mu}{\lambda} \sqrt{p} + 8e^2 \frac{\mu}{\lambda}\sqrt{p} + e^2 \frac{\mu\sqrt{\gamma}p}{\lambda^{3/2}}\bigg)  \|\mathbf{1}_{q=0} \|_{L^p} \notag \\
    &\quad + \sqrt{2}\bigg( e\sqrt{2\pi}\frac{\gamma \sqrt{p}}{\lambda} + 8e^2\frac{\gamma\sqrt{p}}{\lambda} + e^2\frac{\gamma^{3/2}p}{\lambda^{3/2}}   \bigg) \| q - \frac{\lambda-\mu}{\gamma}\|_{L^p} \label{eq: C_10, lambda, mu}\\
    &\leq C_{10, \lambda, \mu} \cdot \big(p\sqrt{\gamma} + \frac{1}{\sqrt{\gamma}} \|\mathbf{1}_{q=0} \|_{L^p} +\gamma\sqrt{p} \| q - \frac{\lambda-\mu}{\gamma}\|_{L^p} \big) \notag
\end{align}
With constant $C_{10, \lambda, \mu}$ defined as 
\begin{align*}
C_{10,\lambda,\mu}
  &:= \sqrt{2}\,\max\left\{
      \frac{4e\sqrt{2\pi}}{\sqrt{\lambda}} + \frac{2e^{2}}{\sqrt{\lambda}},\,
      \frac{\mu}{\sqrt{\lambda}} + e\sqrt{2\pi}\frac{\mu}{\lambda}
        + 8e^{2}\frac{\mu}{\sqrt{\lambda}} + \frac{e^{2}\mu}{2\sqrt{\lambda}},\,
      \frac{e\sqrt{2\pi}}{\lambda} + 8e^{2}\frac{1}{\lambda}
        + \frac{e^{2}}{2\lambda}
      \right\}.
\end{align*}
\end{proof}

\begin{proof}[Proof for Lemma \ref{lem: sub-exponential Lp norm}]
From Chernoff bound with optimal choice of $\theta$, we derive sub-gaussian and sub-exponential tail bounds from the condition $\mathbb{E}[e^{\theta|X|}]  \leq e^{\beta^2 \theta^2/2}, \forall \theta \in [0,1/\alpha]$:
\begin{align*}
    \mathbb{P}(|X| > a) &\leq \begin{cases}
        A\cdot e^{-a^2/2\beta^2} & \text{if } a \leq \beta^2/\alpha \\
        A\cdot e^{-a/2\alpha} & \text{if } a > \beta^2/\alpha
    \end{cases}
\end{align*}
Thus from the p-th moment identity, there is :
\begin{align*}
    \mathbb{E}[|X|^p] &= \int_0^\infty p\cdot t^{p-1} \mathbb{P}(|X| \geq t) dt \\
    &\leq Ap\cdot \int_0^{\beta^2/\alpha} t^{p-1} e^{-r^2/2\beta^2} dt + p\cdot \int_{\beta^2/\alpha}^\infty t^{p-1} e^{-r/2\alpha} dt \\
    &\leq Ap\cdot \int_0^{\infty} t^{p-1} e^{-r^2/2\beta^2} dt + p\cdot \int_{0}^\infty t^{p-1} e^{-r/2\alpha} dt \\
    &= 2^{\frac{p-1}{2}} \beta^p A\cdot \Gamma\left(\frac{p}{2}\right) + (2\alpha)^p \cdot \Gamma(p) 
\end{align*}
Thus for $L^p$ norm, we have
\begin{align*}
    \|X \|_{L^p} &\leq A^{1/p} [2^{\frac{p-1}{2}} \beta^p \cdot \Gamma\left(\frac{p}{2}\right)]^{1/p} + [(2\alpha)^p \cdot \Gamma(p)]^{1/p} \\
    &\overset{(a)}{\leq}A^{1/p} \bigg[ 2e\sqrt{2\pi} \sqrt{2}^{p-1} \beta^p \cdot \left(\frac{p}{2}\right)^{p/2-1/2} e^{-p/2} \bigg]^{1/p} + \bigg[ 2e\sqrt{2\pi} (2\alpha)^p \cdot (1/e)^p p^{p-1/2} \bigg]^{1/p} \\
    &\leq 2e\sqrt{2\pi}A^{1/p} \cdot \beta p + 2e\sqrt{2\pi} \cdot \alpha p 
\end{align*}
Inequality $(a)$ holds from Stirling approximation for Gamma function \cite[Lemma 1]{zouzias2010low}: $\Gamma(m) \leq 2e \sqrt{2\pi} m^{m-1/2} e^{-m}$.
\end{proof}

\subsection{Proof for Lemma \ref{lem: tight bound on P(q_infty = 0)}: Tight Bound on $\mathbb{P}(q = 0)$}
We provide tight upper and lower bounds for $\mathbb{P}(q = 0)$
The key technique is to approximate the summation term by Riemann
integral and Laplace-style method \cite{azevedo1994laplace}.
\begin{proof}
Since the Assumption \ref{ass:heavy_overload} is invariant under scaling of $\mu$, without loss of generality, we assume $\mu = 1$ in the following proof. One can rescale parameters back for the final result.
With detailed balance equation for SSQ, we have
\begin{align*}
    \mathbb{P}(q = i) = \frac{\lambda}{\gamma * i + 1} \mathbb{P}(q = i-1)
        = \prod_{j=1}^{i} \frac{\lambda}{\gamma * j + 1} \mathbb{P}(q = 0), \quad \forall i \geq 1
\end{align*}
Let $Z:=1+\sum_{n=1}^{\infty} \prod_{j=1}^{n}\frac{\lambda}{\gamma * j + 1}$, we have $\mathbb{P}(q = 0) = \frac{1}{Z}$. Let $a_n := \prod_{j=1}^{n}\frac{\lambda}{\gamma * j + 1}$, we have $a_n = \exp(n \ln\lambda - \sum_{j=1}^{n} \ln(\gamma * j + 1))$. Then we have
\begin{align*}
    &a_n \geq \exp(n \ln\lambda - \frac{1}{\gamma} \int_{0}^{\gamma n + \gamma} \ln(u + 1) du) 
    = \exp \bigg( n \ln\lambda - \frac{1}{\gamma}\big((1 + \gamma n +\gamma) \ln(\gamma n + \gamma+ 1) - \gamma n -\gamma\big)\bigg) \\
    &a_n \leq \exp(n \ln\lambda - \frac{1}{\gamma} \int_{0}^{\gamma n} \ln(u + 1) du) 
    = \exp \bigg( n \ln\lambda - \frac{1}{\gamma}\big((1 + \gamma n) \ln(\gamma n + 1) - \gamma n\big)\bigg)
\end{align*}

Let $g(x) = -x \ln \lambda + (1+x)\ln(1+x) - x$.  
Then $a_n \geq \exp(-\frac{1}{\gamma}(\gamma\ln \lambda - g(\gamma n + \gamma))) = \frac{1}{\lambda}\exp(-\frac{1}{\gamma} g(\gamma n + \gamma))$ and 
$a_n \leq \exp(-\frac{1}{\gamma} g(\gamma n))$.  
Since $g'(x) = -\ln \lambda + \ln(1+x)$ and $g''(x) = 1/(1+x)>0$, $g$ is convex on $(-1,\infty)$ with a unique minimum at $x^* = \lambda - 1$.  We first bound the summation by integrals:
\begin{align*}
    Z &= 1 + \sum_{n=1}^{\infty} a_n 
    \;\leq\; \sum_{n=0}^{\infty} a_n 
    \;=\; \frac{1}{\gamma} \sum_{n=0}^{\infty} \exp\!\left(-\tfrac{1}{\gamma}g(\gamma n)\right)\gamma \\
    &\overset{(a)}{\leq} \frac{1}{\gamma}\int_{0}^{\infty} \exp\!\left(-\tfrac{1}{\gamma}g(x)\right)dx 
       + \exp\!\left(-\tfrac{1}{\gamma}g(x^*)\right).
\end{align*}

$(a)$ follows from unimodal sum-integral comparison, using the convexity of $g(x)$:  
\begin{align*}
\text{If } g(\gamma \lfloor x^*/\gamma \rfloor) \le g(\gamma \lceil x^*/\gamma \rceil)
& \sum_{n=0}^{\lfloor x^*/\gamma \rfloor} \exp\!\left(-\tfrac{1}{\gamma}g(\gamma n)\right)\gamma
  \;\le\; \int_{0}^{\gamma \lfloor x^*/\gamma \rfloor} \exp\!\left(-\tfrac{1}{\gamma}g(x)\right)\,dx
  + \exp\!\left(-\tfrac{1}{\gamma}g(x^*)\right)\gamma,\\
& \sum_{n=\lceil x^*/\gamma \rceil}^{\infty} \exp\!\left(-\tfrac{1}{\gamma}g(\gamma n)\right)\gamma
  \;\le\; \int_{\gamma \lfloor x^*/\gamma \rfloor}^{\infty} \exp\!\left(-\tfrac{1}{\gamma}g(x)\right)\,dx.
\\[1ex]
\text{If } g(\gamma \lfloor x^*/\gamma \rfloor) > g(\gamma \lceil x^*/\gamma \rceil)
& \sum_{n=0}^{\lfloor x^*/\gamma \rfloor} \exp\!\left(-\tfrac{1}{\gamma}g(\gamma n)\right)\gamma
  \;\le\; \int_{0}^{\gamma \lceil x^*/\gamma \rceil} \exp\!\left(-\tfrac{1}{\gamma}g(x)\right)\,dx,\\
& \sum_{n=\lceil x^*/\gamma \rceil}^{\infty} \exp\!\left(-\tfrac{1}{\gamma}g(\gamma n)\right)\gamma
  \;\le\; \int_{\gamma \lceil x^*/\gamma \rceil}^{\infty} \exp\!\left(-\tfrac{1}{\gamma}g(x)\right)\,dx
  + \exp\!\left(-\tfrac{1}{\gamma}g(x^*)\right)\gamma.
\end{align*}

By Taylor Expansion to 2nd order residual, we obtain $g(x) \geq g(x^*) + g'(x^*)(x-x^*) + \frac{(x-x^*)^2}{2(x^* + 1 +(e^{2}-1)\lambda)}$, for $\forall x \in [0, x^* + (e^{2}-1)\lambda]$. Decompose the above integral into two parts at point $(e^{2}-1)\lambda$, we have Gaussian tail via convexity for the first part:
\begin{align}
    \int_{0}^{x^* + (e^2 - 1)\lambda} \exp(- \frac{1}{\gamma} g(x)) dx &\leq \int_{0}^{x^* + (e^2 - 1)\lambda} \exp(- \frac{1}{\gamma} g(x^*)) \exp(-\frac{(x-x^*)^2}{2\gamma(x^* + 1 + R)}) dx \notag\\
    &\leq \exp(- \frac{1}{\gamma} g(x^*)) \int_{-\infty}^{\infty} \exp(-\frac{(x-x^*)^2}{2e^2 \gamma \lambda }) dx \notag\\
    &\leq \exp(- \frac{1}{\gamma} g(x^*)) \sqrt{2\pi e^2 \gamma \lambda} \label{eq: Gaussian Tail for P_0}
\end{align}
Beyond this separating point, we have linear lower bound for $g(x)$ from convexity. $g(x) -g(x^*) = (x+1)\ln(x+1) - \lambda \ln \lambda +(x-x^*) \ln \lambda - (x-x^*) \geq (x+1) (2+\ln \lambda) - \lambda \ln \lambda - (x-x^*)(\ln \lambda -1) \geq (x-x^*),$ for$\forall x\geq x^* + (e^2-1)\lambda$. Thus:
\begin{align}
    \int_{x^* + (e^2 - 1)\lambda}^{\infty} \exp (- \frac{1}{\gamma} g(x) dx &\leq \int_{x^* + (e^2 - 1)\lambda}^{\infty} \exp(- \frac{1}{\gamma} g(x^*)) \exp(-\frac{(x-x^*)}{\gamma}) dx \notag\\
    &\leq \exp(- \frac{1}{\gamma} g(x^*)) \int_{0}^{\infty} \exp(-\frac{1}{\gamma}x ) dx \notag\\
    &= \gamma \exp(- \frac{1}{\gamma} g(x^*)) \label{eq: Exponential Tail for P_0}
\end{align}

Combining \eqref{eq: Gaussian Tail for P_0} and \eqref{eq: Exponential Tail for P_0}, we have the upper bound for $Z$:
\begin{align}
    Z \leq \frac{1}{\gamma} (\gamma+e\sqrt{2\pi \lambda}\sqrt{\gamma} + \gamma) \exp(- \frac{1}{\gamma} g(x^*)) \leq \frac{1}{\sqrt{\gamma}} (2\sqrt{\mu}+\sqrt{2\pi e^2 \lambda}) \exp(\frac{1}{\gamma}(\lambda - 1 - \ln \lambda)) \label{eq: upper bound for Z}
\end{align}

For the lower bound of $Z$, noticing that $Z \geq \frac{1}{\lambda} \sum_{0}^{\infty}\exp(-\frac{1}{\gamma}g(\gamma n + \gamma))$ and by Assumption \ref{ass:heavy_overload} $x^* = \lambda - 1 \geq C \sqrt{\gamma}$, we have
\begin{align}
    Z &\overset{(a)}{\geq} \frac{1}{\lambda} \bigg( \frac{1}{\gamma}\int_{x^*}^{x^* + \sqrt{\gamma}} \exp(-\frac{1}{\gamma}g(x)) dx + \exp(-\frac{1}{\gamma}g(x^*- \gamma)) - \exp(-\frac{1}{\gamma}g(x^*)) \bigg) \notag \\
    &\geq \frac{1}{\lambda} \bigg( \frac{1}{\sqrt{\gamma}} \exp(-\frac{1}{\gamma}g(x^* +\sqrt{\gamma})) + \exp(-\frac{1}{\gamma}g(x^*- \gamma)) - \exp(-\frac{1}{\gamma}g(x^*)) \bigg) \notag \\
    &\overset{(b)}{\geq} \frac{1}{\lambda} \bigg( \frac{1}{\sqrt{\gamma}} \exp(-\frac{1}{\gamma}g(x^*) - \frac{1}{\sqrt{\gamma}}\ln (\frac{\lambda+\sqrt{\gamma}}{\lambda})) + \exp(-\frac{1}{\gamma}g(x^*) - \ln (\frac{\lambda}{\lambda - \gamma})) - \exp(-\frac{1}{\gamma}g(x^*)) \bigg) \notag \\ 
    &= \frac{1}{\lambda} \bigg[ \frac{1}{\sqrt{\gamma}} (\frac{\lambda}{\lambda+\sqrt{\gamma}})^{\frac{1}{\sqrt{\gamma}}} - \frac{\gamma}{\lambda} \bigg] \exp(-\frac{1}{\gamma}g(x^*)) \notag \\
    &\overset{(c)}{\geq}\frac{1}{\lambda\sqrt{\gamma}} \bigg[ (\frac{1 + C\sqrt{\gamma}}{1 + (C+1)\sqrt{\gamma}})^{\frac{1}{\sqrt{\gamma}}} - \frac{\gamma^{3/2}}{1+C\sqrt{\gamma}} \bigg] \exp(-\frac{1}{\gamma}g(x^*)) \notag \\
    &\overset{(d)}{\geq} \frac{1}{\lambda\sqrt{\gamma}} (e^{-1}\wedge \frac{C^2+C-1}{(1+C)(2+C)}) \exp(-\frac{1}{\gamma}g(x^*)) \notag \\
    &= (e^{-1}\wedge \frac{C^2+C-1}{(1+C)(2+C)}) \frac{1}{\lambda\sqrt{\gamma}} \exp(\frac{1}{\gamma}(\lambda - 1 - \ln \lambda)) \label{eq: lower bound for Z}
\end{align}

Where $(a)$ holds from unimodal sum-integral comparison: 
\begin{align*}
    \sum_{n=\lfloor \frac{x^*-\gamma}{\gamma}\rfloor}^{\lceil \frac{x^*+\sqrt{\gamma}}{\gamma}\rceil} \exp(-\frac{1}{\gamma}g(\gamma n + \gamma)) &\geq \frac{1}{\gamma} \int_{\lfloor \frac{x^*}{\gamma}\rfloor \gamma + \gamma}^{x^* + \sqrt{\gamma}} \exp(-\frac{1}{\gamma}g(x)) dx  + \exp(-\frac{1}{\gamma}g(\gamma \lfloor \frac{x^*-\gamma}{\gamma}\rfloor + \gamma)) \\
    &\geq \frac{1}{\gamma} \int_{\lfloor \frac{x^*}{\gamma}\rfloor \gamma + \gamma}^{x^* + \sqrt{\gamma}} \exp(-\frac{1}{\gamma}g(x)) dx  +\frac{1}{\gamma} \int_{x^*}^{\lfloor \frac{x^*}{\gamma}\rfloor \gamma + \gamma} \exp(-\frac{1}{\gamma}g(x)) dx  \\
    & \quad + \exp(-\frac{1}{\gamma}g(x^*- \gamma)) - \exp(-\frac{1}{\gamma}g(x^*)) 
\end{align*}
Here we use $\lambda - 1 > \gamma$ by Assumption \ref{ass:heavy_overload}. $(b)$ holds from the convexity of $g(x)$: $g(x) \leq g(x^*) + g'(x)(x-x^*)$. $(c)$ holds from Assumption \ref{ass:heavy_overload}. $(d)$ holds from the fact that function $f(\gamma):=\Big(\frac{1+C\sqrt{\gamma}}{1+(C+1)\sqrt{\gamma}}\Big)^{\frac{1}{\sqrt{\gamma}}}-\frac{\gamma}{1+C\sqrt{\gamma}}$ obtains its infimum at the boundary $\gamma\in\{0,1\}$. To see this, we check the concavity of $f(x)$ as combination of $h(x):=\Big(\frac{1+Cx}{1+(C+1)x}\Big)^{\frac{1}{x}}$ and $j(x):=-\frac{x^3}{1+Cx}$. For $h(x)$, we define $g(x):=\ln h(x) = \frac{1}{x} \ln(\frac{1 + Cx}{1 + (C+1)x})$, and auxiliary function $a_x(u):= \frac{C+u}{(1+x(C+u))^2} $ then:
\begin{align*}
    h''(x) &= \exp(g(x)) (g'(x)^2 + g(x) g''(x)) \\
    &\overset{(a)}{=} \exp(g(x)) \left( \bigg( \int_{0}^{1} \frac{C+u}{(1+x(C+u))^2} du\bigg)^2 -2\int_{0}^{1} \frac{(C+u)^2}{(1+ x(C+u))^3} du \right) \\
    &\overset{(b)}{=} \exp(g(x)) \left( \int_{0}^{1} \int_{0}^{1} a_x(u) a_x(v) - 2 (1+x(C+u))a_x^2(u) dudv \right) \\
    &\overset{(c)}{\leq}\frac{1}{2} \exp(g(x)) \bigg[  \int_{0}^{1} \int_{0}^{1} 2a_x(u) a_x(v) - 2 (a_x^2(u)+a_x^2(v)) dudv \bigg] < 0
\end{align*} 
Where $(a)$ holds from $g(x) = - \frac{1}{x} \int_{Cx}^{(C+1)x} \frac{1}{1+u} du = \int_{0}^{1} \frac{du}{1+x(C+u)}$. $(b)$ and $(c)$ hold from symmetry of $u, v$. For $j(x)$, we have $j''(x) = -\frac{2x(C^2x^2+3Cx + 3)}{(1+Cx)^3} < 0$. Thus, $f''(x) = h''(x) + j''(x) < 0$, and $f(x)$ obtains its infimum at the boundary, and $(d)$ holds with $f(0) = e^{-1}, f(1) = \frac{C^2+C-1}{(1+C)(2+C)}$.

Combining \eqref{eq: upper bound for Z} and \eqref{eq: lower bound for Z}, we acquire an order-wise tight interval for $\mathbb{P}(q = 0)$ w.r.t. $\gamma$:

\begin{align}
    \mathbb{P}(q = 0) &\leq \lambda(e\vee \frac{(1+C)(2+C)}{C^2+C-1})\cdot \sqrt{\gamma} \exp(-\frac{1}{\gamma}(\lambda - 1 - \ln \lambda)) \label{eq: tight bound for P_0}\\
    \mathbb{P}(q = 0) &\geq \frac{1}{(2\sqrt{\gamma} + \sqrt{2\pi e^2 \lambda})} \cdot \sqrt{\gamma} \exp(-\frac{1}{\gamma}(\lambda - 1 - \ln \lambda)) \notag
\end{align}  \end{proof}

\subsection{Proof for Proposition \ref{pro: Lp norm of q_infty}: Lp norm for normalized queue} \label{apx: proof for Lp norm of q_infty}

Now we establish the $L^p$ norm bound for centered queue length, recalling that $\hat{q}:=q - \frac{\lambda - \mu}{\gamma}$. $L^p$ norm for normalized queue $\tilde{q}$ would follow directly from the relation $\tilde{q} = \frac{\sqrt{\gamma}}{\sqrt{\lambda}} \hat{q}$. We use the Lyapunov drift method with exponential test function to derive bounds for the moment generating function (MGF) of $q$, and then derive $L^p$ norm bounds based on this MGF bound. In this section, we separate the proof into three parts: an upper bound on $\mathbb{E}[\exp(\theta |\hat{q}|)]$ for $\theta>0$, upper bound and lower bounds for $L^p$ norm bounds of $\hat{q}$.

\subsubsection{Upper Bound for MGF of Absolute Centered Queue Length: $\mathbb{E}[\exp(\theta |\hat{q}|)]$ with positive $\theta$}
We use Lyapunov drift method with test functions $\exp(\theta q)$ and $\exp(-\theta q)$, $\theta>0$ to derive upper bounds for $\mathbb{E}[\exp(\theta \hat{q})]$ and $\mathbb{E}[\exp(-\theta \hat{q})]$ respectively . Combining these two bounds, we obtain the one-sided MGF bound for absolute centered queue length $|\hat{q}|$, i.e., $\mathbb{E}[\exp(\theta |\hat{q}|)]$ with $\theta>0$. The final $L^p$ norm upper bounds would then follow from these bounds on $\mathbb{E}[\exp(\theta |\hat{q}|)]$. In particular, the analysis for the exponent of $\mathbb{E}[\exp(\theta |\hat{q}|)]$ will lead to two complementary $L^p$ norm upper bounds, namely the Sub-exponential and Sub-Poisson bounds. We combine these two upper bounds for the final result.

We start with the analysis for $\mathbb{E}[\exp(\theta q)]$ with $\theta>0$. 
From the stability of this process and existence of MGF for $q$ in \eqref{eq: well-definedness of MGF for SSQ}, we can now apply the zero-drift identity $\mathbb{E}[\mathcal{L}_{\text{SSQ}} f(q)]=0$ to exponential Lyapunov function $f(q)=e^{\theta q}$, leading to the following equation.
\begin{align*}
    0 &= \mathbb{E} [\mathcal{L}_{SSQ} f_\theta (q)] \\
    &= \mathbb{E} \bigg[ \lambda[f_\theta(q+1) - f_\theta(q)] + (\mu + \gamma q) [f_\theta(q-1) - f_\theta(q)] \mathbf{1}_{q \geq 1} \bigg] \\
    &= \mathbb{E}[\lambda e^{\theta q} (e^\theta -1) + \mu e^{\theta q}(e^{-\theta}-1) - \mu \mathbf{1}_{q=0}(e^{-\theta}-1) + \gamma q e^{\theta q}(e^{-\theta}-1)] \\
    &\overset{(a)}{=} -\lambda M(\theta)e^\theta + \mu M(\theta) + \gamma M'(\theta) - \mu \mathbb{P}(q=0).
\end{align*}
Equality $(a)$ holds from the definition of MGF $M(\theta) := \mathbb{E}[e^{\theta q}]$ and its derivative $M'(\theta) = \mathbb{E}[q e^{\theta q}]$. We can interchange the derivative and expectation since $\mathbb{E}[q e^{\theta q}]$ is well-defined from the fact that $q e^{\theta q} \leq e^{(\theta + 1) q}$ a.s. for $\theta>0$ and the MGF exists from \eqref{eq: well-definedness of MGF for SSQ}.
Thus, we obtain an Ordinary Differential Equation (ODE) for the stationary $M(\theta):=\mathbb{E}[e^{\theta q}], \theta>0$:
\begin{align*}
\mu \mathbb{P}(q = 0) = M(\theta)[\mu - \lambda e^\theta] + \gamma M'(\theta).
\end{align*}
After solving this ODE, we have an intermediate expression for $M(\theta)$ in terms of $\mathbb{P}(q=0)$,
\begin{align}
    M(\theta) &= \exp\left(-\frac{\mu}{\gamma}\theta + \frac{\lambda}{\gamma}e^\theta - \frac{\lambda}{\gamma}\right)  \notag\\
    &\quad + \exp\left(-\frac{\mu}{\gamma}\theta + \frac{\lambda}{\gamma}e^\theta\right) \cdot 
    \left[\mu \frac{\mathbb{P}(q=0)}{\gamma} \int_{0}^{\theta} 
    \exp\left(\frac{\mu}{\gamma}y - \frac{\lambda}{\gamma}e^{y}\right) dy\right]. \label{MGF for q_infty}
\end{align}
Using the solution of $M(\theta)$ for $q$ in \eqref{MGF for q_infty}, we can bound $\mathbb{E}[e^{\theta \hat{q}}]$ with centered queue length $\hat{q}$.
\begin{align}
    \mathbb{E}\left[ \exp((\hat{q})\theta)\right] &= M(\theta) \cdot \exp(-\frac{\lambda - \mu}{\gamma}\theta) \notag \\
    &= \exp\left(-\frac{\lambda}{\gamma}\theta + \frac{\lambda}{\gamma}e^\theta - \frac{\lambda}{\gamma}\right) (1+\frac{\mu}{\gamma}\mathbb{P}(q=0) \int_{0}^{\theta} \exp\left(\frac{\lambda}{\gamma} + \frac{\mu}{\gamma}y - \frac{\lambda}{\gamma}e^{y}\right) dy) \notag \\
    &\overset{(a)}{\leq } \exp\left(-\frac{\lambda}{\gamma}\theta + \frac{\lambda}{\gamma}e^\theta - \frac{\lambda}{\gamma}\right) (1+\frac{\mu}{\gamma}\mathbb{P}(q=0) \int_{0}^{\infty} \exp\left(-\frac{\lambda - \mu}{\gamma}y\right) dy) \notag \\
    &= \exp\left(-\frac{\lambda}{\gamma}\theta + \frac{\lambda}{\gamma}e^\theta - \frac{\lambda}{\gamma}\right) (1+ \frac{\mu}{\lambda-\mu} \mathbb{P}(q=0) ) \label{eq: MGF upper bound intermediate step SSQ}
\end{align}
Inequality $(a)$ holds since $ \frac{\lambda}{\gamma}e^y \geq \frac{\lambda}{\gamma}(y+1)$.

Next, we analyze $\mathbb{E}[e^{-\theta q}]$ with $\theta>0$. Because the test function $e^{-\theta q}$ is bounded, $\forall \theta > 0$, we can aplly the zero drift identity directly. In particular, we use test function $e^{-\theta q}$, and then bound the Laplace transform of $q$ and $\hat{q}$ as follows.
\begin{align}
    0 &= \mathbb{E}[\mathcal{L}_{SSQ} e^{-\theta q}] = \mathbb{E} \bigg[ \lambda e^{-\theta q} (e^{-\theta} -1) + (\mu + \gamma q) e^{-\theta q}(e^{\theta}-1) \mathbf{1}_{q \geq 1} \bigg] \notag\\
    0&= -\lambda \mathbb{E}[e^{-\theta q}] e^{-\theta} + \mu \mathbb{E}[e^{-\theta q}] + \gamma \mathbb{E}[q e^{-\theta q}] - \mu \mathbb{P}(q=0) \notag\\ 
    &\leq -\lambda \mathbb{E}[e^{-\theta q}] e^{-\theta} + \mu \mathbb{E}[e^{-\theta q}] + \gamma \mathbb{E}[q e^{-\theta q}] \notag\\
    \mathbb{E}[e^{-\theta q}] &\overset{(a)}{\leq} \exp\left(\frac{\lambda}{\gamma}e^{-\theta} + \frac{\mu}{\gamma} \theta - \frac{\lambda}{\gamma}\right) \notag\\
    \mathbb{E}[e^{-\theta \hat{q}}] &= \mathbb{E}[e^{-\theta q}] \cdot \exp(\frac{\lambda - \mu}{\gamma}\theta) \notag\\
    &\leq \exp\left( \frac{\lambda}{\gamma}e^{-\theta} + \frac{\lambda}{\gamma} \theta + \frac{\lambda}{\gamma} \right) \leq \exp\left( \frac{\lambda}{\gamma}e^{\theta} - \frac{\lambda}{\gamma} \theta + \frac{\lambda}{\gamma} \right) \label{eq: Laplace upper bound intermediate step SSQ}
\end{align}
Inequality $(a)$ holds from first writing the differential inequality for $\mathbb{E}[e^{-\theta q}]$ and Gronwall's inequality. Again, we can interchnage the derivative and expectation since $\mathbb{E}[q e^{-\theta q}]$ is well-defined from the boundedness that $q e^{-\theta q} \leq \frac{1}{\theta}$.
Combining the two bounds in \eqref{eq: MGF upper bound intermediate step SSQ} and \eqref{eq: Laplace upper bound intermediate step SSQ}, we can bound the exponential of absolute centered queue length $|\hat{q}|$, i.e., $\mathbb{E}[\exp(\theta |\hat{q}|)]$ with $\theta>0$ using inequality
\begin{align*}
\mathbb{E}\left[ \exp(\theta |\hat{q}|)\right] \leq \mathbb{E} \left[ \exp(\theta (\hat{q}))\right] + \mathbb{E}\left[ \exp(-\theta (\hat{q}))\right].
\end{align*}
 
\subsubsection{Lp norm upper bounds for centered queue length $\hat{q}$}
In this part, we build on the bounds for the exponential moment $\mathbb{E}[\exp(\theta |\hat{q}|)]$ (for $\theta>0$) established in the previous section, and convert them into upper bounds on the $L^p$ norm of the centered queue length $\hat{q}$. In the upper bound \eqref{eq: Sub-Poisson upper bound on L-p norm} of Proposition \ref{pro: Lp norm of q_infty}, we present two types of $L^p$ bounds: a sub-exponential bound and a sub-Poisson bound, stiching together by taking the minimum.
 Mirroring the “small deviation $\leftrightarrow$ Gaussian” and “large deviation $\leftrightarrow$ Poisson” intuition in Figure \ref{fig: Phase Transition Diagram}, these two bounds correspond to analyzing the behavior of $\mathbb{E}[\exp(\theta |\hat{q}|)]$ in the regimes of small and large $\theta$, respectively.

We first exploit the quadratic-in-$\theta$ structure of the exponent to obtain sub-exponential $L^p$ bounds. We then leverage the Poisson-like exponent structure to derive sub-Poisson $L^p$ bounds, which are complementary and particularly useful for large $p$. We begin with the sub-exponential $L^p$ bounds.
\begin{align}
    \mathbb{E}\left[ \exp(\theta |\hat{q}|)\right] &\leq \mathbb{E} \left[ \exp(\theta (\hat{q}))\right] + \mathbb{E}\left[ \exp(-\theta (\hat{q}))\right] \notag\\
    &\overset{(a)}{\leq} (2 + \frac{\mu}{\lambda - \mu} \mathbb{P}(q=0)) \cdot \exp\left( \frac{\lambda}{\gamma}e^{\theta} - \frac{\lambda}{\gamma} \theta + \frac{\lambda}{\gamma} \right) \notag\\
    &\overset{(b)}{\leq} (2 + C_{\lambda,\mu}' \cdot \frac{\mu}{\lambda-\mu} \frac{\sqrt{\gamma}}{\sqrt{\mu}}  \exp(-\frac{1}{\gamma/\mu}(\lambda/\mu - 1 -  \ln (\frac{\lambda}{\mu})))) \cdot \exp\left( \frac{\lambda}{\gamma}e^{\theta} - \frac{\lambda}{\gamma} \theta + \frac{\lambda}{\gamma} \right) \notag\\
    &\overset{(c)}{\leq} (2+ C_{\lambda,\mu}' \frac{\mu}{\lambda-\mu} \frac{\sqrt{\gamma}}{\sqrt{\mu}} \exp(- \frac{(\lambda -\mu)^2}{2\lambda  \gamma})) \cdot \exp\left( \frac{\lambda}{\gamma}e^{\theta} - \frac{\lambda}{\gamma} \theta + \frac{\lambda}{\gamma} \right) \notag\\
    &\overset{(d)}{\leq} (2 + C_{\lambda,\mu}' \cdot \frac{\mu}{\lambda-\mu} \frac{\sqrt{\gamma}}{\sqrt{\mu}} \cdot \frac{\sqrt{2\lambda \gamma}}{(\lambda-\mu)}) \cdot \exp\left( \frac{\lambda}{\gamma}e^{\theta} - \frac{\lambda}{\gamma} \theta + \frac{\lambda}{\gamma} \right) \notag\\ 
    &\overset{(e)}{\leq} (2+ C_{\lambda,\mu}' \sqrt{2})\exp\left( \frac{\lambda}{\gamma}e^{\theta} - \frac{\lambda}{\gamma} \theta + \frac{\lambda}{\gamma} \right) \label{eq: two-sided MGF upper bound SSQ}\\
    &\overset{(f)}{\leq} (2+C_{\lambda,\mu}' \sqrt{2}) \exp\left( \frac{1}{2} \frac{2\lambda}{\gamma}\theta^2\right), \quad \forall \theta\in [0,1]. \label{eq: two-sided MGF sub-expo upper bound SSQ}
\end{align}
Inequality $(a)$ holds from combining \eqref{eq: MGF upper bound intermediate step SSQ} and \eqref{eq: Laplace upper bound intermediate step SSQ}. Inequality $(b)$ holds from Lemma \ref{lem: tight bound on P(q_infty = 0)}. Inequality $(c)$ is from the crude bound $\lambda/\mu - 1 - \ln (\lambda/\mu) \geq \frac{(\lambda -\mu)^2}{2\lambda\mu}$. Inequality $(d)$ is from the fact that polynomial function is upper bounded by 
exponential function, i.e., 
\begin{align}
    e^x &\geq x^y /\Gamma(y+1), \quad \forall x > 0, y > 0. \label{eq: Exp Dominating Polynomial Stirling Bound}
\end{align}
Inequality $(e)$ holds from heavily overloaded Assumption \ref{ass:heavy_overload}, and $C>1$ in Assumption \ref{ass:heavy_overload}.
Inequality $(f)$ holds from taylor expansion for $\exp(\theta)$ with residual term. We elaborate the inequality $(d)$ that polynomial function is upper bounded by exponential function in details here. Specifically, we plug in $x=\frac{(\lambda -\mu)^2}{2\lambda\mu}$ and $y=1/2$ into \eqref{eq: Exp Dominating Polynomial Stirling Bound}. With $\Gamma(3/2)\leq \Gamma(1+1)=1$ , we then have inequality $(d)$.
Now it suffices to obtain the above inequality \eqref{eq: Exp Dominating Polynomial Stirling Bound}. To see this, we study the monotonicity of function $g(x) := e^{x}/x^y$, which has minimizer $x^* = y$ from following derivation.
\begin{align*}
    g'(x) = \frac{e^x}{x^y} (1 - \frac{y}{x}) = 0 \implies x^* = y,\quad  g''(x) = \frac{e^x}{x^y} (1 - \frac{y}{x} + \frac{y(y+1)}{x^2}) > 0, \forall x > 0, y >0
\end{align*}
Thus it suffices to show $\Gamma(y+1) >y^ye^{-y}$ for $y >0$. We define $h(y):=\ln \Gamma(y+1) - y\ln y +y$. We first have its limit at 0,
\begin{align*}
    \lim_{y \to 0^+} h(y) = \lim_{y \to 0^+} \ln \Gamma(y+1) - y\ln y + y = 0
\end{align*}
Then we check the monotonicity of $h(y)$:
\begin{align*}
    h'(y) &= \frac{\Gamma'(y+1)}{\Gamma(y+1)} - \ln y \overset{(a)}{=} \psi(y+1) - \ln y \\
        &\overset{(b)}{=} \psi(y) + \frac{1}{y} - \ln y \overset{(c)}{>} 0 , \quad \forall y >0
\end{align*}
Here $\psi(y)$ is the digamma function in equality $(a)$. Equality $(b)$ holds from the recurrence relation of digamma function. Inequality $(c)$ holds from \cite[Inequality 2.2]{alzer1997some}, where the author shows that $\psi(y)>\frac{1}{y}+\ln(y)$ for all $y>0$. 


Coming back to the sub-exponential bound \eqref{eq: two-sided MGF sub-expo upper bound SSQ}, we have shown that $\hat{q} \in SE(\frac{2\lambda}{\gamma}, 1)$ with pre-exponential constant $2+C_{\lambda,\mu}' \sqrt{2}$.
Using lemma \ref{lem: sub-exponential Lp norm}, we obtain sub-exponential $L^p$ norm bound.
\begin{align}
    \|q - \frac{\lambda - \mu}{\gamma}\|_{L^p} &\leq (2+C_{\lambda,\mu}' \frac{\mu}{2\lambda})^{1/p} 2e\sqrt{2\pi}\bigg( \frac{\sqrt{2\lambda}}{\sqrt{\gamma}} \sqrt{p} + p \bigg) \notag\\
    &\overset{\triangle}{=} C_{1,\lambda,\mu} \sqrt{\frac{\lambda}{\gamma}} \sqrt{p} + C_{2,\lambda,\mu} p \label{eq: sub-expo Lp norm SSQ} 
\end{align}
Where $C_{1,\lambda,\mu} = 4e\sqrt{\pi}(2+C_{\lambda,\mu}' \sqrt{2})$ and $C_{2,\lambda,\mu} = 2e\sqrt{2\pi}(2+C_{\lambda,\mu}' \sqrt{2})$.

Meanwhile, starting from the two-sided MGF bound in \eqref{eq: two-sided MGF upper bound SSQ}, we can also derive sub-Poisson type $L^p$ norm bounds for $\hat{q}$, which is complementary to the above sub-exponential bound, especially for large $p$. The proof below is in the spirit of \cite{ahle2021sharpsimpleboundsraw} (see their proof of Theorem 1). For simplicity, we use $b:= \frac{\lambda}{\gamma}$ in the following derivaion. We set the parameter $\theta$ as $\theta_*:= W(\frac{p}{b})$, where $W(\cdot)$ is the Lambert W function satisfying $\frac{p}{b} = \theta_*e^{\theta_*}$.
\begin{align}
    \mathbb{E}\bigl[|\hat{q}|^p\bigr]
    &\overset{(a)}{\le}
    \mathbb{E}\bigl[e^{\theta_*|\hat{q}|}\bigr]\left(\frac{p}{e\theta_*}\right)^p \notag\\ 
    &\overset{(b)}{\le}
    \bigl(2+ C_{\lambda,\mu}' \sqrt{2}\bigr)
    \exp\!\Bigl(b e^{\theta_*}-b\theta_*+b\Bigr)
    \left(\frac{p}{e\theta_*}\right)^p \notag\\ 
    &=
    \bigl(2+ C_{\lambda,\mu}' \sqrt{2}\bigr)
    \exp\!\Bigl(b e^{\theta_*}-b\theta_*+b\Bigr)
    \left(\frac{p}{e\theta_* b}\right)^p b^p \notag\\ 
    &\overset{(c)}{=}
    \bigl(2+ C_{\lambda,\mu}' \sqrt{2}\bigr)
    \exp\!\Bigl(b\bigl(\tfrac{p}{\theta_* b}-1\bigr)\Bigr)
    \Bigl(\frac{e^{\theta_*}}{e}\Bigr)^p b^p \notag\\ 
    &=
    \bigl(2+ C_{\lambda,\mu}' \sqrt{2}\bigr)
    \exp\!\Bigl(p\Bigl(\tfrac{b}{p}\bigl(\tfrac{p}{\theta_* b}-1\bigr)+\theta_*-1\Bigr)\Bigr)
    b^p \notag\\ 
    &\overset{(d)}{\leq}
    \bigl(2+ C_{\lambda,\mu}' \sqrt{2}\bigr)
    \exp\!\Bigl(p \log\!\Bigl(\frac{p/b}{\log(1+p/b)}\Bigr)\Bigr) b^p \notag\\
    \|\hat{q}\|_{L^p}
    &\overset{(e)}{\leq}
      (2+ C_{\lambda,\mu}' \sqrt{2}) \cdot \frac{p}{\log(1+p/b)} \label{eq: sub-Poisson upper bound on L-p norm SSQ}
\end{align}
Here inequality $(a)$ holds from basic inequality $1+z \leq e^z$, where we substitute with $z = \theta_* |\hat{q}|/p-1$ to derive it.
Inequality $(b)$ holds from bounds on two-sided MGF in \eqref{eq: two-sided MGF upper bound SSQ} with Assumption \ref{ass:heavy_overload} and $\gamma\leq \gamma_0$. 
Equality $(c)$ holds from the definition of $\theta_*:=W(p/b)$ and $\frac{p}{b} = \theta_*e^{\theta_*}$.
Inequality $(d)$ holds from \cite[Lemma 2]{ahle2021sharpsimpleboundsraw}, stating that $\frac{1}{W(x)}+W(x)\leq \frac{y}{x}+\log(\frac{x}{\log y})$ for all $y>1$ and $x>0$. Finally, inequality $(e)$ is by raising both sides of above upper bound to power $1/p$, and the pre-exponent constant is no less than 2.

Having obtained the two types of upper bounds in \eqref{eq: sub-expo Lp norm SSQ}  and \eqref{eq: sub-Poisson upper bound on L-p norm SSQ}, we present a short discussion on their comparisons and complementarities. Specifically, we have the following order-wise comparisons.
\begin{align}
    &\sqrt{\frac{\lambda}{\gamma}}\sqrt{p} + p = O(p),\quad \frac{p}{\log(1+p\gamma/\lambda)} = O\left(\frac{p}{\log p}\right), & \text{for large $p$}, \notag\\
    &\sqrt{\frac{\lambda}{\gamma}}\sqrt{p} + p = O\left(\frac{1}{\sqrt{\gamma}}\right),\quad \frac{p}{\log(1+p\gamma/\lambda)} = O\left(\frac{1}{\gamma}\right), & \text{for small $p$}. \label{eq: discussion for p's dependency on Lp norm SSQ}
\end{align}
Note that $\gamma \ll 1$. Therefore, the sub-exponential bound is tighter for small $p$ regime, while the sub-Poisson bound is tighter for large $p$ regime.

\subsubsection{Lower Bounds for Lp norm of centered queue length $\hat{q}$}
Having established the upper bound for $L^p$ norm of $\hat{q}$, we now derive the lower bound. We first recall the following bound for the MGF of $\hat{q}$ with positive $\theta$. Starting from \eqref{MGF for q_infty}, it can be shown from the above derivation as in \eqref{eq: two-sided MGF upper bound SSQ} that
\begin{align}
    1\leq \frac{\mathbb{E}\left[ \exp((\hat{q})\theta)\right] }{\exp\left(-\frac{\lambda}{\gamma}\theta + \frac{\lambda}{\gamma}e^\theta - \frac{\lambda}{\gamma}\right)} \leq\underbrace{1+C_{\lambda,\mu}' \sqrt{2}}_{:=A_{\lambda,\mu}} \label{eq: MGF for q_infty - lambda/mu/gamma, Poisson style}
\end{align}
with $C_{\lambda,\mu}'$ being the constant in Lemma \ref{lem: tight bound on P(q_infty = 0)}.
    

At the beginning of this part of proof, we again introduce the notations $b:=\frac{\lambda}{\gamma}$ to simplify the expressions. We first utilize the integration by parts identity to connect the $L^p$ norm with tail probability as follows.
\begin{align}
    \mathbb{E}\left[ |\hat{q}|^p \right] &= p\int_{0}^{\infty} t^{p-1} \mathbb{P}(|\hat{q}| \geq t) dt \geq p\int_{1/2f(p)}^{f(p)} t^{p-1} \mathbb{P}(\hat{q} \geq t) dt \label{eq: L-p norm integral identity} 
\end{align}
The boundary terms for integral, $f(p)$ is chosen as $f(x) := x/\log(1+x/b)$. Now it suffices to lower bound the tail probability $\mathbb{P}(\hat{q} \geq t)$ for $t\in [1/2 f(p), f(p)]$. To this end,
we will use the change of measure technique \cite{theodosopoulos2005reversionchernoffbound} to derive a lower bound on tail probability. 
We change measure from the law $\mathcal{L}(\hat{q})$ to tilted measure in the following way.
\begin{align*}
    Q\{x\} =\frac{e^{\theta_* x}}{M(\theta_*)}\,\mathcal{L}(\hat{q})\{x\}, \; \forall x\in \text{supp}(\mathcal{L}(\hat{q})), \quad M(\theta_*) = \mathbb{E}_{x\sim\mathcal{L}(\hat{q})}\left[ e^{\theta_*x}\right], \quad \theta_* := \log (\frac{8t}{b}+1)
\end{align*}
The tilted measure then enables us to lower bound tail probability as follows.
\begin{align*}
    \mathbb{P}(\hat{q} \geq t) &= \mathbb{E}_{x\sim \mathcal{L}(\hat{q})}\left[ \mathbf{1}_{\{x\geq t\}}\right] \\
    &= \mathbb{E}_{x\sim Q}\left[ \mathbf{1}_{\{x\geq t\}}M(\theta_*) e^{-\theta_* x} \right] \\ 
    &\overset{(a)}{\geq} \mathbb{E}_{x\sim Q}\left[ \mathbf{1}_{\{t\leq x \leq t+\delta\}}M(\theta_*) e^{-\theta_* x} \right] \\
    &\geq \underbrace{Q\left(\{t\leq x\leq t+\delta\}\right)}_{(a_1)} \cdot \underbrace{M(\theta_*) e^{-\theta_* (t+\delta)}}_{(a_2)}
\end{align*}
Inequality $(a)$ holds from the choice of $\delta = (14 + 2A_{\lambda,\mu})t \geq 0$.
Now it suffices to lower bound the two terms $(a_1)$ and $(a_2)$ above. In the following, we will use Cantelli inequality to lower bound $(a_1)$, and use the MGF bounds in \eqref{eq: MGF for q_infty - lambda/mu/gamma, Poisson style} to lower bound $(a_2)$. We start with lower bounding $(a_1)$ while denoting $\mu_t := \mathbb{E}_{X\sim Q}[X]$ and $v_t := \mathbb{E}_{X\sim Q}[(X - \mu_t)^2]$ as the mean and variance under tilted measure $Q$.
\begin{align}
    (a_1) &\overset{(a)}{\geq} \frac{(\mu_t - t)^2}{v_t + (\mu_t - t)^2} - \frac{v_t}{v_t + (t+\delta - \mu_t)^2} \notag 
\end{align}
Inequality $(a)$ holds from the two sided Cantelli inequality.
\begin{align*}
    Q(\{\mu_t - c \leq x \leq \mu_t + d\}) &\geq \frac{c^2}{c^2 + v_t} - \frac{v_t}{d^2 + v_t} 
\end{align*}
In order to obtain an explicit lower bound for $(a_1)$, 
we need to check the range of mean $\mu_t$ and variance $v_t$ with respect to deviation $t$ under tilted measure $Q$. We first show that the choice of $t\in [1/2 f(p), f(p)]$ ensures that $t\geq 1/2 b$. Thus $t\geq 1/2 b \geq 1/2$ from Assumption \ref{ass:heavy_overload}. To see this, we use the monotonicity of $f(x)$ illustrated as followed,
\begin{align*}
    f'(x) &= \frac{1}{\log^2(1+x/b)} \cdot \bigg( \log(1+x/b) - \frac{x}{b+x} \bigg) =: \frac{1}{\log^2(1+x/b)} \cdot g(x/b) \\
    g'(x) &= \frac{1}{1+x} - \frac{1}{(1+x)^2} = \frac{x}{(1+x)^2} \geq 0, \forall x \geq 0 \implies g(x/b) \geq g(0) = 0, \forall x \geq 0 \\
    f(x) &\geq \lim_{x\to 0} f(x) = \lim_{x\to 0} \frac{1+x/b}{1/b} = b
\end{align*}
Now we check the range of $\mu_t$ and $v_t$ via cumulant generating function, i.e., $ \log \mathbb{E}[e^{\theta X}]$ of $\hat{q}$ for tilted measure $Q$.
\begin{align}
    \mu_t &= \int_{\mathbb{R}} x \frac{e^{\theta_*x}}{M(\theta_*)} d\pi(x) = \frac{M'(\theta_*)}{M(\theta_*)} = \frac{d}{d\theta} \log M(\theta) \bigg|_{\theta = \theta_*} \notag \\
    &= b (e^{\theta_*} - 1) + \frac{
        \mu/\gamma \cdot \mathbb{P}(q = 0) \exp\left(\frac{\mu}{\gamma}\theta_* +\frac{\lambda}{\gamma} - \frac{\lambda}{\gamma}e^{\theta_*} \right)
    }{
        1 + \frac{\mu}{\gamma} \mathbb{P}(q = 0) \int_{0}^{\theta_*} \exp\left(\frac{\mu}{\gamma}y + \frac{\lambda}{\gamma} - \frac{\lambda}{\gamma}e^{y} \right) dy
    } \notag \\
    &\leq b (e^{\theta_*} - 1) + A_{\lambda,\mu}= 8t + A_{\lambda,\mu} \notag \\
     &\leq (8+2A_{\lambda,\mu})t \label{eq: upper bound on mu_t} \\
    \mu_t &\geq b (e^{\theta_*} - 1) = 8t \label{eq: lower bound on mu_t} \\
    v_t &= \int_{\mathbb{R}} (x - \mu_t)^2 \frac{e^{\theta_*x}}{M(\theta_*)} d\pi(x) = \frac{M''(\theta_*)}{M(\theta_*)} - \mu_t^2 = \frac{d^2}{d\theta^2} \log M(\theta) \bigg|_{\theta = \theta_*} \notag \\
    &=  b e^{\theta_*} - \frac{
        \frac{\mu}{\gamma} \mathbb{P}(q = 0) (e^{\theta_*}\lambda/\gamma - \mu/\gamma) \exp\left(\frac{\mu}{\gamma}\theta_* +\frac{\lambda}{\gamma} - \frac{\lambda}{\gamma}e^{\theta_*} \right) 
    }{
        \left(1 + \frac{\mu}{\gamma} \mathbb{P}(q = 0) \int_{0}^{\theta_*} \exp\left(\frac{\mu}{\gamma}y + \frac{\lambda}{\gamma} - \frac{\lambda}{\gamma}e^{y} \right) dy\right)^2
    } \notag \\
    &\quad \quad - \frac{
         (\mu/\gamma \cdot \mathbb{P}(q = 0) \exp\left(\frac{\mu}{\gamma}\theta_* +\frac{\lambda}{\gamma} - \frac{\lambda}{\gamma}e^{\theta_*} \right))^2
    }{\left(1 + \frac{\mu}{\gamma} \mathbb{P}(q = 0) \int_{0}^{\theta_*} \exp\left(\frac{\mu}{\gamma}y + \frac{\lambda}{\gamma} - \frac{\lambda}{\gamma}e^{y} \right) dy\right)^2} \leq be^{\theta_*} =8t + b \notag\\
    &\overset{(a)}{\leq} 10t \label{eq: upper bound on niu_t}
\end{align}
Inequality $(a)$ holds from $t\geq 1/2 b$.
Now for $(a_1)$, we can lower bound the probability under tilted measure $Q$ by constant as follows.
\begin{align}
    (a_1) &\geq \frac{(\mu_t - t)^2}{v_t + (\mu_t - t)^2} - \frac{v_t}{v_t + (t+\delta - \mu_t)^2} \notag \\
    &\overset{(a)}{\geq} \frac{(7t)^2}{10t + (7t)^2} - \frac{10t}{10t + ((15+2A_{\lambda,\mu} - (8+2A_{\lambda,\mu})t))^2} \notag \\
    &= \frac{49t^2 - 10t}{10t + 49t^2} \notag\\
    &\overset{(b)}{\geq}  \frac{49 \times 1/4 - 10\times 1/2}{10\times1/2 + 49\times1/4} = \frac{29}{69} \label{eq: lower bound on tilted measure, constant}
\end{align}
Inequality $(a)$ holds from combining the ranges on $\mu_t$ and $v_t$ with respect to $t$ in \eqref{eq: upper bound on mu_t}, \eqref{eq: lower bound on mu_t} and \eqref{eq: upper bound on niu_t}. Inequality $(b)$ holds from the range of $t\in [1/2 f(p), f(p)]$ and $f(p)\geq b \geq 1$.


For $(a_2)$, first we recall that $M(\theta_*) \geq \exp(-\frac{\lambda}{\gamma}\theta_* -\frac{\lambda}{\gamma} + \frac{\lambda}{\gamma}e^{\theta_*})$ from \eqref{eq: MGF for q_infty - lambda/mu/gamma, Poisson style}. Thus we have the following lower bound.
\begin{align}
    (a_2) &= M(\theta_*) e^{-\theta_* (t+\delta)} \geq \exp(-\frac{\lambda}{\gamma}\theta_* -\frac{\lambda}{\gamma} + \frac{\lambda}{\gamma}e^{\theta_*} - \theta_*(t+\delta)) \notag \\
    &= \exp\left( b e^{\log(1+\frac{8t}{b})} - b \log(1+\frac{8t}{b}) - b - \log(1+\frac{8t}{b})(t+(14 + 2A_{\lambda,\mu})t) \right) \notag \\
    &= \exp\left( 8t - (b+ (15+A_{\lambda,\mu})t)\log(1+\frac{8t}{b})\right) \notag \\
    &= \exp\left( 8t - b \log(1+\frac{8t}{b}) - (15+A_{\lambda,\mu})t\log(1+\frac{8t}{b}) \right) \notag \\
    &\overset{(a)}{\geq} \exp\left( -(15+A_{\lambda,\mu})t\log(1+\frac{8t}{b}) \right) \notag \\
    &\overset{(b)}{\geq} \exp\left( -(15+A_{\lambda,\mu})f(p)\log(1+\frac{8f(p)}{b}) \right) \notag \\
    &\overset{(b)}{\geq} \exp\left( -(15+A_{\lambda,\mu}) (9p + 8b) \right) \label{eq: lower bound on a_2 for L-p norm for centered queue length} 
\end{align}
Inequality $(a)$ holds from the bound $\log(1+x) \leq x$ for $x\geq 0$. Inequality $(b)$ holds from the choice of $t\in [1/2 f(p), f(p)]$. Inequality
$(c)$ holds from the following fact.
\begin{align*}
    (15+A_{\lambda,\mu})f(p)\log(1+\frac{8f(p)}{b}) &\leq (15+A_{\lambda,\mu}) \bigg( 9p + 8b \bigg) 
\end{align*}
To obtain inequality $(c)$, we develop the monotonicity of the following function, $F(x) := -\log(1+\frac{8x}{\log(1+x)}) + 9\log(1+x) + \frac{8\log(1+x)}{x}$. For $x\geq0$, we have
\begin{align*}
    F'(x) &= \frac{9}{1+x} + \frac{8}{x(1+x)} - \frac{8\log(1+x)}{x^2} \\
    &\quad - \frac{8}{(1+\frac{8x}{\log(1+x)})}(\frac{1}{\log(1+x)} - \frac{x}{(1+x)\log^2(1+x)}) \\
    &= \frac{1}{1+x} + \bigg( \frac{8}{1+x} + \frac{8}{x(1+x)} - \frac{8\log(1+x)}{x^2}\bigg) \\
    &\quad - \frac{8}{(1+\frac{8x}{\log(1+x)})}(\frac{1}{\log(1+x)} - \frac{x}{(1+x)\log^2(1+x)}) \\
    &\overset{(a)}{\geq} \frac{1}{1+x} - \frac{8}{(1+\frac{8x}{\log(1+x)})}(\frac{1}{\log(1+x)} - \frac{x}{(1+x)\log^2(1+x)}) \\
    &\overset{(b)}{\ge} \frac{1}{1+x} - \frac{\log(1+x)}{x} (\frac{1}{\log(1+x)} - \frac{x}{(1+x)\log^2(1+x)}) \\
    &= \frac{1}{1+x} - \frac{1}{x} + \frac{1}{(1+x)\log(1+x)} \\
    &= - \frac{1}{x(1+x)} + \frac{1}{(1+x)\log(1+x)} \overset{(c)}{\geq} 0, \quad \forall x\ge 0. \\
    F(x) &\geq \lim_{x\to 0^+} F(x) =8 - \log(9) >0.
\end{align*}
Inequality $(a)$ holds from the fact that $\log(1+x) \leq x$. Thus $\frac{8}{1+x} + \frac{8}{x} - \frac{8\log(1+x)}{x^2}\geq 0$. Inequality $(b)$ holds from the fact that $1+\frac{8x}{\log(1+x)} \geq \frac{8x}{\log(1+x)}$ for $x\geq 0$. Inequality $(c)$ again holds from $\log(1+x) \leq x$. Using the monotonicity of $F(x)$, we plug in $x:=p/b$, and thus show \eqref{eq: lower bound on a_2 for L-p norm for centered queue length}.
\begin{align*}
    \log(1+\frac{8f(p)}{b}) &\leq 9 \log(1+p/b) + \frac{b}{p}\log(1+p/b) \\
     (15+A_{\lambda,\mu})f(p)\log(1+\frac{8f(p)}{b}) &\leq (15+A_{\lambda,\mu}) \bigg( 9p + 8b \bigg) 
\end{align*} Combining bounds for $(a_1)$ and $(a_2)$ in
\eqref{eq: lower bound on tilted measure, constant} and \eqref{eq: lower bound on a_2 for L-p norm for centered queue length}, we have the lower bound on tail probability as follows.
\begin{align*}
    \mathbb{P}(\hat{q} \geq t) &\geq \frac{29}{69} \exp\left( -(15+A_{\lambda,\mu})(9p + 8b) \right), \quad \forall t\in [1/2 f(p), f(p)] 
\end{align*}
Recall that $f(p) = p/\log(1+p/b)$, and $b = \lambda/\gamma$. Continuing from \eqref{eq: L-p norm integral identity}, we use the above lower bound on tail probability to derive the lower bound on $L^p$ norm of $\hat{q}$ as follows.
\begin{align}
    \mathbb{E}\left[ |\hat{q}|^p \right]
    &\geq p\int_{1/2f(p)}^{f(p)} t^{p-1} \mathbb{P}(\hat{q} \geq t) dt \notag \\
    &\geq p (1/2f(p))^{p-1} \int_{1/2f(p)}^{f(p)} \frac{29}{69} \exp\left( -(15+A_{\lambda,\mu})(9p + 8b) \right) dt \notag \\
    &= \frac{29}{69} p (1/2f(p))^{p} \exp\left( -(15+A_{\lambda,\mu})(9p + 8b) \right) \notag \\
     \left\| \hat{q} \right\|_{L^p} &\geq (\frac{29}{69} p )^{1/p} 1/2f(p) \exp (-(15+A_{\lambda,\mu})(9 + 8b/p)) \notag \\
    &= \underbrace{(\frac{29}{69} p )^{1/p} \exp(-9(15+A_{\lambda,\mu}))}_{:= C_{4,\lambda,\mu}} \exp(-8(15+A_{\lambda,\mu})b/p) \frac{p}{2\log(1+p/b)} \label{eq: sub-Poisson lower bound on L-p norm SSQ}
\end{align}


\paragraph{Results} Combining the sub-exponential upper bound in \eqref{eq: sub-expo Lp norm SSQ}, the sub-Poisson upper and lower bounds in \eqref{eq: sub-Poisson lower bound on L-p norm SSQ}, we summarize the final result for $L^p$ norm of centered queue length $\hat{q}$ as follows.
\begin{align}
    \| \hat{q} \|_{L^p} &\leq \min\{C_{1,\lambda,\mu}\sqrt{\frac{1}{\gamma}} \sqrt{p} + C_{2,\lambda,\mu} p \notag \\
    &\quad, C_{3,\lambda,\mu} \frac{p}{\log(1+\gamma p/\lambda)}\} \notag \\
    &\geq C_{4,\lambda,\mu} \exp\left( -C_{5,\lambda,\mu} \frac{1}{ \gamma p} \right) \frac{p}{\log(1+\gamma p/\lambda)} \label{eq: constants for L-p norm}
\end{align}
Where $C_{1,\lambda,\mu} = 2e\sqrt{2\pi\lambda}(2+ \frac{ C_{\lambda,\mu}'^2}{4})$ and $C_{2,\lambda,\mu} = 2e\sqrt{2\pi}$. $C_{3,\lambda,\mu} = 2+ 2A_{\lambda,\mu}$, $C_{4,\lambda,\mu} = (\frac{29}{69}*2) \cdot \exp(-9(15+A_{\lambda,\mu}))$, $C_{5,\lambda,\mu} = (120 + 8A_{\lambda,\mu})\lambda$, and $A_{\lambda,\mu} = C_{\lambda,\mu}' \cdot (\frac{\lambda}{C^2\mu})^{1/(2-4\alpha)} \Gamma(1/(2-4\alpha) +1)$. Constant $C_{\lambda,\mu}'$ is defined in Lemma \ref{lem: tight bound on P(q_infty = 0)} and $C$ is in Assumption \ref{ass:heavy_overload}.

\subsection[Proof for Theorem \ref{thm: Gaussian Wasserstein-$p$ upper bound for M/M/1+M}: Wasserstein-$p$ distance between q and Z]{Proof for Theorem \ref{thm: Gaussian Wasserstein-$p$ upper bound for M/M/1+M}: Wasserstein-$p$ distance between $\tilde{q}$ and $Z$} \label{sec: proof for W-p SSQ}
In this section, we will obtain different upper and lower bounds for Wasserstein-$p$ distance between $\mathcal{L}(\tilde{q})$ and $\mathcal{L}(Z)$. As mentioned in the proof sketch section \ref{sec: proof sketch}, we utilize Stein's method and traingle inequality to acquire different upper bounds, while we use traingle inequality and quantile coupling to acquire different lower bounds. We will illustrate these results via different techniques in the sequel. At the start of this section, we recall our assumption that $\gamma \leq \gamma_0$. This range of $\gamma$ provides pre-limit version of the limit results since the latter studies $\gamma \to 0$.

We start with Stein's method. In this part of proof, we assume that $p$ satisfies $1 \leq p \leq \frac{\lambda}{2\gamma}$. From Lemma \ref{lem: Wasserstein-$p$ bound SSQ reduction}, we reduce upper bounding Wasserstein-$p$ distance between $\tilde{q}$ and $\mathcal{L}(Z)$ into upper bounding idle time proportion $\mathbb{P}(q=0)$ and $L^p$ norm bound $\|\hat{q}\|_{L^p}$ as follows,
    \begin{align*}
        W_p(\mathcal{L}(\tilde{q}), \mathcal{L}(Z))&\leq C_{10, \lambda, \mu} \cdot \big(p\sqrt{\gamma} + \frac{1}{\sqrt{\gamma}} \|\mathbf{1}_{q=0} \|_{L^p} +\gamma\sqrt{p} \| q - \frac{\lambda-\mu}{\gamma}\|_{L^p} \big).
    \end{align*}
Using this reduced form, we combine the bounds for $\mathbb{P}(q=0)$ in Lemma \ref{lem: tight bound on P(q_infty = 0)} and $L^p$ norm of $\hat{q}$ in Proposition \ref{pro: Lp norm of q_infty} to obtain the following upper bound for Wasserstein-$p$ distance between $\tilde{q}$ and $\mathcal{L}(Z)$.
\begin{align}
    W_p(&\mathcal{L}(\tilde{q}) , \mathcal{L}(Z))\leq \sqrt{2}\bigg( \frac{4e\sqrt{2\pi}+2e^2}{\sqrt{\lambda}} + ((e+8e^2+\sqrt{2}e^2)\sqrt{2\pi}/\sqrt{\lambda}) (C_{1,\lambda,\mu} + C_{2,\lambda,\mu}\frac{\lambda}{\sqrt{2}}) \bigg)\cdot\gamma\sqrt{p} \notag\\ 
    &+ \sqrt{2}(1+2e\sqrt{\pi}+ 8\sqrt{2}e^2 + 2e^2)\frac{\mu}{\sqrt{\lambda}}(C_{\lambda,\mu}' \lor 1) \cdot \frac{1}{\sqrt{\gamma}} \exp \left( -\frac{1}{\gamma/\mu} (\lambda/\mu - 1 - \ln (\lambda/\mu))\right) \label{eq: Wasserstein-$p$ upper bound SSQ, Stein}
\end{align}

\paragraph{Triangle Inequality}
Meanwhile, from triangle inequality, we can relax the assumption on $p$ in above Stein's method. Specifically, we can drop the assumption that $p \leq \frac{\lambda}{2\gamma}$. Using triangle inequality, we upper bound Wasserstein-$p$ distance for all $p > 1$ as 
\begin{align}
    W_p(\mathcal{L}(\tilde{q}), \mathcal{L}(Z)) &= \inf_{\pi \in \Pi(\pi, \mathcal{L}(Z))} \mathbb{E}_{(X,Y) \sim \pi} [|X-Y|^p]^{1/p} \notag\\
    &\leq  \sqrt{\frac{\gamma}{\lambda}}\left\| (q - \frac{\lambda - \mu}{\gamma}) \right\|_{L^p} + \left\| Z \right\|_{L^p} \overset{(a)}{\leq} \sqrt{\frac{\gamma}{\lambda}}\left\| \hat{q} \right\|_{L^p} + 4e \sqrt{2\pi} \sqrt{p} 
\label{eq: Wasserstein-$p$ upper bound SSQ, Triangle}
\end{align}
Here $(a)$ is from the $L^p$ norm of Gaussian random variable as in \eqref{eq: Lp norm of Z}. The opposite direction also holds from triangle inequality as follows.
\begin{align}
    \mathcal{W}_{p}(\mathcal{L}(\tilde{q}), \mathcal{L}(Z)) &= \inf_{\pi \in \Pi(\pi, \mathcal{L}(Z))} \mathbb{E}_{(X,Y) \sim \pi} [|X-Y|^p]^{1/p} \notag\\
    &\geq  \sqrt{\frac{\gamma}{\lambda}}\left\| \hat{q} \right\|_{L^p} - \left\| Z \right\|_{L^p} \label{eq: Gaussian Wasserstein-$p$ Lower L^p norm, Triangle} 
\end{align}
Such upper and lower bounds via triangle inequality reduce the problem of upper and lower bounding Wasserstein-$p$ distance into bounding $L^p$ norm of centered queue length $\hat{q}$. Thus, combining the bounds for $L^p$ norm of $\hat{q}$ in Proposition \ref{pro: Lp norm of q_infty}, we have the following upper bound for Wasserstein-$p$ distance between $\tilde{q}$ and $\mathcal{L}(Z)$.
\begin{align*}
    W_p(&\mathcal{L}_{\tilde{q}} , \mathcal{L}(Z))\leq  
    \bigg((C_{4, \lambda, \mu}' \sqrt{p} + C_{2, \lambda, \mu}' \cdot p\sqrt{\gamma}) \land( C_{5, \lambda, \mu}' \frac{p\sqrt{\gamma}}{\log(1+\gamma p/\lambda)} ) \bigg) \\
    &\geq \frac{C_{4,\lambda,\mu}}{\sqrt{\lambda}} \exp\left( -C_{5,\lambda,\mu} \frac{1}{\gamma p} \right) \frac{p\sqrt{\gamma}}{\log(1+\gamma p/\lambda)} - C_{6,\lambda,\mu}'\sqrt{p}.
\end{align*}
We delay the definitions of constants $C_{i,\lambda,\mu}'$ for $i=2,4,5,6$ to the end of this proof (see \eqref{eq: constants for SSQ Wasserstein-$p$ bounds}). The result from triangle inequality holds for all $p\geq 1$, and it provides an order-wise tight upper bound for large $p$, when $p=\Omega(\frac{1}{\gamma})$. Compared to the result from Stein's method in \eqref{eq: Wasserstein-$p$ upper bound SSQ, Stein}, the result from triangle inequality provides a better upper bound for large $p$, while the result from Stein's method provides a better upper bound for small $p$.

\paragraph{Quantile Coupling: Proof for Proposition \ref{pro: Quantile Coupling Lower Bound for SSQ}} Note that the above lower bound via triangle inequality in \eqref{eq: Gaussian Wasserstein-$p$ Lower L^p norm, Triangle} gives trivial lower bound, i.e., negative value, when $p$ is in the range $p=o(\frac{1}{\gamma})$. The following quantile coupling technique provides a non-trivial lower bound for Wasserstein-$p$ distance and complements the lower bound via triangle inequality in small $p$ regime. In the following, we prove Proposition \ref{pro: Quantile Coupling Lower Bound for SSQ}. For simplicity of expressions, we let $P_0:=\mathbb{P}(q=0)$. Without loss of generality, we assume $\mu=1$ in the following proof due to the scaling property of the Assumption \ref{ass:heavy_overload}. 
We use $F_\pi^{-1}(u)$ to denote the generalized quantile function for distribution $\pi$, defined as $F_\pi^{-1}(u):= \inf\{x: F_\pi(x) \geq u\}$. From \cite[Theorem~8.1]{MR520959}, the quantile coupling between the two $1$-dimensional distributions gives the optimal coupling in terms of Wasserstein-$p$ distance. Thus we have the lower bound as follows.
\begin{align}
    W_p(\mathcal{L}_{\tilde{q}}, \mathcal{L}(Z)) &\overset{(a)}{=} \left[ \int_{0}^{1} \left| F_{\mathcal{L}_{\tilde{q}}}^{-1}(u) - F_{\mathcal{L}(Z)}^{-1}(u) \right|^p du \right]^{1/p} \notag\\
    &\overset{(b)}{\geq} \left[ \int_{0}^{cP_0} \left| F_{\mathcal{L}_{\tilde{q}}}^{-1}(u) - F_{\mathcal{L}(Z)}^{-1}(u) \right|^p du \right]^{1/p} \notag\\
    &\overset{(c)}{=} \left[ \int_{0}^{cP_0} \left|  -\frac{\lambda-\mu}{\sqrt{\lambda \gamma}} -\Phi^{-1}(u) \right|^p du \right]^{1/p} \notag\\
    &= \left[ \int_{0}^{cP_0} \left|  \frac{\lambda-\mu}{\sqrt{\lambda \gamma}} +\Phi^{-1}(u) \right|^p du \right]^{1/p} \notag\\
    &\overset{(d)}{\geq} \left[ \int_{0}^{cP_0}  \frac{1}{\sqrt{\frac{\lambda-\ln \lambda-1}{\gamma} - \ln(cC_{\lambda,\mu}'\sqrt{\gamma})}^p}  du \right]^{1/p} \notag\\
    &= \frac{1}{\sqrt{\frac{\lambda-\ln \lambda-1}{\gamma} - \ln(cC_{\lambda,\mu}'\sqrt{\gamma})}} (cP_0)^{1/p} \notag\\
    &\overset{(e)}{\geq} \frac{\sqrt{2}}{\sqrt{\lambda-\ln \lambda -1 }} \cdot\sqrt{\gamma} (cP_0)^{1/p} \notag\\
    &\overset{(f)}{=}  \underbrace{\frac{\sqrt{2}[2C_{\lambda,\mu}' e^{2\sqrt{2}} \sqrt{\lambda}]^{-1}}{\sqrt{\lambda-\ln \lambda -1 }}}_{:=C_{quantile, \lambda, \mu}} \cdot\sqrt{\gamma} P_0^{1/p}  \label{eq: Gaussian Wasserstein-$p$ Lower bound, P_0}
\end{align}
Equality $(a)$ holds from the quantile coupling. Note that $F_{\mathcal{L}(Z)}^{-1}(U)$ with $U\sim \text{Uniform}(0,1)$ has the same distribution as $Z$, same for $F_{\mathcal{L}_{\tilde{q}}}^{-1}(U)$. Thus the coupling $(F_{\mathcal{L}_{\tilde{q}}}^{-1}(U), F_{\mathcal{L}(Z)}^{-1}(U))$ is a feasible coupling between $\mathcal{L}(\tilde{q})$ and $\mathcal{L}(Z)$. Its optimality follows from \cite[Theorem~8.1]{MR520959}. Inequality $(b)$ holds from restricting the integration domain to $(0, cP_0)$ for constant $c := [2C_{\lambda,\mu}' e^{2\sqrt{2}} \sqrt{\lambda}]^{-1} \leq 1$. For equality $(c)$, we use the fact that $F_{\mathcal{L}_{\tilde{q}}}^{-1}(u) = -\frac{\lambda-\mu}{\sqrt{\lambda \gamma}}$ for all $u \in (0, P_0)$ due to the point mass at $-\frac{\lambda-\mu}{\sqrt{\lambda \gamma}}$ in distribution $\mathcal{L}(\tilde{q})$. The fact $F_{\mathcal{L}_{\tilde{q}}}^{-1}(u) = -\frac{\lambda-\mu}{\sqrt{\lambda \gamma}}$ follows from the definition of generalized quantile function and right-continuity of CDF. Inequality $(d)$ holds from the following bound on quantile function, which we will prove later,
\begin{align}
    \Phi^{-1}(u) + \frac{\lambda-\mu}{\sqrt{\lambda \gamma}} \leq -\frac{1}{\sqrt{\left(\frac{\lambda-\ln \lambda -1}{\gamma} - \ln(cC_{\lambda,\mu}'\sqrt{\gamma})\right)}}, \quad \forall u \in (0, cP_0). \label{eq: quantile function bound for SSQ, used in quantile W-p}
\end{align}
Inequality $(e)$ holds from the fact that $\lambda - \ln \lambda -1 \geq 0 \geq2\gamma \ln(cC_{\lambda,\mu}'\sqrt{\gamma})$ under the assumption that $\gamma \leq \gamma_0$. Equality $(f)$ holds from the definition of $c=[2C_{\lambda,\mu}' e^{2\sqrt{2}} \sqrt{\lambda}]^{-1}$. Thus we have the lower bound for Wasserstein-$p$ distance between $\tilde{q}$ and $\mathcal{L}(Z)$ in terms of idle time proportion $P_0$ as in \eqref{eq: Gaussian Wasserstein-$p$ Lower bound, P_0}. The result in Proposition \ref{pro: Quantile Coupling Lower Bound for SSQ} follows directly from \eqref{eq: Gaussian Wasserstein-$p$ Lower bound, P_0}. Combined with the lower bound for $P_0$ in Lemma \ref{lem: tight bound on P(q_infty = 0)}, we have the final result in Proposition \ref{pro: Quantile Coupling Lower Bound for SSQ}.
\begin{align*}
    W_p(\mathcal{L}_{\tilde{q}}, \mathcal{L}(Z)) &\geq C'_{7, \lambda, \mu} \cdot \gamma \exp\left(-\frac{\lambda/\mu - \ln(\lambda/\mu) -1 }{p\gamma}\right) 
\end{align*}
We again delay the definition of constant $C'_{7, \lambda, \mu}$ to later part of this proof (see \eqref{eq: constants for SSQ Wasserstein-$p$ bounds}).
To this end, it suffices to prove \eqref{eq: quantile function bound for SSQ, used in quantile W-p} to complete the proof for Proposition \ref{pro: Quantile Coupling Lower Bound for SSQ}.

To show \eqref{eq: quantile function bound for SSQ, used in quantile W-p},
 we start from the bound on quantile function for standard normal random variable $Z$. Specifically, we use the following bound on lower quantile $\Phi^{-1}(u)$ from \cite[Theorem~2.1]{inglot2010inequalities}.
\begin{align*}
    \Phi^{-1}(u) &+ \frac{\lambda-\mu}{\sqrt{\lambda \gamma}}\overset{(a)}{\leq} -\sqrt{2\ln(1/u)} + \frac{\ln (4\ln(1/u))+2}{2\sqrt{2\ln(1/u)}} + \frac{\lambda-\mu}{\sqrt{\lambda \gamma}} \\
    &\overset{(b)}{\leq} -\sqrt{2\ln(1/cP_0)} + \frac{\ln (4\ln(1/cP_0))+2}{2\sqrt{2\ln(1/cP_0)}} + \frac{\lambda-\mu}{\sqrt{\lambda \gamma}} \\
    &\overset{(c)}{\leq} -\sqrt{2(A+B)} + \frac{\ln (4(A+B))+2}{2\sqrt{2(A+B)}} + \frac{\lambda-\mu}{\sqrt{\lambda \gamma}} \\
    &= -\frac{\sqrt{2}B+\sqrt{2}A}{\sqrt{A+B}} + \frac{ \ln (4(A+B))+2}{2\sqrt{2(A+B)}} + \frac{\lambda-\mu}{\sqrt{\lambda \gamma}} \\
    &\overset{(d)}{\leq} -\frac{\sqrt{2}B+\sqrt{2}A}{\sqrt{A+B}} + \frac{ \ln (4(A+B))+2}{2\sqrt{2(A+B)}} + \sqrt{2B} \\
    &= \left( -\frac{\sqrt{2}B}{\sqrt{A+B}} + \sqrt{2B} \right) + \left( -\frac{\sqrt{2}A}{\sqrt{A+B}} + \frac{ \ln (4(A+B))+2}{2\sqrt{2(A+B)}} \right) \\
    &= \frac{A\sqrt{2B}}{\sqrt{A+B}(\sqrt{A+B}+\sqrt{B})} + \frac{-4A + \ln (4(A+B))+2}{2\sqrt{2(A+B)}}\\
    &\overset{(e)}{\leq} \frac{A}{\sqrt{A+B}}\cdot \frac{1}{\sqrt{2}} + \frac{-4A + \ln (4(A+B))+2}{2\sqrt{2(A+B)}} \\
    &= \frac{-2A + \ln (4(A+B))+2}{2\sqrt{2(A+B)}} \\
    &= \frac{1}{2\sqrt{2(A+B)}} (2\ln c + 2\ln C_{\lambda,\mu}' + 2 + \ln(4(\lambda-\ln \lambda -1) - 4\gamma\ln(cC_{\lambda,\mu}'\sqrt{\gamma}))) \\
    &\overset{(f)}{\leq}  \frac{1}{2\sqrt{2(A+B)}} (2\ln c+ 2\ln C_{\lambda,\mu}' + 2 + \ln(4\lambda))\\
    &\overset{(g)}{\leq} -\frac{2\sqrt{2}}{2\sqrt{2(A+B)}} 
\end{align*}
Inequality $(a)$ holds from \cite[Theorem~2.1]{inglot2010inequalities} which is quoted as follows.
\begin{align*}
    \Phi^{-1}(u) &\leq -\sqrt{2\ln(1/u)} + \frac{\ln (4\ln(1/u))+2}{2\sqrt{2\ln(1/u)}}, \quad \forall u \in (0,0.1)
\end{align*}
Due to the assumption $\gamma \leq \gamma_0$, we have $cP_0 \leq 1/10$ from the following derivation using the upper bound for $P_0$ in Lemma \ref{lem: tight bound on P(q_infty = 0)}.
\begin{align*}
    cP_0 \leq cC_{\lambda,\mu}' \sqrt{\gamma} \exp(-\frac{\lambda-1-\ln\lambda}{\gamma}) \leq 1/10 \impliedby \gamma \leq \frac{1}{(10cC_{\lambda,\mu}')^2} = \frac{1}{(2^2 5^2 e^{4\sqrt{2}}\lambda)}
\end{align*}
Thus the domain $u \in (0, cP_0)$ is a subset of $(0,0.1)$. Inequality $(b)$ holds from the monotonicity of the right hand side function in $u$ in inequality $(a)$. Thus we plug in $u = cP_0$ and yield right hand side of $(b)$. Inequality $(c)$ holds from the monotonicity of the right hand side function in $P_0$. We plug in the upper bound for $P_0$ from Lemma \ref{lem: tight bound on P(q_infty = 0)} as follows,
\begin{align*}
    P_0 &\leq C_{\lambda,\mu}' \sqrt{\gamma} \exp\left( -\frac{1}{\gamma} (\lambda - \ln \lambda -1) \right).
\end{align*}
We then apply change of variable after plugging in $P_0$ to achieve the right hand side of $(c)$, where we define
$A$ and $B$ as follows
\begin{align*}
    A &:= -\ln(cC_{\lambda,\mu}'\sqrt{\gamma}), \quad B := \frac{\lambda-\ln(\lambda)-1}{\gamma}.
\end{align*}
Inequality $(d)$ holds from the crude bound $\lambda - \ln \lambda - 1 \leq \frac{(\lambda - 1)^2}{2\lambda}$. Thus we have $\frac{\lambda-\mu}{\sqrt{\lambda \gamma}} = \frac{\lambda-1}{\sqrt{\lambda \gamma}} \leq \sqrt{2B}$. Inequality $(e)$ holds from the fact that $A, B \geq 0$ and thus $\sqrt{A+B} + \sqrt{B} \geq 2 \sqrt{B}$. Inequality $(f)$ holds from the assumption on $\gamma$ such that
\begin{align*}
    -4\gamma \ln(cC_{\lambda,\mu}'\sqrt{\gamma}) \leq 1, \impliedby \gamma \leq \frac{c^2C_{\lambda,\mu}'^2}{16} = \frac{1}{(2^2 e^{4\sqrt{2}}\lambda)}
\end{align*}
Combined with $\lambda \geq 1$, we have $4\lambda - 4\ln \lambda -4 -4\gamma \ln(cC_{\lambda,\mu}'\sqrt{\gamma}) \leq 4\lambda$. Inequality $(g)$ holds from the choice of $c$ such that $2\ln c + 2\ln C_{\lambda,\mu}' + 2 + \ln(4\lambda) \leq -2\sqrt{2}$.

\paragraph{Result}
In summary, combining the upper bounds from Stein's method in \eqref{eq: Wasserstein-$p$ upper bound SSQ, Stein} and triangle inequality in \eqref{eq: Wasserstein-$p$ upper bound SSQ, Triangle}, and the lower bounds from triangle inequality in \eqref{eq: Gaussian Wasserstein-$p$ Lower L^p norm, Triangle} and quantile coupling in Proposition \ref{pro: Quantile Coupling Lower Bound for SSQ},
 we have the following bounds for $\mathcal{W}_p(\mathcal{L}(\tilde{q}), \mathcal{L}(Z))$ as 
\begin{align*}
    \mathcal{W}_p(\mathcal{L}(\tilde{q}), \mathcal{L}(Z)) 
    &\leq  C_{1, \lambda, \mu}' (\frac{1}{\sqrt{\gamma}})\exp(-\frac{C^2}{2\lambda}\frac{1}{p\gamma^{1-2\alpha}}) + C_{2, \lambda, \mu}' p\sqrt{\gamma} \quad \text{if } 1< p\leq \lambda/(2\gamma) \\
    \mathcal{W}_p(\mathcal{L}(\tilde{q}), \mathcal{L}(Z)) 
    &\leq \bigg((C_{4, \lambda, \mu}' \sqrt{p} + C_{2, \lambda, \mu}' \cdot p\sqrt{\gamma}) \land( C_{5, \lambda, \mu}' \frac{p\sqrt{\gamma}}{\log(1+\gamma p/\lambda)} ) \bigg) \\
    \mathcal{W}_p(\mathcal{L}(\tilde{q}), \mathcal{L}(Z)) &\geq \max \Bigg\{\frac{C_{4,\lambda,\mu}}{\sqrt{\lambda}} \exp\left( -C_{5,\lambda,\mu} \frac{1}{\gamma p} \right) \frac{p\sqrt{\gamma}}{\log(1+\gamma p/\lambda)} - C_{6,\lambda,\mu}'\sqrt{p} \\
    &\quad, C'_{7, \lambda, \mu} \cdot \gamma \exp\left(-\frac{\lambda/\mu - \ln(\lambda/\mu) -1 }{p\gamma}\right) \Bigg\}
\end{align*}
Where the constants $C_{i,\lambda,\mu}, C_{i,\lambda,\mu}'$ are defined as 
\begin{align}
    C_{1,\lambda,\mu}'&:= \sqrt{2}(1+2e\sqrt{\pi}+ 8\sqrt{2}e^2 + 2e^2)\frac{\mu}{\sqrt{\lambda}}(C_{\lambda,\mu}' \lor 1) \notag\\
    C_{2,\lambda,\mu}' &:=\sqrt{2}\bigg( \frac{4e\sqrt{2\pi}+2e^2}{\sqrt{\lambda}} + ((e+8e^2+\sqrt{2}e^2)\sqrt{2\pi}/\sqrt{\lambda}) (C_{1,\lambda,\mu} + 2e\sqrt{\pi}\lambda) \bigg)  \notag\\
    C_{4, \lambda, \mu}' &:= C_{1,\lambda,\mu}+ 2e\sqrt{2\pi}, \quad C_{5, \lambda, \mu}' := 1/\sqrt{\lambda}\cdot (6+ 2C_{\lambda,\mu}') + 2e\sqrt{2\pi}, \notag\\
    C'_{7, \lambda, \mu} &:= \frac{\sqrt{2}}{2C_{\lambda,\mu}' e^{2\sqrt{2}} \sqrt{\lambda} \sqrt{\lambda-\ln \lambda -1 }} \min\left\{\frac{D_{\lambda,\mu}'}{\sqrt{\mu}},1\right\}, \quad C_{1,\lambda,\mu} = 2e\sqrt{2\pi\lambda}(2+ \frac{ C_{\lambda,\mu}'^2}{4}) \notag\\
    C_{4,\lambda,\mu} &:= (\frac{29}{69}*2) \cdot \exp(-9(15+A_{\lambda,\mu}), \quad C_{5,\lambda,\mu} := (120 + 8A_{\lambda,\mu})\lambda\notag\\
    C_{\lambda,\mu}' &=  \lambda/\mu(e\vee \frac{(1+C)(2+C)}{C^2+C-1}), \quad D_{\lambda,\mu}' = \frac{\sqrt{\mu}}{(2\sqrt{\mu} + \sqrt{2\pi e^2 \lambda})} \label{eq: constants for SSQ Wasserstein-$p$ bounds}
\end{align}

Now in order to obtain the final result in Theorem \ref{thm: Gaussian Wasserstein-$p$ upper bound for M/M/1+M}, it remains to analyze the dominant terms in above upper bounds for Wasserstein-$p$ distance between $\mathcal{L}(\tilde{q})$ and $\mathcal{L}(Z)$. In particular, we will focus on showing the order of joint dependency on $p$ and $\gamma$ in above upper and lower bounds. Since there are two variables $p$ and $\gamma$ in above bounds, in the following analysis, we fix $\gamma \leq \gamma_0$ a constant, and view $p$ as our independent variable. Our analysis for dominating terms will discuss different regimes of $p$ as a function of $\gamma$. 
Consequently, we have the following three stage behavior for Wasserstein-$p$ distance as in \eqref{eq: Wasserstein-$p$ bound, first regime}, \eqref{eq: Wasserstein-$p$ bound, second regime}, \eqref{eq: Wasserstein-$p$ bound, third regime} below.

First, when $\ln(p) /\ln(1/\gamma) < 1-2\alpha$, i.e., $\lim_{\gamma \to 0}\gamma^{1-2\alpha}p \to 0$, we show the dominating term for Wasserstein-$p$ is $p\sqrt{\gamma}$ using Wasserstein-$p$ bound \eqref{eq: Stein Wasserstein-$p$ upper p's function}. Particularly, suppose $\ln(p)/\ln(1/\gamma) \leq 1-2\alpha - 2\epsilon,$ with $ 0< \epsilon < 1/2-\alpha$, we have the following bound.
\begin{align}
    \mathcal{W}_{p}(\mathcal{L}(\tilde{q}), \mathcal{L}(Z)) &\leq C_{1, \lambda, \mu}' (\frac{1}{\sqrt{\gamma}}+\sqrt{p})\exp(-\frac{C^2}{2\lambda}\frac{1}{p\gamma^{1-2\alpha}}) + C_{2, \lambda, \mu}' p\sqrt{\gamma} + C_{3, \lambda, \mu}' p^{3/2}\gamma \notag\\
    &\overset{(a)}{\leq}  C_{1, \lambda, \mu}' (\frac{1}{\sqrt{\gamma}}+\sqrt{p}) (\frac{2\lambda}{C^2})^A \Gamma(A+1) \cdot p^A \gamma^{A(1-2\alpha)} +\notag\\
    &\quad \quad C_{2, \lambda, \mu}' p\sqrt{\gamma} + C_{3, \lambda, \mu}' p^{3/2}\gamma \\
    &\overset{(b)}{\leq} D_{1,\lambda,\mu} \cdot p\sqrt{\gamma} \label{eq: Wasserstein-$p$ bound, first regime}
\end{align}
Where $(a)$ holds with $A:=1+\alpha/\epsilon$ from the fact that polynomial functions are upper bounded by exponential functions as in  \eqref{eq: Exp Dominating Polynomial Stirling Bound}. And $(b)$ holds from choosing constant $D_{1,\lambda,\mu}$ as the following:
\begin{align*}
    D_{1,\lambda,\mu} &:= 2C_{1,\lambda,\mu}' (\frac{2\lambda}{C^2})^A \Gamma(A+1) + C_{2, \lambda, \mu}' + C_{3, \lambda, \mu}' 
\end{align*}

When $ 1- 2\alpha \leq \ln(p) /\ln(1/\gamma) \leq 1/2$, from analyzing on \eqref{eq: Triangle Wasserstein-$p$ upper p's function}, we show that the dominant term shifts from $p\sqrt{\gamma}$ to $\sqrt{p}$. We denote constant $D_{2,\lambda,\mu} := C_{4, \lambda, \mu}' + C_{2, \lambda, \mu}'$.
\begin{align}
    \mathcal{W}_{p}(\mathcal{L}(\tilde{q}), \mathcal{L}(Z)) &\leq (C_{4, \lambda, \mu}' \sqrt{p} + C_{2, \lambda, \mu}' \cdot p\sqrt{\gamma}) \leq D_{2,\lambda,\mu} \cdot \sqrt{p}  \label{eq: Wasserstein-$p$ bound, second regime}
\end{align}

Finally, when $\ln(p) /\ln(1/\gamma) \geq 1/2$, the dominant term is $ \frac{p\sqrt{\gamma}}{\log (1+p\gamma/\lambda)}$. Here we denote $D_{3,\lambda,\mu} := C_{5, \lambda, \mu}' + C_{6, \lambda, \mu}'$ and $C_{5, \lambda, \mu}', C_{6, \lambda, \mu}'$ are defined in \eqref{eq: constants for SSQ Wasserstein-$p$ bounds}.
\begin{align}
    \mathcal{W}_{p}(\mathcal{L}(\tilde{q}), \mathcal{L}(Z)) &\leq C_{5, \lambda, \mu}' \frac{p\sqrt{\gamma}}{\log (1+p\gamma/\lambda)} + C_{6, \lambda, \mu}' \sqrt{p} 
    \overset{(a)}{\leq} D_{3,\lambda,\mu} \cdot \frac{p\sqrt{\gamma}}{\log (1+p\gamma/\lambda)} \label{eq: Wasserstein-$p$ bound, third regime}
\end{align}

$(a)$ holds since $\sqrt{p}$ is always upper bounded by $ \sqrt{\gamma}/\sqrt{\lambda} \cdot p/\ln(1+p\gamma/\lambda)$, as can be seen from the following derivation.
\begin{align*}
    \log (1+x) \leq \sqrt{x} \implies \log(1+p\gamma/\lambda) \leq \sqrt{p\gamma/\lambda} \implies \sqrt{p} \leq \frac{p\sqrt{\gamma}}{\sqrt{\lambda}\log(1+p\gamma/\lambda)}
\end{align*}

Meanwhile, for the lower bound, it can be derived that, when p is in the range $1 \leq p < D_{5,\lambda,\mu}/\gamma$, we have the following bound.
\begin{align*}
    \mathcal{W}_p(\mathcal{L}(\tilde{q}), \mathcal{L}(Z)) &\geq C'_{7, \lambda, \mu} \cdot \gamma \exp\left(-\frac{\lambda/\mu - \ln(\lambda/\mu) -1 }{p\gamma/\mu}\right),
\end{align*}
where constant $D_{5,\lambda,\mu}$ is defined as in \eqref{eq: D_4, lambda, mu SSQ} below. 

When $p\geq D_{5,\lambda,\mu}/\gamma$, Wasserstein-$p$ distance is lower bounded by order $\Omega(\frac{p\sqrt{\gamma}}{\log(1+p\gamma/\lambda)})$. Specifically, we have the following bound,
\begin{align*}
    \mathcal{W}_p(\mathcal{L}(\tilde{q}), \mathcal{L}(Z)) &\geq  D_{4,\lambda,\mu}' \cdot \frac{p\sqrt{\gamma}}{\log(1+\gamma p/\lambda)}.
\end{align*}
Constant $D_{4,\lambda,\mu},D_{5,\lambda,\mu}$ are defiend as 
\begin{align}
    D_{5,\lambda,\mu} &:= \frac{2^{12}\lambda}{C_{4,\lambda,\mu}^4}\cdot\exp(4C_{5,\lambda,\mu}\lambda), D_{4,\lambda,\mu} := \frac{1}{2} \frac{C_{4,\lambda,\mu}}{\sqrt{\lambda}}\exp(-C_{5,\lambda,\mu}/D_{5,\lambda,\mu}) \label{eq: D_4, lambda, mu SSQ}
\end{align}

\subsection{Proof for Theorem \ref{thm: SSQ Tail bound}: SSQ Tail Bound}
In this section, we present the proof for Theorem \ref{thm: SSQ Tail bound} on tail bounds for steady-state queue length in SSQ. We will present two different approaches to prove Theorem \ref{thm: SSQ Tail bound}. The first approach is based on concentration via Wasserstein-$p$, i.e., Lemma \ref{lemma: tail bound via wasserstein}, which directly connects Wasserstein-$p$ distance bounds to tail bounds. This argument will cover Theorem \ref{thm: SSQ Tail bound, first regime} and Theorem \ref{thm: SSQ Tail bound, second regime} in two different regimes of deviation and provide efficient concentration inequalities. The second approach is based on transform method, i.e., moment generating function (MGF) bound and Markov inequality. This approach will cover Theorem \ref{thm: SSQ Tail bound, third regime} and Theorem \ref{thm: SSQ Tail bound, fourth regime} in large deviation regimes with order-wise tight exponents.
\subsubsection{Concentration via Wasserstein-$p$} \label{subsubsec: SSQ Tail Bound via Concentration}
In this section, we prove Theorem \ref{thm: SSQ Tail bound, first regime} and Theorem \ref{thm: SSQ Tail bound, second regime} using Lemma \eqref{eq: Tail bound from Wasserstein-$p$ distance}. This argument also allow us to cover third and fourth regimes in Theorem \ref{thm: SSQ Tail bound, third regime} and Theorem \ref{thm: SSQ Tail bound, fourth regime}. However, the exponent terms in these two regimes are not as sharp as those derived from transform method in next section, and we cannot achieve order-wise tight lower bounds. So we merely present the third and fourth regimes here for completeness. We will work with sharper bounds in next section using transform method.

Recall that we have different dominant terms of Wasserstein-$p$ distance in different regimes of $p$ as in \eqref{eq: Wasserstein-$p$ bound, first regime}, \eqref{eq: Wasserstein-$p$ bound, second regime}, \eqref{eq: Wasserstein-$p$ bound, third regime}. We use them here to derive tail bounds in different regimes of deviation $a$. Also recall our assumption on $\gamma$ in \eqref{eq: gamma assumption for SSQ W-p} that
\begin{align*}
    \gamma \leq \gamma_0 := \min \left\{
        \frac{e^{4\sqrt{2}}\lambda}{25},
        \frac{\mu^2}{(2^2 e^{4\sqrt{2}}\lambda)},
        \mu, (1/2eD_{1,\lambda,\mu})^{-\frac{1}{1/2-\alpha-\epsilon}}, 
        (2eD_{1,\lambda,\mu})^{-\frac{2}{1/2-\alpha-\epsilon}},
        (2\sqrt{2}eD_{1,\lambda,\mu})^{-1}, \lambda/(A_{\lambda,\mu}^2). 
    \right\}
\end{align*}

For simplicity of all below derivations, we will use the notation $a$ for deviation level, instead of $a_\gamma$ as in Theorem \ref{thm: SSQ Tail bound}. 
\paragraph{Proof for Theorem \ref{thm: SSQ Tail bound, first regime}} 
For the setup of this deviation regime, we consider a constant $a>0$. Moreover, we assume $\gamma$ is small enough depending on $a$ as follows. 
\begin{align}
    \gamma_{a}:=\min \{ \gamma_0, (2\alpha)^{\frac{1}{1-2\alpha}} a^{-\frac{1}{1-2\alpha -\epsilon}}, (\frac{\alpha}{eD_{1,\lambda,\mu}} a)^{\frac{1}{1/2 -\alpha}}, (\frac{1}{eD_{1,\lambda,\mu}a})^2, (\frac{1}{2}-\alpha -\epsilon)^{\frac{1}{1/2-\alpha -\epsilon}}, \frac{4}{a^6} \}. \label{eq: gamma_a}
\end{align}
Here, $0<\epsilon<1/2-\alpha$ and $\gamma_0$ is defined in \eqref{eq: gamma assumption for SSQ W-p}. In order to apply Lemma \ref{lemma: tail bound via wasserstein} for tail bounds, we need to pick proper $p$ and $\rho$ depending on $a$ and $\gamma$. Particularly, we will pick $p$ and $\rho$ such that the regime of $p$ falls into the first regime in \eqref{eq: Wasserstein-$p$ bound, first regime} as $\ln(p)/\ln(1/\gamma) < 1-2\alpha$. Our choice is as follows.
\begin{align}
    \rho &:= 1 - e \mathcal{W}_p(\mathcal{L}(\tilde{q}), \mathcal{L}(Z)) / a, \quad
    p := a^2/2 + \ln(1/\sqrt{\gamma}) \label{eq: p and rho for Wasserstein-$p$ distance}
\end{align}
We first verify the validity of the above choice of $p$ and $\rho$, i.e., $p$ falls into the first regime in \eqref{eq: Wasserstein-$p$ bound, first regime} and $\rho \in [0,1]$. First notice that, from the assumption \eqref{eq: gamma assumption for SSQ W-p} that $\gamma$ is small enough, we have the following condition on $p$:
\begin{align*}
    1< p &:= a^2/2 + \ln(1/\sqrt{\gamma}) \overset{(a)}{<}  a^2/2+\frac{1}{1-2\alpha-2\epsilon} (\frac{1}{\gamma})^{1/2-\alpha-\epsilon}\overset{(b)}{\leq} \frac{1}{2} (\frac{1}{\gamma})^{1-2\alpha-2\epsilon} + a^2/2 \overset{(c)}{\leq} (\frac{1}{\gamma})^{1-2\alpha -2\epsilon}
\end{align*}
Inequality $(a)$ holds from the fact that logarithmic function is upper bounded by polynomial function, i.e., 
\begin{align}
    \ln(1/\sqrt{\gamma}) = \frac{1}{2\alpha} \ln (1/\gamma^\alpha) \leq \frac{1}{2\alpha} \gamma^{-\alpha} \label{eq: ln inequality}
\end{align}
Inequality $(b)$ and $(c)$ hold from the assumption \eqref{eq: gamma_a} and can be shown from $\gamma \leq (1/2 -\alpha - \epsilon)^{\frac{1}{1/2 -\alpha -\epsilon}}$ and $\gamma \leq (\frac{1}{a})^{\frac{1}{1/2 -\alpha -\epsilon}}$ respectively. Therefore, we have verified that $p$ falls into the first regime in \eqref{eq: Wasserstein-$p$ bound, first regime}, i.e., $\ln(p)/\ln(1/\gamma) < 1-2\alpha$ and $p> 1$. Thus, from the choice of $p$ in \eqref{eq: p and rho for Wasserstein-$p$ distance} and Wasserstein-$p$ bound \eqref{eq: Wasserstein-$p$ bound, first regime}, we have
\begin{align*}
    \mathcal{W}_p(\mathcal{L}(\tilde{q}), \mathcal{L}(Z)) &\leq D_{1,\lambda,\mu} \cdot p\sqrt{\gamma} \\
    & = D_{1,\lambda,\mu} \cdot \sqrt{\gamma} (\ln(1/\sqrt{\gamma}) +a^2/2)
\end{align*}

Now we verify that $\rho \in [0,1]$. Recall the choice of $\rho$ in \eqref{eq: p and rho for Wasserstein-$p$ distance}. We have $\rho \leq 1$ trivially. For the lower bound of $\rho$,
\begin{align*}
    \rho &:= 1 - e \mathcal{W}_p(\mathcal{L}(\tilde{q}), \mathcal{L}(Z)) / a \\
    &\geq 1 - \frac{1}{a}\bigg(eD_{1,\lambda,\mu} \sqrt{\gamma} (\ln(1/\sqrt{\gamma}) +a^2/2)\bigg) \\
    &= 1 - \frac{1}{a}\bigg(eD_{1,\lambda,\mu} \sqrt{\gamma} \ln(1/\sqrt{\gamma}) + e D_{1,\lambda,\mu} \sqrt{\gamma} a^2/2\bigg)\\
    &\overset{(a)}{\geq} 1 - \frac{1}{a}\bigg(eD_{1,\lambda,\mu} \sqrt{\gamma} \ln(1/\sqrt{\gamma}) + a/2\bigg) \\
    &\overset{(b)}{\geq} 1 - \frac{1}{a}\bigg(a/2 + a/2\bigg) = 0
\end{align*}
Inequality $(a)$ holds from the assumption \eqref{eq: gamma_a} of $\gamma \leq (\frac{1}{eD_{1,\lambda,\mu}a})^2$. This assumption leads to $e D_{1,\lambda,\mu} \sqrt{\gamma} a^2/2 \leq a/2$. Inequality $(b)$ can be shown equivalence with the assumption \eqref{eq: gamma_a} of $\gamma \leq (\frac{\alpha}{eD_{1,\lambda,\mu}} a)^{\frac{1}{1/2 -\alpha}}$ via the same argument used in \eqref{eq: ln inequality}.


Having verified the validity of $p$ and $\rho$, we can now apply Lemma \ref{lemma: tail bound via wasserstein} to derive the tail bound.
\begin{align*}
    \left|\mathbb{P}(\tilde{q} > a) - \mathbb{P}(Z > a)\right| &\leq eD_{1,\lambda,\mu}\cdot  \sqrt{\gamma} (\ln(1/\sqrt{\gamma}) +a^2/2) \phi(\rho a) + \sqrt{\gamma} e^{-a^2/2} \\
    &\overset{(a)}{\leq} eD_{1,\lambda,\mu} \cdot \sqrt{\gamma} (\ln(1/\sqrt{\gamma}) +a^2/2) \exp((1-\rho)a^2)\phi(-a^2/2) + \sqrt{\gamma} e^{-a^2/2} \\
    &\overset{(b)}{\leq} \frac{D_{1,\lambda,\mu}\exp(1+2eD_{1,\lambda,\mu})}{\sqrt{2\pi}}\cdot \sqrt{\gamma} (\ln(1/\sqrt{\gamma}) +a^2/2) \exp(-a^2/2) + \sqrt{\gamma} e^{-a^2/2} 
\end{align*}

Inequality $(a)$ can be shown equivalent to $(\rho - 1)^2\geq 0$. Inequality $(b)$ is derived by plugging in the definition of $\phi(\cdot)$ and bounding $\exp((1-\rho)a^2)$ as follows.
\begin{align*}
    \exp((1-\rho)a^2) &= \exp( e W_p(\mathcal{L}(\tilde{q}), \mathcal{L}(Z)) a) \leq \exp\bigg( e D_{1,\lambda,\mu} \sqrt{\gamma} ( (\ln(1/\sqrt{\gamma}) +a^2/2)) a \bigg) \\
    &\overset{(a)}{\leq} \exp(eD_{1,\lambda,\mu} (1+1)) 
\end{align*}
Inequality $(a)$ is equal to the assumption \eqref{eq: gamma assumption for SSQ W-p} of $\gamma \leq (2\alpha)^{\frac{1}{1-2\alpha}} a^{-\frac{1}{1-2\alpha -\epsilon}}$, and $\gamma \leq \frac{4}{a^6}$. With the above derivation, we obtain the efficient concentration inequality for constant deviation regime as follows,
\begin{align}
    \left|\mathbb{P}(\tilde{q} > a) - \mathbb{P}(Z > a)\right| \leq D_{1,\lambda,\mu}' \cdot \sqrt{\gamma} (\ln(1/\sqrt{\gamma}) +a^2/2) e(-a^2/2) \label{eq: Tail bound, first regime},
\end{align}
with constant $D_{1,\lambda,\mu}' := \frac{D_{1,\lambda,\mu}\exp(1+2eD_{1,\lambda,\mu})}{\sqrt{2\pi}}+1$ and $D_{1,\lambda,\mu}$ is defined in \eqref{eq: Wasserstein-$p$ bound, first regime} and \eqref{eq: constants for SSQ Wasserstein-$p$ bounds}.

Beyond constant regime, we consider deviation $a$ the independent variable and $\gamma$ a constant. In the following, the range of $a$ w.r.t. $\gamma$ is separated into three regimes, named as near constant deviation, moderate deviation and large deviation. We obtain different tail bounds in these three regimes respectively using Lemma \ref{lemma: tail bound via wasserstein} and different dominant terms of Wasserstein-$p$ distance in \eqref{eq: Wasserstein-$p$ bound, first regime}, \eqref{eq: Wasserstein-$p$ bound, second regime}, \eqref{eq: Wasserstein-$p$ bound, third regime}.


\paragraph{Proof for Theorem \ref{thm: SSQ Tail bound, second regime}} 
For this regime, we consider deviation $a$ in the following range depending on $\gamma$,
\begin{align}
    \underbrace{2eD_{1,\lambda,\mu}}_{D_{Tail, 2, l}} \leq
    a \leq \underbrace{\min\left\{1,\frac{1}{2eD_{1,\lambda,\mu}}\right\} }_{D_{Tail, 2, u}}(1/\gamma)^{1/2-\alpha - \epsilon}, \label{eq: near constant deviation setup}
\end{align}
with fixed $ 0<\epsilon<1/2-\alpha$ and $D_{1,\lambda,\mu}$ defined in \eqref{eq: Wasserstein-$p$ bound, first regime}. We first note that the $\delta$ in Theorem \ref{thm: SSQ Tail bound, second regime} is defined as $\delta := 1/2 - \alpha - \epsilon$, and that the definition of $\gamma_0$ in \eqref{eq: gamma assumption for SSQ W-p} also depends on this fixed $\epsilon$. Also recall that we use notation $a$ for deviation level for simplicity.
 From the assumption $\gamma \leq \gamma_0$ in \eqref{eq: gamma assumption for SSQ W-p}, the above interval is non-empty, since $\gamma \leq (2eD_{1,\lambda,\mu})^{-\frac{2}{1/2-\alpha-\epsilon}}$ and $\gamma \leq (1/2eD_{1,\lambda,\mu})^{-\frac{1}{1/2-\alpha-\epsilon}}$.

In this regime, we pick $p$ and $\rho$ same as in \eqref{eq: p and rho for Wasserstein-$p$ distance}, i.e., $p= a^2/2 + \ln(1/\sqrt{\gamma})$ and $\rho = 1 - e \mathcal{W}_p(\mathcal{L}(\tilde{q}), \mathcal{L}(Z)) / a$. Similar to the constant deviation regime, we need to verify the validity of this choice of $p$ and $\rho$, i.e., $p$ falls into the first regime in \eqref{eq: Wasserstein-$p$ bound, first regime} and $\rho \in [0,1]$. First, we check that $p$ falls into the first regime in \eqref{eq: Wasserstein-$p$ bound, first regime},
\begin{align*}
    1\leq p := a^2/2 + \ln(1/\sqrt{\gamma}) &\overset{(a)}{\leq} a^2/2 + \frac{1}{2}(1/\gamma)^{1-2\alpha - 2\epsilon}  \overset{(b)}{\leq} (1/\gamma)^{1-2\alpha - 2\epsilon}.
\end{align*}
Inequality $(a)$ is from assumption \eqref{eq: gamma assumption for SSQ W-p} and inequality $(b)$ is from the regime of deviation $a$ in \eqref{eq: near constant deviation setup}. Thus, we have verified that $p$ falls into the first regime in \eqref{eq: Wasserstein-$p$ bound, first regime}, i.e., $\ln(p)/\ln(1/\gamma) < 1-2\alpha$ and $p> 1$. From Wasserstein-$p$ bound in the first regime \eqref{eq: Wasserstein-$p$ bound, first regime}, we have
\begin{align*}
    \rho &= 1 - e \mathcal{W}_p(\mathcal{L}(\tilde{q}), \mathcal{L}(Z)) / a \\
    &\geq 1 - e D_{1,\lambda,\mu} \sqrt{\gamma} (a^2/2 + \ln(1/\sqrt{\gamma})) / a \\
    &= 1 - \frac{1}{a}\bigg(e D_{1,\lambda,\mu} \sqrt{\gamma} a^2/2 + e D_{1,\lambda,\mu} \sqrt{\gamma} \ln(1/\sqrt{\gamma})\bigg) \\
    &\overset{(a)}{\geq} 1 - 1/4 - e D_{1,\lambda,\mu} \sqrt{\gamma} \ln(1/\sqrt{\gamma}) / a\\
    &\overset{(b)}{\geq} 1- 1/4 - 1/4 = 1/2
\end{align*}
Inequality $(a)$ holds from the regime of deviation $a$ in \eqref{eq: near constant deviation setup}. Specifically, the setup of upper bound on $a$ in \eqref{eq: near constant deviation setup} ensures that
\begin{align}
    \sqrt{\gamma}(a^2/2 + \ln(1/\sqrt{\gamma})) \leq a/(2D_{1,\lambda,\mu}e), \label{eq: quadratic inequality for near constant deviation, SSQ}
\end{align} 
which is proved below. Inequality $(b)$ also holds from the regime of deviation $a$ in \eqref{eq: near constant deviation setup}, which ensures that $\sqrt{\gamma} e D_{1,\lambda,\mu} \ln(1/\sqrt{\gamma}) / a \leq 1/4$. Now, to prove inequality \eqref{eq: quadratic inequality for near constant deviation, SSQ}, we rearrange it into a quadratic inequality of $a$ as $\sqrt{\gamma}(a^2/2 + \ln(1/\sqrt{\gamma})) - a/(2D_{1,\lambda,\mu}e) \leq 0$ and check its roots. 
\begin{align*}
    a & \leq \frac{1}{2eD_{1,\lambda,\mu} (\gamma)^{1/2-\alpha - \epsilon}} \leq \frac{1}{2D_{1,\lambda,\mu}e\sqrt{\gamma}} \leq \frac{1}{\sqrt{\gamma}} \bigg( \frac{1}{2eD_{1,\lambda,\mu}}+\sqrt{\frac{1}{4e^2D_{1,\lambda,\mu}^2} - 2\gamma\ln(1/\sqrt{\gamma})}\bigg) \\
    a &\geq 2eD_{1,\lambda,\mu}\sqrt{\gamma}\ln(\frac{1}{\sqrt{\gamma}}) \geq \frac{1}{\sqrt{\gamma}} (\frac{2\gamma\ln(1/\sqrt{\gamma}) }{1/eD_{1,\lambda,\mu}}) \\
    &\geq \frac{1}{\sqrt{\gamma}} \frac{2\gamma\ln(1/\sqrt{\gamma})}{\bigg( \frac{1}{2D_{1,\lambda,\mu}e}+\sqrt{\frac{1}{4e^2D_{1,\lambda,\mu}^2} - 2\gamma\ln(1/\sqrt{\gamma})}\bigg)} \\
    &= \frac{1}{\sqrt{\gamma}} \bigg( \frac{1}{2eD_{1,\lambda,\mu}}-\sqrt{\frac{1}{4e^2D_{1,\lambda,\mu}^2} - 2\gamma\ln(1/\sqrt{\gamma})}\bigg)
\end{align*}
Note that from $\gamma\leq\gamma_0$ \eqref{eq: gamma assumption for SSQ W-p}, we know that $\gamma\ln(1/\sqrt{\gamma}) \leq 1/(8e^2D_{1,\lambda,\mu}^2)$ and thus the above square root are real. From the verification of the square root bounds, we have proved inequality \eqref{eq: quadratic inequality for near constant deviation, SSQ} and thus verified that $\rho \in [0,1]$. With this pair of $(p, \rho)$, we can now apply Lemma \ref{lemma: tail bound via wasserstein} to derive the tail bound in this regime.
\begin{align*}
    \left|\mathbb{P}(\tilde{q} > a) - \mathbb{P}(Z > a)\right| &\leq eD_{1,\lambda,\mu}\cdot  \sqrt{\gamma} ( (\ln(1/\sqrt{\gamma}) +a^2/2)) \phi(\rho a) + \sqrt{\gamma} e^{-a^2/2} \\
    & = \frac{eD_{1,\lambda,\mu}}{\sqrt{2\pi}}\cdot \sqrt{\gamma} ( (\ln(1/\sqrt{\gamma}) +a^2/2))\\
    &\quad e^{-a^2/2 (1 - D_{1,\lambda,\mu}\sqrt{\gamma}(\ln(1/\sqrt{\gamma}) + a^2/2)/a)}  + \sqrt{\gamma} e^{-a^2/2} 
\end{align*}

\paragraph{Proof for Upper bound in Theorem \ref{thm: SSQ Tail bound, third regime}}
In this regime, we consider deviation $a$ in the following range depending on $\gamma$,
\begin{align}
    \underbrace{\sqrt{2}e D_{2,\lambda,\mu}}_{D_{Tail, 3, l}} \frac{1}{\gamma^{1/2-\alpha} }\leq  a \leq \underbrace{\sqrt{2}e D_{2,\lambda,\mu} }_{D_{Tail, 3, u}}\frac{1}{\sqrt{\gamma} }, \label{eq: sub-Gaussian deviation setup}
\end{align}
with $D_{2,\lambda,\mu}$ defined in \eqref{eq: Wasserstein-$p$ bound, second regime}.

In this regime, we will use the following choice of $p$ and $\rho$,
\begin{align}
    \rho := 1 - e W_p / a, p:=a^2/(2 e^2 D_{2,\lambda,\mu}^2). \label{eq: p and rho for sub-Gaussian Wasserstein-$p$ distance}
\end{align}
Again, we need to verify the validity of $p$ and $\rho$. First, we check that $p$ falls into the second regime in \eqref{eq: Wasserstein-$p$ bound, second regime}, i.e., $1-2\alpha \leq \ln(p) /\ln(1/\gamma) \leq 1/2$.
\begin{align*}
    (1/\gamma)^{1 - 2\alpha}  \leq p &= a^2/(2 e^2 D_{2,\lambda,\mu}^2) \leq 1/\sqrt{\gamma} 
\end{align*}
This holds from the regime of deviation $a$ in \eqref{eq: sub-Gaussian deviation setup}. Next, we verify that $\rho \in [0,1]$. Trivially, we have $\rho \leq 1$. For the lower bound of $\rho$, we upper bound Wasserstein-$p$ from \eqref{eq: Wasserstein-$p$ bound, second regime} as follows,
\begin{align*}
    \mathcal{W}_p(\mathcal{L}(\tilde{q}), \mathcal{L}(Z)) &\leq D_{2,\lambda,\mu} \sqrt{p} = D_{2,\lambda,\mu} \sqrt{a^2/(2 e^2 D_{2,\lambda,\mu}^2)} = a/( \sqrt{2} e)
\end{align*}
Then the verification of lower bound of $\rho$ follows directly,
\begin{align*}
    \rho &= 1 - e W_p / a \geq 1 - e D_{2,\lambda,\mu} \sqrt{p}/(a) \geq 1 - 1/\sqrt{2} 
\end{align*}
Therefore, with this combination of $\rho$ and $p$, it can be shown from \eqref{eq: Tail bound from Wasserstein-$p$ distance}:
\begin{align*}
    \left|\mathbb{P}(\tilde{q} > a) - \mathbb{P} (Z > a)\right| &\leq 
    e/2\sqrt{\pi}\cdot a \exp(-\frac{a^2}{2}(1-\sqrt{1/2})^2) + \exp(-\frac{a^2}{2e^2D_{2,\lambda,\mu}})
\end{align*}

\paragraph{Proof for Upper bound in Theorem \ref{thm: SSQ Tail bound, fourth regime}}
In this regime, we consider deviation $a$ in the following range depending on $\gamma$,
\begin{align}
    a \geq  \underbrace{\exp(2eD_{3,\lambda,\mu}\sqrt{\lambda})}_{D_{Tail, 4, l}} \cdot\frac{1}{\sqrt{\gamma}} \label{eq: large deviation setup}
\end{align}
To prove the tail bound in this regime, we again need to pick proper $p$ and $\rho$. Specifically, we choose $p$ and $\rho$ as follows.
\begin{align}
    \rho := 1 - e W_p / a, \quad p:= \frac{\sqrt{\lambda}}{2eD_{3,\lambda,\mu}} \cdot \frac{a}{\sqrt{\gamma}}\log(1+a\sqrt{\gamma}) \label{eq: p and rho for large deviation Wasserstein-$p$ distance}
\end{align}
We verify the validity of this choice of $p$ and $\rho$, i.e., $p$ falls into the third regime in \eqref{eq: Wasserstein-$p$ bound, third regime} and $\rho \in [0,1]$. First, we check that $p$ falls into the third regime in \eqref{eq: Wasserstein-$p$ bound, third regime}, i.e., $p \geq \lambda/\gamma$.
\begin{align*}
    p &= \frac{\sqrt{\lambda}}{2eD_{3,\lambda,\mu}} \cdot \frac{a}{\sqrt{\gamma}}\log(1+a\sqrt{\gamma}) \overset{(a)}{\geq} \frac{1}{\gamma}
\end{align*}
Inequality $(a)$ holds from the regime of deviation $a$ in \eqref{eq: large deviation setup}. For $\rho$, we first upper bound Wasserstein-$p$ from \eqref{eq: Wasserstein-$p$ bound, third regime} as follows,
\begin{align*}
    e\mathcal{W}_p(\mathcal{L}(\tilde{q}), \mathcal{L}(Z)) &\overset{(a)}{\leq} eD_{3,\lambda,\mu} 
    \frac{\sqrt{\gamma}}{\sqrt{\lambda}} p/\ln(1+p\gamma/\lambda) \\
    &= eD_{3,\lambda,\mu} \frac{\sqrt{\gamma}}{\sqrt{\lambda}} \cdot \frac{\frac{\sqrt{\lambda}}{2eD_{3,\lambda,\mu}} \cdot \frac{a}{\sqrt{\gamma}}\log(1+a\sqrt{\gamma})}{\log\left(1+ \frac{\sqrt{\lambda}}{2eD_{3,\lambda,\mu}} \cdot \frac{a\sqrt{\gamma}}{\sqrt{\lambda}}\log(1+a\sqrt{\gamma})\right)} \\
    &= \frac{1}{2}\frac{a\log(1+a\sqrt{\gamma})}{\log\left(1+ \frac{1}{2e D_{3,\lambda,\mu}} \cdot a\sqrt{\gamma} \log(1+a\sqrt{\gamma})\right)} \\
    &\overset{(b)}{\leq} \frac{1}{2}\cdot \frac{a \log(1+a\sqrt{\gamma})}{\log(1+ a\sqrt{\gamma})} = 1/2\cdot a 
\end{align*}
Inequality $(a)$ is from the dominating term for Wasserstein-$p$ \eqref{eq: Wasserstein-$p$ bound, third regime}. Inequality $(b)$ holds from the assumption of this regime \eqref{eq: large deviation setup}, which ensures that $\frac{1}{2e D_{3,\lambda,\mu}} \cdot a\sqrt{\gamma} \log(1+a\sqrt{\gamma}) \geq 1$. thus, we have verified that $\rho \in [0,1]$ as follows,
\begin{align*}
    \rho &= 1 - e W_p / a \geq 1 - e D_{3,\lambda,\mu} \frac{\sqrt{\gamma}}{\sqrt{\lambda}} p/(a \ln(1+p\gamma/\lambda)) \geq 1 - 1/2 = 1/2
\end{align*}
It can be shown from this choice of $\rho, p$, with Lemma \ref{lemma: tail bound via wasserstein} that
\begin{align*}
    \left|\mathbb{P}(\tilde{q} > a) - \mathbb{P} (Z > a)\right| &\leq \notag\\
        e/2\sqrt{2\pi}\cdot a \exp(-\frac{a^2}{8}) &+ \exp\left(-\frac{\sqrt{\lambda}}{2eD_{3,\lambda,\mu}} \cdot \frac{a}{\sqrt{\gamma}}\log(1+a\sqrt{\gamma}) \right)
\end{align*}


\subsubsection{Transform Method Argument}
From the above derivation using Wasserstein-$p$ distance, we have obtained efficient tail bounds in Theorem \ref{thm: SSQ Tail bound, first regime}, and \ref{thm: SSQ Tail bound, second regime}. The derivation provides upper bounds for all regimes, but the lower bound is only non-trivial in Theorem \ref{thm: SSQ Tail bound, first regime} and \ref{thm: SSQ Tail bound, second regime}. In this section, we revisit the tail bound using MGF control for $\tilde{q}$ to derive tail bound, which complements the Wasserstein-$p$ distance approach. It provides asymptotically tight upper and lower bound for Theorem \ref{thm: SSQ Tail bound, third regime}, and \ref{thm: SSQ Tail bound, fourth regime}. Thus it provides moderate and large deviation principles (MDP, LDP). Particularly, the moderate and large deviation principles state that the exponent of tail probability converges to a quadratic form and the rate function, i.e., Legendre transform of log-MGF, respectively. The deviation principles from transform method cover the whole range of deviation $a$, whereas the Wasserstein-$p$ distance approach only covers up to $a=o((1/\gamma)^{\min\{1/2-\alpha,1/6\}})$ deviation.

\paragraph{Proof for Theorem \ref{thm: SSQ Tail bound, third regime} and \ref{thm: SSQ Tail bound, fourth regime}}

For simplicity of notation, we denote $a:= a_{\gamma}\sqrt{\lambda}/\sqrt{\gamma}$. Since we are considering tail probability $\mathbb{P}(\tilde{q} > a)$, it is equivalent to consider tail probability for $\hat{q} := \tilde{q} - \frac{\lambda-\mu}{\gamma}$ since 
\begin{align*}
    \mathbb{P}(\tilde{q} > a_\gamma) = \mathbb{P}(\hat{q} > a_\gamma \sqrt{\lambda}/\sqrt{\gamma}) = \mathbb{P}(\hat{q} > a)
\end{align*}
 Recall from \eqref{eq: MGF for q_infty - lambda/mu/gamma, Poisson style} that the MGF for $\hat{q}=q - \frac{\lambda-\mu}{\gamma}$ is controlled as
\begin{align*}
    exp\left(\frac{\lambda}{\gamma}e^\theta - \frac{\lambda}{\gamma} - \theta \frac{\lambda}{\gamma}\right) \leq 
    \mathbb{E}[\exp\{\theta \hat{q}\}] &\leq A_{\lambda,\mu} \cdot \exp\left(\frac{\lambda}{\gamma}e^\theta - \frac{\lambda}{\gamma} - \theta \frac{\lambda}{\gamma}\right), \forall \theta \geq 0
\end{align*}
We can obtain the following pre-limit tail bounds for $\hat{q}$ using Markov inequality for upper bounds.
\begin{align}
    \mathbb{P}(\hat{q} > a) &\leq A_{\lambda,\mu} \cdot\inf_{\theta \geq 0} e^{-\theta a} \exp\left(\frac{\lambda}{\gamma}e^\theta - \frac{\lambda}{\gamma} - \theta \frac{\lambda}{\gamma}\right) \notag\\
        &= A_{\lambda,\mu} \cdot \exp\left(a - (\frac{\lambda}{\gamma} + a )\log\left(1 + \frac{a}{\lambda/\gamma}\right)\right) \label{eq: LDP MGF upper}\\
        &\overset{(a)}{\leq} A_{\lambda,\mu} \cdot \exp\left(-\frac{a^2}{2\lambda/\gamma} + \frac{a^3}{2(\lambda/\gamma)^2}\right) 
        \label{eq: MDP MGF upper}
\end{align}
Inequality $(a)$ is from Taylor expansion for $\log(1+x)$ and the bound $\log(1+x) \geq x - x^2/2$, for $x>0$.

For lower bound, we again use Markov inequality but with change of measure technique. Specifically, we tilt the original measure $\mathcal{L}(\hat{q})$ to a new measure $Q$ defined as
\begin{align*}
    \frac{Q\{x\}}{\mathcal{L}(\hat{q})\{x\}} = \exp(\theta^* x ) / \mathbb{E}[\exp\{\theta^* \hat{q}\}], \forall x \in \text{supp}(\hat{q}), \theta^* = \log(1 + \frac{a+\triangle/2}{\lambda/\gamma}), \triangle := 8\sqrt{a + \lambda/\gamma}
\end{align*}
With this tilted distribution $Q$, we have
\begin{align*}
    \mathbb{P}(\hat{q} > a) &\overset{(a)}{\geq} 
        \underbrace{Q(a \leq X \leq a + \triangle)}_{(a_1)}
        \cdot 
        \underbrace{\exp\left\{ \lambda/\gamma \cdot(e^{\theta^*} - 1 - \theta^*) - \theta^* (a + \triangle) \right\}}_{(a_2)}
        \notag
\end{align*}
Inequality $(a)$ holds from restricting the event $\{\hat{q} > a\}$ to $\{a \leq X \leq a + \triangle\}$ and change of measure, recalling that $\triangle := 8\sqrt{a + \lambda/\gamma}>0$.
Now in order to bound the lower bound in \eqref{eq: MDP MGF upper}, we need to bound $(a_1)$ and $(a_2)$ respectively.

To bound $(a_1)$, we use Chebyshev's inequality since the tilted distribution $Q$ is centered around $a$. We first check the mean and variance of $Q$, denoting them as $\mathbb{E}_Q[X]$ and $\text{Var}_Q[X]$ respectively.
From similar derivation of \eqref{eq: upper bound on mu_t}, \eqref{eq: lower bound on mu_t}, \eqref{eq: upper bound on niu_t}, it can be shown that the mean and variance of $Q$ is bounded as 
\begin{align*}
    \mathbb{E}_Q[X] &= \frac{d}{d\theta} \log \mathbb{E}[\exp\{\theta \hat{q}\}]\mid_{\theta = \theta^*} \in [a + \triangle/2, a + \triangle/2+ A_{\lambda,\mu}] \\
    \text{Var}_Q[X] &= \frac{d^2}{d\theta^2} \log \mathbb{E}[\exp\{\theta \hat{q}\}]\mid_{\theta = \theta^*} \leq \lambda/\gamma + a+ \triangle/2
\end{align*}
With the above bounds, we can bound $(a_1)$ as follows.
\begin{align*}
    (a_1) = Q(a \leq X \leq a + \triangle) &\overset{(a)}{\geq} 1 - \frac{\text{Var}_Q[X]}{(\triangle/2 - A_{\lambda,\mu})^2}\\
     &\overset{(b)}{\geq} 1 - \frac{\lambda/\gamma + a+ \triangle/2}{(\triangle/2 - A_{\lambda,\mu})^2} \\
     &= 1 - \frac{\lambda/\gamma + a+ \triangle/2}{(4\sqrt{a + \lambda/\gamma} - A_{\lambda,\mu})^2} \\
     &\overset{(c)}{\geq} 1 - \frac{\lambda/\gamma + a+ \triangle/2}{(3\sqrt{a + \lambda/\gamma})^2} \\
     &= 1 - \frac{\lambda/\gamma + a+ 4\sqrt{a + \lambda/\gamma}}{9(a + \lambda/\gamma)} \\
     &\overset{(d)}{\geq} \frac{2}{3}
\end{align*}
Inequality $(a)$ is from Chebyshev's inequality. Inequality $(b)$ is from the bound on variance of $Q$.
 Inequality $(c)$ holds from the assumption that $ \gamma \leq \lambda/(A_{\lambda,\mu}^2)$. Thus we have $A_{\lambda,\mu} \leq \sqrt{\lambda/\gamma +a}$. Inequality $(d)$ holds from the fact that $\gamma\leq \mu \leq \lambda$.

To bound $(a_2)$, we use Taylor expansion for logarithm functions.
\begin{align}
    (a_2) &= \exp\left\{ -(a+ \triangle+ \lambda/\gamma)\log(1+\frac{a+\triangle/2}{\lambda/\gamma}) + a + \triangle/2 \right\} \label{eq: Change of measure Exponential MDP SSQ} \\
    &= \exp\left\{ -\triangle \ln (1+\frac{a+\triangle/2}{\lambda/\gamma}) + a+ \triangle/2 - (a+\lambda/\gamma)\log(1+\frac{a+\triangle/2}{\lambda/\gamma}) \right\} \notag\\
    &\overset{(a)}{\geq} \exp\left\{ \triangle \ln (1+\frac{a+\triangle/2}{\lambda/\gamma}) + \frac{(a+\triangle/2)^2}{2\lambda/\gamma} - \frac{(a+\triangle/2)^3}{3(\lambda/\gamma)^2} - \frac{a(a+\triangle/2)}{\lambda/\gamma} \right\} \notag \\
    &= \exp\left\{ \triangle \ln (1+\frac{a+\triangle/2}{\lambda/\gamma}) + \frac{\triangle^2}{8\lambda/\gamma} - \frac{a^2}{2\lambda/\gamma} - \frac{(a+\triangle/2)^3}{3(\lambda/\gamma)^2} \right\} \notag
\end{align}
Inequality $(a)$ is from Taylor expansion for $\log(1+x)$ and the bound $\log(1+x) \leq x - x^2/2 + x^3/3$, $\log(1+x) \geq x$, for $x>0$. Combined with the bound on $(a_1)$, we have lower bound as
\begin{align}
    \mathbb{P}(\hat{q} > a) &\geq \frac{2}{3} \exp\left\{ \triangle \ln (1+\frac{a+\triangle/2}{\lambda/\gamma}) + \frac{\triangle^2}{8\lambda/\gamma} - \frac{a^2}{2\lambda/\gamma} - \frac{(a+\triangle/2)^3}{3(\lambda/\gamma)^2} \right\}\notag\\
    &\geq \frac{2}{3} \exp\left\{ -\frac{a^2}{2\lambda/\gamma} - \frac{(a+\triangle/2)^3}{3(\lambda/\gamma)^2} + 8\sqrt{a + \lambda/\gamma} \ln (1+\frac{a+4\sqrt{a + \lambda/\gamma}}{\lambda/\gamma}) + \frac{8(a + \lambda/\gamma)}{\lambda/\gamma} \right\}  \label{eq: MDP MGF lower}\\
    \mathbb{P}(\hat{q} > a)&\leq A_{\lambda,\mu} \cdot \exp\left(-\frac{a^2}{2\lambda/\gamma} + \frac{a^3}{2(\lambda/\gamma)^2}\right)\notag
\end{align}
where we recall the upper bound from \eqref{eq: MDP MGF upper} in the last line. This two sided bound completes the proof for Theorem \ref{thm: SSQ Tail bound, third regime}. In the main content, we also include the lower bound \eqref{eq: MDP MGF lower} in Theorem \ref{thm: SSQ Tail bound, second regime} for completeness.

Meanwhile, we can also derive a Poisson-style tail bound from \eqref{eq: LDP MGF upper} and \eqref{eq: Change of measure Exponential MDP SSQ} as follows. For upper bound, we recall \eqref{eq: LDP MGF upper}
\begin{align*}
    \mathbb{P}(\hat{q} > a) &\leq A_{\lambda,\mu} \cdot \exp\left(a - (\frac{\lambda}{\gamma} + a )\log\left(1 + \frac{a}{\lambda/\gamma}\right)\right)
\end{align*}
For lower bound, we recall our constant lower bound for $(a_1)$ as $2/3$ and we use \eqref{eq: Change of measure Exponential MDP SSQ} to lower bound $(a_2)$. The lower bound is then given as
\begin{align*}
    \mathbb{P}(\hat{q} > a) &\geq \frac{2}{3}\exp\left\{ -(a+ \triangle+ \lambda/\gamma)\log(1+\frac{a+\triangle/2}{\lambda/\gamma}) + a + \triangle/2 \right\}
\end{align*}
Together, the upper and lower bound provide us with the right hand side in Theorem \ref{thm: SSQ Tail bound, fourth regime}. 
Note that in the main content, for Theorem \ref{thm: SSQ Tail bound, third regime} and \ref{thm: SSQ Tail bound, fourth regime}, there are two versions of upper bounds, combined together by taking the minimum. One version is from transform method as above, and the other version is from concentration argument using Wasserstein-$p$ distance, i.e., the third and fourth paragraph in Section \ref{subsubsec: SSQ Tail Bound via Concentration}.

\paragraph{Proof for Corollary \ref{cor: Deviation Principle for M/M/1+M}: Deviation Principles}

Starting from the tail bounds in formulas \eqref{eq: MDP MGF upper} and \eqref{eq: MDP MGF lower}, we can derive moderate deviaiton principle for SSQ, i.e., the second limit result in Corollary \ref{cor: Deviation Principle for M/M/1+M}. Specifically, we focus on the moderate deviation setting as follows,
\begin{align}
    a_\gamma = 1/\gamma^\delta, \quad 0<\delta<1/2, \quad, a:= a_{\gamma}\sqrt{\lambda}/\sqrt{\gamma} \label{eq: MDP setting}
\end{align}

First we have an order analysis for the exponent terms in \eqref{eq: MDP MGF upper} and \eqref{eq: MDP MGF lower}. Recalling that we have the change of variable from $a_\gamma$ to $a$ and from $\tilde{q}$ to $\hat{q}$, and the setup in \eqref{eq: MDP setting}, the following order analysis holds.
\begin{align*}
    \Omega(1)\cdot \exp\left(-\frac{a^2}{2\lambda/\gamma}(1 + o(1))\right)  \overset{(a)}{\leq} \mathbb{P}(\hat{q} > a) &\overset{(b)}{\leq} O(1)\cdot \exp\left(-\frac{a^2}{2\lambda/\gamma}(1 - o(1))\right)
\end{align*}
To see inequality $(a)$ and $(b)$, we analyze the exponent terms in \eqref{eq: MDP MGF upper} and \eqref{eq: MDP MGF lower} respectively.
\begin{align*}
    \ln \mathbb{P}(\hat{q} > a) &\geq -\frac{a^2}{2\lambda/\gamma} - \frac{(a+\triangle/2)^3}{3(\lambda/\gamma)^2} + 8\sqrt{a + \lambda/\gamma} \ln (1+\frac{a+4\sqrt{a + \lambda/\gamma}}{\lambda/\gamma}) + \frac{8(a + \lambda/\gamma)}{\lambda/\gamma} + \ln(2/3) \\
    &\overset{(a)}{\geq} -\frac{a^2}{2\lambda/\gamma} - o(1) \cdot \frac{a^2}{2\lambda/\gamma}  + 8\sqrt{a + \lambda/\gamma} \ln (1+\frac{a+4\sqrt{a + \lambda/\gamma}}{\lambda/\gamma}) + \frac{8(a + \lambda/\gamma)}{\lambda/\gamma} + \ln(2/3) \\
    &\overset{(b)}{\geq} -\frac{a^2}{2\lambda/\gamma} - o(1) \cdot \frac{a^2}{2\lambda/\gamma}  +\ln(2/3)\\
    &\overset{(c)}{\geq} -\frac{a^2}{2\lambda/\gamma} - o(1) \cdot \frac{a^2}{2\lambda/\gamma}  + o(1)\cdot \frac{a^2}{2\lambda/\gamma}  \\
    \ln \mathbb{P}(\hat{q} > a) &\leq -\frac{a^2}{2\lambda/\gamma} + \frac{a^3}{2(\lambda/\gamma)^2} + \ln(A_{\lambda,\mu})\\
    &\overset{(d)}{\leq} \frac{a^2}{2\lambda/\gamma}(-1 + o(1))
\end{align*}
Inequality $(a)$ holds from the order analysis of $\frac{(a+\triangle/2)^3}{3(\lambda/\gamma)^2} = o(a^2/(\lambda/\gamma))$ in moderate deviation regime, i.e., $a_\gamma = 1/\gamma^\delta, 0<\delta<1/2$ and thus $a = 1/\gamma^{\delta'}, 1/2<\delta'<1$. Inequality $(b)$ holds from the positivity of term $8\sqrt{a + \lambda/\gamma} \ln (1+\frac{a+4\sqrt{a + \lambda/\gamma}}{\lambda/\gamma}) + \frac{8(a + \lambda/\gamma)}{\lambda/\gamma}$. Inequality $(c)$ holds from the order analysis of $ \ln(3/2) = o(a^2/(\lambda/\gamma))$. Inequality $(d)$ holds from the order analysis of $a^3/(\lambda/\gamma)^2 = o(a^2/(\lambda/\gamma))$ and the constant term $\ln(A_{\lambda,\mu}) = o(a^2/(\lambda/\gamma))$ in moderate deviation regime.

With the above order analysis, upper bound \eqref{eq: MDP MGF upper} and lower bound \eqref{eq: MDP MGF lower} provide asymptotically tight tail bound in moderate deviation regime. We show the exponential term is precisely Gaussian form of $-a_\gamma^2/2$ and prove the second limit theorem in Corollary \ref{cor: Deviation Principle for M/M/1+M}, i.e.,
\begin{align*}
    \lim_{\gamma \to 0} -\frac{2}{a_\gamma^2} \ln \mathbb{P} (\tilde{q} > a_\gamma) = 1, \quad a_\gamma = 1/\gamma^\delta, \forall 0<\delta<1/2.
\end{align*}

In large deviation regime, we study the below deviation setting,
\begin{align}
    a_\gamma = 1/\gamma^\delta, \quad \delta>1/2, a:= a_{\gamma}\sqrt{\lambda}/\sqrt{\gamma} \label{eq: LDP setting}
\end{align}
In this regime, we can derive large deviation principle for SSQ, i.e., the third limit result in Corollary \ref{cor: Deviation Principle for M/M/1+M}. Specifically, we start from the tail bounds in formulas \eqref{eq: LDP MGF upper} and \eqref{eq: Change of measure Exponential MDP SSQ}. With those two-sided bounds, we can analyze the order of exponent terms and achieve asymptotically tight tail bound. Our large deviation principles exhibit a shift from sub-Gaussian exponent to sub-Poisson exponent.

From bounds in \eqref{eq: LDP MGF upper} and \eqref{eq: Change of measure Exponential MDP SSQ}, we analyze the order of exponent term. Particularly, we have    
\begin{align*}
    \ln \mathbb{P}(\hat{q} > a) &\leq a - (\frac{\lambda}{\gamma} + a )\log\left(1 + \frac{a}{\lambda/\gamma}\right) + \ln(A_{\lambda,\mu}) \\
    &= - a\log(1+\frac{a}{\lambda/\gamma}) + (a - \frac{\lambda}{\gamma}\log(1+\frac{a}{\lambda/\gamma})) + \ln(A_{\lambda,\mu}) \\
    &\overset{(a)}{\leq} - a\log(1+\frac{a}{\lambda/\gamma}) + \frac{a^2}{2\lambda/\gamma} + \ln(A_{\lambda,\mu}) \\
    &\overset{(b)}{\leq} a\ln(1+\frac{a}{\lambda/\gamma}) + o(1)\cdot a\ln(1+\frac{a}{\lambda/\gamma}) + \ln(A_{\lambda,\mu}) \\
    &\leq - a\ln(1+\frac{a}{\lambda/\gamma}) + o(1)\cdot a\ln(1+\frac{a}{\lambda/\gamma})\\
    \ln \mathbb{P}(\hat{q} > a) &\geq -(a+ \triangle+ \lambda/\gamma)\log(1+\frac{a+\triangle/2}{\lambda/\gamma}) + a + \triangle/2 + \ln(2/3) \\
    &= - (a+\triangle)\log(1+\frac{a+\triangle/2}{\lambda/\gamma}) + (a+\triangle/2 - \frac{\lambda}{\gamma}\log(1+\frac{a+\triangle/2}{\lambda/\gamma})) + \ln(2/3) \\
    &\overset{(c)}{\geq} -(a+\triangle)\ln(1+\frac{a+\triangle/2}{\lambda/\gamma}) + \ln(2/3)\\
    &\overset{(d)}{\geq} a\ln(1+\frac{a}{\lambda/\gamma}) - o(1)\cdot a\ln(1+\frac{a}{\lambda/\gamma})
\end{align*}
Inequality $(a)$ holds from the bound $\log(1+x) \geq x - x^2/2$ for $x>0$. Inequality $(b)$ holds from the order analysis of $ \frac{a^2}{2\lambda/\gamma} = o(a\ln(1+a/(\lambda/\gamma)))$ in large deviation regime, i.e., $a_\gamma = 1/\gamma^\delta, \delta>1/2$ and thus $a = 1/\gamma^{\delta'}, \delta'>1$. Inequality $(c)$ holds from the basic fact that $\log(1+x)\leq x$ for $x>0$. Inequality $(d)$ holds from the order analysis of $\triangle \log(1+(a+\triangle/2)/(\lambda/\gamma)) = o(a)$ in large deviation regime, same for constant term $\ln(3/2) = o(a\ln(1+a/(\lambda/\gamma)))$.

Now we change variable back to $a_\gamma$ and $\tilde{q}$, and we have the following asymptotically tight tail bound in large deviation regime,
\begin{align*}
    \lim_{\gamma \to 0} -\frac{\sqrt{\gamma}}{\sqrt{\lambda}a_\gamma \log(1+a_\gamma \sqrt{\gamma}/\sqrt{\lambda})} \ln \mathbb{P}(\tilde{q} > a_\gamma) = 1, \quad a_\gamma = 1/\gamma^\delta, \forall \delta>1/2
\end{align*}

The result is analogous to Cramer's theorem \cite{dembo2009large} in CLT, here the Fenchel-Legendre transform of log-MGF provides the tight rate function for large deviation.

\paragraph{Discussion}
In this section, we can build a methodological connection between the bounds for different deviation regimes and those in classical concentration inequalities. Specifically, we can draw the connection in three folds:
\begin{enumerate}
  \item For constant deviation, Wasserstein-$p$ bounds via Stein's method yield non-uniform Berry-Esseen-style rates of convergence for tails. Non-uniform Berry-Esseen bounds here refer to a non-uniform dependency of the error on the deviation $a$. 
  \item For moderate deviations, Wasserstein-$p$ bounds via triangle inequality recover Bernstein-type inequalities.
  \item For large deviations, the transform method produces Cramér's theorem and Bennett-type bounds.
\end{enumerate}

\section{Proof for Load Balancing System: Join-The-Shortest-queue (JSQ)} \label{sec: proof for JSQ}
In this section, we study join-the-shortest-queue (JSQ) policy for heterogeneous server setting. Recall that we have the notation
\begin{align}
    \tilde{q}_\Sigma := \frac{1}{n} \langle n\sqrt{\gamma/\lambda} (\mathbf{q} - \frac{\lambda-\mu}{n\gamma}\mathbf{1}), \mathbf{1} \rangle = \frac{\sqrt{\gamma}}{\sqrt{\lambda}} (\sum_i q_i - \frac{\lambda-\mu}{\gamma}) \overset{\triangle}{=} \frac{\sqrt{\gamma}}{\sqrt{\lambda}} \hat{q}_\Sigma. \label{eq: definition for JSQ}
\end{align}

\noindent Also recall that we have assumption on relationship of $(\lambda, \gamma, \mu)$ satisfying \ref{ass:heavy_overload} with $\mu:=\sum_{i=1}^{n} \mu_i$. And we use $\mathbf{q}_\perp$ to denote the projection of $\mathbf{q}$ onto the subspace orthogonal to $\mathbf{1}$, i.e., $\mathbf{q}_\perp = \mathbf{q} - \bar{q}_\Sigma\mathbf{1}$, where $\bar{q}_\Sigma := 1/n \sum_{i=1}^{n}q_i$. We use $\mathbf{q}$ to denote the steady-state queue length vector under JSQ policy, while we will use $\mathbf{x}$ for deterministic vector in $\mathbb{R}^n$ to distinguish. The projection $x_\perp$ and $x_\parallel$ are defined in the same way as $\mathbf{q}_\perp$ and $\mathbf{q}_\parallel$. First, we begin the proof by showing the stability and existence of expectation for JSQ system. 

\paragraph{Stability and Existence of MGF} \label{par: stability and existence of MGF for JSQ}
We use test function $f(\mathbf{q}) = \sum_{i=1}^{n} q_i^2$ into the generator \eqref{eq: generator for join shortest queue} to show positive recurrence of the Markov process $\mathbf{q}(t)$:
\begin{align*}
    \mathcal{L}_{JSQ} \| \mathbf{q}\|^2 &= \lambda (2\min_i q_i + 1) + \sum_{i=1}^{n} (\mu_i \mathbf{1}_{\{ q_i \neq 0\}} + \gamma q_i)(-2q_i + 1) \\
    &= -2\gamma\sum_{i=1}^{n} q_i^2 + 2\lambda \min_i q_i + \gamma \sum_{i=1}^{n} q_i - \sum_{i=1}^{n} \mu_iq_i + \lambda + \mu - \sum_{i=1}^{n} \mu_i \mathbf{1}_{\{ q_i = 0\}} \\
    &\leq -2\gamma \| \mathbf{q}\|^2 + (2\lambda/n +\gamma)\sum_{i=1}^{n} q_i + \lambda + \mu \\
    &\overset{(a)}{\leq} -\gamma \| \mathbf{q}\|^2 + \frac{n}{4\gamma}(2\lambda/n + \gamma)^2 + \lambda + \mu 
\end{align*}
Here inequality $(a)$ is AMGM. Since the drift is negative outside the compact set, which is defined as $\{ \mathbf{q} \in \mathbb{R}^n:
 \|\mathbf{q}\|^2 \leq (\frac{n}{4\gamma}(2\lambda/n + \gamma)^2 + \lambda + \mu)/\gamma\}$, the Markov process $\mathbf{q}(t)$ is positive recurrent from Foster-Lyapunov theorem \cite{hajek1982hitting}. In the following, we show existence of expectation $\mathbb{E}[\sum_i \exp(\theta q_i)]$ for $\theta\geq 0$ by applying the coupling argument in Section \ref{sec: coupling argument}. We use $q_{JSQ,i}$ to denote the steady-state queue length for $i$-th server under JSQ policy, and $q_{M/M/\infty,i}$ to denote the steady-state queue length for $i$-th server under M/M/$\infty$ queueing system with arrival rate $\lambda$ and service rate $\gamma$. From Section \ref{sec: coupling argument}, we have $q_{JSQ,i} \leq_{st} q_{M/M/\infty,i}$ for each $i$. Here $\leq_{st}$ denotes stochastic dominance. Note that we have explicit form for MGF of $q_{M/M/\infty,i}$ as $\mathbb{E}[\exp(\theta q_{M/M/\infty,i})] = \exp(\lambda/\gamma (e^{\theta} - 1)), \forall \theta \geq 0$.
 Thus we have
\begin{align}
    \sum_{i=1}^{n} \mathbb{E}[\exp(\theta q_{JSQ,i})] \leq \sum_{i=1}^{n} \mathbb{E}[\exp(\theta q_{M/M/\infty,i})] = n \exp(\lambda/\gamma (e^{\theta} - 1)) \vee 1 < \infty, \forall \theta \geq 0 \label{eq: well-defined MGF for JSQ}
\end{align}
Thus $\mathbb{E}[\exp(\theta q_{JSQ,i})] < \infty, \forall i, \theta \in \mathbb{R}$ since if $\theta <0$, $\exp(\theta q_{JSQ,i}) \leq 1$, whereas if $\theta \geq 0$, we have the above bound. This completes the proof for existence of MGF.

Before the proof for theorems under JSQ policy, we first pin down the assumption for $\gamma\leq \gamma_1$ as follows.
\begin{align}
    \gamma \leq \gamma_1&:= \min\{e^{-1}, \mu,\frac{\lambda^2}{C^2} \ln^2(\lambda/\mu), \frac{e^{4\sqrt{2}}\lambda}{25},
        \frac{\mu^2}{(2^2 5^2e^{4\sqrt{2}}\lambda)}, (\xi/2)^{2/\xi},\\
         &\frac{1}{2^{16} e^4 B_{1,\lambda,\mu,n}^4}, \min\{
        \langle\boldsymbol{\phi}, \boldsymbol{1}\rangle^2, \frac{1}{2} (\frac{\langle\boldsymbol{\phi}, \boldsymbol{1}\rangle^4}{eB_{1,\lambda,\mu,n}})^{1/3}\}^{1/\min\{1/6,\xi/2\}} 
        \} \label{eq: gamma assumption for JSQ W-p}
\end{align}
where $\xi = 1/2-\alpha-\epsilon$ for arbitrary $\epsilon \in (0,1/2-\alpha)$. $B_{1,\lambda,\mu,n}$ defined in \eqref{eq: B1, B2 for JSQ tail bound}. 
We will revisit this assumption in the proof for Theorem \ref{thm: Wasserstein-$p$ distance JSQ} and Theorem \ref{thm: JSQ Tail bound}. In the following, we follow the same order as Section \ref{sec: proof for JSQ} and present the proof fo State Space Collapse for JSQ policy, two ingredients for Wasserstein-$p$ distance, Wasserstein-$p$ distance bound and tail bound for JSQ policy sequentially.

\subsection{Proof for Theorem \ref{thm: SSC JSQ}: State Space Collapse}
In this section, we prove Theorem \ref{thm: SSC JSQ} for state space collapse of JSQ policy. Recall that we use $\mathbf{q}_\parallel$ to denote the projection of queue length vector $\mathbf{q}$ onto the subspace spanned by $\boldsymbol{1}$, and $\mathbf{q}_\perp$ to denote the projection onto the orthogonal subspace. We further recall the definition of $\tilde{q}_\Sigma$ in \eqref{eq: definition for JSQ} and the properties that the parallel component $\mathbf{q}_\parallel$ is unique and has closed form $\mathbf{q}_\perp = \frac{1}{n}(\sum_{i=1}^n q_i)\boldsymbol{1}$. Meanwhile, pythagoras identity gives us closed form expression for $\|\mathbf{q}_\perp\|$ as discrepancy among different queues $\|\mathbf{q}_\perp\|^2 = \|\mathbf{q}\|^2 - \|\mathbf{q}_\parallel\|^2 = \sum_{i=1}^n (q_i - \frac{1}{n}\sum_{j=1}^n q_j)^2$.

We first derive the p-th moment bound on $\gamma \langle \mathbf{q}, \boldsymbol{1} \rangle$ for explicit dependency on $p$. This $p$-th moment bound is an ingredient for later proof of SSC when we are using Lyapunov drift argument.
For $i$-th coordinate of queue length vector $\mathbf{q}$, we coupled the queue length process with the process $q_i'$ where the dispatcher no longer follows JSQ policy, but deterministically routes to the $i$-th queue. In this case, the queue length $q_i'$ is a single server queue (SSQ) with arrival rate $\lambda$, service rate $\mu_i$ and abandonment rate $\gamma$. This SSQ stochastically dominates $q_i$ under JSQ policy, i.e., $q_i \leq_{st} q_i'$. Therefore, in terms of MGF, we set in the below $\theta := 1$,
\begin{align*}
    \mathbb{E}\left[ e^{\theta \gamma q_i} \right] &\leq \mathbb{E}\left[ e^{\theta \gamma q_i'} \right] \overset{(a)}{\leq} A_{\lambda,\mu} \exp\left( \frac{\lambda}{\gamma} e^{\gamma} - \frac{\lambda}{\gamma} - \mu_i \right) \overset{(b)}{\leq} A_{\lambda,\mu} \exp ( 2\lambda - \mu_i)
\end{align*}
Inequality $(a)$ holds from the bound on MGF with positive $\theta$ for SSQ \eqref{eq: MGF for q_infty - lambda/mu/gamma, Poisson style}. Note that the bound in \eqref{eq: MGF for q_infty - lambda/mu/gamma, Poisson style} is for $\hat{q}$, we move the centered term $\frac{\lambda -\mu}{\gamma}$ to the right hand side and obtain the bound for $q_i'$. 
 Inequality $(b)$ holds from the fact $\exp(\gamma) \leq 2\gamma +1$ and the assumption \eqref{eq: gamma assumption for JSQ W-p} that $\gamma \leq 1$.

Then from the fact that polynomial functions are upper bounded by exponential functions, we have
\begin{align}
    \mathbb{E}\left[ \gamma^p \langle \mathbf{q}, \boldsymbol{1} \rangle^p \right] &\leq n^p \Gamma(p+1) \cdot \mathbb{E}\left[ e^{\frac{1}{n}\theta \gamma \langle \mathbf{q},\boldsymbol{1} \rangle} \right] \notag\\
    & \leq n^p \Gamma(p+1) \cdot \Pi_{i=1}^{n} \mathbb{E}\left[ \left(e^{\theta \gamma q_i'}\right) \right]^{1/n} \notag\\
    &\leq n^p \Gamma(p+1) \cdot A_{\lambda,\mu} \exp \left( 2\lambda - \frac{1}{n}\mu \right), \label{eq: p-moment for gamma queue length}
\end{align}
which exhibits the independence of $\gamma$ in right hand side of the upper bound. The dependency on $p$ for the upper bound is through the term $\Gamma(p+1)$, which is of order $p^p$.


Next, we show a two-stage control on drift for $\|\mathbf{q}_\perp\|$, i.e., $\mathcal{L}_{JSQ} \|\mathbf{q}_\perp\|$, via the following boundedness of a single jump and drift analysis on $\|\mathbf{q}_\perp\|^2$. Recall that $\|\cdot\|$ denotes the Euclidean norm for vector in $\mathbb{R}^n$, unlike $\|\cdot\|_{L^p}$ which denotes the $L^p$ norm for random variable. This two-stage control of $\mathcal{L}_{JSQ} \|\mathbf{q}_\perp\|$ will also be used in the proof for SSC late when we apply Lyapunov drift method. We define $f(\mathbf{x}):=\|\mathbf{x}_\perp\| = \|\mathbf{x} - 1/n \langle \mathbf{x}, \boldsymbol{1}\rangle \|$, we first get a uniform bound on the difference $\|\mathbf{x}'_\perp\| - \|\mathbf{x}_\perp\|$, for every pair $ (\mathbf{x}, \mathbf{x}')$ such that $\mathbf{Q}_{JSQ}(\mathbf{x}, \mathbf{x}') > 0$. Recall that $\mathbf{Q}_{JSQ}(\mathbf{x}, \mathbf{x}')$ is the transition rate from state $\mathbf{x}$ to $\mathbf{x}'$ under JSQ policy defined in \eqref{eq: generator for join shortest queue}. We have
\begin{align}
    \big| \|\mathbf{x}'_\perp\| - \|\mathbf{x}_\perp\|\big| &= \frac{\big| \|\mathbf{x}'_\perp\|^2 - \|\mathbf{x}_\perp\|^2 \big|}{\big| \|\mathbf{x}'_\perp\| + \|\mathbf{x}_\perp\| \big|}\notag\\
     &\overset{(a)}{\leq} \frac{2|x_i - 1/n\sum_i x_i| +(1-1/n)^2 }{\big| \|\mathbf{x}'_\perp\| + \|\mathbf{x}_\perp\| \big|} \notag\\
     &= \frac{2|x_i - 1/n\sum_i x_i| }{\big| \|\mathbf{x}'_\perp\| + \|\mathbf{x}_\perp\| \big|} + \frac{(1-1/n)^2}{\big| \|\mathbf{x}'_\perp\| + \|\mathbf{x}_\perp\| \big|} \notag\\
     &\overset{(b)}{\leq} \frac{2|x_i - 1/n\sum_i x_i| }{\sqrt{\sum_i (x_i - 1/n \sum_i x_i)^2}} + \frac{(1-1/n)^2}{\big| \|\mathbf{x}'_\perp\| + \|\mathbf{x}_\perp\| \big|} \notag\\
    &\overset{(c)}{\leq} \frac{2|x_i - 1/n\sum_i x_i| }{\sqrt{\sum_i (x_i - 1/n \sum_i x_i)^2}} + \frac{(1-1/n)^2}{1-1/n}\notag \\
     &\leq 2 +1 = 3,
      \label{eq: first order difference constant control}
\end{align}
where $(a)$ holds from the fact that $\|\mathbf{x}_\perp\|^2 = \|\mathbf{x}\|^2 - \| 1/n \langle \mathbf{x}, \boldsymbol{1}\rangle\|^2 = \sum_i (x_i - 1/n \sum_i x_i)^2$, and the fact that $(\mathbf{x}, \mathbf{x}'): \mathbf{Q}_{JSQ}(\mathbf{x}, \mathbf{x}') > 0$ have difference 1 in exactly one coordinate. Inequality $(b)$ holds from basic bound $\|\mathbf{x}'_\perp\| + \|\mathbf{x}_\perp\| \geq \|\mathbf{x}_\perp\|$, and the closed form of $\|\mathbf{x}_\perp\|$. Inequality
 $(c)$ holds from the following lower bound on $\big| \|\mathbf{x}'_\perp\| + \|\mathbf{x}_\perp\| \big|$ as a property of projection. Without loss of generality, we assume $ \mathbf{x}' = \mathbf{x} \pm \mathbf{e}_i$ for a fix $i$, then we have
\begin{align*}
    \big| \|\mathbf{x}'_\perp\| + \|\mathbf{x}_\perp\| \big| &= \sqrt{\sum_{j=1}^{n} (x_j \pm \mathbf{1}_{\{j=i\}} - 1/n(\sum_i x_i \pm 1))^2} + \sqrt{\sum_{j=1}^{n} (x_j - 1/n\sum_i x_i)^2}\\
    &\geq |x_i - 1/n\sum_i x_i| + |x_i - 1/n\sum_i x_i \pm (1-1/n)| \\
    &\geq 1-1/n
\end{align*}
Now, with bounded single jump \eqref{eq: first order difference constant control}, we can move to the drift analysis on $\|\mathbf{q}_\perp\|$, which provides us the two-stage control on $\mathcal{L}_{JSQ} \|\mathbf{q}_\perp\|$.
Specifically, we apply infinitesimal generator $\mathcal{L}_{JSQ}$ to $f(\mathbf{q}) = \|\mathbf{q}_\perp\|^2$, while we denote $q_* := \min_{i} q_i$ in the following. Recall that $\bar{q}$ is the average queue length $\bar{q} = 1/n \sum_{i=1}^{n} q_i$. We have
\begin{align}
    \mathcal{L}_{JSQ} \|\mathbf{q}_\perp\|^2 &= \lambda (2 q_* - 2\bar{q} +1 - 1/n) +\sum_{i=1}^{n}(\mu_i \boldsymbol{1}_{\{ q_i \neq 0\}} +\gamma q_i) (-2q_i + 2\bar{q} +1-1/n) \notag\\
    &= \lambda (2 q_* - 2\bar{q} +1 - 1/n) -2\gamma \sum_{i=1}^{n}(q_i -\bar{q})^2 \notag\\
    & \quad + (1-1/n) \gamma\sum_{i=1}^{n}q_i - \bar{q}\sum_{i=1}^{n} \mu_i \boldsymbol{1}_{\{ q_i =0\}} + (1-1/n)\sum_{i=1}^{n} \mu_i \boldsymbol{1}_{\{ q_i \neq 0\}} \notag\\
    &= \lambda(1-1/n) + (1-1/n)\sum_{i=1}^{n} \mu_i\boldsymbol{1}_{\{ q_i \neq 0\}} + \gamma \sum_{i=1}^{n} q_i +2\lambda (q_* - \bar{q}) - 2\gamma \sum_{i=1}^{n}(q_i -\bar{q})^2 - \bar{q}\sum_{i=1}^{n} \mu_i \boldsymbol{1}_{\{ q_i =0\}} \notag\\
    &\overset{(a)}{\leq} (\lambda + \mu) + \gamma \sum_{i=1}^{n} q_i +2\lambda (q_* - \bar{q}) \notag\\
    &\overset{(b)}{\leq} (\lambda + \mu) + \gamma \sum_{i=1}^{n} q_i -\frac{2\lambda }{\sqrt{n(n-1)}} \|\mathbf{q}_\perp\| \label{eq: drift analysis on perpendicular, secon moment}
\end{align}
Inequality $(a)$ holds from the positivity of terms in last line. Inequality $(b)$ holds from the following lower bound on $\bar{q} - q_*$ as a property of $q_*$ is the minimum coordinate of $\mathbf{q}$:
\begin{align*}
    \bar{q} - q_* & = 1/n \sum_{i=1}^{n} (q_i - q_*) \geq \frac{1}{\sqrt{n(n-1)}} \|\mathbf{q}_\perp\|
\end{align*}
To see this, we first use change of variables  $s_i:=q_i - q_*$ for simplicity, then we have $\bar{q} - q_* = 1/n \sum_{i=1}^{n} s_i$, and
\begin{align*}
    \sum_{i=1}^{n} (q_i - \bar{q})^2 &= \sum_{i=1}^{n} (s_i - 1/n \sum_{j=1}^{n} s_j)^2 = \sum_{i=1}^{n} s_i^2 - \frac{1}{n}\left(\sum_{i=1}^{n} s_i\right)^2 \leq \frac{n-1}{n} (\sum_{i=1}^{n} s_i)^2 \\
     \bar{q} - q_* &= 1/n \sum_{i=1}^{n} s_i \geq \frac{1}{\sqrt{n(n-1)}} \sqrt{\sum_{i=1}^{n} (q_i -\bar{q})^2}= \frac{1}{\sqrt{n(n-1)}} \|\mathbf{q}_\perp\| 
\end{align*}

Now, the control on $\mathcal{L}_{JSQ} \|\mathbf{q}_\perp\|^2$ gives us two stage control on $\mathcal{L}_{JSQ} \|\mathbf{q}_\perp\|$ as follows.
\begin{align}
    \mathcal{L}_{JSQ} \|\mathbf{q}_\perp\| &= \sum_{\mathbf{q}'\in \mathbb{Z}^n, \mathbf{q}' \neq \mathbf{q}} \mathbf{Q}_{JSQ}(\mathbf{q}, \mathbf{q}') (\|\mathbf{q}'_\perp\| - \|\mathbf{q}_\perp\|) \label{eq: definition of gen first order diff}\\ 
    &\leq \sum_{\mathbf{q}'\in \mathbb{Z}^n, \mathbf{q}' \neq \mathbf{q}} \mathbf{Q}_{JSQ}(\mathbf{q}, \mathbf{q}') \frac{\|\mathbf{q}'_\perp\|^2 - \|\mathbf{q}_\perp\|^2}{2\|\mathbf{q}_\perp\|} = \frac{\mathcal{L}_{JSQ} \|\mathbf{q}_\perp\|^2}{2\|\mathbf{q}_\perp\|} \notag\\
    &\overset{(a)}{\leq} \begin{cases}
         - \underbrace{\frac{\lambda}{\sqrt{n(n-1)}}}_{:=\zeta >0}, & \text{if } \|\mathbf{q}_\perp\| > L(\mathbf{q}) \label{eq: gen first order diff two control}\\
        3(\lambda + \mu +\gamma \sum_{i=1}^{n} q_i) , & \text{if } \|\mathbf{q}_\perp\| \leq L(\mathbf{q}) 
    \end{cases}
\end{align}
with the definition of $L(\mathbf{q}) := \frac{2(n-1)/n[\gamma\sum_{i=1}^{n}q_i + \mu +\lambda )]}{\lambda /\sqrt{n(n-1)}}$. To see inequality $(a)$, we consider the two cases correspondingly. The first case is when $\|\mathbf{q}_\perp\| > L(\mathbf{q})$, we have from \eqref{eq: drift analysis on perpendicular, secon moment} that
\begin{align*}
    \mathcal{L}_{JSQ} \|\mathbf{q}_\perp\| &\leq \frac{(\lambda + \mu) + \gamma \sum_{i=1}^{n} q_i -\frac{2\lambda }{\sqrt{n(n-1)}} \|\mathbf{q}_\perp\|}{2\|\mathbf{q}_\perp\|} \leq -\frac{\lambda}{\sqrt{n(n-1)}}
\end{align*}
The second case is when $\|\mathbf{q}_\perp\| \leq L(\mathbf{q})$, we have from \eqref{eq: definition of gen first order diff} and \eqref{eq: first order difference constant control} that
\begin{align*}
    \mathcal{L}_{JSQ} \|\mathbf{q}_\perp\| &\leq \sum_{\mathbf{q}'\in \mathbb{Z}^n, \mathbf{q}' \neq \mathbf{q}} \mathbf{Q}_{JSQ}(\mathbf{q}, \mathbf{q}') \cdot 3 = 3(\lambda + \mu +\gamma \sum_{i=1}^{n} q_i)
\end{align*}
Note that the two-stage control \eqref{eq: gen first order diff two control} identifies the negative drift region and the positive drift region for $\mathcal{L}_{JSQ} \|\mathbf{q}_\perp\|$. However, unlike classical Lyapunov drift method, we do not obtain a uniform negative drift outside a compact set of $\mathbf{q}_\perp$, since the negative drift region depends not only on $\|\mathbf{q}_\perp\|$ but also on the state $\mathbf{q}$ via $L(\mathbf{q})$.

With the two-stage control on $\mathcal{L}_{JSQ} \|\mathbf{q}_\perp\|$ in \eqref{eq: gen first order diff two control} and the $p$-th moment bound on $\gamma \langle \mathbf{q}, \boldsymbol{1} \rangle$ in \eqref{eq: p-moment for gamma queue length}, we are ready to prove Theorem \ref{thm: SSC JSQ} for state space collapse of JSQ policy. We aim to bound the $p$-th moment of $\|\mathbf{q}_\perp\|$ for arbitrary $p\geq 1$. We choose Lyapunov function $f(\mathbf{q}) = \|\mathbf{q}_\perp\|^{p+1}$ to obtain this $p$-th moment bound on $\|\mathbf{q}_\perp\|$ via polynomial root. 
With JSQ generator applied to $f(\mathbf{q}) = \|\mathbf{q}_\perp\|^{p+1}$, we have the following bound
\begin{align}
    \mathcal{L}_{JSQ} f(\mathbf{q}) &= \sum_{\mathbf{q}'\in \mathbb{Z}^n, \mathbf{q}' \neq \mathbf{q}} \mathbf{Q}_{JSQ}(\mathbf{q}, \mathbf{q}') (\|\mathbf{q}'_\perp\|^{p+1} - \|\mathbf{q}_\perp\|^{p+1}) \notag\\
    &\overset{(a)}{\leq} \sum_{\mathbf{q}'\in \mathbb{Z}^n, \mathbf{q}' \neq \mathbf{q}} \mathbf{Q}_{JSQ}(\mathbf{q}, \mathbf{q}') \bigg( (p+1)\|\mathbf{q}_\perp\|^{p} \cdot (\|\mathbf{q}'_\perp\| - \|\mathbf{q}_\perp\|) + \notag\\
    &\quad \frac{(p+1)p}{2}\max\{\|\mathbf{q}'_\perp\|, \|\mathbf{q}_\perp\|\}^{p-1} (\|\mathbf{q}'_\perp\| - \|\mathbf{q}_\perp\|)^2 \bigg) \notag\\
    &= (p+1)\|\mathbf{q}_\perp\|^{p} \cdot \sum_{\mathbf{q}'\in \mathbb{Z}^n, \mathbf{q}' \neq \mathbf{q}} \mathbf{Q}_{JSQ}(\mathbf{q}, \mathbf{q}') (\|\mathbf{q}'_\perp\| - \|\mathbf{q}_\perp\|) + \notag\\
    &\quad  \sum_{\mathbf{q}'\in \mathbb{Z}^n, \mathbf{q}' \neq \mathbf{q}} \mathbf{Q}_{JSQ}(\mathbf{q}, \mathbf{q}') \bigg( \frac{(p+1)p}{2}\max\{\|\mathbf{q}'_\perp\|, \|\mathbf{q}_\perp\|\}^{p-1} (\|\mathbf{q}'_\perp\| - \|\mathbf{q}_\perp\|)^2 \bigg) \notag\\
    &\overset{(b)}{\leq} -\zeta (p+1)\|\mathbf{q}_\perp\|^{p} + (p+1)(3\lambda + 3\mu + 3\gamma \sum_{i=1}^{n} q_i + \zeta) L^p (\mathbf{q}) +  \notag\\ 
    &\quad  \sum_{\mathbf{q}'\in \mathbb{Z}^n, \mathbf{q}' \neq \mathbf{q}} \mathbf{Q}_{JSQ}(\mathbf{q}, \mathbf{q}') \bigg( \frac{(p+1)p}{2}\max\{\|\mathbf{q}'_\perp\|, \|\mathbf{q}_\perp\|\}^{p-1} (\|\mathbf{q}'_\perp\| - \|\mathbf{q}_\perp\|)^2 \bigg) \notag\\
    &\overset{(c)}{\leq} -\zeta (p+1)\|\mathbf{q}_\perp\|^{p} + (p+1)(3\lambda + 3\mu + 3\gamma \sum_{i=1}^{n} q_i + \zeta) L^p (\mathbf{q}) +  \notag\\
    &\quad 9\frac{(p+1)p}{2} \sum_{\mathbf{q}'\in \mathbb{Z}^n, \mathbf{q}' \neq \mathbf{q}} \mathbf{Q}_{JSQ}(\mathbf{q}, \mathbf{q}') (\|\mathbf{q}'_\perp\|^{p-1} + \|\mathbf{q}_\perp\|^{p-1}) \label{eq: gen Taylor expansion for L-p norm}
\end{align}
Inequality $(a)$ holds from Taylor expansion with Lagrange remainder. Inequality $(b)$ holds from the two stage control \eqref{eq: gen first order diff two control}. Specifically, if $\|\mathbf{q}_\perp\| > L(\mathbf{q})$, then we have negative drift $-\zeta$ for the first term in the second last line; else if $\|\mathbf{q}_\perp\| \leq L(\mathbf{q})$, then we upper bound the first term in the second last line by $(p+1)(3\lambda + 3\mu + 3\gamma \sum_{i=1}^{n} q_i) L^p (\mathbf{q})$. Note that we add a $\zeta$ in the right hand side of $(b)$ to use $\|\mathbf{q}_\perp\| \leq L(\mathbf{q})$ in the second case. Inequality $(c)$ holds from the bounded single jump \eqref{eq: first order difference constant control}, i.e., $\big| \|\mathbf{q}'_\perp\| - \|\mathbf{q}_\perp\| \big| \leq 3$ for every pair $(\mathbf{q}, \mathbf{q}')$ such that $\mathbf{Q}_{JSQ}(\mathbf{q}, \mathbf{q}') >0$.

In steady state, because of the well-definedness of MGF for JSQ policy \eqref{par: stability and existence of MGF for JSQ}
, we can take expectation on both sides of \eqref{eq: gen Taylor expansion for L-p norm} and set the left hand side to be zero. Then, we have
\begin{align*}
    0 &= \mathbb{E}\left[ \mathcal{L}_{JSQ} f(\mathbf{q}) \right] \\
        &\leq -\zeta (p+1)\mathbb{E}\left[ \|\mathbf{q}_\perp\|^{p} \right] + (p+1)\mathbb{E}\left[(3\lambda + 3\mu + 3\gamma  \sum_{i=1}^{n} q_i  + \zeta)  L^p (\mathbf{q}) \right] +  \\
        &\quad \mathbb{E}\left[ 9\frac{(p+1)p}{2} \sum_{\mathbf{q}'\in \mathbb{Z}^n, \mathbf{q}' \neq \mathbf{q}} \mathbf{Q}_{JSQ}(\mathbf{q}, \mathbf{q}') (\|\mathbf{q}'_\perp\|^{p-1} + \|\mathbf{q}_\perp\|^{p-1}) \right] \\
        &\overset{(a)}{=}  -\zeta (p+1)\mathbb{E}\left[ \|\mathbf{q}_\perp\|^{p} \right] + (p+1)\mathbb{E}\left[(3\lambda + 3\mu + 3\gamma  \sum_{i=1}^{n} q_i  + \zeta)  L^p (\mathbf{q}) \right] +  \\
        &\quad + 9(p+1)p \mathbb{E}\left[ \|\mathbf{q}_\perp\|^{p-1}  \cdot \sum_{\mathbf{q}'\in \mathbb{Z}^n, \mathbf{q}' \neq \mathbf{q}} \mathbf{Q}_{JSQ}(\mathbf{q}, \mathbf{q}')\right] \\
        &\overset{(b)}{=} -\zeta (p+1)\mathbb{E}\left[ \|\mathbf{q}_\perp\|^{p} \right] + (p+1) (2\sqrt{n(n-1)})^p \mathbb{E}[(3\mu+3\lambda+3\gamma \sum_{i=1}^{n} q_i + \zeta)(\mu+\lambda+\gamma \sum_{i=1}^{n} q_i)^p] \\
        &\quad + 9(p+1)p \mathbb{E}\left[ \|\mathbf{q}_\perp\|^{p-1}  \cdot \sum_{\mathbf{q}'\in \mathbb{Z}^n, \mathbf{q}' \neq \mathbf{q}} \mathbf{Q}_{JSQ}(\mathbf{q}, \mathbf{q}')\right] \\
        &\overset{(c)}{\leq }
        -\underbrace{\zeta (p+1)}_{a_{p}} \mathbb{E}\left[ \|\mathbf{q}_\perp\|^{p} \right] 
        + \underbrace{(p+1) (2\sqrt{n(n-1)})^p \mathbb{E}\left[(3\mu+3\lambda+3\gamma \sum_{i=1}^{n} q_i + \zeta)(\mu+\lambda+\gamma \sum_{i=1}^{n} q_i)^p\right]}_{a_0} \\
        &\quad + \underbrace{9(p+1)p \mathbb{E}\left[ (\sum_{\mathbf{q}'\in \mathbb{Z}^n, \mathbf{q}' \neq \mathbf{q}} \mathbf{Q}_{JSQ}(\mathbf{q}, \mathbf{q}'))^p\right]^{1/p}}_{a_{p-1}} \mathbb{E}[ \|\mathbf{q}_\perp\|^{p}]^{(p-1)/p}
\end{align*}
Equality $(a)$ holds from the zero drift identity on $\|\mathbf{q}_\perp\|^{p-1}$, i.e.,
\begin{align*}
    0=\mathbb{E}[\mathcal{L}_{JSQ}\|\mathbf{q}_\perp\|^{p-1}]= \mathbb{E}[\sum_{\mathbf{q}'\in \mathbb{Z}^n, \mathbf{q}' \neq \mathbf{q}} \mathbf{Q}_{JSQ}(\mathbf{q}, \mathbf{q}')(\|\mathbf{q}'_\perp\|^{p-1} - \|\mathbf{q}_\perp\|^{p-1})] 
\end{align*}
Equality $(b)$ holds from the definition of $L(\mathbf{q})$. Inequality $(c)$ holds from H\"older's Inequality that $\mathbb{E}[\|\mathbf{q}_\perp|^{p-1}] \leq \mathbb{E}[|\mathbf{q}_\perp|^{p}]^{(p-1)/p}$. Rearranging the above inequality, we can see a polynomial inequality on $\mathbb{E}[\|\mathbf{q}_\perp|^{p}]$ as follows,
\begin{align}
    a_p \mathbb{E}[\|\mathbf{q}_\perp\|^{p}] \leq  a_{p-1} \mathbb{E}[\|\mathbf{q}_\perp\|^{p}]^{(p-1)/p} + a_0 \notag
\end{align}
From Fujiwara's argument on the magnitude of polynomials' root \cite{fujiwara1916obere}, we bound the p-th moment as 
\begin{align*}
    \mathbb{E}[\|\mathbf{q}_\perp\|^{p}] ^{1/p} &\leq 2\max\{(\frac{|a_0|}{|a_p|})^{1/p}, \frac{|a_{p-1}|}{|a_p|}\} 
\end{align*}
Using p's dependency from \eqref{eq: p-moment for gamma queue length}, we bound the two terms in sequel. For $(\frac{|a_{p-1}|}{|a_p|})$:
\begin{align*}
    \frac{|a_{p-1}|}{|a_p|} &\overset{(a)}{\leq} \frac{9p}{\zeta} \cdot (\lambda + \mu + \mathbb{E}[\gamma^p \langle \mathbf{q}, \boldsymbol{1} \rangle^p]^{1/p})  \\
    &\overset{(b)}{\leq} \frac{9p}{\zeta} \cdot (\lambda + \mu + \frac{2\sqrt{2\pi}}{e} A_{\lambda,\mu} \exp(2\lambda - \mu/n) \cdot np)
\end{align*}
$(a)$ holds from triangle inequality for $L^p$ norm. Inequality $(b)$ is using the p-th moment bound on $\gamma \langle \mathbf{q}, \boldsymbol{1} \rangle$ in \eqref{eq: p-moment for gamma queue length} along with
 the stirling approximation for $\Gamma$ function in \eqref{eq: Exp Dominating Polynomial Stirling Bound}. For $(\frac{|a_{0}|}{|a_p|})^{1/p}$, we have
\begin{align*}
    (\frac{|a_{0}|}{|a_p|})^{1/p} &= (1/\zeta)^{1/p} (2\sqrt{n(n-1)}) \cdot \mathbb{E}[(3\lambda + 3\mu + 3\gamma \sum_{i=1}^{n} q_i + \zeta) \cdot (\lambda + \mu + \gamma \langle \mathbf{q}, \boldsymbol{1} \rangle)^p])^{1/p} \\
    &\overset{(a)}{\leq} \max(1,1/\zeta) 2n \bigg( \zeta(\lambda + \mu + \|\gamma \langle \mathbf{q}, \boldsymbol{1} \rangle\|_{L^p}) + 3\cdot 2((\lambda + \mu)^{(p+1)/p} + \mathbb{E}[\gamma^{p+1} \langle \mathbf{q}, \boldsymbol{1} \rangle^{p+1}]^{1/p} )\bigg)\\
    &\overset{(b)}{\leq} 2n\max(1,1/\zeta) \bigg( \zeta(\lambda + \mu + \frac{2\sqrt{2\pi}}{e} A_{\lambda,\mu} \exp(2\lambda - \mu/n) np) \\
    &\quad + 6\cdot \big((\lambda + \mu)^{2} + \frac{4\sqrt{2\pi}}{e} A_{\lambda,\mu} \exp(2\lambda - \mu/n) n^2 p \big) \bigg)
\end{align*}

Inequality $(a)$ holds from $(x+y)^{1/p} \leq x^{1/p} + y^{1/p}$, and $(x+y)^{p+1} \leq 2^p (x^{p+1} + y^{p+1})$ for $x,y \geq 0$. Inequality $(b)$ is from $(p+1)^{1/p} \leq 2, \forall p > 1$ and the bound on p-th moment of $\gamma \langle \mathbf{q}, \boldsymbol{1} \rangle$ in \eqref{eq: p-moment for gamma queue length}.

Combining the two above bounds, we obtain the quadratic growth of p-th moment for $\|\mathbf{q}_\perp\|$:
\begin{align}
    \mathbb{E}[\|\mathbf{q}_\perp\|^{p}]^{1/p} &\leq \max\{ E_{n,\lambda,\mu}^{(1)} p^2, E_{n,\lambda,\mu}^{(2)} p \} \leq E_{n,\lambda,\mu} p^2 \label{eq: p-moment for q_perp, state space collapse}
\end{align}
Where constant $E_{n,\lambda,\mu}^{(1)} := \frac{18}{\zeta} (\lambda +\mu +\frac{2\sqrt{2\pi}}{e}A_{\lambda,\mu} \exp(2\lambda - \mu/n) n)$, and $E_{n,\lambda,\mu}^{(2)} := 4n\max(1,1/\zeta) \bigg( \zeta(\lambda + \mu + \frac{2\sqrt{2\pi}}{e} A_{\lambda,\mu} \exp(2\lambda - \mu/n) n) + 6\cdot \big((\lambda + \mu)^{2} + \frac{4\sqrt{2\pi}}{e} A_{\lambda,\mu} \exp(2\lambda - \mu/n) n^2 \big) \bigg)$. And constant $E_{n,\lambda,\mu} := \max\{ E_{n,\lambda,\mu}^{(1)}, E_{n,\lambda,\mu}^{(2)} \}$. Recalling that $\zeta = \frac{\lambda}{\sqrt{n(n-1)}}$.

\subsection{Proof for Proposition \ref{pro: zero probability sum for JSQ}: Probability mass at zero}
In this section, we will prove Proposition \ref{pro: zero probability sum for JSQ} that bounds the probability mass at zero for the queue length process under JSQ policy in heavy overload regime. The proof is based on the Lyapunov drift method with properly chosen exponential Lyapunov function, the raitonale of choice is explained in Section \ref{sec: proof sketch}.

We apply the generator $\mathcal{L}_{JSQ}$ in \eqref{eq: generator for join shortest queue} on the Lyapunov function $f_\theta(\mathbf{q}) = \sum_{i=1}^{n} \exp(-\sqrt{\gamma}\theta q_i)$ with $\theta$ in the following range, recall $C$ is from Assumption \ref{ass:heavy_overload}:
\begin{align}
    \theta \in \left[0, \frac{C}{\lambda}\right] \label{eq: theta interval for Laplace Transform, JSQ}
\end{align}
Then, we have the following derviation:
\begin{align}
    0 &= \mathbb{E}[\mathcal{L}_{JSQ} f_\theta(\mathbf{q}) ] \notag\\
    &= \mathbb{E}\left[ \lambda \left( \exp(-\sqrt{\gamma}\theta (q_{\min} + 1)) - \exp(-\sqrt{\gamma}\theta q_{\min}) \right) \right] \notag\\
    &\quad + \sum_{i=1}^{n} \mathbb{E}\left[ (\mu_i \mathbf{1}_{\{ q_i \neq 0\}} + \gamma q_i) \left( \exp(-\sqrt{\gamma}\theta (q_i - 1)) - \exp(-\sqrt{\gamma}\theta q_i) \right) \right] \notag\\
      0&= \lambda e^{-\sqrt{\gamma}\theta} \mathbb{E}[ \exp(-\sqrt{\gamma}\theta q_{\min}) ] +  \mathbb{E}[\sum_{i=1}^{n}\mu_i \exp(-\sqrt{\gamma}\theta q_i)] \notag\\
    &\quad + \gamma  \mathbb{E}[ \sum_{i=1}^{n}q_i \exp(-\sqrt{\gamma}\theta q_i)] -\mathbb{E}[\sum_{i=1}^{n}\mu_i \mathbf{1}_{\{ q_i \neq 0\}}] \notag\\
    &\overset{(a)}{=} \lambda e^{-\sqrt{\gamma}\theta} N_*(\theta) +  \sum_{i=1}^{n}\mu_i N_i(\theta) + \sqrt{\gamma}N'(\theta) -\mathbb{E}[\sum_{i=1}^{n}\mu_i \mathbf{1}_{\{ q_i \neq 0\}}] \notag\\
    &\leq \lambda e^{-\sqrt{\gamma}\theta} N_*(\theta) +  \sum_{i=1}^{n}\mu_i N_i(\theta) + \sqrt{\gamma}N'(\theta) \notag\\
    &\overset{(b)}{\leq} -\frac{\lambda}{n} e^{-\sqrt{\gamma}\theta} N(\theta) +  \frac{\mu}{n} N(\theta) -\sqrt{\gamma} N'(\theta), \text{if } \sqrt{\gamma}\theta \leq \ln(\frac{\lambda}{\mu}) \label{eq: ODE for Laplace Transform JSQ}
\end{align}

Equality $(a)$ is change of variables. We use $N_*(\theta) := \mathbb{E}[ \exp(-\sqrt{\gamma}\theta q_{\min}) ]$, $N_i(\theta) := \mathbb{E}[\exp(-\sqrt{\gamma}\theta q_i)]$, $N(\theta) := \mathbb{E}[\sum_{i=1}^{n}\exp(-\sqrt{\gamma}\theta q_i)]$, and $N'(\theta) = \frac{d N(\theta)}{d \theta}$. $(b)$ holds from the relationship between $N_*(\theta)$ and $N_i(\theta)$. Specifically, we first change of variables that $D_i := N_*(\theta) - N_i(\theta)\geq 0$. Then, we have the following computation
\begin{align}
    \frac{1}{n} \lambda e^{-\sqrt{\gamma}\theta} \sum_{i=1}^{n} 
    (N_*(\theta) - N_i(\theta))&= \frac{1}{n} \lambda e^{-\sqrt{\gamma}\theta} \sum_{i=1}^{n} D_i \notag\\
    &\overset{(a)}{\geq} \frac{\mu}{n} \sum_{i=1}^{n} D_i\notag\\
    &\geq \frac{\mu}{n} \sum_{i=1}^{n} D_i - \sum_{i=1}^{n} \mu_i D_i(\theta) \notag\\
    &\overset{(b)}{=} \sum_{i=1}^{n} \mu_i N_i(\theta) - \frac{\mu}{n} N(\theta) \label{eq: algebraic computation for Laplace Transform JSQ}
\end{align}
Inequality $(a)$ holds from the fact that $\lambda e^{-\sqrt{\gamma}\theta} \geq \mu$, which is guaranteed by the choice of $\theta$ range in \eqref{eq: theta interval for Laplace Transform, JSQ} that $\theta \leq \frac{C}{\lambda}$ as well as $\gamma\leq \gamma_1$ from \eqref{eq: gamma assumption for JSQ W-p} as follows.
\begin{align*}
    \gamma \leq \frac{\lambda^2}{C^2} \ln^2(\lambda/\mu), \implies \lambda e^{-\sqrt{\gamma}\theta} &\geq \lambda e^{-\sqrt{\gamma}\frac{C}{\lambda}} \geq \mu.
\end{align*}
Equality $(b)$ holds from plugging back the definition of $D_i$.

Starting from \eqref{eq: ODE for Laplace Transform JSQ}, we bound the moment generating function with negative $\theta$ of $\mathbf{q}$ as follows.
\begin{align*}
    N(\theta) &\overset{(a)}{\leq} \exp \left( \frac{\lambda}{n\gamma} e^{-\sqrt{\gamma}\theta} + \frac{\mu}{n\sqrt{\gamma}} \theta - \frac{\lambda}{n\gamma} \right) \\
    &\overset{(b)}{\leq} \exp \left( \frac{\lambda}{n\gamma} (1-\sqrt{\gamma}\theta + \frac{1}{2}\gamma\theta^2) + \frac{\mu}{n\sqrt{\gamma}} \theta - \frac{\lambda}{n\gamma} \right) \overset{(c)}{\leq} \exp \left( -\frac{\lambda - \mu}{2n\sqrt{\gamma}} \theta\right) 
\end{align*}
Inequality $(a)$ is from Gronwall's inequality applied to the ODE in \eqref{eq: ODE for Laplace Transform JSQ}.
Inequality $(b)$ is bounded from Taylor Series: $e^{-x} \leq 1 - x + \frac{1}{2}x^2$ for $x \geq 0$. Inequality $(c)$ holds from the interval of $\theta$ in \eqref{eq: theta interval for Laplace Transform, JSQ} 

Then we can upper bound the zero's probability mass from this MGF bound as follows.
\begin{align}
    \sum_{i=1}^{n} \mathbb{P}(q_i = 0) &= \sum_{i=1}^{n} \mathbb{E}[\mathbf{1}_{\{ q_i = 0\}}] =\sum_{i=1}^{n} \mathbb{E}[\mathbf{1}_{\{ q_i = 0\}}\exp(-\sqrt{\gamma}\theta q_i)] \notag\\
    &\leq \sum_{i=1}^{n} \mathbb{E}[\exp(-\sqrt{\gamma}\theta q_i)] = N(\theta), \quad \forall \theta \in [0, \frac{C}{\lambda}] \notag\\
    \sum_{i=1}^{n} \mathbb{P}(q_i = 0) &\leq \exp \left( -\frac{C}{2n\lambda}\cdot \frac{\lambda - \mu}{\sqrt{\gamma}} \right) \label{eq: zero's probability mass upper bound JSQ}
\end{align}

Now we obtain a lower bound on $\mathbb{P}(\sum_{i=1}^{n} q_i = 0)$ by coupling the JSQ system with the same arrival and abandonment process, but with all servers' service rates being 0. This is equivalent to a degenerated $M/M/\infty$ queue with arrival rate $\lambda$ and service rate $\gamma$. It is shown in Section \ref{sec: coupling argument} that for stationary distribution, the total queue length in $M/M/\infty$ stochastically dominates the total queue length in JSQ system:
\begin{align}
    \mathbb{P}(\sum_{i=1}^{n} q_i = 0) &\geq \mathbb{P}(\sum_{i=1}^{n} q_{i, M/M/\infty} = 0 ) = e^{-\lambda/\gamma} \label{eq: zero's probability mass lower bound JSQ}
\end{align}
Here we denote $q_{i, M/M/\infty}$ as the queue length in server $i$ under the $M/M/\infty$ system. Since the stationary distribution of total queue length in $M/M/\infty$ system is Poisson with parameter $\lambda/\gamma$, the probability that the total queue length is zero is $e^{-\lambda/\gamma}$.


Meanwhile, we can also couple JSQ system with SSQ, where the server is now a resource pooled server with service rate $\mu = \sum_{i=1}^{n} \mu_i$ but the same arrival rate $\lambda$ and abandonment rate $\gamma$. From the coupling, we have the following stochastic domination:
\begin{align}
    \mathbb{P}(\sum_{i=1}^{n} q_i = 0) &\leq \mathbb{P}(\mathbf{q}_{SSQ} = 0 ) \overset{(a)}{\leq} C_{\lambda,\mu}' \sqrt{\gamma/\mu} \exp\left( -\frac{1}{\gamma/\mu} (\lambda/\mu - 1 - \ln (\lambda/\mu)) \right) \label{eq: zero's probability mass upper bound JSQ homogeneous}
\end{align}
Here $\mathbf{q}_{SSQ}$ is the queue length in the SSQ system.
Inequality $(a)$ is from the upper bound in Lemma \ref{lem: tight bound on P(q_infty = 0)}. 

Combining \eqref{eq: zero's probability mass upper bound JSQ} and \eqref{eq: zero's probability mass lower bound JSQ}, we complete the proof for Proposition \ref{pro: zero probability sum for JSQ}.



\subsection{Proof for Proposition \ref{pro: Moment Bounds on hat q for JSQ}: Moment Bounds on $\hat{q}_\Sigma$}
In this section, we will prove Proposition \ref{pro: Moment Bounds on hat q for JSQ} that bounds the $p$-th moment of the centered total queue length $\hat{q}_\Sigma = \sum_i q_i - \frac{\lambda-\mu}{\gamma}$. The proof is based on the Lyapunov drift method with properly chosen exponential Lyapunov function, the raitonale of choice is again explained in Section \ref{sec: proof sketch}.

We use the following two Lyapunov functions as test functions into the generator $\mathcal{L}_{JSQ}$ in \eqref{eq: generator for join shortest queue}.
\begin{align*}
    f_\theta(\mathbf{q}) = \sum_{i=1}^{n} \exp (-\theta q_i), \; g_\theta(\mathbf{q})= \sum_{i=1}^{n} \exp (\theta q_i),  \theta \geq 0 
\end{align*}
Since the MGF of $\mathbf{q}$ is well-defined in steady state \eqref{par: stability and existence of MGF for JSQ}, we can take expectation on both sides of $\mathcal{L}_{JSQ} f_\theta(\mathbf{q})$ and $\mathcal{L}_{JSQ} g_\theta(\mathbf{q})$ respectively, and set the left hand side to be zero.
\begin{align}
    0 &= \mathbb{E}[\mathcal{L}_{JSQ} g_\theta(\mathbf{q})] \notag\\
    &= \lambda \left( e^{\theta (q_{\min} + 1)} - e^{\theta q_{\min}} \right) + \sum_{i=1}^{n} (\mu_i + \gamma q_i) \left( e^{\theta (q_i - 1)} - e^{\theta q_i} \right) \notag\\
     0&= \lambda e^{\theta} \mathbb{E}[e^{\theta q_{\min}}] - \sum_{i=1}^{n} \mu_i \mathbb{E}[e^{\theta q_i}] + \gamma \sum_{i=1}^{n} \mathbb{E}[q_i e^{\theta q_i}] + \sum_{i=1}^{n} \mu_i \mathbb{E}[\mathbf{1}_{\{ q_i = 0\}}] \notag\\ 
    &\overset{(a)}{=} \gamma M'(\theta) + \sum_{i=1}^{n} \mu_i M_i(\theta) - \lambda e^{\theta} M_*(\theta) - P_0 \notag\\
    &\overset{(b)}{\geq} \gamma M'(\theta) + \frac{\mu}{n} M(\theta) - \frac{\lambda}{n} e^{\theta} M(\theta) - P_0 \label{eq: ODE upper bound for MGF JSQ}
\end{align}

Inequality $(a)$ is change of variables. We use $M_*(\theta) := \mathbb{E}[e^{\theta q_{\min}}]$, $M_i(\theta) := \mathbb{E}[e^{\theta q_i}]$, $M(\theta) := \mathbb{E}[\sum_{i=1}^{n} e^{\theta q_i}]$, and $M'(\theta) = \frac{d M(\theta)}{d \theta}$. And $P_0 := \sum_{i=1}^{n} \mu_i \mathbb{E}[\mathbf{1}_{\{ q_i = 0\}}]$ is the weighted probability mass at zero. Inequality $(b)$ holds from similar shifting argument as in \eqref{eq: algebraic computation for Laplace Transform JSQ}. Note that here $(b)$ holds for $\forall \theta \geq 0$, while \eqref{eq: algebraic computation for Laplace Transform JSQ} has upper bound conditions for $\theta$.

Starting from the ODE \eqref{eq: ODE upper bound for MGF JSQ}, we re-organize the terms in this inequality to achieve
\begin{align*}
    \frac{d}{d\theta} \left( e^{-\frac{\lambda}{n\gamma} e^{\theta} + \frac{\mu}{n\gamma}\theta} M(\theta) \right) &\leq -\frac{\lambda}{n\gamma} e^{-\frac{\lambda}{n\gamma} e^{\theta} + \frac{\mu}{n\gamma}\theta} P_0
\end{align*}
Integrating both sides , we have the following upper bound on $M(\theta)$
\begin{align*}
    M(\theta) &\leq e^{\frac{\lambda}{n\gamma} e^{\theta} - \frac{\mu}{n\gamma}\theta} \left( \exp(-\frac{\lambda}{n\gamma}) + \frac{P_0}{\gamma} \int_0^\theta e^{-\frac{\lambda}{n\gamma} e^{x} + \frac{\mu}{n\gamma}x} dx \right) \\
    &\overset{(a)}{\leq} e^{\frac{\lambda}{n\gamma} e^{\theta} - \frac{\mu}{n\gamma}\theta - \frac{\lambda}{n\gamma}} \left(
        1 + \frac{P_0}{\gamma} \cdot \frac{n\gamma}{\lambda - \mu} 
    \right) \\
    &\overset{(b)}{\leq} (1 + \mu \Gamma(\frac{1}{1/2 - \alpha}+1)(\frac{\lambda}{2n})^{-\frac{1}{1/2 - \alpha}} \frac{n\gamma}{\lambda - \mu} ) e^{\frac{\lambda}{n\gamma} e^{\theta} - \frac{\mu}{n\gamma}\theta - \frac{\lambda}{n\gamma}} \\
    &\overset{(c)}{\leq}  (1 +n \mu \Gamma(\frac{1}{1/2 - \alpha}+1)(\frac{\lambda}{2n})^{-\frac{1}{1/2 - \alpha}})e^{\frac{\lambda}{n\gamma} e^{\theta} - \frac{\mu}{n\gamma}\theta - \frac{\lambda}{n\gamma}}
\end{align*}
Inequality $(a)$ holds from the following bound on the integral term:
\begin{align*}
    \int_0^\theta e^{-\frac{\lambda}{n\gamma} e^{x} + \frac{\mu}{n\gamma}x + -\frac{\lambda}{n\gamma}} dx \leq \int_0^\theta e^{-\frac{\lambda - \mu}{n\gamma}x} dx \leq \frac{n\gamma}{\lambda - \mu} 
\end{align*}
Inequality $(b)$ holds from the upper bound on $P_0$ using \eqref{eq: zero's probability mass upper bound JSQ} and the dominance of exponential on polynomials \eqref{eq: Exp Dominating Polynomial Stirling Bound}. Inequality $(c)$ is the Assumption \ref{ass:heavy_overload}. 
With the upper bound on $M(\theta)$, we can bound $\mathbb{E}[\exp(\theta \sum_{i=1}^{n} q_i)]$ as follows.
\begin{align}
    \mathbb{E}[\exp(\theta \sum_{i=1}^{n} q_i)] &\overset{(a)}{\leq} \frac{1}{n} \sum_{i=1}^{n} \mathbb{E}[\exp(n\theta q_i)] \notag\\
    &\leq \frac{1}{n} (1 +n \mu \Gamma(\frac{1}{1/2 - \alpha}+1)(\frac{C}{2n\lambda})^{-\frac{1}{1/2 - \alpha}})e^{\frac{\lambda}{n\gamma} e^{n\theta} - \frac{\mu}{\gamma}\theta - \frac{\lambda}{n\gamma}} \notag\\
    &\overset{\triangle}{=} A_{\lambda,\mu,n} \exp\left( \frac{\lambda}{n\gamma} e^{n\theta} - \frac{\mu}{\gamma}\theta - \frac{\lambda}{n\gamma} \right) \label{eq: MGF upper bound JSQ}
\end{align}
Inequality $(a)$ is from Jensen's inequality that 
\begin{align}
    \exp (\frac{\theta}{n}\sum_{i=1}^{n} q_i) \leq \frac{1}{n} \sum_{i=1}^{n}  g_{\theta}(\mathbf{q}), \quad \text{for } \theta \geq 0 \label{eq: Jensen's inequality for MGF JSQ}
\end{align}

Meanwhile, applying the  generator to $f_\theta(\mathbf{q}) = \sum_{i=1}^{n} \exp (-\theta q_i)$, we use similar argument as we use in \eqref{eq: ODE for Laplace Transform JSQ} to obtain the following inequality. We use the change of variable: $\tilde{N}_i(\theta) := \mathbb{E} [e^{-\theta q_i}]$, $\tilde{N}_*(\theta) := \mathbb{E} [e^{-\theta q_{min}}]$, $\tilde{N}(\theta)=\sum_{i=1}^{n} \tilde{N}_i(\theta)$, then we have
\begin{align}
    \gamma \tilde{N}'(\theta) &= -\lambda e^{-\theta} \tilde{N}_*(\theta) + \sum_{i=1}^{n} \mu_i \tilde{N}_i(\theta) - P_0 \notag\\
    &\leq -\frac{\lambda}{n} e^{-\theta} \tilde{N}(\theta) + \frac{\mu}{n} \tilde{N}(\theta) \notag\\
     \tilde{N}(\theta) &\leq \exp\left( \frac{\lambda}{n\gamma} e^{-\theta} + \frac{\mu}{n\gamma}\theta - \frac{\lambda}{n\gamma} \right), \notag
\end{align}
The above inequality holds for all $\theta \geq 0$, especially $\theta \leq \ln (\frac{\lambda}{\mu})$ which will be used later. Moreover, since $\tilde{N}(\theta)$ is monotone decreasing in $\theta$, we have the following upper bound for $\theta \geq \ln (\frac{\lambda}{\mu})$:
\begin{align}
    \tilde{N}(\theta) \leq 
    \begin{cases}
        \frac{1}{n}\exp\left( \frac{\lambda}{n\gamma} e^{-n\theta} + \frac{\mu}{\gamma}\theta - \frac{\lambda}{n\gamma} \right), & \text{if } \theta \leq \frac{1}{n}\ln (\frac{\lambda}{\mu}) \\
        \frac{1}{n}\exp\left( \frac{\mu}{n\gamma} + \frac{\mu}{n\gamma}\ln(\frac{\lambda}{\mu}) - \frac{\lambda}{n\gamma} \right), & \text{if } \theta \geq \frac{1}{n}\ln (\frac{\lambda}{\mu}), \\
    \end{cases} \label{eq: negative tilde MGF upper bound JSQ}
\end{align}
where the first case is from the previous upper bound, and the second case is from monotonicity of $\tilde{N}(\theta)$ that $\tilde{N}(\theta) \leq \tilde{N}(\frac{1}{n}\ln (\frac{\lambda}{\mu}))$ for $\theta \geq \frac{1}{n}\ln (\frac{\lambda}{\mu})$.

Using \eqref{eq: negative tilde MGF upper bound JSQ} we have control on $\mathbb{E}[\exp(-\theta \hat{q}_\Sigma)]$. Firstly, for $\theta \leq \frac{1}{n}\ln (\frac{\lambda}{\mu})$, we have
\begin{align*}
    \mathbb{E}[\exp(-\theta \hat{q}_\Sigma)]&= \mathbb{E}[\exp(-\theta ( \hat{q}_{\Sigma}))] \\
    &= \tilde{N}( \theta) \exp\left( \frac{\lambda-\mu}{\gamma}\theta \right) \\
    &\overset{(a)}{\leq} \frac{1}{n} \sum_{i=1}^{n} \mathbb{E}[\exp(-n\theta q_i)] \exp\left( \frac{\lambda-\mu}{\gamma}\theta \right) \\
    &\overset{(b)}{\leq} \frac{1}{n}\exp\left( \frac{\lambda}{n\gamma} e^{-n\theta} + \frac{\lambda}{\gamma}\theta - \frac{\lambda}{n\gamma} \right) \\
    &\leq \frac{1}{n}\exp\left( \frac{\lambda}{n\gamma} e^{n\theta} - \frac{\lambda}{\gamma}\theta - \frac{\lambda}{n\gamma} \right)
\end{align*}
Inequality $(a)$ is from Jensen's inequality. Inequality $(b)$ is from the first case in \eqref{eq: negative tilde MGF upper bound JSQ}. 

When $\theta \geq \frac{1}{n}\ln (\frac{\lambda}{\mu})$, we have the following bound.
\begin{align*}
    \mathbb{E}[\exp(-\theta \hat{q}_{\Sigma} )] &\overset{(a)}{\leq} \frac{1}{n}\exp\left( \frac{\mu}{n\gamma} + \frac{\mu}{n\gamma}\ln(\frac{\lambda}{\mu}) - \frac{\lambda}{n\gamma} + \frac{\lambda-\mu}{\gamma}\theta \right) \\ 
    &\overset{(b)}{\leq} \frac{1}{n}\exp\left( \frac{\lambda}{n\gamma} e^{n\theta} - \frac{\lambda}{\gamma}\theta - \frac{\lambda}{n\gamma} \right)
\end{align*}
Inequality $(a)$ is from \eqref{eq: negative tilde MGF upper bound JSQ}. Inequality $(b)$ holds from checking the following inequality holds, where we use the notation $\lambda'=\lambda/\mu$,
\begin{align}
    f_{\lambda'}(\theta):=\lambda' e^{n\theta} - \left(n (2\lambda' -1) \theta + 1+ \ln(\lambda')\right) \geq 0 ,\; \text{if } \theta \geq \frac{1}{n}\ln(\lambda'). \label{eq: f_lambda'_inequality}
\end{align}
To see this inequality holds, we first check the monotonicity of $f_{\lambda'}(\theta)$. First order derivative of $f_{\lambda'}(\theta)$ gives us monotonicity when $\theta \geq \frac{1}{n}\ln(\lambda')$,
\begin{align*}
    f_{\lambda'}'(\theta) &= n \lambda' e^{n\theta} - n(2\lambda' - 1) \geq n(\lambda'^2 - (2\lambda' - 1)) \geq 0,\; \text{if } \theta \geq \frac{1}{n}\ln(\lambda')
\end{align*}
It suffices to verify $f_{\lambda'}(\theta)$ at $\theta = \frac{1}{n}\ln(\lambda')$,
\begin{align*}
    f_{\lambda'}(\frac{1}{n}\ln(\lambda')) &= \lambda'^2 - (2\lambda')\ln \lambda' -1 \overset{(a)}{\geq} 0
\end{align*}
Inequality $(a)$ holds from the assumption that $\lambda' \geq 1$. Therefore inequality \eqref{eq: f_lambda'_inequality} holds.

Hence, combining the two cases, we obtain upper bound on $\mathbb{E}[\exp(-\theta \hat{q}_{\Sigma} )]$ for $\forall \theta \geq 0$:
\begin{align}
    \mathbb{E}[\exp(-\theta \hat{q}_{\Sigma} )] &\leq \frac{1}{n}\exp\left( \frac{\lambda}{n\gamma} e^{n\theta} - \frac{\lambda}{\gamma}\theta - \frac{\lambda}{n\gamma} \right), \forall \theta \geq 0 \label{eq: negative tilde MGF upper bound whole domain JSQ}
\end{align}
Combining \eqref{eq: MGF upper bound JSQ} and \eqref{eq: negative tilde MGF upper bound whole domain JSQ}, we obtain upper bound on $\mathbb{E}[\exp(\theta | \hat{q}_{\Sigma}|)]$,
\begin{align}
    \mathbb{E}[\exp(\theta | \hat{q}_{\Sigma}|)] &\leq \mathbb{E}[\exp(\theta ( \hat{q}_{\Sigma}))] + \mathbb{E}[\exp(-\theta ( \hat{q}_{\Sigma}))] \\
    &\leq (A_{\lambda,\mu,n} +\frac{1}{n}) \exp\left( \frac{\lambda}{n\gamma} e^{n\theta} - \frac{\lambda}{\gamma}\theta - \frac{\lambda}{n\gamma} \right) \label{eq: absolute MGF upper bound JSQ}
\end{align}

Then we can use the above upper bound to control $\|  \hat{q}_{\Sigma} \|_{L^p}$ for $p > 1$. Recalling arguments in proposition \ref{pro: Lp norm of q_infty}, we first identify the sub-exponential variance of $ \hat{q}_{\Sigma}$. For $\theta \in [0,1]$, 
\begin{align*}
    \mathbb{E}[\exp(\theta | \hat{q}_{\Sigma}|)] &\leq(A_{\lambda,\mu,n} +\frac{1}{n}) \exp\left( \frac{\lambda}{n\gamma} e^{n\theta} - \frac{\lambda}{n\gamma}\theta - \frac{\lambda}{n\gamma} \right) \\
    &\overset{(a)}{\leq} (A_{\lambda,\mu,n} +\frac{1}{n}) \exp\left( \frac{n\lambda}{\gamma}\int_{0}^{1} (1-x)e^{n\theta x}dx\cdot \theta^2 \right) \\
    &= (A_{\lambda,\mu,n} +\frac{1}{n}) \exp\left( \frac{\lambda}{\gamma} (\frac{1}{n}e^n - 1- \frac{1}{n}) \theta^2 \right) 
\end{align*}
Inequality $(a)$ holds from the integral residual form of Taylor series expansion for $e^{n\theta}$. 
From lemma \ref{lem: sub-exponential Lp norm}, we have
\begin{align}
    \|  \hat{q}_{\Sigma} \|_{L^p} &\leq C_{\lambda,\mu,n}'' \cdot\frac{\sqrt{p}}{\sqrt{\gamma}} + C_{\lambda,\mu,n}'\cdot p \label{eq: subexpo L-p norm JSQ}
\end{align}
where $C_{\lambda,\mu,n}'$ and $C_{\lambda,\mu,n}''$ are constants defined as followed:
\begin{align*}
    C_{\lambda,\mu,n}' &:= (A_{\lambda,\mu,n} +\frac{1}{n}) \sqrt{2 (\frac{1}{n}e^n - 1- \frac{1}{n})\lambda} \cdot 2e\sqrt{2\pi}, \quad C_{\lambda,\mu,n}'' := 2e\sqrt{2\pi} 
\end{align*}

Notice that without the Taylor expansion step $(a)$, if we directly use the upper bound \eqref{eq: absolute MGF upper bound JSQ} as follows,
\begin{align*}
    \mathbb{E}[\exp(\theta \frac{| \hat{q}_{\Sigma}|}{n})] &\leq (A_{\lambda,\mu,n} +\frac{1}{n}) \exp\left( \frac{\lambda}{n\gamma} e^{\theta} - \frac{\lambda}{n\gamma}\theta - \frac{\lambda}{n\gamma} \right).
\end{align*}
Using the same derivation in proposition \ref{pro: Lp norm of q_infty} that leads to \eqref{eq: sub-Poisson upper bound on L-p norm SSQ} , we have sub-Poisson bound on $\|  \hat{q}_{\Sigma} \|_{L^p}$ for $p > 1$:
\begin{align}
    \|  \hat{q}_{\Sigma} \|_{L^p} &\leq n(A_{\lambda,\mu,n} +\frac{1}{n}) \frac{p}{\log(1+\frac{n\gamma}{\lambda}p)}.\label{eq: Lp norm Sub poisson upper bound q_infty JSQ}
\end{align}

On the other hand, we can obtain the lower bound for $\|  \hat{q}_{\Sigma} \|_{L^p}$ from the coupling argument as in \eqref{eq: zero's probability mass upper bound JSQ homogeneous}. We coupled the JSQ system with the SSQ with resource-pooled server that has service rate $\mu = \sum_{i=1}^{n} \mu_i$ and the same arrival rate $\lambda$ and abandonment rate $\gamma$.
 In steady state, JSQ's total queue length stochastically dominates the SSQ's queue length, hence we have for any deviation $a \geq 0$
\begin{align*}
    \mathbb{P}(\sum_{i=1}^{n} q_i \geq a) \geq \mathbb{P} ( q_{SSQ} \geq a).
\end{align*}
The notation $ q_{SSQ}$ denotes the queue length in the SSQ system in steady state, following the same notation above.

Thus, following the same computation starting from \eqref{eq: L-p norm integral identity}, we have
\begin{align}
    \| \hat{q}_{\Sigma} \|_{L^p} &\geq \left(p \int_{0}^{\infty} t^{p-1} \mathbb{P}( \hat{q}_{\Sigma} \geq t) dt \right)^{\frac{1}{p}} \notag\\
    &\geq \left(p \int_{0}^{\infty} t^{p-1} \mathbb{P}(q_{SSQ} \geq t) dt 
    \right)^{\frac{1}{p}} \notag\\
    &\overset{(a)}{\geq} C_{4,\lambda,\mu} \exp(-8(15+A_{\lambda,\mu})\frac{\lambda}{\gamma p}) \frac{p}{\log(1+\frac{p\gamma}{\lambda})} \label{eq: Lp norm lower bound  JSQ}
\end{align}
Inequality $(a)$ is from the lower bound in proposition \ref{pro: Lp norm of q_infty} for SSQ, where $A_{\lambda,\mu}$ is defined in \eqref{eq: MGF for q_infty - lambda/mu/gamma, Poisson style}.
Combining \eqref{eq: subexpo L-p norm JSQ}, \eqref{eq: Lp norm Sub poisson upper bound q_infty JSQ}, and \eqref{eq: Lp norm lower bound  JSQ}, we complete the proof for Proposition \ref{pro: Moment Bounds on hat q for JSQ}.

\paragraph{Discussion} The dominating terms for different $p$'s regime in the upper bound for Proposition \ref{pro: Moment Bounds on hat q for JSQ} 
 is the same as in \eqref{eq: discussion for p's dependency on Lp norm SSQ}. This further exhibits the effect of SSC that total queue length behaves approximately equal and therefore the moment bounds behave similarly
 in SSQ manner.

\subsection{Proof for Theorem \ref{thm: Wasserstein-$p$ distance JSQ}: \texorpdfstring{$\mathcal{W}_p$}{W\_p} Distance for JSQ}
In this section, we will prove Theorem \ref{thm: Wasserstein-$p$ distance JSQ} that bounds the Wasserstein-$p$ distance between the law of the total queue length $\hat{q}_\Sigma$ and the standard normal distribution $\mathcal{N}(0,1)$ in steady state. We will use Stein's Method, triangle inequality, and the quantile coupling technique to achieve the bound. The proof is similar to the one in Section \ref{sec: proof for W-p SSQ}, we will quote similar steps and highlight the differences in the proof for JSQ system.

\subsubsection{Stein's Method} 
For the following analysis that uses Stein's method, we study the $p$'s range in 
\begin{align}
    1 < p \leq \frac{\lambda}{2} (\frac{1}{\gamma}) \label{eq: Initial p's regime for Wasserstein distance JSQ}
\end{align}

We use Proposition \ref{pro: Wasserstein-$p$ bound without exchangeability final version} that enables us to bound the Wasserstein-$p$ distance by generator coupling. For the setup of Proposition \ref{pro: Wasserstein-$p$ bound without exchangeability final version}, we choose steady-state queue length vector $\mathbf{q}$ as $X$. Transformation $T_2$ is defiend as $T_2(\mathbf{x}) = \frac{\sqrt{\gamma}}{\sqrt{\lambda}} ( \sum_{i=1}^{n} x_i - \frac{\lambda-\mu}{\gamma} )$, thus $\tilde{q}_{\Sigma} = T_2(\mathbf{q})$. From the well-definedness of MGF in \eqref{eq: well-defined MGF for JSQ}, $T_2(\mathbf{q})$ has finite moments of all order $p$, therefore we can bound any Wasserstein-$p$ distance when $p > 1$.
 We define the operator $\mathcal{A}$ acting on function $g:\mathbb{R}^n \to \mathbb{R}$ as  the scaled infinitesimal generator of the queue length process in JSQ system. 
\begin{align}
    \mathcal{A} g(\mathbf{x}) := \frac{1}{\gamma}\mathcal{L}_{JSQ} g(\mathbf{x}) &=
     \lim_{t \to 0} \frac{1}{t} \bigg( \mathbb{E}[g(\mathbf{q}) \mid \mathbf{q} = \mathbf{x}] - g(\mathbf{x}) \bigg), \quad \text{for } \mathbf{x} \in \operatorname{supp}(\mathcal{L}(\mathbf{q})) \label{eq: operator A for JSQ, implicit}
\end{align}
With the definition of $\mathcal{L}_{JSQ}$ in \eqref{eq: generator for join shortest queue}, we can compute $\mathcal{A} (f\circ T_2)$ explicitly as follows,
\begin{align}
    \mathcal{A} (f\circ T_2)(\mathbf{x}) &= \frac{1}{\gamma}\mathcal{L}_{JSQ} (f\circ T_2)(\mathbf{x}) \notag\\
        &= \frac{\lambda}{\gamma} \biggl( f\biggl( T_2(\mathbf{x}) + \frac{\sqrt{\gamma}}{\sqrt{\lambda}} \biggr) - f\bigl(T_2(\mathbf{x})\bigr) \biggr) \notag\\
        &\quad + \frac{1}{\gamma}\sum_{i=1}^{n} (\mu_i \mathbf{1}_{\{x_i \neq 0\}} + \gamma x_i) \biggl( f\biggl( T_2(\mathbf{x}) - \frac{\sqrt{\gamma}}{\sqrt{\lambda}} \biggr) - f\bigl(T_2(\mathbf{x})\bigr) \biggr) \notag\\
        &= \frac{\lambda}{\gamma} \biggl( f\biggl( \frac{\sqrt{\gamma}}{\sqrt{\lambda}} (\sum_{j=1}^{n} x_j+1) - \frac{\lambda-\mu}{\sqrt{\lambda\gamma}} \biggr) - f\biggl(\frac{\sqrt{\gamma}}{\sqrt{\lambda}} \sum_{j=1}^{n} x_j
         - \frac{\lambda-\mu}{\sqrt{\lambda\gamma}}\biggr) \biggr) \notag\\
        &\quad + \frac{1}{\gamma}\sum_{i=1}^{n} (\mu_i \mathbf{1}_{\{x_i \neq 0\}} + \gamma x_i)  \biggl( f\biggl( \frac{\sqrt{\gamma}}{\sqrt{\lambda}} \biggl(\sum_{j=1}^{n} x_j - 1\biggr) - \frac{\lambda-\mu}{\sqrt{\lambda\gamma}} \biggr) \notag\\
        &\qquad - f\biggl(\frac{\sqrt{\gamma}}{\sqrt{\lambda}} \sum_{j=1}^{n} x_j - \frac{\lambda-\mu}{\sqrt{\lambda\gamma}}\biggr) \biggr) \label{eq: operator A for JSQ, explicit}
\end{align}
 Again, we verify the assumptions in Proposition \ref{pro: Wasserstein-$p$ bound without exchangeability final version}. Similarly to SSQ, it can be shown that $\mathcal{A} f\circ T_2$ has finite expectation for all $f \in C_b(\mathbb{R})$ from the boundedness of $f$ and the finiteness of moments of $\tilde{q}_{\Sigma}$ in Proposition \ref{pro: Moment Bounds on hat q for JSQ}. Also, we have
 $\mathbb{E}[\mathcal{A} (f\circ T_2)
(\tilde{q}_{\Sigma})] = 0$ for all $f \in C_b(\mathbb{R})$ from the definition of generator. Thus the Assumption \ref{assumption: Stein identity final version} holds.
 
For Assumption \ref{assumption: Kramers-Moyal expansion final version} and \ref{assumption: integrability of Kramers-Moyal coefficients final version}, we first use the explicit form of $\mathcal{A}$ in \eqref{eq: operator A for JSQ, explicit} to show
\begin{align}
    a^{\mathcal{A}, T_2}_j(\mathbf{x}) = \frac{1}{j!\gamma} (\sqrt{\gamma/\lambda})^j \left( \lambda + \sum_{i=1}^{n} (\mu_i \mathbf{1}_{\{x_i \neq 0\}} + \gamma x_i) (-1)^j \right) \label{eq: conditional expectation of tilde q JSQ} 
\end{align}
by applying Lemma \ref{lemma: identity for Kramers-Moyal coefficients final version}.


With the above expression, we can show that for all constant $D > 0$,
\begin{align*}
    \mathbb{E}[\sum_{j=1}^{\infty}D^j\sqrt{j!}|a^{\mathcal{A}, T_2}_j(\mathbf{x})|] &\leq \frac{1}{\gamma}\sum_{j=1}^{\infty} \frac{D^j}{\sqrt{j!}}
    (\sqrt{\gamma/\lambda})^j \left( \lambda + \sum_{i=1}^{n} (\mu_i + \gamma \mathbb{E}[|q_i|]) \right) <\infty,
\end{align*}
thus verify Assumption \ref{assumption: integrability of Kramers-Moyal coefficients final version}. For Assumption \ref{assumption: Kramers-Moyal expansion final version}, using the condition $f \in \mathcal{F}$ that there exists constant $D_f$ and $m$ such that $|f^{(j)}(x)| \leq D_f^j \sqrt{j!}(1+x)^m$, combined with the Taylor's theorem for Lagrange form of remainder, we can use exact same argument in SSQ when proving Lemma \ref{lem: Wasserstein-$p$ bound SSQ reduction} to show that the remainder term in the Taylor series expansion of $\mathcal{A} (f\circ T_2)(\mathbf{x})$ converges to zero.  Thus, Assumption \ref{assumption: Kramers-Moyal expansion final version} holds.


Now we apply Proposition \ref{pro: Wasserstein-$p$ bound without exchangeability final version} to bound $\mathcal{W}_p(\mathcal{L}(\tilde{q}_{\Sigma}), \mathcal{L}(Z))$. Particularly, we choose the parameters in Proposition \ref{pro: Wasserstein-$p$ bound without exchangeability final version} as $t_0 = -\frac{1}{2}\log(1-p\gamma/\lambda)$. In the sequel, we bound the terms $g_0(t_0,p),(b_0),(b_1),(b_2)$ in the right hand side of the inequality respectively. First recall that under our assumption \eqref{eq: Initial p's regime for Wasserstein distance JSQ}, we have $e^{t_0}\leq \sqrt{2}$. From upper bound \eqref{eq: b0 upper bound SSQ}, we have the same upper bound for $(g_0(t_0,p)$ in JSQ system, 
\begin{align}
    g_0(t_0,p) &\leq   \frac{8e\sqrt{\pi}}{\sqrt{\lambda}} \sqrt{\gamma} p \label{eq: g0 upper bound JSQ}
\end{align}

Next, we bound the terms $(b_1), (b_2), (b_3)$ respectively. We use the identity in \eqref{eq: conditional expectation of tilde q JSQ} to compute the quantities involving operator $\mathcal{A}$ in each of these three terms. We omit the computation that is similar to SSQ, and only highlight the key steps.

For the first term $(b_1)$, we use \eqref{eq: conditional expectation of tilde q JSQ} to bound it as 
\begin{align}
    (b_1) &= e^{t_0} \int_{-\frac{1}{2}\log(1-p\gamma/\lambda)}^{\infty} e^{-r} dr \cdot \| a^{\mathcal{A}, T_2}_1(\mathbf{q}) + \tilde{q}_{\Sigma}\|_{L^p} \notag\\
    &= \sqrt{2}\int_{-\frac{1}{2}\log(1-p\gamma/\lambda)}^{\infty} e^{-r} dr \cdot \frac{\sqrt{\gamma}}{\sqrt{\lambda}} \cdot \| \frac{\lambda-\mu}{\gamma} - \sum_{i=1}^{n} q_i + \frac{1}{\gamma} \sum_{i=1}^{n} \mu_i \mathbf{1}_{\{ q_i = 0\}} +  \hat{q}_{\Sigma} \|_{L^p} \notag\\
    &\leq \sqrt{2}\int_{0}^{\infty} e^{-r} dr\cdot \frac{\mu}{\sqrt{\lambda}\sqrt{\gamma}} \|\sum_{i=1}^{n} \mathbf{1}_{\{ q_i = 0\}}\|_{L^p} = \frac{\mu}{\sqrt{\lambda}\sqrt{\gamma}} \|\sum_{i=1}^{n} \mathbf{1}_{\{ q_i = 0\}}\|_{L^p}\label{eq: b1 upper bound JSQ}
\end{align}
Similarly, for $(b_2)$, we have
\begin{align}
    (b_2) &= e^{t_0}\lim_{t\to 0} \int_{-\frac{1}{2}\log(1-p\gamma/\lambda)}^{\infty} \frac{e^{-2r}\|h_1(Z)\|_{L^p}}{\sqrt{1-e^{-2r}}} dr \cdot \|  a^{\mathcal{A}, T_2}_2(\mathbf{q}) - 1 \|_{L^p} \notag\\
    &= \sqrt{2}\int_{-\frac{1}{2}\log(1-p\gamma/\lambda)}^{\infty} \frac{e^{-2r}\|h_1(Z)\|_{L^p}}{\sqrt{1-e^{-2r}}} dr \cdot \frac{\gamma}{\lambda}\| \frac{1}{2\gamma} \bigg(
        \lambda + \mu + \gamma \sum_{i=1}^{n} q_i - \sum_{i=1}^{n} \mu_i \mathbf{1}_{\{ q_i = 0\}} - 1 \bigg)\|_{L^p} \notag\\
    &\leq \sqrt{2}\int_{-\frac{1}{2}\log(1-p\gamma/\lambda)}^{\infty} \frac{e^{-2r}\|h_1(Z)\|_{L^p}}{\sqrt{1-e^{-2r}}} dr \cdot \frac{\gamma}{\lambda} (\frac{1}{2}\| \hat{q}_{\Sigma}\|_{L^p} + \frac{\mu}{2\gamma} \| \sum_{i=1}^{n} \mathbf{1}_{\{ q_i = 0\}}\|_{L^p}) \notag\\
    &\overset{(a)}{\leq} 4e\sqrt{\pi}(1 - \sqrt{1 - (1- \frac{p\gamma}{\lambda})}) \cdot \frac{\gamma\sqrt{p}}{2\lambda} \cdot (\| \hat{q}_{\Sigma} \|_{L^p} + \frac{\mu}{\gamma} \cdot \|\sum_{i=1}^{n} \mathbf{1}_{\{ q_i = 0\}}\|_{L^p}) \notag\\
    &\leq 4e\sqrt{\pi} \frac{\gamma\sqrt{p}}{2\lambda} \cdot (\| \hat{q}_{\Sigma} \|_{L^p} + \frac{\mu}{\gamma} \cdot \|\sum_{i=1}^{n} \mathbf{1}_{\{ q_i = 0\}}\|_{L^p}) \label{eq: b_2 upper bound JSQ}
\end{align}
Hre inequality $(a)$ is from the $L^p$ norm for $h_1(Z)=Z$ as in \eqref{eq: Lp norm of Z} and the explicit computation for integration. 
For $(b_3)$, we split the integral into two parts, based on whether $k$ is even or odd. For odd $k$, we have
\begin{align}
    (b_3,1) &:= e^{t_0}\sum_{k\geq 3, odd} \int_{-\frac{1}{2}\log(1-p\gamma/\lambda)}^{\infty} \frac{e^{-kr}\|h_{k-1}(Z)\|_{L^p}}{\sqrt{1-e^{-2r}}^{k-1}} dr \cdot \| a^{\mathcal{A}, T_2}_k(\mathbf{q}) \|_{L^p} \notag\\
    &\overset{(a)}{\leq} \sqrt{2}\sum_{k\geq 3, odd} \int_{-\frac{1}{2}\log(1-p\gamma/\lambda)}^{\infty} \frac{e^{-kr}\|h_{k-1}(Z)\|_{L^p}}{k!\sqrt{1-e^{-2r}}^{k-1}} dr \cdot \notag\\
    &\quad \frac{\sqrt{\gamma}^k}{\sqrt{\lambda}^k} \bigg(\frac{\mu}{\gamma} \| \sum_{i=1}^{n} \mathbf{1}_{q_i=0}\|_{L^p} + \|  \hat{q}_{\Sigma}\|_{L^p} \bigg) \notag
\end{align}
Here inequality $(a)$ holds from identity \eqref{eq: conditional expectation of tilde q JSQ}. Following the same procedure as in \eqref{eq: b_3,1 integral upper bound SSQ}, we compute the integral above and obtain the upper bound for odd index summation,
\begin{align}
    (b_{3,1}) \leq 8\sqrt{2}e^2 \frac{\gamma \sqrt{p}}{\lambda} \bigg( \|  \hat{q}_{\Sigma}\|_{L^p} + \frac{\mu}{\gamma}\|\sum_{i=1}^{n}\mathbf{1}_{q_i=0}\|_{L^p}\bigg) \label{eq: b_3,1 upper bound JSQ}
\end{align}
For the even index $k$, we handle it in the following way,
\begin{align}
    (b_{3,2}) &:= 
    e^{t_0}\sum_{k\geq 4, even} \int_{-\frac{1}{2}\log(1-p\gamma/\lambda)}^{\infty} \frac{e^{-kr}\|h_{k-1}(Z)\|_{L^p}}{\sqrt{1-e^{-2r}}^{k-1}} dr \cdot \| a^{\mathcal{A}, T_2}_k(\mathbf{q}) \|_{L^p} \notag\\
    &\overset{(a)}{\leq} \sqrt{2}\sum_{k\geq 4, even} \int_{-\frac{1}{2}\log(1-p\gamma/\lambda)}^{\infty} \frac{e^{-kr}\|h_{k-1}(Z)\|_{L^p}}{k!\sqrt{1-e^{-2r}}^{k-1}} dr \cdot \notag\\
    &\quad \frac{\sqrt{\gamma}^k}{\sqrt{\lambda}^k} \bigg(\frac{\mu}{\gamma} \| \sum_{i=1}^{n} \mathbf{1}_{q_i=0}\|_{L^p} + \|  \hat{q}_{\Sigma}\|_{L^p} + \frac{2\lambda}{\gamma} \bigg) \notag 
\end{align}
$(a)$ is from identity \eqref{eq: conditional expectation of tilde q JSQ} and triangle inequality. The integral above can also be bounded from the treatment in \eqref{eq: b_3,2 integral upper bound SSQ}. Thus we have
\begin{align}
    (b_{3,2}) \leq \sqrt{2}e^2 \frac{\gamma^{3/2}p}{\lambda^{3/2}} \bigg(\frac{\mu}{\gamma} \| \sum_{i=1}^{n} \mathbf{1}_{q_i=0}\|_{L^p} + \|  \hat{q}_{\Sigma}\|_{L^p} + \frac{2\lambda}{\gamma} \bigg) \label{eq: b_3,2 upper bound JSQ}
\end{align} 
With the above bounds, i.e., \eqref{eq: g0 upper bound JSQ}, \eqref{eq: b1 upper bound JSQ}, \eqref{eq: b_2 upper bound JSQ}, \eqref{eq: b_3,1 upper bound JSQ}, \eqref{eq: b_3,2 upper bound JSQ}, we reduce bounding Wasserstein-$p$ distance to bounding the $L^p$ norms $\| \hat{q}_{\Sigma}\|_{L^p}$ and $\|\sum_{i=1}^{n} \mathbf{1}_{q_i=0}\|_{L^p}$, which are handled in Proposition \ref{pro: Moment Bounds on hat q for JSQ} and Proposition \ref{pro: zero probability sum for JSQ} respectively.
Combining all the bounds together,
we now bound Wasserstein-$p$ distance as 
\begin{align}
    W_p(\mathcal{L}(\tilde{q}_{\Sigma}), \mathcal{L}(Z))&\overset{(a)}{\leq} \sqrt{2}\bigg(((e+8e^2+\sqrt{2}e^2)\sqrt{2\pi}/\sqrt{\lambda}) (C_{\lambda,\mu,n}'' + C_{\lambda,\mu,n}'\frac{\lambda}{\sqrt{2}}) + \frac{4e\sqrt{2\pi}+2e^2}{\sqrt{\lambda}} \bigg)\cdot\gamma\sqrt{p} \notag\\ 
    &\quad + \sqrt{2}(1+2e\sqrt{\pi}+ 8\sqrt{2}e^2 + 2e^2)\frac{\mu}{\sqrt{\lambda}} \cdot \frac{1}{\sqrt{\gamma}} \exp \left( -\frac{C}{2n\lambda}\cdot \frac{\lambda - \mu}{p\sqrt{\gamma}} \right) \label{eq: Wasserstein-$p$ upper bound JSQ, Stein}
\end{align}

\subsubsection{Triangle Inequality and Quantile Coupling}
Using similar argument as SSQ, we can adopt triangle inequality to obtain different versions of bound for Wasserstein-$p$ distance. We first use triangle inequality to upper bound Wasserstein-$p$ distance as
\begin{align}
    W_p(\mathcal{L}(\tilde{q}_{\Sigma}),\mathcal{L}(Z))&\leq \|\tilde{q}_{\Sigma}\|_{L^p} + \|Z\|_{L^p} \leq (C_{\lambda,\mu,n}''\frac{\lambda}{\sqrt{2}} + C_{\lambda,\mu,n}' + 4e\sqrt{2\pi}) \sqrt{p}\label{eq: Wasserstein-$p$ upper bound JSQ, Triangle}
\end{align}

Meanwhile, we obtain the lower bound from triangle inequality and quantile coupling, similarly as in SSQ \eqref{eq: Gaussian Wasserstein-$p$ Lower L^p norm, Triangle} and \eqref{eq: Gaussian Wasserstein-$p$ Lower bound, P_0}.
\begin{align}
    W_p(\mathcal{L}(\tilde{q}_{\Sigma}),\mathcal{L}(Z))&\geq \frac{\sqrt{\gamma}}{\sqrt{\lambda}} \|\hat{q}_{\Sigma}\|_{L^p} - \|Z\|_{L^p} \notag\\
    &\overset{(a)}{\geq} C_{4,\lambda,\mu} \exp(-8(15+A_{\lambda,\mu})\frac{\lambda}{\gamma p}) \frac{p}{\log(1+\frac{p\gamma}{\lambda})} - 4e\sqrt{2\pi}\sqrt{p} \label{eq: Wasserstein-$p$ lower bound JSQ, Triangle} \\
    W_p(\mathcal{L}(\tilde{q}_{\Sigma}),\mathcal{L}(Z)) &=\left[ \int_{0}^{1} \left| F_\mathcal{L}(\tilde{q}_{\Sigma})^{-1}(u) - F_\mathcal{L}(Z)^{-1}(u) \right|^p du \right]^{1/p} \notag\\
    &\overset{(b)}{\geq} \frac{\sqrt{2}[2C_{\lambda,\mu}' e^{2\sqrt{2}} \sqrt{\lambda}]^{-1}}{\sqrt{\lambda-\ln \lambda -1 }} \cdot\sqrt{\gamma} \mathbb{P}(\sum_{i=1}^{n} q_i = 0)^{1/p} \label{eq: Wasserstein-$p$ lower bound JSQ, P_0}
\end{align}
Inequality $(a)$ is from lower bound in Proposition \ref{pro: Moment Bounds on hat q for JSQ}. 
Inequality $(b)$ is from the same argument as in SSQ (see inequality \eqref{eq: Gaussian Wasserstein-$p$ Lower bound, P_0}), where we use quantile coupling to lower bound the Wasserstein-$p$ distance by the probability mass on zero state. Specifically, we have the exact same upper bound for $\mathbb{P}(\sum_{i=1}^{n} q_i = 0)$ in JSQ, i.e., \eqref{eq: zero's probability mass upper bound JSQ homogeneous},
 as in SSQ (see upper bound in Lemma \ref{lem: tight bound on P(q_infty = 0)}), and we have same assumption on $\gamma$i.e., $\gamma \leq \frac{\mu^2}{(2^2 5^2e^{4\sqrt{2}}\lambda)}$ that is used in SSQ to derive \eqref{eq: Gaussian Wasserstein-$p$ Lower bound, P_0}.
 Thus inequality $(b)$ holds from the exactly same compuation in \eqref{eq: Gaussian Wasserstein-$p$ Lower bound, P_0}.

Combining the above upper and lower bounds on Wasserstein-$p$ distance in \eqref{eq: Wasserstein-$p$ lower bound JSQ, P_0}, \eqref{eq: Wasserstein-$p$ lower bound JSQ, Triangle}, \eqref{eq: Wasserstein-$p$ upper bound JSQ, Triangle}, \eqref{eq: Wasserstein-$p$ upper bound JSQ, Stein}, we obtain the final result for Wasserstein-$p$ in JSQ setting, under the assumption of regime of p \eqref{eq: Initial p's regime for Wasserstein distance JSQ}:

\begin{align}
    \mathcal{W}_p(\mathcal{L}(\tilde{q}_{\Sigma}), \mathcal{L}(Z)) 
    &\leq  A_{1,\lambda,\mu,n}\cdot p\sqrt{\gamma} + A_{1,\lambda,\mu,n}'\cdot\frac{1}{\sqrt{\gamma}}\exp \left( -\frac{C}{2n\lambda}\cdot \frac{\lambda - \mu}{p\sqrt{\gamma}} \right), \; \text{if} \; 1< p \leq \frac{\lambda}{2} (\frac{1}{\gamma}) 
    , \notag\\
    \mathcal{W}_p(\mathcal{L}(\tilde{q}_{\Sigma}), \mathcal{L}(Z)) &\leq \min \Bigg\{  A_{2,\lambda,\mu,n}\cdot\sqrt{p}, \quad A_{2,\lambda,\mu,n}'\cdot\frac{p\sqrt{\gamma}}{\log(1+\frac{n\gamma}{\lambda}p)} \Bigg\}\notag\\
    \mathcal{W}_p(\mathcal{L}(\tilde{q}_{\Sigma}), \mathcal{L}(Z)) &\geq \max \Bigg\{\frac{C_{4,\lambda,\mu}}{\sqrt{\lambda}} \exp\left( -C_{5,\lambda,\mu} \frac{1}{\gamma p} \right) \frac{p\sqrt{\gamma}}{\log(1+\gamma p/\lambda)} - C_{6,\lambda,\mu}'\sqrt{p} \notag\\
    &\quad, C_{\text{OT}, \lambda, \mu}' \cdot \sqrt{\gamma} \exp\left(-\frac{\lambda }{p\gamma}\right) \Bigg\} \notag
\end{align}
And we can use same analysis on dominating terms as in SSQ to show that
\begin{subnumcases}{\mathcal{W}_p(\mathcal{L}(\tilde{q}_{\Sigma}), \mathcal{L}(Z)) \leq}
    E_{1,\lambda,\mu,n} \cdot p\sqrt{\gamma}, & if $p \in [1, 1/\gamma^{1/2-\alpha-\epsilon}]$ \\
    E_{2,\lambda,\mu,n} \cdot \sqrt{p} , & if $p \in [1/\gamma^{1/2-\alpha}, 1/\gamma)$ \\
    E_{3,\lambda,\mu,n} \cdot \frac{p\sqrt{\gamma}}{\log(1+\frac{n\gamma}{\lambda}p)}, & if $p \in [1/\gamma, \infty)$ 
\end{subnumcases}
\begin{subnumcases}{\mathcal{W}_p(\mathcal{L}(\tilde{q}_{\Sigma}), \mathcal{L}(Z)) \geq}
        C_{OT,\lambda,\mu}' \cdot \sqrt{\gamma} \exp\left(-\frac{\lambda }{p\gamma}\right), & $p \in [1, D_{5,\lambda,\mu}/\gamma]$\label{eq: OT Wasserstein-$p$ lower bound p's function, JSQ, 2}\\
        D_{4,\lambda,\mu}' \cdot \frac{p\sqrt{\gamma}}{\log(1+\gamma p/\lambda)} & $p \in [D_{5,\lambda,\mu}/\gamma, \infty)$\label{eq: Triangle Wasserstein-$p$ lower bound p's function, JSQ, 2}
    \end{subnumcases}
with constants defined as 
\begin{align*}
    E_{k,\lambda,\mu,n} &:= A_{1,\lambda,\mu,n} + A_{1,\lambda,\mu,n}'(\frac{2n\lambda}{C^2})^A \Gamma(A+1) + A_{2,\lambda,\mu,n} + A_{2,\lambda,\mu,n}' \; \text{for} \; k=1,2,3
\end{align*} and $D'_{4,\lambda,\mu}$ defined in \eqref{eq: D_4, lambda, mu SSQ}.
Here is a list of constants used in the above bounds, collected from constants in \eqref{eq: constants for SSQ Wasserstein-$p$ bounds}.
\begin{align}
    A_{1,\lambda,\mu,n} &:= \sqrt{2}\bigg(((e+8e^2+\sqrt{2}e^2)\sqrt{2\pi}/\sqrt{\lambda}) (C_{\lambda,\mu,n}'' + C_{\lambda,\mu,n}'\frac{\lambda}{\sqrt{2}}) + \frac{4e\sqrt{2\pi}+2e^2}{\sqrt{\lambda}} \bigg) \notag\\ 
    A_{1,\lambda,\mu,n}' &:= \sqrt{2}(1+2e\sqrt{\pi}+ 8\sqrt{2}e^2 + 2e^2)\frac{\mu}{\sqrt{\lambda}} \notag\\
    A_{2,\lambda,\mu,n} &:= (C_{\lambda,\mu,n}''\frac{\lambda}{\sqrt{2}}  + \frac{4e\sqrt{2\pi}}{\sqrt{\lambda}}), \quad  A_{2,\lambda,\mu,n}' := C_{\lambda,\mu,n}' \notag\\
    A_{2,\lambda,\mu,n}''&:= \bigg(n(A_{\lambda,\mu,n} +\frac{1}{n})+\frac{1}{\sqrt{2\lambda}}\bigg) \notag\\
    C_{4,\lambda,\mu} &:= (\frac{29}{69}*2) \cdot \exp(-9(15+A_{\lambda,\mu})) \notag\\
    C_{6,\lambda,\mu}' &:= 4e\sqrt{2\pi}, 
    C_{\text{OT}, \lambda, \mu}' := \frac{\sqrt{2}[2C_{\lambda,\mu}' e^{2\sqrt{2}} \sqrt{\lambda}]^{-1}}{\sqrt{\lambda-\ln \lambda -1 }} \label{eq: constants for Wasserstein-$p$ JSQ}
\end{align}


\subsection{Proof for Theorem \ref{thm: JSQ Tail bound}: Tail Bound for JSQ} \label{sec: Proof for Theorem JSQ Tail Bound}

Recalling the upper bound on Wasserstein-$p$ distance as followed:
\begin{align*}
    \mathcal{W}_p\left (\mathcal{L} \big( \langle \tilde{\mathbf{q}}, \boldsymbol{\phi}\rangle\big),\mathcal{L}\big(\langle \boldsymbol{\phi}, Z\cdot\boldsymbol{1} \rangle\big)\right ) &\leq \mathbb{E}[\|\mathbf{q}_{\perp}\|^p]^{1/p} + W_p(\mathcal{L}(\tilde{q}_{\Sigma}),\mathcal{L}(Z)) 
\end{align*}
We can now combine Theorem \ref{thm: Wasserstein-$p$ distance JSQ} and Theorem \ref{thm: SSC JSQ} to get the final upper bound for the Wasserstein-$p$ distance given any $\epsilon \in (0,1/2-\alpha)$ fixed,
\begin{align}
   \mathcal{W}_p\left (\mathcal{L} \big( \langle \tilde{\mathbf{q}}, \boldsymbol{\phi}\rangle\big),\mathcal{L}\big(\langle \boldsymbol{\phi}, Z\cdot\boldsymbol{1} \rangle\big)\right ) &\leq \begin{cases}
        B_{1,\lambda,\mu,n}\cdot p^2\sqrt{\gamma}, & 1 \leq p \leq (\frac{1}{\gamma})^{\min\{1/3,1/2-\alpha-\epsilon\}} \\
        B_{2,\lambda,\mu,n}\cdot \max\{\sqrt{p}, p^2\sqrt{\gamma}\}, & (\frac{1}{\gamma})^{\min\{1/3,1/2-\alpha\}} \leq p  \\
    \end{cases} \label{eq: Wasserstein-$p$ upper bound JSQ, final}
\end{align}
We introduce the following notations for above constants $B_{1,\lambda,\mu,n}, B_{2,\lambda,\mu,n}$, we also denote $\iota$ for later use,
\begin{align}
    B_{1,\lambda,\mu,n} &:= A_{1,\lambda,\mu,n} + A_{1,\lambda,\mu,n}'(\frac{2n\lambda}{C^2})^A \Gamma(A+1) + E_{\lambda,\mu,n}/\sqrt{\lambda} \label{eq: B1, B2 for JSQ tail bound}\\
    B_{2,\lambda,\mu,n} &:= A_{2,\lambda,\mu,n} + E_{\lambda,\mu,n}/\sqrt{\lambda}, \quad \iota := \min\{1/3,1/2-\alpha-\epsilon\}. \notag
\end{align}
Here $A:=1+\frac{1+2\alpha}{2\epsilon}$. And we recall that $A_{1,\lambda,\mu,n}, A_{1,\lambda,\mu,n}', A_{2,\lambda,\mu,n}$ are defined in \eqref{eq: constants for Wasserstein-$p$ JSQ}, and $E_{\lambda,\mu,n}$ is defined in \eqref{eq: p-moment for q_perp, state space collapse}.
The above analysis on dominating term is obtained by similar argument of the fact polynomial function is upper bounded by exponential function in \eqref{eq: Exp Dominating Polynomial Stirling Bound}.
 




The proof for tail bound in Theorem \ref{thm: JSQ Tail bound} follows the same concentration argument in\eqref{eq: Tail bound from Wasserstein-$p$ distance} as in SSQ. We repeatedly choose different $(p,\rho)$ pairs , verify their legitimacy and check which regime of Wasserstein-$p$ they fall into according to \eqref{eq: Wasserstein-$p$ upper bound JSQ, final}. Then we obtain tail bound according to different $(p,\rho)$ pairs.  We first slightly modify the concentration argument for Inequality \eqref{eq: Tail bound from Wasserstein-$p$ distance}
to fit in the high dimensional setting. Specifically, we have an arbitrary unit vector $\boldsymbol{\phi} \in \mathbb{R}^n$, and we consider the projection $\langle \tilde{\mathbf{q}}, \boldsymbol{\phi}\rangle$ and $\langle \boldsymbol{\phi}, Z\cdot\boldsymbol{1}\rangle$. We first study the case when $\langle\boldsymbol{\phi}, \boldsymbol{1}\rangle \ne 0$. For any $\rho \in [0,1]$ and $a>0$, we have

\begin{align*}
    |\mathbb{P}(\langle \tilde{\mathbf{q}}, \boldsymbol{\phi}\rangle > a) - \mathbb{P}(\langle \boldsymbol{\phi}, Z\cdot\boldsymbol{1}\rangle > a) |\leq \frac{(1-\rho)a}{|\langle\boldsymbol{\phi}, \boldsymbol{1}\rangle|}\phi(\frac{\rho a}{|\langle\boldsymbol{\phi}, \boldsymbol{1}\rangle|}) + \frac{\mathcal{W}_p(\mathcal{L}(\langle \tilde{\mathbf{q}}, \boldsymbol{\phi}\rangle), \mathcal{L}(\langle \boldsymbol{\phi}, Z\cdot\boldsymbol{1}\rangle)^p}{(1-\rho)^p a^p}
\end{align*}
We will use this inequality which connects tail boudns with Wasserstein-$p$ distance
to show tail bounds for different regimes of deviation $a$. In the following, we will choose $\rho:= 1 - e\mathcal{W}_p/a$ and $p$ based on different regimes of deviation $a$, the choices of $p$ are summarized as below,
\begin{align*}
    p:= \begin{cases}
        a^2/(2\langle\boldsymbol{\phi}, \boldsymbol{1}\rangle^2) + \ln (1/\sqrt{\gamma}), & a\in[2, c_{\lambda,\mu,n}/\gamma^\beta], \beta\leq \min\{1/4-\alpha/2,1/6\} \\
        c a^{\min\{1/(4\beta)+1/2, 2\}}, & a = 1/\gamma^\beta, \delta\in[\min\{1/4-\alpha/2,1/6\}, \infty) 
    \end{cases}
\end{align*}

First we assume $\langle\boldsymbol{\phi}, \boldsymbol{1}\rangle \ne 0$, and we emphasize that when proving the following tail bounds, we will assume $\gamma$ is sufficiently small such that 
\begin{align}
    \gamma \leq \gamma_2 := \min\{ \gamma_1,  \min\{
        \langle\boldsymbol{\phi}, \boldsymbol{1}\rangle^2, \frac{1}{2} (\frac{\langle\boldsymbol{\phi}, \boldsymbol{1}\rangle^4}{eB_{1,\lambda,\mu,n}})^{1/3}\}^{2/\iota} 
        \} \label{eq: gamma assumption for JSQ tail bound}
\end{align}
with $\gamma_1$ defined in Theorem \ref{thm: Wasserstein-$p$ distance JSQ}. 

\subsubsection{
    Proof for Constant Deviation Regime in Theorem \ref{thm: JSQ Tail bound, first regime}
} \label{subsub: constant deviation JSQ}
In this regime, we consider any constant deviation $a$ and we assume $\gamma\leq \gamma_a$ defined as
\begin{align}
    \gamma\leq \gamma_a = \min\{\gamma_2,
        \frac{
        \langle\boldsymbol{\phi}, \boldsymbol{1}\rangle^8
    }{e^2 B_{1,\lambda,\mu,n}^2 a^6},  (\langle\boldsymbol{\phi}, \boldsymbol{1}\rangle^2/a^2)^{1/\iota},  \left(\frac{a}{64eB_{1,\lambda,\mu,n}}\right)^4, \frac{4\langle\boldsymbol{\phi}, \boldsymbol{1}\rangle^{12}}{e^2 B_{1,\lambda,\mu,n}^2 a^{10}}, \frac{\langle\boldsymbol{\phi}, \boldsymbol{1}\rangle^8}{2^{20} e^4 B_{1,\lambda,\mu,n}^4 a^4}
    \}
     \label{eq: gamma_a definition JSQ first regime}
\end{align}

We first choose the following $(p,\rho)$ pair as
\begin{align*}
    p:= \frac{a^2}{2\langle\boldsymbol{\phi}, \boldsymbol{1}\rangle^2} + \ln (1/\sqrt{\gamma}), \quad \rho := 1- e\mathcal{W}_p/a
\end{align*}
We verify that this $(p,\rho)$ pair is valid, i.e., $\rho \in [0,1]$ and $p> 1$,
 and falls into the first regime of Wasserstein-$p$ distance upper bound \eqref{eq: Wasserstein-$p$ upper bound JSQ, final}.
\begin{align*}
    p &= \frac{a^2}{2\langle\boldsymbol{\phi}, \boldsymbol{1}\rangle^2} + \ln (1/\sqrt{\gamma}) \geq 1 \\
    p &\overset{(a)}{\leq} a^2/(2\langle\boldsymbol{\phi}, \boldsymbol{1}\rangle^2) + \frac{1}{\iota} (\frac{1}{\gamma})^{\iota/2} \overset{(b)}{\leq} a^2/(2\langle\boldsymbol{\phi}, \boldsymbol{1}\rangle^2) + \frac{1}{2} (\frac{1}{\gamma})^{\iota} \overset{(c)}{\leq} (\frac{1}{\gamma})^{\iota}.
\end{align*}
Here inequality $(a)$ is from the fact logarithm function is upper bounded by any polynomial function.
Inequality $(b)$ is from $\gamma \leq (\iota/2)^{1/(\iota/2)}$, inequality $(c)$ is from $\gamma \leq (\langle\boldsymbol{\phi}, \boldsymbol{1}\rangle^2/a^2)^{1/\iota}$. Therefore we are in the first regime for Wasserstein-$p$ distance\eqref{eq: Wasserstein-$p$ upper bound JSQ, final}. So we plug in the upper bound on Wasserstein-$p$ distance in \eqref{eq: Wasserstein-$p$ upper bound JSQ, final} into the definition of $\rho$. We have
\begin{align*}
    1> \rho &:= 1 - eW_p/a \geq 1 - \frac{eB_{1,\lambda,\mu,n} p^2 \sqrt{\gamma}}{a} \\
    &\geq 1 - eB_{1,\lambda,\mu,n} (a^2/(2\langle\boldsymbol{\phi}, \boldsymbol{1}\rangle^2) + \ln(1/\sqrt{\gamma}))^2 \sqrt{\gamma}/a \\
    &\overset{(a)}{\geq} 1 - 2eB_{1,\lambda,\mu,n} (a^4/(4\langle\boldsymbol{\phi}, \boldsymbol{1}\rangle^4) + \ln^2(1/\sqrt{\gamma}) \sqrt{\gamma}/a \\
    &\overset{(b)}{\geq} 1 - (a/2 + a/2)/a = 0
\end{align*}
Inequality $(a)$ is AMGM.
Inequality $(b)$ is from our assumption that $\gamma$ is sufficiently small
\begin{align*}
    \frac{1}{4\langle\boldsymbol{\phi}, \boldsymbol{1}\rangle^4} a^3\sqrt{\gamma} &\leq \frac{1}{4eB_{1,\lambda,\mu,n}} \impliedby \gamma \leq \frac{
        \langle\boldsymbol{\phi}, \boldsymbol{1}\rangle^8
    }{e^2 B_{1,\lambda,\mu,n}^2 a^6}\\
    \sqrt{\gamma}\ln^2(1/\sqrt{\gamma}) &\leq \frac{a}{4eB_{1,\lambda,\mu,n}} \impliedby \gamma \leq \left(\frac{a}{64eB_{1,\lambda,\mu,n}}\right)^4
\end{align*}
Meanwhile, we also verify the following property for $(1-\rho)a^2/\langle \boldsymbol{\phi}, \boldsymbol{1}\rangle^2$. This property allows us to extract the exact $\frac{a^2}{\langle \boldsymbol{\phi}, \boldsymbol{1}\rangle^2}$ term in exponent.
\begin{align*}
    (1-\rho)a^2/\langle \boldsymbol{\phi}, \boldsymbol{1}\rangle^2 &\leq \frac{2eB_{1,\lambda,\mu,n}}{\langle \boldsymbol{\phi}, \boldsymbol{1}\rangle^2} (a^2/2 + \ln(1/\sqrt{\gamma}))^2 \sqrt{\gamma} a \\
    &= \frac{2eB_{1,\lambda,\mu,n}}{4\langle \boldsymbol{\phi}, \boldsymbol{1}\rangle^6}a^5 \sqrt{\gamma} + \frac{2eB_{1,\lambda,\mu,n}}{\langle \boldsymbol{\phi}, \boldsymbol{1}\rangle^2} a \sqrt{\gamma} \ln^2(1/\sqrt{\gamma}) \overset{(a)}{\leq} 1 + 1 = 2
\end{align*}
Here $(a)$ is from the same assumption as above: 
\begin{align*}
    \frac{2eB_{1,\lambda,\mu,n}}{4\langle \boldsymbol{\phi}, \boldsymbol{1}\rangle^6}a^5 \sqrt{\gamma} &\leq 1 \impliedby \gamma \leq \frac{4\langle\boldsymbol{\phi}, \boldsymbol{1}\rangle^{12}}{e^2 B_{1,\lambda,\mu,n}^2 a^{10}}\\
    \frac{2eB_{1,\lambda,\mu,n}}{\langle \boldsymbol{\phi}, \boldsymbol{1}\rangle^2} a \sqrt{\gamma} \ln^2(1/\sqrt{\gamma}) &\leq 1 \impliedby \gamma \leq \frac{\langle\boldsymbol{\phi}, \boldsymbol{1}\rangle^8}{2^{20} e^4 B_{1,\lambda,\mu,n}^4 a^4}
\end{align*}
The above conditions on $(p,\rho)$ enable us to use the concentration argument and derivation from \eqref{eq: Tail bound, first regime} to obtain: 
\begin{align*}
    |\mathbb{P}( \tilde{q} > a) - \mathbb{P}(Z > a)| &\leq \frac{eB_{1,\lambda,\mu,n}}{|\langle \boldsymbol{\phi}, \boldsymbol{1}\rangle|} \sqrt{\gamma} (a^2/(2\langle \boldsymbol{\phi},\boldsymbol{1}\rangle) + \ln(1/\sqrt{\gamma}))^2 \phi(\rho a) + \sqrt{\gamma} e^{-a^2/2} \\
    & \leq \frac{e^3B_{1,\lambda,\mu,n} }{\sqrt{2\pi}|\langle \boldsymbol{\phi},\boldsymbol{1}\rangle|} \cdot\sqrt{\gamma}\big((a^2/(2\langle \boldsymbol{\phi},\boldsymbol{1}\rangle^2) + \ln(1/\sqrt{\gamma}))^2 \\
            &\quad +\sqrt{2\pi}\big) \exp(-a^2/(2\langle \boldsymbol{\phi},\boldsymbol{1}\rangle^2)) 
\end{align*}

\subsubsection{
    Proof for Moderate Deviation Regime in Theorem \ref{thm: JSQ Tail bound, second regime}
} \label{subsub: MDP JSQ}
We study the regime where deviation $a$ satisfies the below condition
\begin{align}
    2 \leq a 
    \leq \underbrace{\min\{2\langle\boldsymbol{\phi}, \boldsymbol{1}\rangle^2, (\frac{\langle\boldsymbol{\phi}, \boldsymbol{1}\rangle^4}{eB_{1,\lambda,\mu,n}})^{1/3} \}}_{E_{Tail, 1, u}}\cdot (\frac{1}{\gamma})^{\iota/2}
     \label{eq: a's regime for MDP JSQ}
\end{align}
with $B_{1,\lambda,\mu,n}, \iota$ defined in \eqref{eq: B1, B2 for JSQ tail bound}. We note that the above range $\iota/2$ for the exponent of $\gamma$ is equivalent to $\min\{1/4-\alpha/2, 1/6\}$ which is the range of $\delta$ in the definition of moderate deviation regime in Theorem \ref{thm: JSQ Tail bound, second regime}. We first verify that the above interval for $a$ is non-empty using the assumption on $\gamma$ in \eqref{eq: gamma assumption for JSQ tail bound}.  
\begin{align*}
    \gamma &\leq \gamma_2 \leq \min\{
        \langle\boldsymbol{\phi}, \boldsymbol{1}\rangle^2, \frac{1}{2} (\frac{\langle\boldsymbol{\phi}, \boldsymbol{1}\rangle^4}{eB_{1,\lambda,\mu,n}})^{1/3}\}^{2/\iota} 
\end{align*}
 First, we choose the same $(p,\rho)$ pair as in the constant deviation regime: $p:= a^2/(2\langle\boldsymbol{\phi}, \boldsymbol{1}\rangle^2) + \ln(1/\sqrt{\gamma})$, $\rho := 1 - e\mathcal{W}_p/a$.
We verify that $\rho\in[0,1]$ and $p>1$, and that we are in the first regime of Wasserstein-$p$ distance upper bound \eqref{eq: Wasserstein-$p$ upper bound JSQ, final} as follows,
\begin{align*}
    1\leq p &:= a^2/(2\langle\boldsymbol{\phi}, \boldsymbol{1}\rangle^2) + \ln(1/\sqrt{\gamma}) \overset{(a)}{\leq} (\frac{1}{\gamma})^{\iota}\\
    &\impliedby \gamma \leq (\iota/2)^{2/\iota}\; \& \;a \leq \langle\boldsymbol{\phi}, \boldsymbol{1}\rangle^2 \cdot 2(\frac{1}{\gamma})^{\iota/2}\\
    1 > \rho &:= 1 - \frac{eW_p}{a} \leq 1 - \frac{eB_{1,\lambda,\mu,n}}{a} (a^2/(2\langle\boldsymbol{\phi}, \boldsymbol{1}\rangle^2) + \ln(1/\sqrt{\gamma})^2)^2 \sqrt{\gamma} \\
    &\geq 1 - eB_{1,\lambda,\mu,n}(\frac{1}{2\langle\boldsymbol{\phi}, \boldsymbol{1}\rangle^4}a^4\sqrt{\gamma} + \sqrt{\gamma}\ln^2(1/\sqrt{\gamma}))/a \\
    &\overset{(b)}{\geq} 1 - (\frac{1}{2\langle\boldsymbol{\phi}, \boldsymbol{1}\rangle^4}eB_{1,\lambda,\mu,n} a^4\sqrt{\gamma} + a/2)/a \overset{(c)}{\geq} 1 - 1 =0
\end{align*}
Here $(a)$ is from the assumption on deviation regime in \eqref{eq: a's regime for MDP JSQ} and assumption $\gamma \leq (\iota/2)^{2/\iota}$ with the fact logarithm function is upper bounded by any polynomial function.
Inequalities $(b)$, $(c)$ are from the assumption of this deviation regime and assumption on $\gamma\leq \gamma_2$ in \eqref{eq: gamma assumption for JSQ tail bound}.
\begin{align*}
    \frac{1}{2\langle\boldsymbol{\phi}, \boldsymbol{1}\rangle^4}eB_{1,\lambda,\mu,n} a^4\sqrt{\gamma} &\leq \frac{1}{2}a, \impliedby a\leq (\frac{\langle\boldsymbol{\phi}, \boldsymbol{1}\rangle^4}{eB_{1,\lambda,\mu,n}})^{1/3} \cdot (\frac{1}{\gamma})^{1/6}\\
    eB_{1,\lambda,\mu,n}\cdot\sqrt{\gamma}\ln^2(1/\sqrt{\gamma}) &\leq a/2, \impliedby a \geq 2, \;\&\; \gamma \leq \frac{1}{2^{16} e^4 B_{1,\lambda,\mu,n}^4}
\end{align*}
Witht the above validity of $(p,\rho)$, we derive tail bound using concentration argument:
\begin{align*}
            \left|\mathbb{P}(\langle \tilde{\mathbf{q}},\boldsymbol{\phi}\rangle > a) - \mathbb{P}(Z\cdot\langle\boldsymbol{\phi},\boldsymbol{1}\rangle > a)\right| &\leq \frac{eB_{1,\lambda,\mu,n}}{\sqrt{2\pi}|\langle\boldsymbol{\phi},\boldsymbol{1}\rangle|}\cdot \sqrt{\gamma} \big((a^2/(2\langle \boldsymbol{\phi},\boldsymbol{1}\rangle^2) + \ln(1/\sqrt{\gamma}))^2+\sqrt{2\pi}\big)\\
    &\exp(-\frac{a^2}{2(\langle \boldsymbol{\phi},\boldsymbol{1}\rangle^2)}(1 - eB_{1,\lambda,\mu,n}(\frac{a^4}{2\langle \boldsymbol{\phi},\boldsymbol{1}\rangle^4}\sqrt{\gamma} + 2\sqrt{\gamma}\ln^2(1/\sqrt{\gamma})/a)) 
\end{align*}

\paragraph{Large Deviation} In this setting, deviation $a$ is in the following regime:
\begin{align*}
    a = (\frac{1}{\gamma})^{\beta}, \beta \in [\min\{1/4-\alpha/2,1/6\}, \infty)
\end{align*}
We will utilize the pair $(p,\rho)$ defined as $p:= ca^{\min\{1/(4\beta)+1/2, 2 \}}$, $\rho := 1 - eW_p/a$. Here $c$ is some constant defined as
\begin{align}
    c &:= \min\{ (1/(2eB_{2,\lambda,\mu,n}))^{1/2}, (1/(2eB_{2,\lambda,\mu,n}))^{2} \} \label{eq: constant c in subweibull regime, JSQ}
\end{align}
Again, we first verify the conditions for $\rho \in[0,1]$ and $p>1$, and that we are in the second regime of Wasserstein-$p$ distance upper bound \eqref{eq: Wasserstein-$p$ upper bound JSQ, final} according to our choice of $p$. We now verify these conditions
\begin{align*}
    p &:= ca^{\min\{1/(4\beta)+1/2, 2 \}} \geq 1 \\
    p &\leq c a^{\min\{1/(4\beta)+1/2, 2 \}} \leq (1/\gamma)^{\iota} \\
    \rho &:= 1 - \frac{eW_p}{a} < 1 \\
    \rho &\geq 1 - \frac{eB_{2,\lambda,\mu,n}}{a} \max\{\sqrt{p}, p^2\sqrt{\gamma}\} \\
    &= \min \{ 1 - eB_{2,\lambda,\mu,n} c^{1/2} a^{\min\{1/(8\beta) + 1/4, 1\}}/a, 1 - eB_{2,\lambda,\mu,n} c^2 a^{\min\{1/(2\beta) + 1, 4\}}\sqrt{\gamma}/a \} \\
    &\overset{(a)}{\geq} \min\{ 1 - \frac{1}{2} a^{\min\{1/(8\beta) + 1/4, 1\}}/a, 1 - \frac{1}{2} a^{\min\{1/(2\beta) + 1, 4\}}\sqrt{\gamma}/a \} \overset{(b)}{\geq} 1/2
\end{align*}
Here $(a)$ and $(b)$ are from the assumption of this deviation regime, the choice of $c$ \eqref{eq: constant c in subweibull regime, JSQ} and $\gamma < 1$. With the above verification, we use the fact that $\rho \geq 1/2$ to derive:
\begin{align*}
    |\mathbb{P}( \tilde{q} > a) - \mathbb{P}(Z > a)| &\leq \frac{e}{2\sqrt{2\pi}} a\exp(-a^2/8) + \exp(-ca^{\min\{1/(4\beta)+1/2, 2\}})
\end{align*}

\paragraph{Large Deviation with Positive Direction} 


For the large deviation regime: $a \geq \frac{e^2n\sqrt{\lambda}}{\sqrt{\gamma}} $, when direction vector $\boldsymbol{\phi}$ is in positive orthant, i.e., $\boldsymbol{\phi} \geq 0$, we can use coupling argument that $\mathbf{q}_{JSQ} \leq_{st} \mathbf{q}_{M/M/\infty}$. Here we use $\mathbf{q}_{M/M/\infty}$ to denote the steady-state queue length vector, where each coordinate is a  steady-state queue length random variable for $M/M/\infty$ queue with arrival rate $\lambda/n$ and service rate $\gamma$. To make distinguishion, we use $\mathbf{q}_{JSQ}$ to denote the steady-state queue length vector under JSQ policy. 
The stochastic domination is coordinate-wise, thus for any non-negative direction vector $\boldsymbol{\phi}$, we have
\begin{align*}
    \langle \tilde{\mathbf{q}}_{JSQ}, \boldsymbol{\phi}\rangle &\leq_{st} \langle \tilde{\mathbf{q}}_{M/M/\infty}, \boldsymbol{\phi}\rangle
\end{align*}
Thus we have the following tail bound,
\begin{align*}
    \mathbb{P}(\langle \tilde{\mathbf{q}}_{JSQ}, \boldsymbol{\phi}\rangle> a)
    &= \mathbb{P}\left( \langle \mathbf{q}_{JSQ} - \frac{\lambda-\mu}{n\gamma}\boldsymbol{1}, \boldsymbol{\phi} \rangle > \frac{\sqrt{\lambda}}{n\sqrt{\gamma}} a \right) \\
    &= \mathbb{P}\left( \langle \mathbf{q}_{JSQ}, \frac{1}{\langle \boldsymbol{\phi}, \boldsymbol{1}\rangle} \boldsymbol{\phi} \rangle > \frac{1}{\langle \boldsymbol{\phi}, \boldsymbol{1}\rangle}\frac{\sqrt{\lambda}}{n\sqrt{\gamma}} a + \frac{\lambda-\mu}{n\gamma}\right)\\
    &\leq \mathbb{P}\left( \left\langle \mathbf{q}_{M/M/\infty}, \frac{1}{\langle \boldsymbol{\phi}, \boldsymbol{1}\rangle} \cdot \boldsymbol{\phi} \right\rangle \geq \frac{1}{\langle \boldsymbol{\phi}, \boldsymbol{1}\rangle}  \frac{\sqrt{\lambda}}{n\sqrt{\gamma}}a + \frac{\lambda-\mu}{n\gamma}  \right) \\
    &\overset{(a)}{\leq} \inf_{t\geq0} \exp\left( -tb + \sum_{i=1}^{n} \frac{\lambda}{\gamma} \left( e^{t\phi_i/\langle \boldsymbol{\phi}, \boldsymbol{1}\rangle} - 1 \right) \right) \\
    &\overset{(b)}{\leq} \inf_{t\geq0} \exp\left( -tb + \frac{\lambda}{\gamma} (e^{t} - 1) \right) \\
    &\overset{(c)}{=} \exp\left( -b\log\left( \frac{b\gamma}{\lambda} \right) + b - \frac{\lambda}{\gamma} \right) \\
    &\overset{(d)}{\leq} \exp\left( - \left(\frac{1}{\langle \boldsymbol{\phi}, \boldsymbol{1}\rangle} \frac{\sqrt{\lambda}}{2n\sqrt{\gamma}}a + \frac{\lambda-\mu}{n\gamma} \right) \log\left( \frac{1}{\langle \boldsymbol{\phi}, \boldsymbol{1}\rangle} \frac{\sqrt{\gamma}}{n\sqrt{\lambda}}a + \frac{\lambda-\mu}{n\sqrt{\lambda\gamma}}  \right) \right)
\end{align*}
$(a)$ is Markov Inequality and we change variable to $b:=\frac{1}{\langle \boldsymbol{\phi}, \boldsymbol{1}\rangle}  \frac{\sqrt{\lambda}}{n\sqrt{\gamma}}a + \frac{\lambda-\mu}{n\gamma}$ for simplicity. $(b)$ is from AMGM and definition of $\mathbb{\phi}$. $(c)$ is due to the fact $a=\Omega(1/\sqrt{\gamma})$, thus we can solve the infimum explicitly. $(d)$ is from the assumption of this deviation regime.



Now we study the case when $\langle\boldsymbol{\phi}, \boldsymbol{1}\rangle = 0$. In this case, the projection is completely in the perpendicular space of $\boldsymbol{1}$. The tail bound is now dominated by the sub-Weibull heavy tail of $\mathbf{q}_\perp$. 
\paragraph{Worst Case Direction}
For worst case direction $\boldsymbol{\phi}$, i.e., $\langle \boldsymbol{\phi}, \boldsymbol{1}\rangle = 0$, the projection is completely in the perpendicular space, i.e., $\langle\tilde{\mathbf{q}}, \boldsymbol{\phi}\rangle = \frac{n\sqrt{\gamma}}{\sqrt{\lambda}} \langle \mathbf{q}_\perp, \boldsymbol{\phi}\rangle=\pm\frac{n\sqrt{\gamma}}{\sqrt{\lambda}}\|\mathbf{q}_\perp\|$. If we assume deviation $a$ is in the following regime
\begin{align}
    a\geq \underbrace{\frac{e^2 n\sqrt{E_{\lambda,\mu,n}}}{\sqrt{\lambda}}}_{E_{Tail, 2, u}} \cdot\sqrt{\gamma}, \label{eq: a's regime for worst case direction JSQ}
\end{align} we can use the sub-Weibull norm bound on $\mathbf{q}_\perp$ to derive tail bound: 
\begin{align}
    \mathbb{P}(\langle\tilde{\mathbf{q}}, \boldsymbol{\phi}\rangle > a) &\leq \mathbb{P}(\|\mathbf{q}_\perp\| > \frac{\sqrt{\lambda}}{n\sqrt{\gamma}}a) \overset{(a)}{\leq} \sup_{p > 1} E_{\lambda,\mu,n}^p \frac{p^{2p}}{(a\sqrt{\lambda})^p/(n\sqrt{\gamma})^p} \notag \\
    & = \exp\left( -p\log(a\sqrt{\lambda}/(n\sqrt{\gamma})) + p\log E_{\lambda,\mu,n} + 2p\log p \right) \notag\\
    &\overset{(b)}{=} \exp\left( -\frac{2\lambda^{1/4}}{e\sqrt{n E_{\lambda,\mu,n}}} \cdot \frac{a^{1/2}}{\gamma^{1/4}} \right) := \exp(-G_{\lambda,\mu,n} \cdot a^{1/2}/\gamma^{1/4}) \label{eq: tail bound for worst case direction}
\end{align}
$(a)$ is Markov inequality and optimizing over $p$, we have $p^* = \frac{\sqrt{a\sqrt{\lambda}/(n\sqrt{\gamma})}}{e\sqrt{E_{\lambda,\mu,n}}}$ once $p^* \geq 1$, yielding $(b)$. $p^* \geq 1$ is equivalent to $a \geq \frac{e^2 n\sqrt{\gamma E_{\lambda,\mu,n}}}{\sqrt{\lambda}}$ in \eqref{eq: a's regime for worst case direction JSQ}.

\subsection{Coupling Argument} \label{sec: coupling argument}
In this section, we provide all the coupling argument for all above sections. 
\paragraph{JSQ vs $M/M/\infty$}

We path-wise couple the join-the-shortest system with the system where all server's service rate is 0 (degenerating into $M/M/\infty$). Arrivals for both systems are generated from the same Poisson process with rate $\lambda$. Each customer in both systems has an independent exponential patience time $P_k$ with rate $\gamma$, starting from its arrival time $A_k$. JSQ has additional service process, which is an independent Poisson process with rate $\mu$ and each event is labeled as service from queue $i$ with probability $\frac{\mu_i}{\mu}$. As convention, if the chosen queue is empty, the service event is wasted.

The following induction over event epochs shows the pathwise domination of total queue length in $M/M/\infty$ over that in JSQ. Let $L(t)$ and $L'(t)$ be the label sets of customers at time $t$ in JSQ and $M/M/\infty$ respectively. We want to show $L(t) \subseteq L'(t)$ for all $t$. At time 0, both systems are empty, so the claim holds. Assume the claim holds before event time $t$, we consider the following cases:
\begin{itemize}
    \item Arrival event: both systems insert the new customer $k$ into $L(t)$ and $L'(t)$, so the inclusion is preserved.
    \item Abandonment event: if $t=P_k$, i.e., customer $k$'s patience time is up. Then $k \in L(t)$ implies $k \in L'(t)$ by induction hypothesis, so either both systems remove $k$ from their label sets, or only $M/M/\infty$ removes $k$ if $k \notin L(t)$, and inclusion is preserved.
    \item Service event: if $t$ is a service event from queue $i$, then $L(t)$ can only becomes smaller, while $L'(t)$ remains unchanged, so inclusion is preserved.
\end{itemize}
Thus for all $t \geq 0$, $\sum_{i=1}^{n} q_{JSQ,i}(t) \leq \sum_{i=1}^{n} q_{M/M/\infty,i}(t)$ almost surely. From positive recurrence of both systems, taking limit of time $t \to \infty$, the stochastic domination is preserved for steady state distribution, as from Theorem 6.B.16(d) in \cite{shaked2007stochastic}. 

\paragraph{JSQ vs SSQ}

Similarly, we path-wise couple the JSQ system with a single pooled SSQ with arrival rate $\lambda$, service rate $\mu$ and abandonment rate $\gamma$. For both systems, the arrival, service, and abandonment events are generated from the same Poisson processes as above. Notice that in SSQ, there is no splitting of service events. We want to show that in steady state, the total queue length in JSQ stochastically dominates the total queue length in SSQ. We will use the following non-anticipated coupling and argue for all $t \geq 0$, the label set $L''(t)$ in SSQ is contained in the label set $L(t)$ in JSQ. We first discuss the above three types of events and the case for arrival and abandonment are in similar manner so we omit the details. For service event, if customer $k$ is served in JSQ, it must also be contained in $L''(t)$. We let $k$ be served in SSQ as well, which preserves the law of queue length in SSQ and the inclusion $L''(t) \subseteq L(t)$. This stochastic domination is preserved for steady state distribution as from Theorem 6.B.16(d) in \cite{shaked2007stochastic}.

\end{document}